%% file: S-HopfConstruction-clean.tex
\documentclass{amsart}

\def\titlepaper{How to build a Hopf algebra}
\def\authorpaper{Theo Johnson-Freyd \& David Reutter}

\include{preamble}

\title{How to build a Hopf algebra}

\author{Theo Johnson-Freyd}
\address{
Perimeter Institute for Theoretical Physics}
\email{theojf@pitp.ca}
\urladdr{https://categorified.net/}

\author{David Reutter}
\address{Universit\"at Hamburg, Fachbereich Mathematik}
\email{david.reutter@uni-hamburg.de}
\urladdr{https://www.davidreutter.com/}

\begin{document}

\begin{abstract}
We construct a functor that inputs a retract in an $(\infty,3)$-category  satisfying some adjunctibility conditions and outputs a Hopf algebra in a braided monoidal $(\infty,1)$-category. Provided the braided monoidal category is presentable, any Hopf algebra can be obtained in this way. Our functor specializes to --- and provides a higher-categorical explanation for --- the Tannakian reconstruction of a Hopf algebra from a monoidal category with duals and a fiber functor. 

Towards this end, we review and develop the lax (aka Gray) tensor product $\otimes$ of $(\infty,\infty)$-categories, and we analyze the ``lax smash product'' $\owedge$ of pointed $(\infty,\infty)$-categories. We compute the $\owedge$-square of the ``walking adjunction'' and show that its $3$-localization corepresents retracts with some adjunctibilty conditions, whereas the $3$-localization of the $\owedge$-square of the ``walking monad'' corepresents  bialgebras. In these terms, our functor is restriction along the $\owedge$-square of the inclusion $\{\text{walking monad}\} \to \{\text{walking adjunction}\}$. After $3$-localization, we show that this restriction inverts a certain shear map,  proving the existence of an antipode. 

We discuss generalizations of this construction  to Hopf monads, analyze additional adjunctibility conditions and their interplay with integrals and cointegrals, and finally explain how variants of classical Tannakian reconstruction fit into our scheme. 
\end{abstract}

\maketitle

\tableofcontents

\section{Introduction}\label{sec:intro}

This article describes a very general form of \define{Tannakian reconstruction}: we will explain how to (re)construct a Hopf algebra given some basic ``fibre functor'' data. Before stating our main result, let us indicate how our construction generalizes the reconstruction of a Hopf algebra from a linear monoidal category $\cA$ and a linear monoidal functor $G:\cA \to \Vect$; this is described in detail in \S\ref{subsec:tannakian}.
Choose a convenient\footnote{For this exposition, we will be indifferent about exactly what ``linear category'' and ``linear functor'' means; the reader should choose their favourite setting in which to do categorified linear algebra, see \S\ref{subsec:tannakian} for one such choice.}
 $3$-category $ \cat{3Vec}$ of linear $2$-categories, and write $ \cat{2Vec} \in \cat{3Vec}$ for the $2$-category of linear categories. 
Our linear monoidal category $\cA$ has a $2$-category  $\cat{Mod}(\cA)$ of module categories equipped with a functor $f : \cat{2Vec} \to \cat{Mod}(\cA)$ selecting the regular module $\cA_{\cA}$. Base-changing along the fiber functor $G: \cA \to \Vect$ induces a  functor  $g: \cat{Mod}(\cA) \to \cat{2Vec}$ so that $gf \simeq \id_{\cat{2Vec}}$. 
Under  various finiteness conditions on $\cA$ and $G$ (such as the existence of duals in $\cA$), classical Tannakian reconstruction recovers   a Hopf algebra in $\Vec$. In \S\ref{subsec:tannakian} we explain how these finiteness conditions can be directly encoded in terms of adjunctibility conditions on the retract $(f,g,gf \simeq \id)$ in $\cat{3Vec}$. From this perspective, our main theorem is a far-reaching generalization of Tannakian reconstruction:  starting with \emph{any} $(\infty,3)$-category and \emph{any} retract therein satisfying certain adjunctibility conditions, we $\infty$-categorically-coherently construct a Hopf algebra.

\begin{maintheorem}[Lemma~\ref{lem:unpacked} and Corollary~\ref{cor:anyHopf}] \label{thm:mainHopftheorem} 
  Let $\cC$ be an $(\infty,3)$-category with  a chosen object $1_\cC \in \cC$.
  An \emph{adjunctible retract} in $\cC$ is a pair of $1$-morphisms $1_{\cC} \to[f]  X \to[g] 1_{\cC}$ together with a $2$-isomorphism $\alpha: \id_{1_{\cC}}\Isom gf$ satisfying the following adjunctibility conditions:
  \begin{enumerate}
  \item[{[i]}] $f$ is a left adjoint and $g$ is a right adjoint;
  \item[{[ii]}] the canonical $2$-morphism $f^R \To g$ (built from $\alpha$ and the adjunction data) is a right adjoint; equivalently, the canonical $2$-morphism $g^L \To f$ is a left adjoint.
  \end{enumerate}
  Consider the following composite in the $\EE_2$-monoidal $(\infty,1)$-category  $\Omega^2\cC := \End_{\End_\cC(1_\cC)}(\id_{1_\cC})$ of $2$-endo\-mor\-phisms of the basepoint $1_\cC \in \cC $: 
 \begin{equation}\label{eqn:formulaforH}
\begin{tikzcd}[sep=5em]
1_\cC\arrow[r, equal] \arrow[d, equal]& 1_\cC \arrow[r, equal] \arrow[d, "f" description]& 1_\cC \arrow[d, equal] \\
1_\cC  \arrow[ur, Rightarrow, shorten <=7pt, shorten >=7pt, "\alpha^{\lmate}" description]\arrow[r, "g^L" description] \arrow[d, equal] & X  \arrow[ur, Rightarrow, shorten <=7pt, shorten >=7pt,
"\sim" sloped, "\alpha^{-1}" description] \arrow[r, "g" description] \arrow[d, "f^R"description] & 1_\cC\arrow[d, equal]\\
1_\cC \arrow[ur, Rightarrow, shorten <=5pt, shorten >=5pt, "\tiny\substack{((\alpha^{\rmate})^L)^{\lmate} \\ \rotatebox[origin=c]{270}{\ensuremath\simeq} \\ ((\alpha^{\lmate})^R)^{\rmate}}" description]  \arrow[r, equal] & 1_\cC \arrow[ur, Rightarrow, shorten <=7pt, shorten >=7pt, "\alpha^{\rmate}" description] \arrow[r, equal] & 1_\cC
\end{tikzcd}\qquad \in \Omega^2\cC.
\end{equation}
Here, $(-)^L, (-)^R$ denote the left and right adjoints of a given $k$-morphism, and $(-)^\lmate,(-)^\rmate$ denote the left and right mates of a square, defined as compositions with certain adjunction data of the sides of the square, as detailed in \S\ref{subsec:mates}.

Then the composite \eqref{eqn:formulaforH} carries the structure of a coHopf\footnote{The category $\coAlg(\cB)$ of coalgebras in an $\EE_2 \simeq \EE_1 \otimes \EE_1$-monoidal $(\infty,1)$-category $\cB$ inherits a monoidal structure. A \emph{bialgebra} in $\cB$ is an algebra in $\coAlg(\cB)$. A bialgebra is \emph{Hopf} if it admits a (not-necessarily invertible) antipode or equivalently if the  \define{shear map} $(\id_B \otimes m_B) \circ (\Delta_B \otimes \id_B ) : B \otimes B \to B \otimes B$ is invertible. We call a bialgebra  \emph{coHopf} if the coopposite bialgebra, in which the comultiplication is flipped, is Hopf; or equivalently if the \define{coshear map} $(m_B \otimes \id_B) \circ (\id_B \otimes \br_{B,B}) \circ (\Delta_B \otimes \id_B)$ is invertible.   (Note that the coopposite bialgebra of a bialgebra in $\cB$ is a bialgebra in $\cB^{\rev}$, i.e.\ uses the opposite braiding for the bialgebra law.)  A Hopf algebra is coHopf if and only if its antipode is invertible.  See \S\ref{subsec:finalproof} for details.} algebra.
Moreover, any (not necessarily dualizable) coHopf algebra in any presentably $\EE_2$-monoidal $(\infty,1)$-category $\cB$ (such as $\cB=\Vect$, the category of all vector spaces) arises from this construction in a certain $3$-category $\cC$ with $\Omega^2 \cC = \cB$. 
\end{maintheorem}

      The unital algebra and counital coalgebra structures on the composite~\eqref{eqn:formulaforH}  come from recognizing it as either a horizontal composition of an adjoint pair of $2$-morphisms in $\cC$ or as a vertical composition of a different adjoint pair of $2$-morphisms, see Remark~\ref{rem:simplesquare}.
 
\begin{example}
One of our motivations for Theorem~\ref{thm:mainHopftheorem}  comes from ongoing investigations into topological quantum field theories. Let $M$ be an $n$-manifold with an embedded normally framed $k$-disk $D^k \hookrightarrow M$. Surgering along the boundary $S^{k-1}$ produces a new $n$-manifold $M'$ and a handle-attachment cobordism $M \leadsto M'$. After surgery, the core of the original $D^k$ becomes an embedded sphere $S^k \hookrightarrow M$. Surgering $M'$ along this embedded $S^k$ returns  (a manifold diffeomorphic to) the original $M$, and the composite cobordism $M \leadsto M' \leadsto M$ is (diffeomorphic to) the identity. In other words, for every embedded disk in a manifold, the above handle-cancellation defines a retract in the cobordism category.  Hence, any sufficiently-extended cobordism higher category contains a host of adjunctible retracts, and hence by Theorem~\ref{thm:mainHopftheorem} a host of Hopf algebra objects. In particular, any functorial TQFT produces Hopf algebras in the codomain of the TQFT.  
In upcoming work, described for example in the talk \cite{TheoTalkPI}, we investigate these Hopf algebras and how they turn out to encode nontrivial physical information about the TQFT.
\end{example}

We make  the construction in Theorem~\ref{thm:mainHopftheorem}  functorial by defining an $(\infty,3)$-category $\Ret^{\someadj}(\cC)$ (Definition~\ref{defn:retadjcategory}) of adjunctible retracts in the sense of Theorem~\ref{thm:mainHopftheorem} and oplax morphisms of retracts\footnote{An \emph{oplax morphism of retracts} $(X_0, f_0, g_0, \alpha_0) \to (X_1, f_1, g_1, \alpha_1)$ consists of a $1$-morphism $X_< : X_0 \to X_1$,  $2$-morphisms $f_< : X_< f_0 \Rightarrow f_1$ and $g_< : g_0 \Rightarrow g_1 X_<$ and a $3$-isomorphism  $(g_0 f_<) \circ (g_< f_0) \circ \alpha_0 \Iisom \alpha_1$. A useful sufficient, but not necessary, condition for it to satisfy the adjunctibility conditions to be a morphism in $\Ret^{\someadj}(\cC)$ is  if $f_<$ is a left adjoint and  the canonical $2$-morphism $g_1^L \Rightarrow C_< g_0^L$ (built from $g_<$ and the adjunction data) is invertible, see Corollary~\ref{cor:somesubcategoriesofRetC}.} satisfying certain adjunctible conditions described in Proposition~\ref{prop:retadj}.

 \begin{maintheorem}[Corollaries~\ref{cor:part2ofmaintheorem} and~\ref{thm:adjinduceshopf}]\label{thm:maintheoremfunctorial}
 The construction from Theorem~\ref{thm:mainHopftheorem} assembles into a functor 
 \[ \Ret^\someadj(\cC) \to \coHopf(\Omega^2 \cC)
 \]
 to the full subcategory $\coHopf(\Omega^2 \cC) \subseteq \BiAlg(\Omega^2\cC):= \Alg(\coAlg(\Omega^2 \cC))$ on the coHopf algebras. 
 \end{maintheorem}
  
  \begin{remark}\label{rem:oppositevariants}
  Applying Theorem~\ref{thm:mainHopftheorem} to various opposite categories, one  obtains the following three alternative adjunctibility conditions on a retract $gf \simeq \id_{1_{\cC}}$ which give rise to either a Hopf algebra or a coHopf algebra in $\Omega^2 \cC$: 
  
  \begin{center}
    \begin{tabular}{cll}
      1-op: &
      \begin{tabular}{rl}
        {}[i] & $f$ is a left adjoint and $g$ is a right adjoint \\ {}[ii] & $f^R \To g$ is a left adjoint, equiv.\ $g^L\To f$ is a right adjoint
      \end{tabular}
       & $\leadsto$ Hopf
      \\ \hline \\[-8pt]
      2-op: &
      \begin{tabular}{rl}
        {}[i] & $f$ is a right adjoint and $g$ is a left adjoint \\ {}[ii] & $f \To g^R$ is a right adjoint, equiv.\ $g\To f^L $ is a left adjoint
      \end{tabular}
      & $\leadsto$ Hopf
      \\ \hline \\[-8pt]
      1-op, 2-op: &
      \begin{tabular}{rl}
        {}[i] & $f$ is a right adjoint and $g$ is a left adjoint \\ {}[ii] & $f \To g^R$ is a left adjoint, equiv.\ $g\To f^L $ is a right adjoint
      \end{tabular}
      & $\leadsto$ coHopf
    \end{tabular}
  \end{center}

  To see this, note that the notion of bialgebra in $\Omega^2 \cC$ is invariant under all these opposites: for any braided monoidal $(\infty,1)$-category there are canonical equivalences $\BiAlg(\cB) \simeq \BiAlg(\cB^{\rev}) \simeq \BiAlg(\cB^\mop)$. Under these equivalences, the full subcategories $\Hopf(\cB)$ and $\coHopf(\cB)$ are exchanged depending on the parity of the number of opposites.  On the other hand, taking opposites does change adjunctibility conditions and can reverse the roles of $f$ and $g$ in a retract. This leads to the above four variants of our construction. The notion of an adjunctible retract is fixed under $(-)^{\text{1-op, 3-op}}$ (see Remark~\ref{rem:adjunctibleretractsinCop}), so these are all variants we can extract. 
  \end{remark}

\begin{remark}
Drawing {\color{black!50!red}$f$} and its adjoints in {\color{black!50!red}red} and {\color{black!50!blue}$g$} and its adjoints in {\color{black!50!blue}blue}, and abbreviating $\alpha^\sharp := ((\alpha^{\rmate})^L)^{\lmate} \simeq  ((\alpha^{\lmate})^R)^{\rmate}$, the composite \eqref{eqn:formulaforH} can be drawn Poincar\'e-dually as a square:
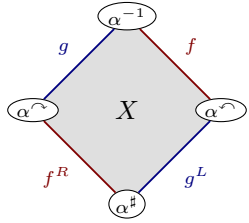
\[
\begin{tikzpicture}[scale=1.25,baseline=(M.base)]
\path[fill=gray,opacity=0.25] (-1,0) -- (0,-1) -- (+1,0) -- (0,+1) -- cycle;
\path 
(-1,0) node[draw, ellipse, inner sep=0.5pt,fill=white] (W) {$\scriptstyle \alpha^\rmate$}
(+1,0) node[draw, ellipse, inner sep=0.5pt,fill=white] (E) {$\scriptstyle \alpha^\lmate$}
(0,+1) node[draw, ellipse, inner sep=0.5pt,fill=white] (N) {$\scriptstyle \alpha^{-1}$}
(0,-1) node[draw, ellipse, inner sep=0.5pt,fill=white] (S) {$\scriptstyle \alpha^\sharp$}
(0,0) node (M) {$X$}
;
\draw[string, black!50!red] (S) -- node[auto] {$\scriptstyle f^R$} (W);
\draw[string, black!50!blue] (S) -- node[auto,swap] {$\scriptstyle g^L$} (E);
\draw[string, black!50!blue] (W) -- node[auto] {$\scriptstyle g$} (N);
\draw[string, black!50!red] (E) -- node[auto,swap] {$\scriptstyle f$} (N);
\end{tikzpicture}
\]
The multiplication, unit, counit, and comultiplication can then be visualized as the following three-dimensional diagrams:
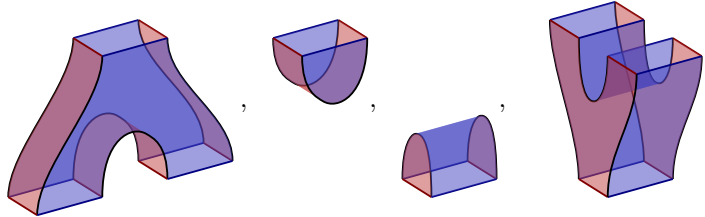
\[
\begin{tikzpicture}[scale=.625]
  \path 
  (0,0,0)
  +(-1,0,0)  coordinate (Wm)
  +(+1,0,0)  coordinate (Em)
  +(0,0,+1)  coordinate (Sm)
  +(0,0,-1)  coordinate (Nm)
  ++(+.25,-1.85,0)
  +(-.5,0,-.5) coordinate (pi)
  +(+.5,0,+.5) coordinate (pii)
  (-1,-3,+1)
  +(-1,0,0)  coordinate (Wm1) 
  +(+1,0,0)  coordinate (Em1) 
  +(0,0,+1)  coordinate (Sm1) 
  +(0,0,-1)  coordinate (Nm1) 
  (+1,-3,-1)
  +(-1,0,0)  coordinate (Wm2) 
  +(+1,0,0)  coordinate (Em2) 
  +(0,0,+1)  coordinate (Sm2) 
  +(0,0,-1)  coordinate (Nm2) 
  ;
  \draw[string, black!50!red] (Em) -- (Nm);
  \draw[string, black!50!blue] (Nm) -- (Wm);
  \draw[string, black!50!red] (Em1) -- (Nm1);
  \draw[string, black!50!blue] (Nm1) -- (Wm1);
  \draw[string, black!50!red] (Em2) -- (Nm2);
  \draw[string, black!50!blue] (Nm2) -- (Wm2);
  \draw[string] (Nm2) .. controls +(0,+1,0) and +(0,-1,0) .. (Nm);
  \draw[string] (Nm1) .. controls +(0,+1.25,0) and +(-.5,0,-.5) .. (pi) .. controls +(+.2,0,+.2) and +(0,+.75,0)  .. (Wm2);
  \path[fill=gray!50!blue,opacity=0.5] (Wm1) .. controls +(0,+1,0) and +(0,-1,0) .. (Wm) -- (Nm) .. controls +(0,-1,0) and +(0,+1,0) .. (Nm2) -- (Wm2) 
  .. controls +(0,+.75,0) and +(+.2,0,+.2) .. (pi)  .. controls +(-.5,0,-.5) and +(0,+1.25,0)  ..  (Nm1) -- (Wm1);
  \path[fill=gray!50!red,opacity=0.5] (Nm2) .. controls +(0,+1,0) and +(0,-1,0) .. (Nm) -- (Em) .. controls +(0,-1,0) and +(0,+1,0) .. (Em2) -- (Nm2);
  \draw[string] (Wm1) .. controls +(0,+1,0) and +(0,-1,0) .. (Wm);
  \draw[string] (Em2) .. controls +(0,+1,0) and +(0,-1,0) .. (Em);
  \path[fill=gray!50!red,opacity=0.5] (Wm1) .. controls +(0,+1,0) and +(0,-1,0) .. (Wm) -- (Sm) .. controls +(0,-1,0) and +(0,+1,0) .. (Sm1) -- (Wm1);
  \path[fill=gray!50!red,opacity=0.5] (Nm1) .. controls +(0,+1.25,0) and +(-.5,0,-.5) .. (pi) -- (pii) .. controls +(-.5,0,-.5) and +(0,+1.25,0)  ..  (Em1) -- (Nm1);
  \path[fill=gray!50!red,opacity=0.5] (Sm2) .. controls +(0,+.75,0) and +(+.2,0,+.2) .. (pii) -- (pi) .. controls +(+.2,0,+.2) and +(0,+.75,0)  .. (Wm2) -- (Sm2);
  \path[fill=gray!50!blue,opacity=0.5] (Sm1) .. controls +(0,+1,0) and +(0,-1,0) .. (Sm) -- (Em) .. controls +(0,-1,0) and +(0,+1,0) .. (Em2) -- (Sm2) .. controls +(0,+.75,0) and +(+.2,0,+.2) .. (pii)  .. controls +(-.5,0,-.5) and +(0,+1.25,0)  ..  (Em1) -- (Sm1);
  \draw[string] (Sm1) .. controls +(0,+1,0) and +(0,-1,0) .. (Sm);
  \draw[string] (Em1) .. controls +(0,+1.25,0) and +(-.5,0,-.5) .. (pii) .. controls +(+.2,0,+.2) and +(0,+.75,0)  .. (Sm2);
  \draw[string, black!50!red] (Wm1) -- (Sm1);
  \draw[string, black!50!blue] (Sm1) -- (Em1);
  \draw[string, black!50!red] (Wm2) -- (Sm2);
  \draw[string, black!50!blue] (Sm2) -- (Em2);
  \draw[string, black!50!red] (Wm) -- (Sm);
  \draw[string, black!50!blue] (Sm) -- (Em);
  \path (2.625,-1.5) node {,};
  \path 
  (4.25,0,0)
  +(-1,0,0)  coordinate (Wu)
  +(+1,0,0)  coordinate (Eu)
  +(0,0,+1)  coordinate (Su)
  +(0,0,-1)  coordinate (Nu)
  ++(-.25,-1.15,0)
  +(-.5,0,-.5) coordinate (boti)
  +(+.5,0,+.5) coordinate (botii)
  ;
  \draw[string, black!50!red]  (Eu) -- (Nu);
  \draw[string, black!50!blue] (Nu) -- (Wu);
  \draw[string] (Wu) .. controls +(0,-.75,0) and +(-.2,0,-.2) .. (boti) .. controls +(+.5,0,+.5) and +(0,-1.25,0) .. (Nu);
  \path[fill=gray!50!blue,opacity=0.5] (Wu) .. controls +(0,-.75,0) and +(-.2,0,-.2) .. (boti) .. controls +(+.5,0,+.5) and +(0,-1.25,0) .. (Nu) -- (Wu);
  \path[fill=gray!50!red,opacity=0.5] (boti) .. controls +(+.5,0,+.5) and +(0,-1.25,0) .. (Nu) -- (Eu) .. controls +(0,-1.25,0) and +(+.5,0,+.5) .. (botii) -- (boti);
  \path[fill=gray!50!red,opacity=0.5] (Wu) .. controls +(0,-.75,0) and +(-.2,0,-.2) .. (boti) -- (botii) .. controls +(-.2,0,-.2) and +(0,-.75,0)  .. (Su) -- (Wu);
  \path[fill=gray!50!blue,opacity=0.5] (Su) .. controls +(0,-.75,0) and +(-.2,0,-.2) .. (botii) .. controls +(+.5,0,+.5) and +(0,-1.25,0) .. (Eu) -- (Su);
  \draw[string, black!50!red]  (Wu) -- (Su);
  \draw[string, black!50!blue] (Su) -- (Eu);
  \draw[string] (Su) .. controls +(0,-.75,0) and +(-.2,0,-.2) .. (botii) .. controls +(+.5,0,+.5) and +(0,-1.25,0) .. (Eu);
  \path (5.375,-1.5) node {,};
  \path 
  (7,-3,0)
  +(-1,0,0)  coordinate (Wcu)
  +(+1,0,0)  coordinate (Ecu)
  +(0,0,+1)  coordinate (Scu)
  +(0,0,-1)  coordinate (Ncu)
  ++(-.05,1.15,0)
  +(-.5,0,+.5) coordinate (topi)
  +(+.5,0,-.5) coordinate (topii)
  ;
  \draw[string, black!50!red]  (Ecu) -- (Ncu);
  \draw[string, black!50!blue] (Ncu) -- (Wcu);
  \draw[string] (Wcu) .. controls +(0,+.75,0) and +(-.05,0,+.05) .. (topi) .. controls +(+.15,0,-.15) and +(0,+1,0) .. (Scu);
  \draw[string] (Ncu) .. controls +(0,+.75,0) and +(-.05,0,+.05) .. (topii) .. controls +(+.15,0,-.15) and +(0,+1,0) .. (Ecu);
  \path[fill=gray!50!red,opacity=0.5] (Ncu) .. controls +(0,+.75,0) and +(-.05,0,+.05) .. (topii) .. controls +(+.15,0,-.15) and +(0,+1,0) .. (Ecu) -- (Ncu);
  \path[fill=gray!50!blue,opacity=0.5] (Wcu) .. controls +(0,+.75,0) and +(-.05,0,+.05) .. (topi) -- (topii) .. controls +(-.05,0,+.05) and +(0,+.75,0) .. (Ncu) -- (Wcu);
  \path[fill=gray!50!red,opacity=0.5] (Wcu) .. controls +(0,+.75,0) and +(-.05,0,+.05) .. (topi) .. controls +(+.15,0,-.15) and +(0,+1,0) .. (Scu) -- (Wcu);
  \draw[string, black!50!red]  (Wcu) -- (Scu);
  \draw[string, black!50!blue] (Scu) -- (Ecu);
  \path[fill=gray!50!blue,opacity=0.5] (topi) .. controls +(+.15,0,-.15) and +(0,+1,0) .. (Scu) -- (Ecu) .. controls +(0,+1,0)  and +(+.15,0,-.15) .. (topii) -- (topi);
  \path (8.125,-1.5) node {,};
  \path
  (10.75,-3,0)
  +(-1,0,0)  coordinate (Wc)
  +(+1,0,0)  coordinate (Ec)
  +(0,0,+1)  coordinate (Sc)
  +(0,0,-1)  coordinate (Nc)
  ++(+.05,1.85,0)
  +(-.5,0,+.5) coordinate (qi)
  +(+.5,0,-.5) coordinate (qii)
  (9.75,0,-1)
  +(-1,0,0)  coordinate (Wc1) 
  +(+1,0,0)  coordinate (Ec1) 
  +(0,0,+1)  coordinate (Sc1) 
  +(0,0,-1)  coordinate (Nc1) 
  (11.75,0,+1)
  +(-1,0,0)  coordinate (Wc2) 
  +(+1,0,0)  coordinate (Ec2) 
  +(0,0,+1)  coordinate (Sc2) 
  +(0,0,-1)  coordinate (Nc2) 
  ;
  \draw[string, black!50!red]  (Ec) -- (Nc);
  \draw[string, black!50!blue] (Nc) -- (Wc);
  \draw[string, black!50!red]  (Ec1) -- (Nc1);
  \draw[string, black!50!blue] (Nc1) -- (Wc1);
  \draw[string, black!50!red]  (Ec2) -- (Nc2);
  \draw[string, black!50!blue] (Nc2) -- (Wc2);
  \draw[string] (Wc1) .. controls +(0,-1,0) and +(0,+1,0) .. (Wc);
  \draw[string] (Nc1) .. controls +(0,-1,0) and +(0,+1,0) .. (Nc);
  \draw[string] (Ec2) .. controls +(0,-1,0) and +(0,+1,0) .. (Ec);
  \draw[string] (Nc2) .. controls +(0,-.75,0) and +(+.05,0,-.05) .. (qii) .. controls +(-.15,0,+.15) and +(0,-1,0) .. (Ec1);
  \path[fill=gray!50!blue,opacity=0.5] (Wc1) .. controls +(0,-1,0) and +(0,+1,0) .. (Wc) -- (Nc) .. controls +(0,+1,0) and +(0,-1,0) .. (Nc1) -- (Wc1);
  \path[fill=gray!50!red,opacity=0.5] (Nc1) .. controls +(0,-1,0) and +(0,+1,0) .. (Nc) -- (Ec) .. controls +(0,+1,0) and +(0,-1,0) .. (Ec2) -- (Nc2) .. controls +(0,-.75,0) and +(+.05,0,-.05) .. (qii) .. controls +(-.15,0,+.15) and +(0,-1,0) .. (Ec1) -- (Nc1);
  \path[fill=gray!50!blue,opacity=0.5] (qii) .. controls +(-.15,0,+.15) and +(0,-1,0) .. (Ec1) -- (Sc1) .. controls +(0,-1,0) and +(-.15,0,+.15) .. (qi) -- (qii);
  \draw[string] (Wc2) .. controls +(0,-.75,0) and +(+.05,0,-.05) .. (qi) .. controls +(-.15,0,+.15) and +(0,-1,0) .. (Sc1);
  \path[fill=gray!50!blue,opacity=0.5] (qi) .. controls +(+.05,0,-.05) and +(0,-.75,0)  .. (Wc2) -- (Nc2) .. controls +(0,-.75,0) and +(+.05,0,-.05) .. (qii) -- (qi);
  \path[fill=gray!50!red,opacity=0.5] (Wc1) .. controls +(0,-1,0) and +(0,+1,0) .. (Wc) -- (Sc) .. controls +(0,+1,0) and +(0,-1,0) .. (Sc2) -- (Wc2) .. controls +(0,-.75,0) and +(+.05,0,-.05) .. (qi) .. controls +(-.15,0,+.15) and +(0,-1,0) .. (Sc1) -- (Wc1);
  \path[fill=gray!50!blue,opacity=0.5] (Sc2) .. controls +(0,-1,0) and +(0,+1,0) .. (Sc) -- (Ec) .. controls +(0,+1,0) and +(0,-1,0) .. (Ec2) -- (Sc2);
  \draw[string, black!50!red]  (Wc) -- (Sc);
  \draw[string, black!50!blue] (Sc) -- (Ec);
  \draw[string, black!50!red]  (Wc1) -- (Sc1);
  \draw[string, black!50!blue] (Sc1) -- (Ec1);
  \draw[string, black!50!red]  (Wc2) -- (Sc2);
  \draw[string, black!50!blue] (Sc2) -- (Ec2);
  \draw[string] (Sc2) .. controls +(0,-1,0) and +(0,+1,0) .. (Sc);
\end{tikzpicture}
\]
As explained by the second author in \cite{David2017Talk}, bordism-categorical manipulations do suggest that these should form a bialgebra; compare our~\S\ref{subsec:unpacked}--\ref{subsec:extradualizability} and also~\cite{TudorWenjun}. Those earlier works focus on cases in which $\cC$ is a $(3,3)$-category,  has all adjoints and enough ``orientability'' to identify left and right adjoints, leading to Hopf--Frobenius algebras with an antipode that squares to the identity and can be represented as a ``$180^\circ$ rotation'' of the square. At the level of generality of our Theorem~\ref{thm:mainHopftheorem}, such topological manipulations can provide inspiration but not a proof: 
 three-dimensional diagrammatics have not been fully established as a language for $(\infty,3)$-categories,  to exhibit the infinite amount of coherence data inherent in a bialgebra in a general braided monoidal $(\infty,1)$-category. Even when working with $(3,3)$-categories without ``orientability'' structures to produce general (non-Frobenius) Hopf algebras, one needs to carefully keep track of framing data, see~\S\ref{subsec:extradualizability}. Lastly, we do not assume that our $(\infty,3)$-category has all adjoints --- and  indeed, our construction does produce non-dualizable Hopf algebras with non-invertible antipodes, see e.g.~\S\ref{subsec:tannakian}. 
\end{remark}

 This ultimately requires a more algebraic approach to Theorem~\ref{thm:mainHopftheorem}:

\subsektion{Lax smash products of adjunctions and monads}
The $\infty$-category $\Cat^*_{(\infty,\infty)}$ of pointed $(\infty,\infty)$-categories carries a canonical (presentably) monoidal structure given by the ``lax smash product'' (a.k.a.\ Gray smash product)~\cite{2311.00205,NarukiThesis,2505.22640}, which we study in detail in Section~\ref{sec:smash}. 

 The corresponding inner homs (defined as the right adjoints respectively to $\cX \owedge -$ and $-\owedge \cX$) are denoted $\Funlax_*(\cX,-)$ and $\Funoplax_*(\cX,-)$; given a pointed $(\infty,\infty)$-category $\cC$, the $(\infty,\infty)$-categories $\Funlax_*(\cX,\cC)$ and $\Funoplax_*(\cX,\cC)$ both have as their spaces of objects the pointed functors $\cX \to \cC$, but the (higher) morphisms are the lax or oplax, respectively, natural (higher) transformations between pointed functors.

Let $\Adj$ denote the ``walking adjunction,'' the free $(\infty,2)$-category on an adjoint $1$-morphism, pointed at the source of the left adjoint. The full subcategory on the source of the left adjoint is the walking monad $\Mnd \hookrightarrow \Adj$; and indeed, Haugseng shows in  \cite{2002.01037} that $\Funoplax_*(\Mnd, \cC)$ is equivalent to the $(\infty,1)$-category $\Alg(\Omega \cC)$ of algebras in $\Omega\cC$. In \S\ref{sec:mainconstruction}, we prove Proposition~\ref{prop:ClaudiaTheo} characterizing (higher) (op)lax morphisms between (higher) adjunctions  and use it to compute the smash square of $\Mnd$ and of $\Adj$.

\begin{maintheorem}[Corollaries~\ref{cor:smashretract} and~\ref{cor:mndsquared}]\label{mainthm:MndMndAdjAdj} Let $\cC$ be a pointed $(\infty,3)$-category. \begin{itemize}
\item The space of pointed functors $\Mnd \owedge \Mnd \to \cC$ is equivalent to the space of bialgebras in $\Omega^2 \cC$ and  $\Funoplax_*(\Mnd \owedge \Mnd, \cC)$ is equivalent to the $(\infty,1)$-category $\BiAlg(\Omega^2\cC)$;
\item The space of pointed functors $\Adj \owedge \Adj \to \cC$ is equivalent to the space of adjunctible retracts in $\cC$ in the sense of Theorem~\ref{thm:mainHopftheorem}.
\end{itemize}
\end{maintheorem}

In these terms, $\Ret^{\someadj}(\cC):= \Funoplax_*(\Adj\owedge \Adj, \cC)$, which we prove in Corollary~\ref{cor:retadjissubofretlax} is indeed a subcategory of the $(\infty,3)$-category  $\Ret^\oplax(\cC)$ of all retracts and their oplax morphisms. Our construction
\[ \Ret^{\someadj}(\cC) = \Funoplax_*(\Adj \owedge \Adj, \cC) \to \Funoplax_*(\Mnd \owedge \Mnd, \cC) \simeq \BiAlg(\Omega^2 \cC) 
\]
 is   given by restriction along the pointed functor 
 \[ \Mnd \owedge \Mnd \to \Adj \owedge \Adj.
 \]

The category $\Mnd \owedge \Mnd$ contains a \emph{universal shear $3$-morphism}, Definition~\ref{def:strictshear} and Construction~\ref{cons:universalshear},  so that coHopf algebras are precisely those pointed functors $\Mnd\owedge \Mnd \to \cC$ which invert this $3$-morphism. In fact, this $3$-morphism can be lifted to the lax product $\Mnd \otimes \Mnd$. The least formal step of our proof is the following theorem, which shows that  every bialgebra arising from our construction is (co)Hopf.

\begin{maintheorem}[Theorem~\ref{thm:invertibleshear}] \label{mainthm:universalshear}Under the composite $\Mnd \otimes \Mnd \to \Adj \otimes \Adj \to L_3(\Adj \otimes \Adj)$ to the localization of the $(\infty,4)$-category $\Adj \otimes \Adj$  to an $(\infty,3)$-category, the universal shear $3$-morphism is sent to an invertible $3$-morphism. 
\end{maintheorem}
Our proof  depends very much on the localization to $(\infty,3)$-categories --- we did not find a similar statement in $(\infty,4)$-categories (where $\Adj \owedge \Adj$ describes a certain type of lax retract, and where $\Mnd \owedge \Mnd$ describes bialgebras in which the bialgebra axiom holds only laxly, see Remark~\ref{rem:laxbialgebra}).

\subsektion{Related work}
Quantum-physics-focused versions of our main Theorem~\ref{thm:mainHopftheorem}  have been explained by \v{S}evera~\cite{MR1915348}, the second-named author~\cite{David2017Talk},  Freed and  Teleman~\cite{MR4520300}, C\'ordova, Holfester, and Ohmori~\cite{2408.11045}, and Dimofte and Niu~\cite{TudorWenjun}.  A significant advantage of our approach  is that we work entirely $\infty$-categorically, setting up a framework in which to coherently handle functoriality, the bialgebra axiom, and the like, and which in particular is sensitive to all questions about framing data. 
More importantly,   by demanding only some, but not all, adjoints, our construction is able to produce \emph{any} Hopf algebra, in particular infinite-dimensional ones with non-invertible antipode.

The identification  of pointed functors $\Mnd \owedge \Mnd\to \cC$ into a $3$-category $\cC$ with bialgebras in $\Omega^2 \cC$ as in Theorem~\ref{mainthm:MndMndAdjAdj} was also  observed by Hadzihasanovic in~\cite{1709.08086,2101.10361} using the model of ``diagrammatic $(\infty,n)$-categories''~\cite{2410.19053, 2505.01387}.  For $n>0$, this model is not yet known to be equivalent to $(\infty,n)$-categories as used in this paper.

The $(\infty,\infty)$-category version of $\owedge$ is also studied in \cite{2311.00205,NarukiThesis,2505.22640}.

In \S\ref{subsec:hopmonad}, we generalize our construction to a construction of Hopf monads in these sense of \cite{MR1887157,MR2355605,MR2500058,MR2793022}, and explain these in terms of the restriction functor  $ \Mnd_+ \owedge \Mnd \to \Adj_+ \owedge \Adj$ where $(-)_+$ freely adjoins a basepoint. 

In \S\ref{subsec:tannakian}, we compare our construction with ``classical'' Tannakian reconstruction, as in the classical categorically-minded works \cite{MR1098991,MR1173027,MR1623637}  and the more recent physics-motivated works \cite{MR1915348,MR4520300,TudorWenjun}.

\subsektion{Questions}
We end with some questions we will not address in this article:

\begin{question}\label{question:lessadjunctibility}
At the level of generality of Theorem~\ref{thm:mainHopftheorem},
what is the meaning/formula for the antipode? If $\cC$ has all adjoints, then the Hopf algebra is dualizable and hence the antipode is invertible~\cite{MR1685417,MR1759389}. As explained in \S\ref{subsec:extradualizability}, in this case we can identify the antipode as a ``$180^\circ$ rotation of the square $H$ in~\eqref{eqn:formulaforH},'' i.e.\ more precisely as the ratio of two isomorphisms $H^{\rmate\rmate} \simeq H$: one  coming from comparing the various mates in the formula~\eqref{eqn:formulaforH}; the other  from the fact that for any square $\beta$ all of whose sides are identities, there is a canonical isomorphism $\beta^{\rmate\rmate} \simeq \beta$.
But in general, the antipode on the square \eqref{eqn:formulaforH} cannot be such a ratio, as it is not necessarily invertible!
\end{question}

\begin{question} If $\cC$ is an $(\infty,2)$-category, $\Funlax(\Mnd, \cC)$ unwinds to an $(\infty,2)$-categorical version of Street's category $\cat{Mnd}(\cC)$ of monads in $\cC$ from \cite{MR299653}.
Let $\globe_1$ denote the walking arrow, i.e.\ the strict 1-category $\bullet \to \bullet$.
 It follows from Proposition~\ref{prop:FunlaxAdj} that the morphism $\globe_1 \to \Adj$ picking out the right adjoint induces an equivalence $\Funlax(\Adj, \cC) \simeq \Fun(\globe_1, \cC)^R $ with the full subcategory of $\Fun(\globe_1, \cC)$ on the right adjoints.  The category $\cC$ is said to admit \emph{Eilenberg-Moore objects} if the restriction functor 
\[ \Fun(\globe_1, \cC)^R \simeq \Funlax(\Adj, \cC) \to \Funlax(\Mnd, \cC) = \cat{Mnd}(\cC)
\]
admits a right adjoint (which sends a monad in $\cC$ to its Eilenberg-Moore adjunction). As shown in~\cite{1712.00555} in an $(\infty,2)$-categorical context, this right adjoint is always fully faithful with image the \emph{monadic adjunctions}. 
Smashing this with itself, which universal property does it imply for our construction and does this recover the usual story of Tannakian reconstruction? 
\end{question}

\begin{question}
Given a Hopf algebra $H$ in a braided monoidal $(\infty,1)$-category $\cB$, is there an $(\infty,3)$-category $\cC$ with $\Omega^2 \cC = \cB$ and an adjunctible retract in $\cC$ with recovers $H$ via  Theorem~\ref{thm:mainHopftheorem}? 
Our Corollary~\ref{cor:anyHopf} proves this if $\cB$ is presentable; how much can this presentability assumption be relaxed? (Any $\cB$ embeds into a presentably braided monoidal category, namely its presheaf category $\PSh(\cB)$. One could imagine starting from  the adjunctible retract in Corollary~\ref{cor:anyHopf} and cutting down to a sub-$(\infty,3)$-category which contains all relevant morphisms but which deloops $\cB$ instead of $\PSh(\cB)$.) 
\end{question}

\begin{question}
  The functor $\Mnd \to \Adj$ is  fully faithful.  Our main Theorem~\ref{thm:mainHopftheorem} shows that after localizing to $(\infty,3)$-categories, the corresponding functor $\Mnd \owedge \Mnd \to \Adj \owedge \Adj$ is not fully faithful: the codomain, but not the domain, contains an inverse to the shear map. What about before localization? Both the domain and codomain are $(\infty,4)$-categories, but our Theorem~\ref{thm:mainHopftheorem} only sees their localizations to $(\infty,3)$-categories --- without localization, is the functor $\Mnd \owedge \Mnd \to \Adj \owedge \Adj$ a fully faithful inclusion of $(\infty,4)$-categories?
\end{question}

\begin{question}
The domain of the functor in Theorem~\ref{thm:maintheoremfunctorial} is the $(\infty,3)$-category $\Ret^\someadj(\cC) = \Funoplax_*(\Adj \owedge \Adj, \cC)$, whereas the codomain is the $(\infty,1)$-category $\BiAlg(\Omega^2\cC)$. It follows that $H$ factors through the $1$-localization of $\Ret^\someadj(\cC)$. What is this $1$-localization?
\end{question}

\subsection{Acknowledgements}

We thank Harshit Yaddav for many helpful comments and hints related to Hopf algebras in braided monoidal categories,  Tim Campion and Naruki Masuda for valuable discussions about lax tensor products of $(\infty,\infty)$-categories, Lukas M\"uller for patient conversation about diagram chases and for useful comments on a draft, and Markus Zetto for many conversations about higher presentable categories.
This work is supported by the NSERC grant Discovery Grant RGPIN-2021-02424, by the Simons Collaboration for Global Categorical Symmetries (Simons Foundation grant 888996), and by the Deutsche Forschungsgemeinschaft (DFG)
under the Emmy Noether program (grant 493608176) and  the Collaborative Research Center (SFB) 1624 “Higher structures, moduli spaces and integrability” (grant 506632645).
 Parts of this work were completed at the programs \define{Quantum Symmetries Reunion}, held at Simons Laufer Mathematical Sciences Institute, and \define{Global Categorical Symmetries}, held at the Perimeter Institute for Theoretical Physics.
 Research at the Perimeter Institute is supported by the Government of Canada through Industry Canada and by the Province of Ontario through the Ministry of Economic Development and Innovation. 
 We congratulate Benjamin on the occasion of his birthday.

\section{The lax smash product \texorpdfstring{$\owedge$}{owedge}} \label{sec:smash}

The goal of this section is to construct and analyze the lax smash product $\owedge$ of pointed $(\infty,\infty)$-categories, built from the lax (aka Gray) tensor product $\otimes$ on $(\infty,\infty)$-categories; much of this material has also been developed, and with a similar flavour, in the recent PhD thesis \cite{NarukiThesis}\footnote{\emph{Warning:} The lax product $\otimes$ is not commutative, and Masuda writes ``$A \otimes B$'' for what we call ``$B \otimes A$''. The literature is not amazingly consistent about the ordering for the lax product; it takes care not to make a ``categorical sign error.'' \label{footnote:Naurkiconvention}}, also see the recent~\cite{2505.22640}.
 The heart of our proof in Section~\ref{sec:mainconstruction} of our main theorem will be an analysis of the $\owedge$-square of the walking adjunction.

\subsection{Preliminaries}

Let $\Cat_{(\infty, n)}$ denote the $\infty$-category\footnote{We follow standard conventions and use the term ``$\infty$-category'' synonymously with $(\infty,1)$-category: between any two objects is a space of morphisms. Throughout the paper, we work in the framework developed in Lurie's books~\cite{HTT, HA}  and freely use the language and technology therein.} of $(\infty, n)$-categories, as e.g.\ axiomatized in~\cite{MR4301559}. The definitions and observations in this section largely follow~\cite{2311.00205}, although we will focus here on \emph{univalent} (or Rezk complete) $(\infty, n)$-categories while Campion works in the more general setting of flagged (aka \emph{valent}) $(\infty,n)$-categories of \cite{MR3869643}.

In fact, it will be  convenient to work with $(\infty,n)$-categories for various $n$ at a time, and more generally with $(\infty,\infty)$-categories allowing non-invertible morphisms at all levels.
The fully faithful inclusion $\Cat_{(\infty,n)} \hookrightarrow \Cat_{(\infty,n+1)}$ admits both adjoints: the left adjoint, denoted $L_n$, inverts all $(n+1)$-morphisms, and the right adjoint, denoted $\iota_n$, forgets the noninvertible $(n+1)$-morphisms. By definition, the $\infty$-category $\Cat_{(\infty, \infty)}$ of $(\infty,\infty)$-categories (with inductive equivalences) is the limit along the forgetful functors:
$$ \Cat_{(\infty,\infty)} := \varprojlim \left( \dots \to \Cat_{(\infty,n+1)} \overset{\iota_{n}}\to \Cat_{(\infty,n)} \to \dots \to \Cat_{(\infty,0)} \right) $$
The construction supplies a forgetful functor $\iota_n : \Cat_{(\infty,\infty)} \to \Cat_{(\infty,n)}$, which is a right adjoint. Its left adjoint is again fully faithful, and admits its own left adjoint $L_n : \Cat_{(\infty,\infty)} \to \Cat_{(\infty,n)}$ that inverts all morphisms of dimension $>n$.
\begin{equation}\label{eq:inclusion-and-localization}
 \begin{tikzcd}[column sep=1.5cm]
\Cat_{(\infty, n)}
\arrow[hookrightarrow]{r}[xshift=-0.cm, yshift=0.2cm]{\bot}[swap, xshift=-0.cm, yshift=-0.2cm]{\bot}&\arrow[bend left=45]{l}{\iota_n}
\arrow[bend right=45]{l}[swap]{L_n}
\Cat_{(\infty, \infty)}
\end{tikzcd}.
\end{equation}
For $(\infty, \infty)$-categories $\cC$ and $\cD$, we write $\Map(\cC, \cD)$ for the mapping space in $\Cat_{(\infty, \infty)}$, i.e.\ for the \emph{space} of functors (and natural isomorphisms, etc.)\ from $\cC$ to $\cD$. The $\infty$-categories $\Cat_{(\infty,\infty)}$ and $\Cat_{(\infty,n)}$ are Cartesian closed, and we will write $\Fun(\cC, \cD) \in \Cat_{(\infty,\infty)}$ for the internal hom with respect to the Cartesian product: the space of objects of $\Fun(\cC, \cD)$ is $\Map(\cC,\cD)$, the 1-morphisms in $\Fun(\cC, \cD)$ are the (strong) natural transformations, the 2-morphisms are the (strong) natural modifications, and so on.

Let $\strCat_n$ denote the $1$-category of strict $n$-categories and strict $n$-functors, inductively defined as categories (strictly) enriched in the $1$-category $\strCat_{n-1}$ of strict $(n-1)$-categories, starting with $\strCat_0 := \cat{Set}$. 
We emphasize that in $\strCat_n$, the isomorphisms are actual isomorphisms of higher categories, not equivalences --- they induce for example bijections of sets of objects.
The strict analogue of an $(\infty,\infty)$-category is a \emph{strict $\omega$-category}, the $1$-category $\strCat_{\omega}$ of which is defined as the limit $$\strCat_\omega := \varprojlim \left( \ldots \to \strCat_{n+1} \overset{\iota_n}\to \strCat_{n} \to \ldots \to \strCat_0\right)$$ in the $1$-category $\strCat_1$ of strict $1$-categories along the functors $\iota_n: \strCat_{n+1} \to \strCat_n$ which forget all $(n+1)$-morphisms.\footnote{The functors $\iota_n$ are isofibrations, and so this limit agrees with the limit computed in the $(2,1)$-category of strict $1$-categories, functors, and natural isomorphisms.}

A strict $\omega$-category is called \define{gaunt}\footnote{The name and theory of gaunt $\omega$-categories is due to~\cite{MR4301559}. The notion appears briefly under the name ``stiff'' in~\cite{MR1953060}, but the notion is not explored therein.} if for all $k \geq 1$, every invertible $k$-morphism is an identity. We denote the full sub-$1$-category of gaunt $\omega$-categories by $\Gaunt \subseteq \strCat_{\omega}$.
 Any strict $\omega$-category gives rise to an $(\infty, \infty)$-category, leading to a functor $\strCat_{\omega} \to \Cat_{(\infty,\infty)}$. When restricted to the full sub-$1$-category $\Gaunt \subseteq \strCat_{\omega}$ this functor $\Gaunt \to \Cat_{(\infty, \infty)}$ becomes fully faithful.\footnote{The inclusion $\Gaunt \subset \Cat_{(\infty, \infty)}$ is an equivalence onto the full sub-$\infty$-category of $0$-truncated objects~\cite{MR4301559}. Another useful way to think about gaunt $\omega$-categories is that they are precisely the higher categories which are simultaneously univalent and strict.}  We will often implicitly use this functor to promote  gaunt $\omega$-categories to $(\infty, \infty)$-categories. The inclusion $\Gaunt \subseteq \strCat_{\omega}$ is a right adjoint; we will write $\Gau : \strCat_{\omega} \to \Gaunt$ for its left adjoint, referred to as \define{gauntification}.

\begin{definition}\label{defn:suspension} Let $\strCat_{n+1}^{**}:= (\strCat_{n+1})_{*\sqcup */}$ denote the $1$-category of strict $(n+1)$-categories with two marked objects. The strict \define{categorical suspension functor} $ \Sigma: \strCat_n \to \strCat^{**}_{n+1}$ sends a strict $n$-category $C$ to the strict $(n+1)$-category with two (marked) objects $0$ and $1$ and 
\[\Sigma C(a,b):= \begin{cases} C, & a = 0, b = 1, \\ *, & a = b, \\ \emptyset & a=1, b=0. \end{cases}
\] 
\end{definition}
The suspension of a gaunt category is again gaunt.

\begin{definition}\label{defn:walkingcell}
For $n\geq 0$, we define the \define{walking $n$-cell} $\globe_n$ by
$$ \globe_n := \Sigma^n(*) = \Sigma(\globe_{n-1}).$$
\end{definition}
If one starts instead with the empty set, rather than a single point, the iterated suspensions 
$\Sigma^n \emptyset$ consists of two parallel $(n-1)$-morphisms: it is precisely the boundary $\partial \globe_n:= \Sigma^n \emptyset$ of the $n$-cell, with boundary inclusion $\partial \globe_n \mono \globe_n$ given by $\Sigma^n(\emptyset \hookrightarrow *)$.
Suspending instead the map $\globe_1 \to \globe_0 = *$ supplies degeneracy maps $\Sigma^n(\globe_1 \to *) : \globe_{n+1} \to \globe_n$.

The walking cells $\globe_n$ are a set of colimit generators\footnote{A \define{set of colimit-generators} of a cocomplete $\infty$-category $\cC$ is a set $S$ of objects such that the smallest full subcategory of $\cC$ which contains $S$ and is closed under colimits is $\cC$ itself.} for any of the (higher) categories $\Gaunt$, $\strCat_\omega$, and $\Cat_{(\infty,\infty)}$.
The walking cells are examples of \define{pasting diagrams}; we let $\Theta$ denote Joyal's full subcategory of $\strCat_\omega$ on the pasting diagrams (see e.g.~\cite{MR2331244} 
 or \cite[Def.~1.4]{2311.00205}) and write $\Theta_n := \Theta \cap \strCat_{n} \subseteq \strCat_n$ for the full subcategory of $n$-dimensional pasting diagrams. 
Since all pasting diagrams are gaunt, we may equally regard $\Theta_n$ and $\Theta$ as full subcategories of $\Cat_{(\infty, n)}$ and $\Cat_{(\infty, \infty)}$, respectively. The full inclusions $\Theta_n \subseteq \Cat_{(\infty,n)}$ and $\Theta \subseteq \Cat_{(\infty, \infty)}$ are \emph{dense}\footnote{A functor $\cA \to \cB$ of $\infty$-categories is \emph{dense}, if the restricted Yoneda embedding $\cB \to \cat{Psh}(\cA)$ is fully faithful.  If $\cB$ has small colimits and $\cX$ is another $\infty$-category with small colimits, this implies that the restriction functor $\Fun^{\mathrm{cocont.}}(\cB, \cX) \to \Fun(\cA, \cX)$ is full faithful, where $\Fun^{\mathrm{cocont.}}$ denotes the full subcategory of the functor category on the cocontinuous functors. In other words,  any functor $\cB \to \cX$ is uniquely determined by its restriction to $\cA$ and any functor $\cA \to \cX$ has at most one cocontinuous extension to $\cB$. } and hence any cocontinuous functor out $\Cat_{(\infty,n)}$ and $\Cat_{(\infty, \infty)}$ is uniquely determined by its restriction to $\Theta_n$ or $\Theta$, respectively, a fact that we will use repeatedly throughout.

The categorical suspension functor restricts to a functor $\Sigma: \Theta_n \to \Theta^{**}_{n+1}:= (\Theta_{n+1})_{* \sqcup */}$, which by \cite[Rem.~1.13]{2311.00205} admits a unique cocontinuous extension \begin{equation}\label{eq:sigmainfinity}
\Sigma : \Cat_{(\infty, n)} \to \Cat_{(\infty, n+1)}^{**}:= (\Cat_{(\infty, n+1)})_{* \sqcup */}.\end{equation} These assemble into a cocontinuous functor 
\[\Sigma : \Cat_{(\infty, \infty)} \to \Cat_{(\infty,\infty)}^{**}.
\]
We will also denote by $\Sigma$ its corestrictions to functors $\Cat_{(\infty,n)} \to \Cat_{(\infty,n+1)}$ and $\Cat_{(\infty,\infty)} \to \Cat_{(\infty,\infty)}$, but we warn that these functors do not preserve coproducts.

\begin{definition}\label{defn:opposites}
For a strict $\omega$-category $C$, let $C^{\op}$ denote the strict $\omega$-category with all odd-morphism directions reversed and $C^{\co}$ the strict $\omega$-category with all even morphism directions reversed. The equivalences $(-)^{\op}$ and $(-)^{\co}$ restrict to autoequivalences of $\Theta$ and extend cocontinuously to autoequivalences of $\Cat_{(\infty, \infty)}$, which we denote by the same names. 
\end{definition}

\begin{lemma}\label{lem:Sigmacoop}
For $\cC \in \Cat_{(\infty,\infty)}$, there are natural equivalences in $\Cat_{(\infty,\infty)}^{**}$: 
\begin{align*}
\left(\{0 \} \sqcup \{1\} \to (\Sigma \cC)^{\co}\right) & \simeq \left(\{0\} \sqcup \{1\}  \to \Sigma \cC^{\op}\right), \\  \left( \{0\} \sqcup \{1\} \to (\Sigma \cC)^{\op}  \right) & \simeq \left(\{1\} \sqcup \{0\} \to \Sigma \cC^{\co}\right).
\end{align*}
\end{lemma}
\begin{proof}
Since $\Sigma:\Cat_{(\infty,\infty)} \to \Cat_{(\infty, \infty)}^{**}$ preserves colimits, it suffices to construct these equivalences on gaunt categories where they are straightforward.
\end{proof}

Since our $(\infty,\infty)$-categories are by convention univalent (aka Rezk-complete), the morphisms $ \Sigma^{n-1}(\globe_1 \to *): \globe_n \to \globe_{n-1}$ are \emph{epimorphisms}\footnote{\label{footnote:epi}A morphism $f:a \to b$ in an $\infty$-category $\cE$ is an \emph{epimorphism} if for all objects $c$ the induced map of spaces $\Map(b, c) \to \Map(a,c)$ is a full subspace inclusion.} in $\Cat_{(\infty,\infty)}$: for any $\cC \in \Cat_{(\infty,\infty)}$, the space $\Map(\globe_n, \cC)$ is the space of $n$-morphisms in $\cC$, and $\Map(\globe_{n+1} \to \globe_{n}, \cC) : \Map(\globe_n, \cC) \to \Map(\globe_{n+1}, \cC)$ is equivalent to the full subspace inclusion onto the space of invertible $(n+1)$-morphisms.

The localization $L_n : \Cat_{(\infty,\infty)} \to \Cat_{(\infty,n)}$ that inverts morphisms of dimension $>n$ is precisely the localization of $\Cat_{(\infty,\infty)}$ along the degeneracy maps $\globe_{m+1} \to \globe_{m}$ for $m > n$. Given an $(\infty, \infty)$-category $\cC$, the categorical suspension $\Sigma (\cC \to L_n \cC)$ is a functor to an $(\infty,n+1)$-category and hence factors through $L_{n+1} \Sigma \cC \to \Sigma L_n \cC$.
\begin{lemma}For any $(\infty, \infty)$-category $\cC$, this canonical functor $L_{n+1} \Sigma \cC \to \Sigma L_n \cC$ is an equivalence.
\end{lemma}
\begin{proof}Since $L_n, L_{n+1}$ and $\Sigma$ preserve colimits, it suffices to verify the statement on the walking $n$-cells 
 $\cC = \globe_m$. The statement is trivial for $m \leq n$, and follows for $m > n$ since the degeneracy maps $\globe_m \to \globe_{m-1} \to \ldots \to \globe_n$ induce equivalences $L_{n} \globe_m \to  \globe_n$. 
\end{proof}

Whereas some functors invert $n$-cells, others detect noninvertibility:
\begin{definition} A functor of $(\infty, n)$-categories is called \emph{conservative on $n$-morphisms} if it is right orthogonal to $\globe_{n} \to \globe_{n-1}$.
\end{definition}
Unpacked, this defines a functor $F:\cA \to \cB$ to be conservative on $n$-morphisms if  every $n$-cell in $\cA$ whose image in $\cB$ is invertible is itself invertible in $\cA$.
\begin{lemma} \label{lem:consequences-conservative} Let $\cA$ be an $(\infty, n)$-category and $F: \cB \to \cD$ be a functor of $(\infty, n)$-categories. 
\begin{enumerate}
\item $\cA$ is an $(\infty, n-1)$-category if and only if $\cA \to *$ is conservative on $n$-morphisms;
\item Conservative-on-n-morphism functors are preserved by pullback in $\Cat_{(\infty, n)}$;
\item If $F$ is conservative on $n$-morphisms, then for every $d\in \cD$, the fiber $\{d \} \times_{\cD} \cC$ (computed in $\Cat_{(\infty, n)}$) is an $(\infty, n-1)$-category. \label{lem:consequences-conservative3}
\end{enumerate}
\end{lemma}
\begin{proof}
The first statement follows by definition, since $\Cat_{(\infty, n-1)}$ is the localization of $\Cat_{(\infty, n)}$ at the morphism $\globe_n \to \globe_{n-1}$. The second statement follows since morphisms right orthogonal to a fixed morphism are preserved under pullback. The third statement follows from the first as a special case of the second. 
\end{proof}

\subsection{The lax tensor product \texorpdfstring{$\otimes^{\strict}$}{otimes^strict} on strict \texorpdfstring{$\omega$}{omega}-categories} \label{subsec:otimesstrict}

The Cartesian product $\cA \times \cB$ of strict $\omega$-categories is built by declaring that, for each cell $a : \globe_m \to \cA$ and each cell $b : \globe_n \to \cB$, there is a strictly-commuting square $a \times b$. In \cite{MR0371990}, Gray introduced a ``lax'' product for strict $2$-categories, which we will denote ``$\otimes^{\strict}$'' in order to reserve ``$\otimes$'' for Campion's generalization to $(\infty,\infty)$-categories reviewed in \S\ref{subsec:otimesinfty}. This product was extended to strict $\omega$-categories by Crans in \cite{CransTensor}, and so $\otimes^{\strict}$ is variously known as the ``Gray product'' and the ``Crans--Gray product.''\footnote{In addition to the names ``Gray product'' and  ``Crans--Gray product,'' we like the names ``strict tensor product'' and ``strict lax product'' for $\otimes^{\strict}$. It is worth emphasizing that ``strict lax'' is not a contradiction. The tensor product $\otimes^{\strict}$ is ``strict'' in the sense that it is a construction on strict $\omega$-categories. It is ``lax'' in the sense that, rather than producing strictly-commuting squares, it fills squares with noninvertible cells. The word replacing ``lax'' for invertible cells is ``strong.''}
The goal of this section is to quickly review the definition of $\otimes^{\strict}$ and to describe a graphical calculus for it developed in~\cite{1709.08086,2101.10361} that we will use in \S\ref{subsec:finalproof}. Thorough developments of $\otimes^{\strict}$ can also be found in~\cite{MR2061574,MR4146146,2404.07273}.

Roughly speaking, given two strict $\omega$-categories $\cA, \cB$, the strict $\omega$-category $\cA \otimes^\strict \cB$ is generated by a cell $a \otimes b$ of dimension $p+q$ for each pair of a $p$-cell in $a \in \cA$ and a $q$-cell in $b \in \cB$, with source and target built from $\partial a \otimes b$ and $a \otimes \partial b$ (glued along $\partial a \otimes \partial b$)\footnote{
Specifically, the rule is that the cell $a \otimes b$ is directed in such a way that ``$\cA$ wins'': $\operatorname{source}(a) \otimes b$ always appears in $\operatorname{source}(a \otimes b)$ and $\operatorname{target}(a) \otimes b$ always appears in $\operatorname{target}(a \otimes b)$. Where $a \otimes \operatorname{source}(b)$ and $a \otimes \operatorname{target}(b)$ appear depends on the dimensions of $a$ and $b$; their locations in $\operatorname{source}(a \otimes b)$ and $\operatorname{target}(a \otimes b)$ are uniquely determined by requiring that $a \otimes b$  be well-formed. The precise combinatorics can be found for example in~\cite{JFS}.
}. 

We will primarily work with pasting diagrams, but it will be occasionally useful to describe formulas in terms of string diagrams (as detailed for example in \cite{MR2767048}). The full theory of  string (and foam) diagrams for $(\infty,\infty)$-categories is currently in progress, and so we will only use string diagrams for strict $\omega$-categories. Indeed, since we are bound to two-dimensional (digital) paper, we will in general lay out commutative diagrams whose entries are diagrams in (strict) 2-categories; for higher dimensional diagrams, we refer the reader to~\cite{1709.08086,2101.10361}. When drawing string diagrams, our direction conventions are the following: composition of 1-morphisms is from right to left; composition of 2-morphisms is from bottom to top. Moreover, we will generally suppress the labels of objects, as these can be reconstructed from the labels on morphisms. 
For example, the laxly commuting square $\alpha : kh \Rightarrow gf$
would be drawn as:
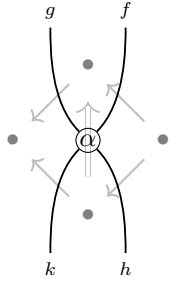
\[ 
\begin{tikzpicture}
  \path 
  (+1,0) node[gray] (e) {$\bullet$}
  (-1,0) node[gray] (w) {$\bullet$}
  (0,+1) node[gray] (n) {$\bullet$}
  (0,-1) node[gray] (s) {$\bullet$}
  ;
  \draw[grayarrow] (e) -- (n);
  \draw[grayarrow] (e) -- (s);
  \draw[grayarrow] (n) -- (w);
  \draw[grayarrow] (s) -- (w);
  \draw[graytwoarrow] (s) -- (n);
  \path (0,0) node[draw,circle,inner sep=0.5pt] (alpha) {$\alpha$}
  +(-.5,-1.5) node[anchor=north] (k) {$\scriptstyle k$}
  +(+.5,-1.5) node[anchor=north] (h) {$\scriptstyle h$}
  +(-.5,+1.5) node[anchor=south] (g) {$\scriptstyle g$}
  +(+.5,+1.5) node[anchor=south] (f) {$\scriptstyle f$}
  ;
  \draw[string] (k) .. controls ++(0,+.5) and +(-.5,-.5) .. (alpha);
  \draw[string] (h) .. controls ++(0,+.5) and +(+.5,-.5) .. (alpha);
  \draw[string] (g) .. controls ++(0,-.5) and +(-.5,+.5) .. (alpha);
  \draw[string] (f) .. controls ++(0,-.5) and +(+.5,+.5) .. (alpha);
\end{tikzpicture}
\]

Suppose now that $a : \globe_1 \to \cA$ and $b : \globe_1 \to \cB$ are 1-morphisms in strict $\omega$-categories $\cA$ and $\cB$. Their tensor product is the laxly commuting square
$$ a \otimes b : \bigl(a \otimes \operatorname{t}(b)\bigr) \underset{\operatorname{s}(a) \otimes \operatorname{t}(b)}\circ \bigl(\operatorname{s}(a) \otimes b\bigr) \Rightarrow \bigl(\operatorname{t}(a) \otimes b\bigr) \underset{\operatorname{t}(a) \otimes \operatorname{s}(b)}\circ \bigl(a \otimes \operatorname{s}(b)\bigr).$$
Here we abbreviate $\operatorname{s}(-) = \operatorname{source}(-)$ and $\operatorname{t}(-) = \operatorname{target}(-)$, and this $a \otimes b$ is a 2-cell between 1-cells of shape $\operatorname{s}(a) \otimes \operatorname{s}(b) \to \operatorname{t}(a) \otimes \operatorname{t}(b)$. We will draw this 2-cell simply as a crossing between an edge labeled $a$, running from southwest to northeast, and an edge labeled $b$, running from southeast to northwest:
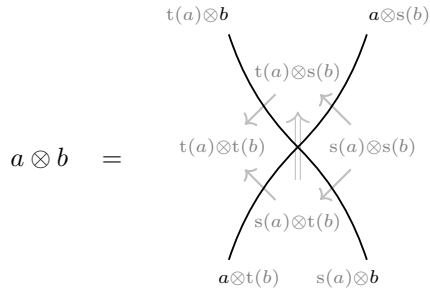
\[ 
a \otimes b \quad = \quad 
\begin{tikzpicture}[baseline=(baseline)]
  \path 
  (+1,0) node[gray] (e) {$\scriptstyle\operatorname{s}(a) \otimes \operatorname{s}(b)$}
  (-1,0) node[gray] (w) {$\scriptstyle\operatorname{t}(a) \otimes \operatorname{t}(b)$}
  (0,+1) node[gray] (n) {$\scriptstyle\operatorname{t}(a) \otimes \operatorname{s}(b)$}
  (0,-1) node[gray] (s) {$\scriptstyle\operatorname{s}(a) \otimes \operatorname{t}(b)$}
  ;
  \draw[grayarrow] (e) -- (n);
  \draw[grayarrow] (e) -- (s);
  \draw[grayarrow] (n) -- (w);
  \draw[grayarrow] (s) -- (w);
  \draw[graytwoarrow] (s) -- (n);
  \path (0,0) coordinate (crossing) +(0,-.25) coordinate (baseline)
  +(-1,-1.5) node[anchor=north] (k) {$\scriptstyle \phantom{\color{gray}{\operatorname{s}(a)}{\otimes}} a {\color{gray}{\otimes}{\operatorname{t}(b)}}$}
  +(+1,-1.5) node[anchor=north] (h) {$\scriptstyle {\color{gray}{\operatorname{s}(a)}{\otimes}} b \phantom{\color{gray}{\otimes}{\operatorname{t}(b)}}$}
  +(-1,+1.5) node[anchor=south] (g) {$\scriptstyle {\color{gray}{\operatorname{t}(a)}{\otimes}} b \phantom{\color{gray}{\otimes}{\operatorname{t}(b)}}$}
  +(+1,+1.5) node[anchor=south] (f) {$\scriptstyle \phantom{\color{gray}{\operatorname{s}(a)}{\otimes}} a {\color{gray}{\otimes}{\operatorname{s}(b)}}$}
  ;
  \draw[string] (k) .. controls ++(+.25,+.75) and +(-.5,-.5) .. (crossing);
  \draw[string] (h) .. controls ++(-.25,+.75) and +(+.5,-.5) .. (crossing);
  \draw[string] (g) .. controls ++(+.25,-.75) and +(-.5,+.5) .. (crossing);
  \draw[string] (f) .. controls ++(-.25,-.75) and +(+.5,+.5) .. (crossing);
\end{tikzpicture}
\]

The tensor products of a 2-cell $\alpha : a \Rightarrow a'$  with a 1-cell $b$  and of a  1-cell $a$  with a 2-cell $\beta : b \Rightarrow b'$ are, respectively, the following 3-cells:
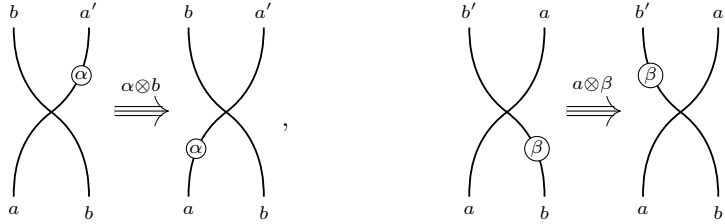
\[
\begin{tikzpicture}[baseline=(baseline),yscale=.75]
  \path (0,0) coordinate (crossing) +(0,-.25) coordinate (baseline)
  +(-.5,-1.5) node[anchor=north] (k) {$\scriptstyle  a $}
  +(+.5,-1.5) node[anchor=north] (h) {$\scriptstyle  b $}
  +(-.5,+1.5) node[anchor=south] (g) {$\scriptstyle  b $}
  +(+.5,+1.5) node[anchor=south] (f) {$\scriptstyle  a' $}
  ;
  \draw[string] (k) .. controls ++(0,+.5) and +(-.5,-.5) .. (crossing);
  \draw[string] (h) .. controls ++(0,+.5) and +(+.5,-.5) .. (crossing);
  \draw[string] (g) .. controls ++(0,-.5) and +(-.5,+.5) .. (crossing);
  \draw[string] (f) .. controls ++(0,-.5) and +(+.5,+.5) .. (crossing);
  \path (+.4,+.65) node[draw,circle,inner sep=0.5pt,fill,fill=white] (bead) {$\scriptstyle\alpha$};
\end{tikzpicture}
\overset{\alpha \otimes b}\threearrowlong
\begin{tikzpicture}[baseline=(baseline),yscale=.75]
  \path (0,0) coordinate (crossing) +(0,-.25) coordinate (baseline)
  +(-.5,-1.5) node[anchor=north] (k) {$\scriptstyle  a $}
  +(+.5,-1.5) node[anchor=north] (h) {$\scriptstyle  b $}
  +(-.5,+1.5) node[anchor=south] (g) {$\scriptstyle  b $}
  +(+.5,+1.5) node[anchor=south] (f) {$\scriptstyle  a' $}
  ;
  \draw[string] (k) .. controls ++(0,+.5) and +(-.5,-.5) .. (crossing);
  \draw[string] (h) .. controls ++(0,+.5) and +(+.5,-.5) .. (crossing);
  \draw[string] (g) .. controls ++(0,-.5) and +(-.5,+.5) .. (crossing);
  \draw[string] (f) .. controls ++(0,-.5) and +(+.5,+.5) .. (crossing);
  \path (-.4,-.65) node[draw,circle,inner sep=0.5pt,fill,fill=white] (bead) {$\scriptstyle\alpha$};
\end{tikzpicture}
, \qquad \qquad \qquad
\begin{tikzpicture}[baseline=(baseline),yscale=.75]
  \path (0,0) coordinate (crossing) +(0,-.25) coordinate (baseline)
  +(-.5,-1.5) node[anchor=north] (k) {$\scriptstyle  a $}
  +(+.5,-1.5) node[anchor=north] (h) {$\scriptstyle  b $}
  +(-.5,+1.5) node[anchor=south] (g) {$\scriptstyle  b' $}
  +(+.5,+1.5) node[anchor=south] (f) {$\scriptstyle  a $}
  ;
  \draw[string] (k) .. controls ++(0,+.5) and +(-.5,-.5) .. (crossing);
  \draw[string] (h) .. controls ++(0,+.5) and +(+.5,-.5) .. (crossing);
  \draw[string] (g) .. controls ++(0,-.5) and +(-.5,+.5) .. (crossing);
  \draw[string] (f) .. controls ++(0,-.5) and +(+.5,+.5) .. (crossing);
  \path (+.4,-.65) node[draw,circle,inner sep=0.5pt,fill,fill=white] (bead) {$\scriptstyle\beta$};
\end{tikzpicture}
\overset{a \otimes \beta}\threearrowlong
\begin{tikzpicture}[baseline=(baseline),yscale=.75]
  \path (0,0) coordinate (crossing) +(0,-.25) coordinate (baseline)
  +(-.5,-1.5) node[anchor=north] (k) {$\scriptstyle  a $}
  +(+.5,-1.5) node[anchor=north] (h) {$\scriptstyle  b $}
  +(-.5,+1.5) node[anchor=south] (g) {$\scriptstyle  b' $}
  +(+.5,+1.5) node[anchor=south] (f) {$\scriptstyle  a $}
  ;
  \draw[string] (k) .. controls ++(0,+.5) and +(-.5,-.5) .. (crossing);
  \draw[string] (h) .. controls ++(0,+.5) and +(+.5,-.5) .. (crossing);
  \draw[string] (g) .. controls ++(0,-.5) and +(-.5,+.5) .. (crossing);
  \draw[string] (f) .. controls ++(0,-.5) and +(+.5,+.5) .. (crossing);
  \path (-.4,+.65) node[draw,circle,inner sep=0.5pt,fill,fill=white] (bead) {$\scriptstyle\beta$};
\end{tikzpicture}
\]
In other words, when working in $\cA \otimes \cB$,
 one can pull $\cA$-cells from  northeast to southwest across wires from $\cB$, and one can pull $\cB$-cells from  southeast to northwest across wires in $\cA$; pulling across a wire is a typically non-invertible operation. The tensor product of two 2-cells is a non-invertible compatibility between these two pulling operations:
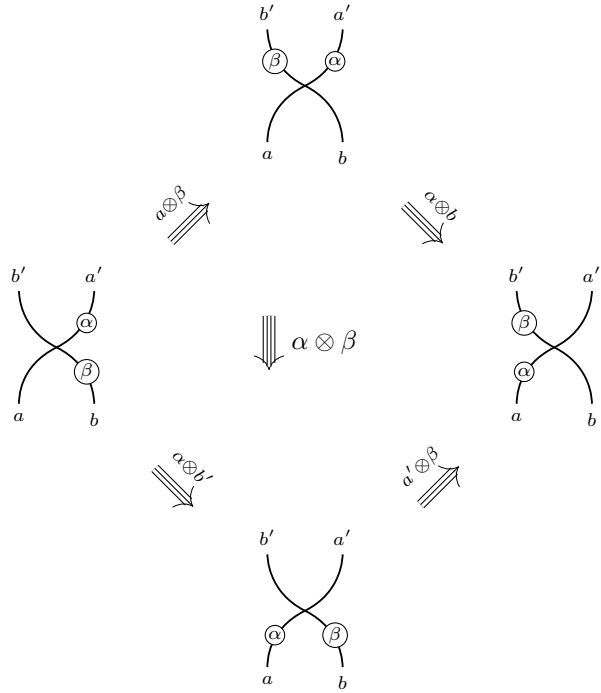
\[
\begin{array}{ccccc}
  && \begin{tikzpicture}[baseline=(baseline),yscale=.5]
  \path (0,0) coordinate (crossing) +(0,-.25) coordinate (baseline)
  +(-.5,-1.5) node[anchor=north] (k) {$\scriptstyle  a $}
  +(+.5,-1.5) node[anchor=north] (h) {$\scriptstyle  b $}
  +(-.5,+1.5) node[anchor=south] (g) {$\scriptstyle  b' $}
  +(+.5,+1.5) node[anchor=south] (f) {$\scriptstyle  a' $}
  ;
  \draw[string] (k) .. controls ++(0,+.5) and +(-.5,-.5) .. (crossing);
  \draw[string] (h) .. controls ++(0,+.5) and +(+.5,-.5) .. (crossing);
  \draw[string] (g) .. controls ++(0,-.5) and +(-.5,+.5) .. (crossing);
  \draw[string] (f) .. controls ++(0,-.5) and +(+.5,+.5) .. (crossing);
  \path (+.4,+.65) node[draw,circle,inner sep=0.5pt,fill,fill=white] (bead) {$\scriptstyle\alpha$};
  \path (-.4,+.65) node[draw,circle,inner sep=0.5pt,fill,fill=white] (bead) {$\scriptstyle\beta$};
\end{tikzpicture}
 && \\
 & \rotatebox[origin=c]{45}{$\overset{a \otimes \beta}\threearrowlong$}&& \rotatebox[origin=c]{-45}{$\overset{\alpha \otimes b}\threearrowlong$}& \\
 \begin{tikzpicture}[baseline=(baseline),yscale=.5]
  \path (0,0) coordinate (crossing) +(0,-.25) coordinate (baseline)
  +(-.5,-1.5) node[anchor=north] (k) {$\scriptstyle  a $}
  +(+.5,-1.5) node[anchor=north] (h) {$\scriptstyle  b $}
  +(-.5,+1.5) node[anchor=south] (g) {$\scriptstyle  b' $}
  +(+.5,+1.5) node[anchor=south] (f) {$\scriptstyle  a' $}
  ;
  \draw[string] (k) .. controls ++(0,+.5) and +(-.5,-.5) .. (crossing);
  \draw[string] (h) .. controls ++(0,+.5) and +(+.5,-.5) .. (crossing);
  \draw[string] (g) .. controls ++(0,-.5) and +(-.5,+.5) .. (crossing);
  \draw[string] (f) .. controls ++(0,-.5) and +(+.5,+.5) .. (crossing);
  \path (+.4,+.65) node[draw,circle,inner sep=0.5pt,fill,fill=white] (bead) {$\scriptstyle\alpha$};
  \path (+.4,-.65) node[draw,circle,inner sep=0.5pt,fill,fill=white] (bead) {$\scriptstyle\beta$};
\end{tikzpicture} 
&&
\rotatebox[origin=c]{-90}{$\fourarrowlong$}\; \raisebox{3pt}{$\alpha \otimes \beta$}
&&
\begin{tikzpicture}[baseline=(baseline),yscale=.5]
  \path (0,0) coordinate (crossing) +(0,-.25) coordinate (baseline)
  +(-.5,-1.5) node[anchor=north] (k) {$\scriptstyle  a $}
  +(+.5,-1.5) node[anchor=north] (h) {$\scriptstyle  b $}
  +(-.5,+1.5) node[anchor=south] (g) {$\scriptstyle  b' $}
  +(+.5,+1.5) node[anchor=south] (f) {$\scriptstyle  a' $}
  ;
  \draw[string] (k) .. controls ++(0,+.5) and +(-.5,-.5) .. (crossing);
  \draw[string] (h) .. controls ++(0,+.5) and +(+.5,-.5) .. (crossing);
  \draw[string] (g) .. controls ++(0,-.5) and +(-.5,+.5) .. (crossing);
  \draw[string] (f) .. controls ++(0,-.5) and +(+.5,+.5) .. (crossing);
  \path (-.4,-.65) node[draw,circle,inner sep=0.5pt,fill,fill=white] (bead) {$\scriptstyle\alpha$};
  \path (-.4,+.65) node[draw,circle,inner sep=0.5pt,fill,fill=white] (bead) {$\scriptstyle\beta$};
\end{tikzpicture} \\
 & \rotatebox[origin=c]{-45}{$\overset{\alpha \otimes b'}\threearrowlong$}&& \rotatebox[origin=c]{45}{$\overset{a' \otimes \beta}\threearrowlong$}& \\
  && \begin{tikzpicture}[baseline=(baseline),yscale=.5]
  \path (0,0) coordinate (crossing) +(0,-.25) coordinate (baseline)
  +(-.5,-1.5) node[anchor=north] (k) {$\scriptstyle  a $}
  +(+.5,-1.5) node[anchor=north] (h) {$\scriptstyle  b $}
  +(-.5,+1.5) node[anchor=south] (g) {$\scriptstyle  b' $}
  +(+.5,+1.5) node[anchor=south] (f) {$\scriptstyle  a' $}
  ;
  \draw[string] (k) .. controls ++(0,+.5) and +(-.5,-.5) .. (crossing);
  \draw[string] (h) .. controls ++(0,+.5) and +(+.5,-.5) .. (crossing);
  \draw[string] (g) .. controls ++(0,-.5) and +(-.5,+.5) .. (crossing);
  \draw[string] (f) .. controls ++(0,-.5) and +(+.5,+.5) .. (crossing);
  \path (-.4,-.65) node[draw,circle,inner sep=0.5pt,fill,fill=white] (bead) {$\scriptstyle\alpha$};
  \path (+.4,-.65) node[draw,circle,inner sep=0.5pt,fill,fill=white] (bead) {$\scriptstyle\beta$};
\end{tikzpicture}
 && \end{array}
\]

\subsection{The lax tensor product \texorpdfstring{$\otimes$}{otimes} on \texorpdfstring{$(\infty,\infty)$}{(oo,oo)}-categories} \label{subsec:otimesinfty}

In~\cite{2311.00205}, Campion uses the strict lax product $\otimes^\strict$ on $\strCat_\omega$ to define an analogous product $\otimes$ on $\Cat_{(\infty,\infty)}$. Campion's product is built as follows. Let $\fancysquare \subset \strCat_\omega$ denote the full subcategory on the $\otimes^\strict$-powers of $\globe_1$: 

\begin{equation}
\label{eq:GrayCubes}
\arraycolsep=20pt
\begin{array}{ccc}
\begin{tikzcd}[sep=2em]
\bullet
\end{tikzcd}
&
\begin{tikzcd}[sep=2em]
\bullet \arrow[r]&\bullet
\end{tikzcd}
&
\begin{tikzcd}[sep=2em]
\bullet \arrow[r]&\bullet
\end{tikzcd}
\otimes 
\begin{tikzcd}[sep=2em]\bullet \arrow[d] \\ \bullet
\end{tikzcd}
=
\begin{tikzcd}[sep=2em]
\bullet \arrow[r] \arrow[d] &\arrow[d] \bullet \\
\bullet \arrow[r] \arrow[Rightarrow, ur,  shorten <=3pt, shorten >=3pt] & \bullet
\end{tikzcd}
\\
\globe_1^{\otimes 0}
&
\globe_1^{\otimes 1}
& \globe_1^{\otimes 2}
\end{array}
\hspace{0.5cm}
\cdots
\end{equation}
These powers are all gaunt, and so we may equally consider $\fancysquare$ as a full subcategory of $\Cat_{(\infty,\infty)}$. Campion proves in \cite{2209.09376} that $\fancysquare \subset \Cat_{(\infty,\infty)}$ is dense; indeed, Campion proves that the idempotent completion of $\fancysquare$ contains $\Theta$. 
Since $\fancysquare \subset \Cat_{(\infty,\infty)}$ is dense,  the lax tensor product on $\fancysquare$ extends in at most one way to a product $\otimes$ on all of $\Cat_{(\infty,\infty)}$ which is cocontinuous in each variable; Theorem~A of~\cite{2311.00205} says that this extension does indeed exist.

\begin{warn}\label{warn:GrayGaunt}

It is expected, but not proven in~\cite{2311.00205}, that the $\infty$-categorical and strict-categorical tensor products agree when restricted along $\Gaunt \subset \Cat_{(\infty,\infty)}$ or $\Gaunt \subset \strCat_\omega$. By~\cite[Thm.~3.14]{2311.00205} this is equivalent to the assertion that the $\infty$-categorical tensor product of two gaunt $\omega$-categories is again a gaunt $\omega$-category (also see~\cite[Assumption~3.5(3)]{2002.01037}), which is unknown at the time of writing. In the sequel, when we write~$\otimes$, we will mean the $\infty$-categorical tensor product, and we will write~$\otimes^{\strict}$ for the strict version (when they are not known to agree).

That said, since Joyal's category $\Theta$ of pasting diagrams is a full subcategory of the idempotent completion of $\fancysquare$, see~\cite{2209.09376}, it does follow from Campion's construction that the $\infty$-categorical and strict-categorical tensor products agree on the monoidal closure of $\Theta \subset \Gaunt$. In particular, $\theta_p \otimes \theta_q$ can be computed strictly, and agrees with the gaunt $\omega$-category denoted $\theta_{(p);(q)}$ in \cite{JFS}.
\end{warn}

\begin{definition}
A \emph{lax square} in an $(\infty, \infty)$-category $\cC$ is a functor $\theta_1 \otimes \theta_1 \to \cC$, the \emph{space of lax squares in $\cC$} is $\Map(\theta_1 \otimes \theta_1, \cC)$. 
\end{definition}
Throughout this paper, we will often informally picture lax squares as diagrams in $\cC$ of the following shape:
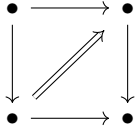
\[\begin{tikzcd}[sep=3em]
\bullet \arrow[r] \arrow[d] &\arrow[d] \bullet \\
\bullet \arrow[r] \arrow[Rightarrow, ur,  shorten <=3pt, shorten >=3pt] & \bullet
\end{tikzcd}
\]

For any pair of $(\infty, \infty)$-categories $\cA, \cB$, the functors $\cA \otimes \cB \to \cA \otimes * \simeq \cA$ and $\cA \otimes \cB \to * \otimes \cB \simeq \cB$ induce a canonical functor $\cA \otimes \cB \to \cA \times \cB$. 

\begin{lemma} \label{lem:laxvstrongsquares} Composition with the canonical functor $\globe_1 \otimes \globe_1 \to \globe_1 \times\globe_1$ exhibits $\Map(\globe_1 \times \globe_1 ,\cC)$ as a full subspace of $\Map(\globe_1 \otimes \globe_1, \cC)$ on those lax squares whose filling $2$-cell is invertible. 
\end{lemma}
\begin{proof}
The diagram 
\[ \begin{tikzcd}
\globe_2 \arrow[r]  \arrow[d]& {\globe_1}\arrow[d]\\ 
{\globe_1} \otimes{ \globe_1 }\arrow[r] & {\globe_1} \times {\globe_1}
\end{tikzcd}
\]
is a pushout in $\Cat_{(\infty,2)}$ by~\cite[Lem.~2.4.2]{2404.03971} and hence in $\Cat_{(\infty, \infty)}$. In particular, $\Map(\globe_1 \times \globe_1, \cC)$ is the pullback $\Map(\globe_1 \otimes \globe_1, \cC) \times_{\Map(\globe_2, \cC)}\Map(\globe_1, \cC)$, i.e.\ the full subspace of $\Map(\globe_1\otimes \globe_1,\cC)$ whose filling $2$-cell is invertible. \end{proof}

Notice that we may build $\theta_2$ by collapsing the vertical $1$-cells in $\theta_1 \otimes \theta_1$. The following key technical lemma generalizes this fact to arbitrary $(\infty, \infty)$-categories:

\begin{lemma}[{\cite[Lem.~3.8]{2311.00205}}]\label{lemma:sigmaaspushout}
  The categorical suspension functor $\Sigma: \Cat_{(\infty, \infty)} \to \Cat_{(\infty, \infty)}^{**}$  agrees with the functor sending an $(\infty, \infty)$-category $\cC$ to the pushout
  \begin{equation}\label{eq:pushoutsigma}
   \colim \left( \begin{tikzcd}
\cC \sqcup \cC \simeq \cC \otimes \partial\globe_1 \arrow[r]  \arrow[d]& \cC \otimes \globe_1 \\
\{0\} \sqcup \{1\} \simeq * \otimes \partial \globe_1  & 
\end{tikzcd} \right) 
\end{equation}
with pointing by $* \sqcup * =\partial \globe_1$ given by the colimit inclusion.
\end{lemma}
\begin{proof}Since both $\Sigma$ as well as the functor constructed from the pushout and $\otimes$ are cocontinuous in $\cC \in \Cat_{(\infty, \infty)}$, it suffices to check the claim on the dense subcategory $\Theta \subset \Cat_{(\infty,\infty)}$. For $\cC \in \Theta$, it is shown in~\cite[Lem.~3.8]{2311.00205} that the pushout~\eqref{eq:pushoutsigma} computed in $\Cat_{(\infty, \infty)}$ agrees with the pushout computed in $\Gaunt$, where it agrees with the categorical suspension, e.g.\ by~\cite[Cor.~B.6.6]{MR4146146}. 
\end{proof}

We end this section by recording a few technical results about pushouts like the one in Lemma~\ref{lemma:sigmaaspushout}. 

\begin{prop}\label{prop:trickypushout}
For all $n \geq 1$, the following is a pushout in $\Cat_{(\infty, n)}$: 
\begin{equation} \label{eq:pushoutcn}
\begin{tikzcd}
\globe_n \sqcup \globe_n \simeq \globe_{n} \otimes \partial\globe_1 \arrow[r]  \arrow[d]& L_n(\globe_n \otimes \globe_1) \arrow[d]\\
\globe_{n-1} \sqcup \globe_{n-1} \simeq \globe_{n-1} \otimes \partial \globe_1 \arrow[r] & \globe_{n-1} \otimes \globe_1
\end{tikzcd}
\end{equation}
\end{prop}

\begin{proof}
We prove the statement by induction. For $n=1$, the statement follows by applying the cocontinuous functor $L_1: \Cat_{(\infty, \infty)} \to \Cat_{(\infty, 1)}$ to Lemma~\ref{lemma:sigmaaspushout} with $\cC = \globe_1$.

Suppose now that we have already proven that~\eqref{eq:pushoutcn} is a pushout for some $n\geq 1$. Consider the following commuting diagram:
\[
\begin{tikzcd}[column sep = small]
& {\Sigma \globe_{n-1} \otimes \partial\globe_1} \arrow[r] \arrow[d, ""{name=D11}] & {\Sigma \globe_{n-1} \otimes^{\mathrm{funny}} \globe_1} \arrow[d, ""{name=D12}] \\
&{ \Sigma \globe_n \otimes \partial \globe_1} \arrow[r] \arrow[d, ""{name=D21}]& {\Sigma \globe_n \otimes^{\mathrm{funny}} \globe_1} \arrow[d, ""{name=D22}] \\
 {\Sigma \globe_n \otimes \partial \globe_1 \simeq \Sigma(\globe_n \otimes \partial \globe_1) } \arrow[r] \arrow[d, ""{name=X1}] &{ \Sigma L_n (\globe_n \otimes \globe_1) \simeq L_{n+1} (\Sigma(\globe_n \otimes \globe_1))} \arrow[r] \arrow[d, ""{name=X2}] & {L_{n+1}((\Sigma \globe_n) \otimes \globe_1) }\arrow[d, ""{name=X3}] \\
 {\Sigma \globe_{n-1} \otimes \partial \globe_1 \simeq \Sigma(\globe_{n-1} \otimes \partial \globe_1)} \arrow[r] &{ \Sigma(\globe_{n-1} \otimes \globe_1)} \arrow[r] &{(\Sigma \globe_{n-1}) \otimes \globe_1 }
 \arrow[phantom, from =D11, to =D12, "(A)"]
  \arrow[phantom, from =D21, to =D22, "(B)"]
    \arrow[phantom, from =X1, to =X2, "(C)"]
    \arrow[phantom, from =X2, to =X3, "(D)"]
\end{tikzcd}
\]
Here, 
for $\cC \in \Cat_{(\infty, n)}$, we follow~\cite[Def.~3.7]{2311.00205} and write 
\begin{equation}\label{eq:funnytensor}
\Sigma \cC \otimes^{\mathrm{funny}}\globe_1 := (\Sigma\cC \cup_{*} \globe_1) \cup_{ \partial \globe_1} (\globe_1\cup_{*} \Sigma \cC) = \colim\left(\begin{tikzcd}
 {\globe_1}  &* \arrow[l, "1"'] \arrow[r, "0"]& \Sigma \cC \\
* \arrow[u, "0"]  \arrow[d, "0"']& &* \arrow[u, "1"'] \arrow[d, "1"]\\
 \Sigma \cC &* \arrow[l, "1"'] \arrow[r, "0"]& {\globe_1} 
\end{tikzcd}.\right)
\end{equation}
 The vertical functors in square $(A)$ are induced by the boundary inclusions $\globe_{n-1} \to \globe_n$ given by $\Sigma^{n-1}(\{0\} \hookrightarrow \globe_1 )$. 
The square $(C)$ is a  pushout diagram since it arises by applying $\Sigma$  to the pushout diagram~\eqref{eq:pushoutcn}. The square $(A)$ is a pushout using the definition~\eqref{eq:funnytensor} and the fact that colimits commute with colimits. The big square $(ABD)$ and the square $(B)$ are pushout squares by~\cite[Lem.~3.6]{2311.00205}, respectively its $(n+1)$-localization. Hence, the glued square $(AB)$ is a pushout, thus $(D)$ is a pushout and therefore $(CD)$ is a pushout, proving the statement for $n+1$.
\end{proof}

\begin{lemma} \label{lem:Sigmapreserves} Let $0 \leq n \leq \infty$. If $\cA \to \cB$ and $\cX \to \cY$ are functors of $(\infty, n)$-categories so that
\begin{equation}\label{eq:assumed-pushout}\begin{tikzcd}
L_n(\cA \otimes \cX) \arrow[r] \arrow[d] & L_n(\cA\otimes \cY)\arrow[d]\\
L_n(\cB \otimes \cX) \arrow[r] & L_n(\cB \otimes \cY)
\end{tikzcd}
\end{equation}
is a pushout square in $\Cat_{(\infty, n)}$, then so is 
\[\begin{tikzcd}
L_n(\cA \otimes  \Sigma \cX) \arrow[r] \arrow[d] & L_n( \cA\otimes  \Sigma \cY )\arrow[d]\\
L_n(\cB \otimes  \Sigma \cX ) \arrow[r] & L_n(\cB \otimes  \Sigma \cY)
\end{tikzcd}.
\]
\end{lemma}
\begin{proof}
  For $(\infty, n)$-categories $\cC, \cD$, let $\cC \otimes_n \cD := L_n (\cC \otimes \cD)$. By~\cite[Cor.~3.15]{2311.00205}, $\otimes_n$ is a biclosed monoidal structure on $\Cat_{(\infty,n)}$, and $L_n : \Cat_{(\infty, \infty)} \to \Cat_{(\infty, n)}$ is monoidal. Applying $L_n$ to the pushout formula in Lemma~\ref{lemma:sigmaaspushout} for $\Sigma$ gives:
  \begin{equation}\label{eq:sigmaformula}L_n(\Sigma \cX)= \cX \otimes_n \globe_1 \underset{\cX \otimes_n \partial \globe_1}\cup \partial \globe_1.\end{equation}
Commuting colimits past one another, we find that the pushout 
\[
\colim\left(
\begin{tikzcd}
L_n(\cA \otimes  \Sigma \cX) \arrow[r] \arrow[d] & L_n( \cA\otimes  \Sigma \cY )\\
L_n(\cB \otimes  \Sigma \cX )  & 
\end{tikzcd}\right)
\]
is also computed by
\[ \simeq
\colim\left(
\begin{tikzcd}[sep=small]
(\cA \otimes_n \cY \otimes_n \partial \globe_1) \underset{(\cA \otimes_n \cX \otimes_n \partial \globe_1)}\cup (\cB \otimes_n \cX \otimes_n \partial \globe_1) \arrow[r] \arrow[d] &
(\cA \otimes_n \partial \globe_1) \underset{\cA \otimes_n \partial \globe_1}\cup \cB \otimes_n \partial \globe_1 \\
(\cA \otimes_n \cY \otimes_n \globe_1) \underset{(\cA \otimes_n \cX \otimes_n \globe_1)}\cup (\cB \otimes_n \cX \otimes_n \globe_1) &
\end{tikzcd}\right).
\]
But $\otimes_n$ preserves colimits in each variable, and since~\eqref{eq:assumed-pushout} is a pushout,  this cospan is equivalent to 
\[
\simeq 
\colim\left(
\begin{tikzcd}
\cB \otimes_n \cY \otimes_n \partial \globe_1 \arrow[r] \arrow[d] &  \cB \otimes_n \partial \globe_1 \\
\cB \otimes_n \cY \otimes_n \globe_1   & 
\end{tikzcd}\right) 
\]
whose colimit is $\cB \otimes_n L_n \Sigma \cY$ by~\eqref{eq:sigmaformula}. 
\end{proof}

\subsection{(Op)lax natural transformations}\label{subsec:oplax}
An $\infty$-category~$\cA$ is \define{presentably monoidal} if~$\cA$ is presentable (in the sense of \cite[Def.~5.5.0.1]{HTT}) and equipped with a monoidal structure $\otimes$ which preserves small colimits in each variable. The $\infty$-categorical adjoint functor theorem~\cite[Cor.~5.5.2.9]{HTT} implies that any presentably monoidal category is \define{biclosed}: for every $A \in \cA$, the endofunctors $A \otimes -$ and $-\otimes A$ of $\cA$ admit right adjoints, called the  \define{internal homs} associated to $\otimes$. Conversely, a biclosed monoidal structure on a presentable category is presentably monoidal.  Here, we will focus on the presentable category $\cA= \Cat_{(\infty, \infty)}$ with its lax tensor product $\otimes$ which by construction preserves colimits in each variable and hence is biclosed. 

\begin{definition}\label{defn:funoplax}
 The internal homs of the biclosed monoidal category $\Cat_{(\infty, \infty)}$ equipped with the lax tensor product are called: 
 \begin{itemize}
  \item $\Funlax(\cC,-)$ is the right adjoint of $ \cC \otimes -: \Cat_{(\infty, \infty)} \to \Cat_{(\infty, \infty)}$.
  \item $\Funoplax(\cC,-)$ is the right adjoint of $ - \otimes \cC: \Cat_{(\infty, \infty)} \to \Cat_{(\infty, \infty)}$.
\end{itemize}
\end{definition}

Being internal homs, $\Funlax(\cC,-)$ and $\Funoplax(\cC,-)$ are contravariantly functorial in $\cC$. For comparison, we remind that there is a third, this time symmetric, closed monoidal structure on $\Cat_{(\infty,\infty)}$, namely its Cartesian product $\times$. Its corresponding internal hom is:
\begin{itemize}
  \item $\Fun(\cC, -) = \Funstrong(\cC,-)$ is the right adjoint of $ \cC \times - \simeq - \times \cC: \Cat_{(\infty, \infty)} \to \Cat_{(\infty, \infty)}$.
\end{itemize}
Note that the terminal category $*$ is the monoidal unit for both $\times$ and $\otimes$. It follows that $\Funlax(\cC,\cD)$, $\Funoplax(\cC,\cD)$, and $\Funstrong(\cC,\cD)$ all have the same space of objects:
\begin{equation}
  \ob \bigl( \Funlax(\cC,\cD)\bigr) \simeq \ob \bigl( \Funoplax(\cC,\cD) \bigr) \simeq \ob \bigl( \Funstrong(\cC,\cD) \bigr) \simeq \Map(\cC,\cD)
\end{equation}
The (higher) morphisms in $\Funlax(\cC,\cD)$, $\Funoplax(\cC,\cD)$, and $\Funstrong(\cC,\cD)$ are, by definition, the \emph{lax/oplax/strong} natural transformations (and higher transfors) between functors. 
Informally, given functors $F, G: \cC \to \cD$, a lax/oplax/strong natural transformation amounts to the data of a $1$-morphism $Fc \to Gc$ for every $c\in \cC$, together with for every $1$-morphism $f:c \to c'$ in $\cC$ a $2$-morphism
\[
\arraycolsep=20pt
\begin{array}{ccc}
\begin{tikzcd}[sep=2em]
Fc\arrow[r] \arrow[d] &\arrow[d] Fc' \\
Gc \arrow[r] \arrow[Rightarrow, ur,  shorten <=3pt, shorten >=3pt] & Gc'
\end{tikzcd}
&
\begin{tikzcd}[sep=2em]
Fc\arrow[r] \arrow[d] &\arrow[d] Fc' \\
Gc \arrow[r] \arrow[Leftarrow, ur,  shorten <=3pt, shorten >=3pt] & Gc'
\end{tikzcd}
&
\begin{tikzcd}[sep=2em]
Fc\arrow[r] \arrow[d] &\arrow[d] Fc' \\
Gc \arrow[r] \arrow[phantom, "\simeq" {rotate=45}, ur,  shorten <=3pt, shorten >=3pt] & Gc'
\end{tikzcd}
\\
\text{lax}
&
\text{oplax}
& \text{strong}
\end{array}
\]
and corresponding higher morphisms for higher morphisms in $\cC$.

\begin{example}\label{exm:funlaxone}
For an $(\infty, \infty)$-category $\cC$, we explicitly unpack the first few layers of morphisms of $\Funlax(\globe_1, \cC)$ and $\Funoplax(\globe_1, \cC)$:
\begin{itemize}
\item The spaces of objects of $\Funlax(\globe_1, \cC)$ and of $\Funoplax(\globe_1, \cC)$ are given by 
\begin{gather*}
\Map(*, \Funlax(\globe_1, \cC)) \simeq \Map(\globe_1 \otimes *, \cC) \simeq \Map(\globe_1, \cC)
\\
\Map(*, \Funoplax(\globe_1, \cC)) \simeq \Map(* \otimes \globe_1, \cC) \simeq \Map(\globe_1, \cC)
\end{gather*}
In other words, an object of either $\Funlax(\globe_1, \cC)$ or  $\Funoplax(\globe_1, \cC)$ is precisely a $1$-morphism $f:a \to b$ in $\cC$.
\item The space of $1$-morphisms in $\Funlax(\globe_1, \cC)$ is given by 
\[\Map(\globe_1, \Funlax(\globe_1, \cC)) \simeq \Map(\globe_1 \otimes \globe_1, \cC) \]
with source and target restrictions induced by pre-composition with the inclusions $\globe_1 \otimes \{0\} \hookrightarrow \globe_1 \otimes \globe_1$ and $\globe_1 \otimes \{1\} \hookrightarrow \globe_1 \otimes \globe_1$.  

The space of $1$-morphisms in $\Funoplax(\globe_1, \cC)$ is given by 
\[\Map(\globe_1, \Funoplax(\globe_1, \cC)) \simeq \Map(\globe_1 \otimes \globe_1, \cC) \]
with source and target restrictions induced by pre-composition with the inclusions $\{0\} \otimes \globe_1 \hookrightarrow \globe_1 \otimes \globe_1$ and $\{1\} \otimes \globe_1 \hookrightarrow \globe_1 \otimes \globe_1$. 

Informally, given  $1$-morphisms $s$ and $t$ in $\cC$, considered as objects of $\Funlax(\globe_1, \cC)$ or $\Funoplax(\globe_1, \cC)$, respectively, $1$-morphisms $\eta:s \to t$ in $\Funlax(\globe_1, \cC)$ and $\mu:s \to t$ in $\Funoplax(\globe_1, \cC)$ are given by diagrams of the following shapes in $\cC$:
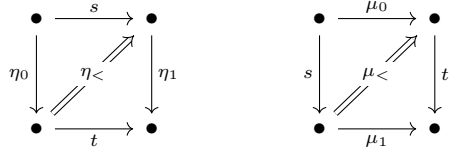
\[\begin{tikzcd}[sep=3em]
\bullet \arrow[r, "s"] \arrow[d, "\eta_0"'] &\arrow[d, "\eta_1"] \bullet \\
\bullet \arrow[r, "t"'] \arrow[Rightarrow, ur,"\eta_{<}" description] & \bullet
\end{tikzcd}
\hspace{1.5cm}
\begin{tikzcd}[sep=3em]
\bullet \arrow[r, "\mu_0"] \arrow[d, "s"'] &\arrow[d, "t"] \bullet \\
\bullet \arrow[r, "\mu_1"'] \arrow[Rightarrow, ur,"\mu_{<}" description] & \bullet
\end{tikzcd}
\]

It is worth emphasizing that the $1$-morphisms in both $\Funlax(\globe_1, \cC))$ and $\Funoplax(\globe_1, \cC))$ are the lax squares $\globe_1 \otimes \globe_1 \to \cC$. 
The difference is in the compositions, which are, respectively, the   \define{vertical} and \define{horizontal} compositions of squares.

\item  The spaces of $2$-morphisms of $\Funlax(\globe_1, \cC)$ and of $\Funoplax(\globe_1, \cC)$ are given by, respectively, 
\begin{gather*}
\Map(\globe_2, \Funlax(\globe_1, \cC)) \simeq \Map(\globe_1 \otimes \globe_2 , \cC) \\
\Map(\globe_2, \Funoplax(\globe_1, \cC)) \simeq \Map(\globe_2 \otimes \globe_1  , \cC)
\end{gather*}
In other words, a $2$-morphism in $\Funlax(\globe_1, \cC)$ between $1$-morphisms $\eta$ and $\eta'$, and a $2$-morphism in $\Funoplax(\globe_1, \cC)$ between $1$-morphisms $\mu$ and $\mu'$, may be  thought of as diagrams of the following shapes in $\cC$, respectively:
$$
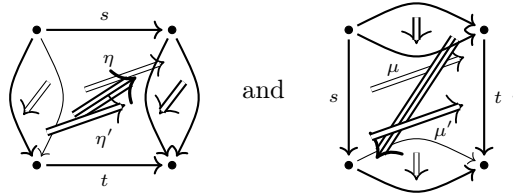
\begin{tikzpicture}[baseline=(middle),scale=0.6]
  \path node[dot] (a) {} +(3,0) node[dot] (b) {} +(0,-3) node[dot] (c) {} +(3,-3) node[dot] (d) {} +(0,-1.5) coordinate (middle);
  \draw[arrow] (a) .. controls +(-.75,-1.5) and +(-.75,1.5) .. coordinate[near end] (lt) coordinate[very near end] (lt2) (c);
  \draw[arrow,thin] (a) .. controls +(.75,-1.5) and +(.75,1.5) .. coordinate[near start] (ls) coordinate (ls2) (c);
  \draw[twoarrowlonger] (ls) -- (lt);
  \draw[arrow] (b) .. controls +(.75,-1.5) and +(.75,1.5) .. coordinate[near start] (rs) coordinate[very near start] (rs0) (d);
  \draw[twoarrow] (ls2) -- coordinate[near end] (s) node[auto] {$\scriptstyle \eta$} (rs0);
  \draw[arrow] (b) .. controls +(-.75,-1.5) and +(-.75,1.5) .. coordinate[near end] (rt) coordinate (rt0) (d);
  \draw[twoarrowlonger,thick] (rs) -- (rt);
  \draw[arrow] (a) node[anchor = south east] {} -- node[auto] {$\scriptstyle s$} (b) node[anchor = south west] {};
  \draw[arrow] (c) node[anchor = north east] {} -- node[below] {$\scriptstyle t$} (d) node[anchor = north west] {};
  \path (lt2) -- coordinate[near start] (t) (rt0);
  \draw[threearrowpart1] (t) --  (s); 
  \draw[threearrowpart2] (t) -- (s); 
  \draw[threearrowpart3] (t) -- (s); 
  \draw[twoarrow,thick] (lt2) -- node[auto,swap] {$\scriptstyle \eta'$} (rt0);
\end{tikzpicture}
\quad \text{ and } \quad
\begin{tikzpicture}[baseline=(middle),scale=0.6]
  \path node[dot] (a) {} node[anchor= south east] {} +(3,0) node[dot] (b) {} node[anchor= south west] {} +(0,-3) node[dot] (c) {} node[anchor= north east] {} +(3,-3) node[dot] (d) {} node[anchor= north west] {}+(0,-1.5) coordinate (middle);
  \path (b) +(-.35,.25) coordinate (alpha1);
  \path (c) +(.35,-.25) coordinate (alpha2);
  \draw[arrow] (a) -- coordinate (l1) node[anchor=east] {$\scriptstyle s$} coordinate[very near end] (r1) (c);
  \draw[arrow] (b) -- coordinate[very near start] (l2) coordinate (r2) node[anchor=west] (rightlable) {$\scriptstyle t$} (d);
  \draw[twoarrow,thin] (l1) -- node[auto,pos=.4,inner sep=1pt] {$\scriptstyle \mu$} (l2);
  \draw[arrow] (a) .. controls +(1.5,.75) and +(-1.5,.75) .. coordinate (ss)  (b);
  \draw[arrow,thin] (c) .. controls +(1.5,.75) and +(-1.5,.75) .. coordinate (ts)  (d);
  \draw[arrow] (c) .. controls +(1.5,-.75) and +(-1.5,-.75) .. coordinate (tt) (d);
  \draw[twoarrowlonger,thin] (ts) --  (tt);
  \draw[threearrowpart1] (alpha1) --  (alpha2);
  \draw[threearrowpart2] (alpha1) --  (alpha2);
  \draw[threearrowpart3] (alpha1) --  (alpha2);
  \draw[arrow,thick] (a) .. controls +(1.5,-.75) and +(-1.5,-.75) .. coordinate (st) (b);
  \draw[twoarrowlonger,thick] (ss) --  (st);
  \draw[twoarrow,thick] (r1) -- node[auto,swap,pos=.6,inner sep=1pt] {$\scriptstyle \mu'$} (r2);
\end{tikzpicture} 
.$$
\end{itemize}
Considering more general pasting diagrams in $\Theta$ and their strict Gray tensor products shows that composition of $1$- and $2$-morphisms works as expected.  We do not provide a proof, but expect that for an $(\infty, n)$-category $\cC$, the $(\infty, n)$-categories $\Funlax(\globe_1, \cC)$ and $\Funoplax(\globe_1, \cC)$ agree, respectively, with the $(\infty,n)$-categories $\cC^{\rightarrow}$ and $\cC^{\downarrow}$ constructed in~\cite{JFS} in terms of $n$-fold Segal spaces. 
\end{example}

\begin{lemma}\label{lem:propertiesofFunoplax:compatwithcoop}
Let $\cC$ and $\cD$ be $(\infty, \infty)$-categories. 
 There are equivalences, natural in $\cC, \cD \in \Cat_{(\infty, \infty)}$:
\begin{align}
\label{lem:propertiesofFunoplax:compatwithcoop:tensor}
\cC^{\op} \otimes \cD^{\op} & \simeq (\cD \otimes \cC)^{\op} & \cC^{\co} \otimes \cD^{\co} & \simeq (\cD \otimes \cC)^{\co} \\
 \Funoplax(\cC, \cD)^{\op} & \simeq \Funlax(\cC^{\op}, \cD^{\op}) & \Funoplax(\cC, \cD)^{\co} &\simeq \Funlax(\cC^{\co},\cD^{\co})
 \label{lem:propertiesofFunoplax:compatwithcoop:fun}
\end{align}
\end{lemma}
\begin{proof}
The autoequivalence $(-)^{\op}$ and $(-)^{\co}$ restrict to autoequivalences of the full subcategory $\fancysquare \subseteq \Cat_{(\infty, \infty)}$.  Since $\fancysquare \subseteq \Cat_{(\infty, \infty)}$ is a dense~\cite{2209.09376} full monoidal subcategory, $\otimes$ is cocontinuous in both variables, and $(-)^{\op}$ and $(-)^{\co}$ are equivalences and hence in particular cocontinuous, it suffices to provide the equivalences \eqref{lem:propertiesofFunoplax:compatwithcoop:tensor} on the strict $1$-category $\fancysquare$. There, it follows from the corresponding property of the strict Gray tensor product, which is established in~\cite[Prop.~A.22]{MR4146146}. The  equivalences \eqref{lem:propertiesofFunoplax:compatwithcoop:fun} follow by adjunction.
\end{proof}
\begin{lemma}\label{lem:propertiesofFunoplax:others}
Let $\cC$ and $\cD$ be $(\infty, \infty)$-categories. 
\begin{enumerate}
\item 
If $\cC$ and $\cD$ are $(\infty, n)$-categories, then so is $\Fun^{\mathrm{(op)lax}}(\cC, \cD)$. 
\item  If $\cC$ and $\cD$ are gaunt, then so is $\Fun^{\mathrm{(op)lax}}(\cC, \cD)$. Moreover, $\Fun^{\mathrm{(op)lax}}(-, -)$  agrees with its  strict variant, i.e.\ with the left/right internal hom of $\otimes^{\strict}$, when restricted to $\Gaunt$. 
\end{enumerate}
\end{lemma}
\begin{proof}
\begin{enumerate}
\item
Corollary~3.15 of~\cite{2311.00205} implies that the localization $L_n: \Cat_{(\infty, \infty)} \to \Cat_{(\infty, n)}$ is compatible with the lax tensor product: the canonical map $ L_n(\cC \otimes \cD) \to L_n(L_n \cC \otimes L_n \cD)$ is an isomorphism for $(\infty, \infty)$-categories $\cC$ and $\cD$. The statement then follows by adjunction. 
\item 
Corollary~3.15 of~\cite{2311.00205} moreover implies  that gauntification $\Gau : \Cat_{(\infty, \infty)} \to \Gaunt$ is compatible with the lax tensor product. The induced monoidal structure $\Gau(-\otimes-)$ on $\Gaunt$ is by construction cocontinuous in each variable. The strict tensor product $\otimes^{\strict}$ is also cocontinuous in each variable~\cite[Th\'eor\`eme~A.15]{MR4146146}. By construction, $\Gau(-\otimes-)$ and $\otimes^{\strict}$ agree on $\fancysquare$, but $\fancysquare \subset \Gaunt$ is dense, and so $\Gau(-\otimes-)$ and $\otimes^{\strict}$ agree on all of $\Gaunt$.
The statement then follows by adjunction. \qedhere
\end{enumerate}
\end{proof}

\begin{lemma}\label{lem:surjectiveergoconservative} If $\cC \in \Cat_{(\infty, n)}$ and $\cA \to \cB$ is an essentially surjective functor of $(\infty, n)$-categories, then $\Funoplax(\cB, \cC) \to \Funoplax(\cA, \cC)$ is conservative on $n$-morphisms, i.e.\ is right orthogonal to $\globe_{n} \to \globe_{n-1}$. 
\end{lemma}
\begin{proof}
The functor $\Funoplax(\cB, \cC) \to \Funoplax(\cA, \cC)$ being right orthogonal to $\globe_n \to \globe_{n-1}$ for all $\cC \in \Cat_{(\infty, n)}$ is by adjunction and Yoneda equivalent to the assertion that the square 
\begin{equation}
\label{eq:pushoutforsurj}\begin{tikzcd}
 \globe_n \otimes \cA \arrow[r] \arrow[d] & \globe_n \otimes \cB \arrow[d] \\
 \globe_{n-1} \otimes \cA \arrow[r] & \globe_{n-1} \otimes \cB
\end{tikzcd}
\end{equation}
is a pushout after $n$-localization for any surjective functor $\cA \to \cB$. 

Recall that a class of morphisms in an $\infty$-category $\cX$ is called \emph{saturated} if it contains all equivalences, is closed under composition, under colimits in $\Fun(\globe_1, \cX)$ and under pushout against arbitrary morphisms. 
As shown in~\cite[Thm.~5.3.7]{2401.02956}, essentially surjective functors of $(\infty, n)$-categories are the left class in a factorization system generated by the set of morphisms $S:=\{\partial \globe_k \to \globe_k\}_{1 \leq k \leq n}$ (with right class the fully faithful functors) and hence by ~\cite[Prop.~5.5.5.7]{HTT} essentially surjective functors form the smallest saturated class of morphisms in $\Cat_{(\infty, n)}$ containing $S$. Since the square~\eqref{eq:pushoutforsurj} is functorial in $(\cA \to \cB) \in \Fun(\globe_1,\Cat_{(\infty, n)})$ and since it being a pushout  is true for equivalences, preserved under composition, colimits in $\Fun(\globe_1, \Cat_{(\infty, n)})$ and pushouts against arbitrary morphisms, it suffices to prove it for $\cA \to \cB$ being $\partial \globe_k \to \globe_k$ for $1\leq k \leq n$. We thus need to prove that for all $n\geq 0$ and all $1\leq k \leq n$, the following is a pushout square:
\[\begin{tikzcd}
L_n( \globe_n \otimes \partial \globe_k )\arrow[r] \arrow[d] & L_n(\globe_n \otimes \globe_k) \arrow[d] \\
L_n( \globe_{n-1} \otimes \partial \globe_k) \arrow[r] & L_n(\globe_{n-1} \otimes \globe_k)
\end{tikzcd}
\]
Since $\partial \globe_k\to \globe_k$ is the suspension $\Sigma(\partial \globe_{k-1} \to \globe_{k-1})$, by Lemma~\ref{lem:Sigmapreserves} it suffices to prove the statement for $k=1$. The $k=1$ case is Proposition~\ref{prop:trickypushout}. 
\end{proof}

\subsection{Hom \texorpdfstring{$(\infty,\infty)$}{(infty,infty)}-categories} \label{subsec:enrichment}

The $(\infty,\infty)$-category $\Funoplax(\globe_1,\cC)$ turns out to be particularly important, and so we give it a special name:
\begin{definition} Given an $(\infty, \infty)$-category $\cC$, we define its $(\infty, \infty)$-category of \emph{arrows} by \[\Arr{\cC}:= \Funoplax(\globe_1, \cC).\] 
\end{definition}

We will often implicitly use the source/target projection $$\Arr{\cC} = \Funoplax(\globe_1, \cC) \to \Funoplax(\partial\globe_1, \cC) = \Funoplax(* \sqcup *, \cC) = \cC \times \cC.$$
This projection can be used to define the hom-categories of an $(\infty, \infty)$-category:
\begin{definition}\label{def:hom-cat} Given objects $a, b: * \to \cC$ of an $(\infty, \infty)$-category $\cC$, we define its \emph{hom-$(\infty, \infty)$-category} \[\cC(a,b):=\{a\} \times_{\cC} \Arr{\cC} \times_{\cC} \{ b\}.\]
\end{definition}

Definition~\ref{def:hom-cat} assembles into a functor 
\begin{equation}\label{eq:hom-functorial} \Funoplax_{*\sqcup */}([1], -): \Cat_{(\infty, \infty)}^{**}:= (\Cat_{(\infty, \infty)})_{* \sqcup */} \to \Cat_{(\infty, \infty)} 
\end{equation} sending an $(\infty, \infty)$-category $\cC$ with two marked objects $a,b$ to $\cC(a, b)$. 

\begin{prop}\label{prop:adjunction-hom-sigma}
The functor~\eqref{eq:hom-functorial} is right adjoint to the functor $\Sigma: \Cat_{(\infty, \infty)} \to  \Cat_{(\infty, \infty)}^{**}$ from~\eqref{eq:sigmainfinity}.
\end{prop}
\begin{proof}
Immediate from Lemma~\ref{lemma:sigmaaspushout}.
\end{proof}

 We first record:
\begin{lemma} \label{lem:takinghomloosescategorylevel}
If $\cC$ is an $(\infty,n)$-category, then for any  $a, b \in \cC$, $\cC(a,b)$ is an $(\infty, n-1)$-category.
\end{lemma}
\begin{proof}
We claim that the functor $\Arr{\cC} \to \cC \times \cC$ is conservative on $n$-morphisms.
By definition, this is the functor $\Funoplax(\globe_1, \cC) \to \Funoplax( \partial \globe_1, \cC)$. By adjunction, this morphism being right orthogonal to $\globe_n \to \globe_{n-1}$ for any $\cC \in \Cat_{(\infty, n)}$ is equivalent to the assertion that the square~\eqref{eq:pushoutcn} is a pushout in $\Cat_{(\infty, n)}$; this is the content of Proposition~\ref{prop:trickypushout}. That $\cC(a, b)$ is an $(\infty, n-1)$-category follows from $n$-conservativity of $\Arr{\cC} \to \cC \times \cC$ by Lemma~\ref{lem:consequences-conservative} part \ref{lem:consequences-conservative3}.
\end{proof}
\begin{lemma} \label{lem:homcatcorrect} If $\cC$ is in $\Gaunt\subseteq \Cat_{(\infty, \infty)}$, then Definition~\ref{def:hom-cat} recovers the usual strict hom-$\omega$-category.
\end{lemma}
\begin{proof} By Lemma~\ref{lem:propertiesofFunoplax:others}, for $\cC \in \Gaunt$, $\Funoplax(\globe_1, \cC)$ is the strict variant of the $\omega$-category of functors and oplax natural transformations. The result then follows from~\cite[Prop.~B.6.2]{MR4146146}.
\end{proof}

Using the theory of enriched $\infty$-categories developed in \cite{MR3345192}, it is shown in \cite{MR3402334} that an $(\infty,n)$-category is precisely a category enriched in the $\infty$-category of $(\infty,n-1)$-categories: There is an equivalence $\Cat_{(\infty,n)} \simeq \Cat(\Cat_{(\infty,n-1)})$, where for a monoidal $\infty$-category $\cV$, $\Cat(\cV)$ denotes the $\infty$-category of $\cV$-enriched univalent (aka Rezk-complete) $\infty$-categories. Taking the limit as $n\to \infty$ supplies an equivalence 
\begin{equation} \label{eqn:catinfinffixed} \Cat_{(\infty,\infty)} \simeq \Cat(\Cat_{(\infty,\infty)}), \end{equation}
which by the main result of~\cite{MR4669810} may be used to universally characterize $\Cat_{(\infty,\infty)}$.
Thus treating an $(\infty, \infty)$-category as a category enriched in $(\infty, \infty)$-categories naturally encodes a hom-$(\infty, \infty)$-category between $(\infty, \infty)$-categories which we will now compare with our Definition~\ref{def:hom-cat}: 

\begin{proposition}\label{prop:enrichmentagreement}
  The hom $(\infty,\infty)$-categories from Definition~\ref{def:hom-cat}, defined in terms of $\Funoplax$, agree with the hom $(\infty,\infty)$-categories defined in terms of the equivalence~\eqref{eqn:catinfinffixed}.
\end{proposition}
\begin{proof} Let $\Cat(\cV)^{**} := \Cat(\cV)_{(*\sqcup *)/}$ denote the $\infty$-category of $\cV$-enriched categories $\cC$ with two distinguished objects $x_0,x_1$.
It follows from the discussion surrounding \cite[Def. 4.3.12]{MR3345192} that the functor  $\Cat(\cV)^{**} \to \cV$ given by $(\cC,x_0,x_1) \mapsto \cC(x_0,x_1)$ has a left adjoint,  the \define{$\cV$-categorical suspension} $\Sigma : \cV \to \Cat(\cV)^{**}$, which sends $V \in \cV$ to the two-object category $\Sigma V$ with object space $\{0,1\}$ and morphisms
  $$ (\Sigma V)(i,j) = \begin{cases} 1_\cV, & i=j, \\ V, & i < j \\ \emptyset_\cV, & i > j \end{cases}.$$
    Here $1_\cV \in \cV$ denotes the monoidal unit in $\cV$, and $\emptyset_\cV \in \cV$ denotes the initial object. 
For $\cV= \Cat_{(\infty, \infty)}$, we claim that our functor $\Sigma: \Cat_{(\infty, \infty)} \to \Cat_{(\infty, \infty)}^{**}$ from~\eqref{eq:sigmainfinity} agrees with the composite of Gepner-Haugseng's categorical suspension with the equivalence from \cite{MR4669810}  $\Cat_{(\infty, \infty)} \to \Cat(\Cat_{(\infty, \infty)})^{**} \simeq \Cat_{(\infty, \infty)}^{**}$.  Since $\Gaunt \subset \Cat_{(\infty, \infty)}$ is dense and both versions of ``$\Sigma$'' are left adjoints, it suffices to compare the functors when restricted to $\Gaunt$. Since both categorical suspension functors send gaunt categories to gaunt categories, their equivalence on $\Gaunt$ follows directly from Lemma \ref{lem:homcatcorrect}.

Hence, their right adjoints agree which by Proposition~\ref{prop:adjunction-hom-sigma} implies the claim. 
\end{proof}

The remainder of this section describes the composition on morphism categories entirely in terms of $\Funoplax$. We expect that this could be used to directly build the comparison map between Campion's model of $\Cat_{(\infty,\infty)}$ in terms of sheaves on $\fancysquare$ and the enriched categories model used by \cite{MR3402334,MR4669810}. We will not give a complete comparison along these lines, but rather develop only what we will need for our application.

\begin{definition}\label{defn:cocategory}
As is standard, we will write $\Delta \subset \Gaunt$ for the category of simplices, and denote its $k$th object by $[k]$. In particular, $[0] = *$ and $[1] = \globe_1$.

An \define{(internal) category}\footnote{We do not impose any sort of (co)univalence axiom for our internal (co)categories.} in an $\infty$-category $\cE$ is a functor $X: \Delta^{\op} \to \cE$ satisfying the \define{Segal conditions} that the face maps $X_{[n]} \to X_{[1]}$ exhibit $X_{[n]}$ as the limit  $X_{[1]} \times_{X_{[0]}} \ldots \times_{X_{[0]}} X_{[1]}$ for all $n\geq 1$. 

  An \define{(internal) cocategory} in an $\infty$-category $\cE$ is a functor $X: \Delta \to \cE$ satisfying the \define{coSegal conditions} that the coface maps $X_{[1]} \to X_{[n]}$ exhibit $X_{[n]}$ as the colimit $X_{[1]} \sqcup_{X_{[0]}} \ldots \sqcup_{X_{[0]}} X_{[1]}$ for all $n\geq 1$.
  
 We denote by $\coCat(\cE)$ the full subcategory of $\Fun(\Delta, \cE)$ on the cocategories.  Since colimits in functor categories are computed pointwise, it follows that for $\infty$-categories $\cA, \cE$, there is a natural equivalence $$\coCat(\Fun(\cA, \cE)) \simeq \Fun(\cA, \coCat(\cE))$$ which we will often use implicitly. 
\end{definition}

\begin{example}\label{exm:internalcat-hom}
A key example of a cocategory in $\Cat_{(\infty, \infty)}$  is the full inclusion $\Delta \hookrightarrow \Cat_{(\infty, \infty)}$. This example of course already lives in the full sub-$1$-category of  gaunt $1$-categories. 

Dualizing, it follows that for any $(\infty, \infty)$-category $\cC$, the functor $\Funoplax(-, \cC): \Delta^{\op} \to \Cat_{(\infty, \infty)}$ is an internal category in $\Cat_{(\infty, \infty)}$ which encodes the coherently associative and unital composition operation on $\Arr{\cC} = \Funoplax([1], \cC)$. This may be used as a starting point for a comparison with $(\infty, \infty)$-categories defined via iterating internal categories. 
\end{example}

Just like the hom-categories are by Proposition~\ref{prop:adjunction-hom-sigma} usefully encoded via the left adjoint $\Sigma$, the composition operation on these hom-categories from Example~\ref{exm:internalcat-hom} are usefully encoded by the following generalization of $\Sigma$ parameterized over $\Delta$:
\begin{construction} \label{const:generalSigma}
Definition~\ref{defn:suspension} naturally generalizes to a functor
$$\Sigma(-,-):  \strCat_{n} \times \Delta \to \strCat_{n+1}$$ sending a $C\in \strCat_n$  and an $[k] \in \Delta$ to the strict $(n+1)$-category $\Sigma(C, [k])$ with objects $0, \ldots, k$ and strict hom-$n$-categories
\[\mathrm{Hom}_{\Sigma(C, [k])}(i, j) = \left\{ \begin{array}{ll} C^{\times(j-i)} & i \leq j, \\ \emptyset & i> j. \end{array} \right.
\]
The composition operations and functoriality are evident. Moreover, $\Sigma(-,-)$ manifestly satisfies the following $1$-categorical universal property: for each $[k] \in \Delta$ and $C,D \in \Gaunt$, there is a natural bijection
\begin{equation}\label{eq:natlSigma}\Map_{\Gaunt}(\Sigma(C, [k]), D) \simeq  \bigsqcup_{x_0, \ldots, x_n \in \ob D} \Map_{\Gaunt}(C, D(x_0, x_1) \times \cdots \times D(x_{k-1}, x_k)).\end{equation}

The functor $\Sigma(-,-)$ restricts to a functor $\Theta_n \times \Delta \to \Theta_{n+1}$. Starting with $\Theta_1 = \Delta$, we may compose the functors $\Theta_n \times \Delta \to \Theta_{n+1}$ to obtain a functor $\Delta \times \ldots \times \Delta \to \Theta_n$, which is used to compare the $n$-fold Segal space model and the $\Theta_n$-space model of $(\infty, n)$-categories. 

In the colimit, the functors $\Sigma: \Theta_n \times \Delta \to \Theta_{n+1}$ induce a functor $\Theta \times \Delta \to \Theta$, which we will also denote by $\Sigma(-,-)$.
\end{construction}{

We will now give a formula for $\Sigma(-,-)$ in terms of the Gray tensor product, generalizing Lemma~\ref{lemma:sigmaaspushout}. For this, we first make the following simple observation:

\begin{lemma} The functor $\Sigma: \Theta \times \Delta  \to \Theta$ defines a cocategory in $\Fun(\Theta, \Theta)$ (or equivalently, a functor $\Theta \to \coCat(\Theta)$), which is preserved by the inclusions $\Theta \to \Gaunt \to \Cat_{(\infty, \infty)}$ (i.e.\ $\Sigma$ also defines a functor $\Theta \to \coCat(\Gaunt)$ and $\Theta \to \coCat(\Cat_{(\infty, \infty)})$). 
\end{lemma}
\begin{proof}The coSegal condition in $\Fun(\Theta, \Theta)$ can be verified pointwise for each $c\in \Theta$ and follow since  the coface maps exhibit $\Sigma(c, [n])$ as the colimit \[ \Sigma(c,[1]) \sqcup_{\Sigma(c,[0])} \cdots \sqcup_{\Sigma(c, [0])} \Sigma(c, [1])\] in  $\Theta$. This colimit is  is one of the ``compositional colimits'' from \cite{MR4301559} which is preserved by the inclusions $\Theta \subseteq \Gaunt \subseteq \Cat_{(\infty, \infty)}$. \end{proof}

In the sequel we will abuse notation and simply write $\Sigma$ also for this functor $\Theta \to \coCat(\Cat_{(\infty, \infty)})$. 

\begin{construction} \label{cons:Sigma-map}
We will now construct a natural transformation of functors  \[ - \otimes - \To \Sigma(-,-) : \Theta \times \Delta \to \Gaunt.\]

Consider the following function of sets, natural in $[n] \in \Delta, C\in \Theta$ and $D \in \Gaunt$ (note that $C\otimes [n]$ agrees with the strict Gray product and hence is gaunt):
\begin{align*}\Map_{\Gaunt}(\Sigma  (C, [n]), D)
 & \simeq \bigsqcup_{x_0, \ldots, x_n \in \mathrm{ob}D} \Map_{\Gaunt}(C, D(x_0, x_1) \times \ldots \times D(x_{n-1}, x_n)) \\ & \to   \Map_{\Gaunt}(C, \Funoplax([1], D) \times_{D} \cdots \times_D \Funoplax([1], D))  \\ & \simeq  \Map_{\Gaunt}(C, \Funoplax([n], D)) \simeq \Map_{\Gaunt}(C\otimes [n] , D).\end{align*}
 Here,
 the first line is precisely {\eqref{eq:natlSigma}}. The non-invertible map in the second line arises by expressing the hom-category $D(x_0, x_1)$ via Lemma~\ref{lem:homcatcorrect} as in Definition~\ref{def:hom-cat} as a fiber of $\Arr{D} \to D\times D$, and using the inclusion $D(x_0, x_1) \to \Arr{D} := \Funoplax([1], D)$.  By the Yoneda lemma, this assembles into the desired natural transformation. 
\end{construction}

In a monoidal $\infty$-category $\cV$ with finite colimits separately preserved by both slots of the tensor product functor,  the tensor product $V\otimes X$ of an object $V\in \cV$ and a cocategory  $X$ is again a cocategory. Hence, the functor $-\otimes - : \Theta \times \Delta \to \Gaunt$ (and its composite $\Theta \times \Delta \to \Gaunt \to \Cat_{(\infty, \infty)}$) defines a cocategory in $\Fun(\Theta, \Gaunt)$ (resp. in $\Fun(\Theta, \Cat_{(\infty, \infty)})$) and thus the functor from Construction~\ref{cons:Sigma-map} is a map of cocategories.

We record the following for use in Corollary~\ref{cor:sigmacircle}:

\begin{definition}\label{defn:coff} A map of cocategories $X\to Y$ in an $\infty$-category $\cE$ is called \emph{co-fully-faithful}, or \emph{co-ff}, if the commutative square 
\[\begin{tikzcd}
X_{[0]} \sqcup X_{[0]} \arrow[r] \arrow[d] & Y_{[0]} \sqcup Y_{[0]} \arrow[d]\\
X_{[1]} \arrow[r] & Y_{[1]}
\end{tikzcd}
\]
is a pushout.
\end{definition}
Indeed, since $X$ and $Y$ are cocategories, the condition of Definition~\ref{defn:coff} is equivalent to the assertion that the squares 
\begin{equation} \label{eq:nsquare}
\begin{tikzcd}
 \bigsqcup_{\ob [n]} X_{[0]}  \arrow[r] \arrow[d] &  \bigsqcup_{\ob [n]} Y_{[0]}  \arrow[d]\\
X_{[n]} \arrow[r] & Y_{[n]}
\end{tikzcd}
\end{equation}
are pushouts for all $n\geq 0$. 

\begin{lemma} \label{lem:coffSigma} Considering the transformation from Construction~\ref{cons:Sigma-map} as a map of cocategories in $\Fun(\Theta, \Gaunt)$ or as a map of cocategories in $\Fun(\Theta, \Cat_{(\infty, \infty)})$, it is co-ff. \end{lemma}
\begin{proof}
Since colimits in functor categories are computed pointwise, it suffices to verify Definition~\ref{defn:coff} for a fixed $C\in \Theta$, where it follows from Lemma~\ref{lemma:sigmaaspushout}.
\end{proof}
In particular, Lemma~\ref{lem:coffSigma} implies that the squares~\eqref{eq:nsquare} are pushouts which unwinds to the following squares being pushouts for all $\cC \in \Theta$ and $[n] \in \Delta$, generalizing Lemma~\ref{lemma:sigmaaspushout}:
\begin{equation}\label{eq:sigma-n-pushout}
\begin{tikzcd}
\cC \otimes \ob [n] \arrow[r] \arrow[d] &  \ob [n]  \arrow[d] \\ \cC \otimes [n] \arrow[r] & \Sigma (\cC, [n]) 
\end{tikzcd} .\end{equation}

\begin{remark}
Just as in Proposition~\ref{prop:adjunction-hom-sigma}, it follows from~\eqref{eq:sigma-n-pushout} that the cocontinuous extension $\Sigma(-,[n]): \Cat_{(\infty, \infty)} \to (\Cat_{(\infty, \infty)})_{\ob [n]/}$ is left adjoint to the functor $\Funoplax_{\ob[n]/}([n], -): (\Cat_{(\infty, \infty)})_{\ob[n]/} \to \Cat_{(\infty, \infty)}$ and hence may be used to characterize the composition law on $(\infty, \infty)$-categories from Example~\ref{exm:internalcat-hom}. 
\end{remark}

\subsection{The lax smash product \texorpdfstring{$\owedge$}{owedge}}
\label{subsec:owedge}

In general, given any $\infty$-category $\cA$ with a terminal object $*$, the under-$\infty$-category $\cA^*:= \cA_{*/}$  is called the \emph{$\infty$-category of pointed objects in $\cA$}. Informally, an object of $\cA^*$ is an object $A \in \cA$ together with a map $a : * \to A$; a 1-morphism in $\cA^*$ is a morphism $f : A \to B$ and a higher isomorphism $fa \simeq b : * \to B$; and  so on.
 If $\cA$ has finite colimits, given any morphism $f:A\to B$ in $\cA$, there is an induced \emph{quotient object} $B/A \in \cA^*$ defined as the colimit in $\cA$ 
\[B/A := \colim(* \leftarrow A \to B )
\]
with evident pointing. This defines a cocontinuous functor $\Arr\cA \to \cA^*$. 

 If $\cA$ has coproducts, the forgetful functor $\cA^* \to \cA$ admits a left adjoint $(-)_+:\cA \to \cA^*, a \mapsto a \sqcup * = A/\emptyset$, where $\emptyset \in \cA$ denotes the initial object. 
 
 \begin{prop}[{c.f.~\cite[Thm.~5.1]{MR3450758}}]
 \label{prop:generalsmash}
 If $(\cA, \otimes)$ is a presentably monoidal $\infty$-category, then there is a unique presentably monoidal structure $\owedge$ on $\cA^*$ for which the left adjoint $(-)_+:\cA \to \cA^*$ is monoidal. Explicitly, the tensor unit in $\cA^*$ is given by $(1_{\cA})_+ = 1_{\cA} \sqcup *$ and the tensor product in $\cA^*$ of two pointed objects $a$ and $b$ is given by the formula \[a\owedge b:= (a\otimes b)/\bigl((a\otimes *) \underset{ * \otimes *} \sqcup (* \otimes b) \bigr).\] 
 \end{prop}
 \begin{proof}
Recall from~\cite[Example~4.8.1.21]{HA} that $\cA^* \simeq \cA \boxtimes \Spaces^*$, where $\boxtimes$ denotes Lurie's tensor product in $\PrL$. The smash product of pointed spaces makes $\Spaces^*$ into a presentably symmetric monoidal category~\cite[Remark~4.8.2.14]{HA}, and hence $\cA^* \simeq \cA \otimes \Spaces^*$ inherits a presentably monoidal structure as a tensor product of algebras in the symmetric monoidal category $\PrL$. Moreover, under this equivalence, $(-)_+: \cA \to \cA^*$ becomes the functor $\cA \simeq \cA \boxtimes \Spaces \stackrel{\id_{\cA} \otimes (-)_+}{\to} \cA \boxtimes \Spaces^*$. Since $\Spaces \to \Spaces^*$ is symmetric monoidal, it follows that also $\cA \to \cA^*$ is monoidal. Uniqueness is proven analogously to~\cite[Thm.~5.1]{MR3450758}. For the explicit formula, note that any pointed object $a$ in $\cA^*$ may be written as the pushout $\colim(* \leftarrow (*)_+ \to a_+)$ in $\cA^*$. The formula then follows using that $(-)_+$ is monoidal and $\owedge$ preserves colimits in both variables. 
 \end{proof}
 \begin{definition} Given a presentably monoidal $\infty$-category $(\cA, \otimes)$, we will refer to the monoidal structure $\owedge$ on $\cA^*$ from Proposition~\ref{prop:generalsmash} as the \emph{smash product  associated to~$\otimes$}. \end{definition}
 
 As a special case, for $x \in \cA$ and $a \in \cA_*$, there are isomorphisms in $\cA_*$ 
  \begin{equation}\label{eq:plussmash}
 x_+ \owedge a \simeq  (x\otimes a)/(x\otimes *) \hspace{1cm} a \owedge x_+ \simeq (a \otimes x)/(* \otimes x) .
 \end{equation}
 
 \begin{lemma}
 \label{lem:generalsmashhom}
 Let $(\cA, \otimes)$ be a presentably monoidal $\infty$-category with associated presentably monoidal $(\cA^*, \owedge)$ as in Proposition~\ref{prop:generalsmash}. Then, the left (resp.\ right) internal hom $\eHom_{\cA^*}(-,-)$ of $\cA^*$ can be computed in terms of the left (resp.\ right) internal hom $\eHom_{\cA}(-, -)$ of $\cA$ as the following pullback in $\cA$ (equipped with its evident pointing): 
\[
\eHom_{\cA^*}((A, a), (B, b)) \simeq \eHom_{\cA}(A, B) \underset{\eHom_{\cA}(*, B)}\times  *  .
\]
Here, $\eHom_{\cA}(A, B) \to \eHom_{\cA}(*, B)$ is given by restriction along $* \overset a \to A$, and
 $* \to \eHom_{\cA}(*, B)$ is given by currying the map $* \otimes * \to * \overset b \to B$. 
 \end{lemma}
 \begin{proof}
 This is a direct consequence of the defining universal property of internal homs and the explicit formal for $\owedge$ from Proposition~\ref{prop:generalsmash}. 
 \end{proof}

We will apply this general theory to the case of $\cA = \Cat_{(\infty,\infty)}$ with its Gray tensor product. An object of $\Cat_{(\infty, \infty)}^* = (\Cat_{(\infty, \infty)})_{*/}$ is then a \define{pointed $(\infty,\infty)$-category}: an $(\infty,\infty)$-category $\cC$ equipped with an object $1_\cC \in \cC$. A 1-morphism from $(\cC, 1_\cC)$ to $(\cD, 1_\cD)$ is a functor $F : \cC \to \cD$ together with an isomorphism $F 1_\cC \simeq 1_\cD$.

\begin{definition}\label{defn:funoplaxstar}
We denote the \define{lax smash product} on $\Cat_{(\infty, \infty)}^*$ by $\owedge$ and its internal homs  --- computed in terms of $\Funoplax$ and $\Funlax$ as in Lemma~\ref{lem:generalsmashhom} --- by $\Funoplax_*$ and $\Funlax_*$. 
\end{definition}

The lax smash product $\owedge$ is also defined and analyzed in~\cite{NarukiThesis}.

\subsection{Directed loop categories}

\begin{definition} The \emph{directed circle} $\Sone \in \Cat_{(\infty, \infty)}^*$ is the quotient $\globe_1/\partial\globe_1$ of the boundary inclusion $* \sqcup * = \partial\globe_1 \hookrightarrow \globe_1$. 
\end{definition}

\begin{remark} Since the full subcategory $\Cat_{(\infty,1)} \subseteq \Cat_{(\infty, \infty)}$ is closed under colimits, $\Sone$ is an object of the full subcategory $\Cat_{(\infty,1)}^*$. It is equivalent to the $1$-category $B \mathbb{N}$, the delooping of the additive monoid of natural numbers; this is because $\bN$ is the free associative monoid on a single generator not only in the strict set-theoretic world but also in the homotopical world of spaces.
\end{remark}

\begin{definition} \label{def:Omega}We define\footnote{Given a pointed object $* \to X$ in an $\infty$-category with finite limits, some authors use the notation ``$\Omega X$'' for the pullback $* \times_X *$; for $(x:*\to X)\in\Cat^*_{(\infty, \infty)}$ this computes the space $\Map_*(S^1, X)$ of \emph{automorphisms of $x$ in $X$}. Our $\Omega$ instead uses the directed circle $\Sone$, hence computes \emph{endomorphisms of $x$ in $X$}, and so perhaps deserves the name $\vec{\Omega}$. As we will not use the undirected circle $S^1$ in the sequel we will write the simpler $\Omega$ rather than $\vec{\Omega}$.}
 the functor \[\Omega  := \Funoplax_*(\Sone,-) :\Cat_{(\infty, \infty)}^* \to \Cat_{(\infty, \infty)}.\]\mbox{}\\[-26pt]\end{definition}

Comparing with Definition~\ref{def:hom-cat}, for a fixed pointed $(\infty, \infty)$-category $(\cC,1_\cC)$, we have:
\begin{equation}
\label{eq:OmegaC}
\Omega \cC:= \cC(1_\cC,1_\cC) = \{1_\cC\} \times_{\cC} {\Arr{\cC}} \times_{\cC}\{1_\cC\}.
\end{equation}\mbox{}\\[-26pt]

\begin{lemma} \label{lem:surjectivereduces}
If $\cC$ is a pointed $(\infty,n)$-category for which the pointing $* \to \cC$ is essentially surjective on objects and if $\cD$ is any pointed $(\infty, n)$-category, then $\Funoplax_*(\cC, \cD)$ is an $(\infty, n-1)$-category. In particular, if $\cC$ is a pointed $(\infty,n)$-category, then $\Omega\cC$ is an $(\infty,n-1)$-category.
\end{lemma}
\begin{proof}
By Lemma~\ref{lem:generalsmashhom},  $\Funoplax_*(\cC, \cD)$ is the fiber of the restriction functor $\Funoplax(\cC, \cD) \to \Funoplax(*, \cD)$ at the pointing of $\cD$. Since $* \to \cC$ is essentially surjective on objects, this functor is conservative on $n$-morphisms by Lemma~\ref{lem:surjectiveergoconservative} and hence the fiber is an $(\infty, n-1)$-category by Lemma~\ref{lem:consequences-conservative}. 
\end{proof}
\begin{remark}
In~\cite[Def. 5.3.1]{2401.02956}, the notions of surjectivity and fullness of functors are generalized to a notion of ``$k$-surjectivity'' of a functor between $(\infty, n)$-categories. We expect an analogous statement and proof as in Lemma~\ref{lem:surjectivereduces}: if $* \to \cC$ is $k$-surjective, then $\Funoplax_*(\cC,-)$ should reduce category number by $k$ levels. Conversely, we expect that if $* \to \cC$ is $k$-surjective and $* \to \cD$ is $r$-surjective, then $* \to \cC \owedge \cD$ should be $(k+r)$-surjective. We will not pursue these more general statements in this paper.
\end{remark}

Intuitively, composition defines a monoidal structure on $\Omega \cC$, which we will now explicitly unpack.

 Recall from Definition~\ref{defn:cocategory} the notion of ``cocategory'' in an $\infty$-category $\cE$: a functor $X : \Delta \to \cE$ satisfying the coSegal conditions.
\begin{definition}
  A \define{comonoid} in an $\infty$-category $\cE$ is a cocategory $X$ for which $X_{[0]}$ is initial. We will denote by $\coMon(\cE)$ the full subcategory of (the full subcategory $\coCat(\cE)$ of) $\Fun(\Delta,\cE)$ on the comonoids. Since colimits in functor categories are computed pointwise, it follows that for $\infty$-categories $\cA, \cE$, there is a natural equivalence $$\coMon(\Fun(\cA, \cE)) \simeq \Fun(\cA, \coMon(\cE))$$ which we will often use implicitly.\end{definition}

\begin{construction} \label{cons:cocattocomon} Let $\cE$ be an $\infty$-category with finite colimits. We construct a functor 
\[\overline{(-)}: \coCat(\cE) \to \coMon(\cE_*)
\]
as follows:
The functor $\mathrm{ev}_{[0]}: \Fun(\Delta, \cE) \to \cE$ has a left adjoint, sending  $e\in \cE$ to the left Kan extension of $\{e\} \to \cE$ along $\{[0]\} \to \Delta$, i.e.\ explicitly to the cosimplicial object $[n] \mapsto   \bigsqcup_{\ob [n]} c$. The counit of this adjunction defines a functor $\Fun(\Delta, \cE) \to \Fun([1], \Fun(\Delta, \cE)) \simeq \Fun(\Delta, \Fun([1], \cE))$, which is cocontinuous since the right adjoint $\mathrm{ev}_0: \Fun(\Delta, \cE) \to \cE$ is cocontinuous. 
 Composing with the cocontinuous quotient functor $\Fun([1], \cE) \to \cE_*$  results in a cocontinuous functor $\Fun(\Delta, \cE) \to \Fun(\Delta, \cE_*)$, which restricts by cocontinuity to a functor $\coCat(\cE) \to \coCat(\cE_*)$. Unpacked, this functor sends a cocategory $X$ in $\cE$ to the cocategory $\overline{X}$ in $\cE_*$ with $n$th object given by $ \overline{X}_{[n]}:= X_{[n]}/(\bigsqcup_{\ob [n]} X_{[0]})$ in $\cE_*$. In particular, $\overline{X}_{[0]} = *$ is initial in $\cE_*$ and hence the image of any cocategory is indeed a comonoid.
\end{construction}

\begin{example}
  Take $\cE = \Spaces$, and input the terminal cocategory $*$ into Construction~\ref{cons:cocattocomon}. The output is the canonical comonoid structure on $S^1 \in \Spaces^*$.
\end{example}

\begin{lemma} \label{lem:coffiso} Let $\cC$ be an $\infty$-category with finite colimits. The functor $\coCat(\cC) \to \coMon(\cC_*)$ from Construction~\ref{cons:cocattocomon} sends co-ff maps of cocategories to isomorphisms.
\end{lemma}
\begin{proof} The induced map at stage $n$ is $X_{[n]}/\sqcup_{\ob [n]} X_{[0]} \to Y_{[n]} / \sqcup_{\ob [n]} Y_{[0]}$ and hence is an isomorphism since~\eqref{eq:nsquare} is a pushout. 
\end{proof}

\begin{lemma} Applying the functor $\coCat(\Cat_{(\infty, \infty)})\to \coMon(\Cat_{(\infty, \infty)}^*)$ to the cocategory given by the full subcategory inclusion  $\Delta \hookrightarrow \Cat_{(\infty, \infty)}$ results in a comonoid structure on $\Sone \in \Cat_{(\infty, \infty)}^*$.
\end{lemma}
\begin{proof} For any cocategory $X$, the underlying object of the resulting comonoid is $X_{[1]}/(X_{[0]} \sqcup X_{[0]})$, i.e.\ in our case $[1]/([0] \sqcup [0])$, the directed circle. 
\end{proof}

\begin{definition}
A \emph{monoid} in an $\infty$-category $\cA$ with finite products is a functor $M: \Delta^{\op} \to \cA$ for which $M_{[0]}$ is terminal and the face maps $M_{[n]} \to M_{[1]}$ exhibit $M_{[n]}$ as an $n$-fold product $M_{[n]} = M_{[1]} \times \ldots \times M_{[1]}$ for all $n \geq 1$. We write $\Mon(\cA)$ for the full subcategory of $\Fun(\Delta^{\op}, \cA)$ on the monoids. For a monoid $M$, we write $M^{\mop}$ for the \emph{reversed monoid} defined as the composite $\Delta^{\op} \overset{(-)^{\op}}{\to} \Delta^{\op} \overset{M}{\to} \cA$, where $(-)^{\op}$ denotes the autoequivalence of $\Cat_{(\infty,1)}$ restricted to the full subcategory $\Delta \subseteq \Cat_{(\infty,1)}$.   A monoid in $ \cA:= \Cat_{(\infty, \infty)}$ is a \emph{monoidal $(\infty, \infty)$-category}.

\end{definition}

\begin{corollary}\label{cor:monoidalstructureonloops}
The functor $\Omega (-) = \Funoplax_*(\Sone,-): \Cat_{(\infty, \infty)}^* \to \Cat_{(\infty, \infty)}$  factors through the $\infty$-category $\Mon(\Cat_{(\infty, \infty)})$ of monoidal $(\infty, \infty)$-categories.
\end{corollary}
\begin{proof} The functor $\Funoplax_*(-,-): ( \Cat_{(\infty, \infty)}^*)^{\op} \times \Cat_{(\infty, \infty)}^* \to \Cat_{(\infty, \infty)}^*$ sends colimits in the first argument to limits, and hence sends comonoids in the first argument to monoids in the functor category $\Fun(\Cat_{(\infty, \infty)}^*, \Cat_{(\infty, \infty)}^*)$, or equivalently to functors $\Cat_{(\infty, \infty)}^* \to \Mon(\Cat_{(\infty, \infty)}^*) \simeq \Mon(\Cat_{(\infty, \infty)})$. \end{proof}

Suppose that $(\cA,\otimes)$ is a monoidal $\infty$-category  with finite colimits which are preserved by $\otimes$ in each variable. Then, for any object $A \in \cA$ and any cocategory object $X\in \cA$, the cosimplicial object $A \otimes X = \{A \otimes X_{[n]}\}_{[n] \in \Delta}$ is again a cocategory; if the cocategory $X$ is a comonoid, then so is $A \otimes X$.

\begin{lemma} \label{lem:auxiso} Suppose that $\cA$ a monoidal $\infty$-category with finite colimits which are preserved by $\otimes$ in each variable. There is an equivalence in $\coMon(\cA_*)$, natural in $A\in \cA$ and $X\in \coCat(\cA)$, between $\overline{A \otimes X}$ and $A_+ \owedge \overline{X}$.
\end{lemma}
\begin{proof}
Consider the squares
\[\begin{tikzcd}
A \otimes \sqcup_{\ob [n]} X_{[0]} \arrow[r] \arrow[d] & A\otimes *\arrow[r] \arrow[d] & *  \arrow[d] \\
A \otimes X_{[n]} \arrow[r] & A\otimes \overline{X}_{[n]} \arrow[r] & A_+ \owedge \overline{X}_{[n]}
\end{tikzcd}.
\]
By~\eqref{eq:plussmash}, the right square is a pushout, and since $\otimes$ preserves colimits, the left square is a pushout. Hence, the total square is a pushout, i.e.\ $A_+ \owedge \overline{X}_{[n]} \simeq A\otimes X_{[n]}/(A \otimes \sqcup_{\ob [n]} X_{[0]}) \simeq A \otimes X_{[n]} /(\sqcup_{\ob [n]} (A\otimes X_{[0]})) =: (\overline{A \otimes X})_{[n]}$.
\end{proof}

\begin{cor}\label{cor:sigmacircle}There is an equivalence of functors $\Theta \to \coMon(\Cat_{(\infty, \infty)}^*)$ between $(-)_+ \owedge \Sone$ and the composite $\Theta \to[\Sigma] \coCat(\Cat_{(\infty, \infty)}) \to[\overline{(-)}] \coMon(\Cat_{(\infty, \infty)}^*)$.
\end{cor}
\begin{proof} By Lemma~\ref{lem:auxiso}, there is an equivalence of comonoids between $(-)_+ \owedge \Sone$ and $\overline{(- \otimes [1])}$, where $-\otimes [1]$ denotes  the functor $\Theta \to \coCat(\Cat_{(\infty, \infty)}) $ sending $C$ to the cocategory with $n$th object $C\otimes [n]$. The statement then follows by applying  Lemma~\ref{lem:coffiso} to Lemma~\ref{lem:coffSigma}.
 \end{proof}

Recall from Definition~\ref{defn:opposites} our notation $(-)^{\op}$ and $(-)^{\co}$ for the autoequivalences of $\Cat_{(\infty, \infty)}$ which reverse all odd morphism directions, respectively all even morphism directions. Corollary~\ref{cor:sigmacircle} implies the following interaction of $\Sigma$ and $\Omega$ with taking opposites. 

\begin{lemma}\label{lemma:movingSigmapastcoop} There is an equivalence of functors $\Theta \times \Delta \to \Theta$
\[ \left( \Sigma(-,-)\right)^{\co} \simeq \Sigma \left((-)^{\op}, -\right). \]
In particular,  the following diagram commutes:
\begin{equation}\label{eq:opcommute}
\begin{tikzcd}
\Theta \arrow[r, "\Sigma"] \arrow[d, "\op"] &   \coCat(\Cat_{(\infty, \infty)}) \arrow[d, "\coCat((-)^{\co})"]\\
\Theta \arrow[r, "\Sigma"] & \coCat(\Cat_{(\infty, \infty)})
\end{tikzcd}
\end{equation}
\end{lemma}
\begin{proof} The first statement is a straight-forward 1-categorical computation. The second statement is an immediate consequence.
\end{proof}

\begin{corollary} \label{cor:Sonecentral} There is an equivalence of functors $\Cat_{(\infty, \infty)}^* \to \coMon(\Cat_{(\infty, \infty)}^*)$: 
\[\Sone \owedge (-)^{\co\op} \simeq (-) \owedge \Sone.\]
\end{corollary}
\begin{proof}
Note that $*$ with the trivial coalgebra structure is a zero object of $\coMon(\Cat_{(\infty, \infty)}^*)$. Recall from~\cite[Prop.~5.4]{MR3450758} that for any presentable $\infty$-category $\cE$ and any presentable $\infty$-category $\cD$ with zero object, restriction along $(-)_+: \cE \to \cE_*$ induces an equivalence $\Fun^L(\cE_*, \cD) \to \Fun^L(\cE, \cD)$, where $\Fun^L(-,-)$ denotes the full subcategory of $\Fun(-,-)$ on the left adjoints. Hence, it suffices to produce an equivalence between the functors $\Cat_{(\infty, \infty)} \to \coMon(\Cat_{(\infty, \infty)}^*)$ 
\[ \Sone \owedge (-)^{\co\op}_+ \simeq (-)_+ \owedge \Sone.
\]
Since $\Theta \subseteq \Cat_{(\infty, \infty)}$ is a dense full subcategory, it in fact suffices to produce this equivalence when further restricted to functors $\Theta \to \coMon(\Cat_{(\infty, \infty)}^*)$. 
Since the functor $\overline{(-)}: \coCat(\cA) \to \coMon(\cA_*)$ from Construction~\ref{cons:cocattocomon} is natural in $\cA$, it follows that the following diagram commutes:
\[\begin{tikzcd}
  \coCat(\Cat_{(\infty, \infty)}) \arrow[d, "\coCat((-)^{\co})"]  \arrow[r, "\overline{(-)}"] &\coMon(\Cat_{(\infty, \infty)}^*) \arrow[d, " \coMon((-)^{\co})"]\\
 \coCat(\Cat_{(\infty, \infty)}) \arrow[r, "\overline{(-)}"] & \coMon(\Cat_{(\infty, \infty)}^*) 
\end{tikzcd}
\]
Composing it with~\eqref{eq:opcommute} from Lemma~\ref{lemma:movingSigmapastcoop}, and using Corollary~\ref{cor:sigmacircle} that $\overline{\Sigma (-)} = (-)_+ \owedge \Sone$, we find that 
\[ \left((-)_+ \owedge \Sone\right)^{\co} \simeq (-)_+^{\op} \owedge \Sone
\]
Using the anti-monoidality of $(-)^{\co}$ from Lemma~\ref{lem:propertiesofFunoplax:compatwithcoop}, and using that $\Sone $ is (a cocategory valued in) a $1$-category with $(\Sone)^{\co} = \Sone$, the left hand side is equivalent to $\Sone \owedge (-)_+^{\co}$, which proves the claim.
\end{proof}

\begin{corollary}\label{cor:opcoloop}
 There are equivalences of functors $\Cat_{(\infty, \infty)}^* \to \Mon(\Cat_{(\infty, \infty)})$: 
\[\Omega \bigl((-)^{\co}\bigr) \simeq (\Omega(-))^{\op}
\hspace{2cm} \Omega \bigl((-)^{\op}\bigr) \simeq (\Omega(-))^{\mop, \co}
\]
\end{corollary}
\begin{proof} With Definition~\ref{def:Omega}, this follows directly from Corollary~\ref{cor:Sonecentral}. 
\end{proof}

Notice that $\Funoplax_*(\cC,-)$ preserves limits and hence monoid objects. In other words, if $\cC$ is an $(\infty, \infty)$-category and $\cM$ is a monoidal $(\infty, \infty)$-category, then  $\Funoplax(\cC, \cM)$ is naturally monoidal. 

\begin{corollary} \label{cor:movingloopsinsidefun}
There is an equivalence of monoidal $(\infty, \infty)$-categories
\[ \Omega \Funoplax_*(\cC, \cD) \simeq \Funoplax_*(\cC^{\co\op}, \Omega \cD)
\]
natural in $\cC, \cD \in \Cat_{(\infty, \infty)}^*$. 
\end{corollary}
\begin{proof}
By definition, we have the following equivalences of monoidal $(\infty,\infty)$-categories
\[\Omega\Funoplax_*(\cC, \cD) := \Funoplax_*(\Sone, \Funoplax_*(\cC, \cD))  \simeq \Funoplax_*(\Sone \owedge \cC, \cD),\]
where the monoidal structures on the latter two are induced by the comonoid structure on $\Sone$. Using Corollary~\ref{cor:Sonecentral}, this is equivalent to 
\[\Funoplax_*(\cC^{\co\op} \owedge \Sone, \cD) \simeq \Funoplax_*(\cC^{\co\op}, \Funoplax_*(\Sone, \cD)) \simeq \Funoplax_*(\cC^{\co\op}, \Omega \cD).\qedhere
\]
\end{proof}

\section{The Main Construction}\label{sec:mainconstruction}

This section proves our main theorems. As explained in the introduction, 
our strategy will be to identify retracts, Hopf algebras, and so on in terms of various lax smash products.
Specifically, \S\ref{subsec:adjoints}--\S\ref{subsec:smashsquares} study adjunctibility in categories of functors and (op)lax transformations, concluding with a description of $\Ret^\someadj(\cC) := \Funoplax_*(\Adj \otimes \Adj, \cC)$ as a subcategory of the category $\Ret^\oplax(\cC)$ of retracts and oplax retract morphisms. In \S\ref{sec:mndsquared}, we identify $\Funoplax_*(\Mnd\owedge\Mnd, \cC)$ with $\BiAlg(\Omega^2 \cC)$, and restricting along the inclusion $\Mnd \hookrightarrow \Adj$ supplies the functor $ \Ret^{\someadj}(\cC) \to \BiAlg(\Omega^2\cC)$.
Finally, in \S\ref{subsec:finalproof}, we prove that this functor factors through the full subcategory $\cat{coHopf}(\cC) \subset \BiAlg(\Omega^2\cC)$.

\subsection{Adjoints and their lax morphisms} \label{subsec:adjoints}

The notion of adjunction in higher categories is well-known: 
\begin{definition}
In a strict $2$-category\footnote{The definition here does not give the correct homotopy-coherent notion when implemented directly at the $\infty$-categorical level. The problem is to implement the phrase ``are identities.'' It is insufficient to simply assert the existence of equivalences $(r\epsilon)(\eta r) \simeq \id_r$ and $(\epsilon l)(l\eta) \simeq \id_l$. One can supply both equivalences, but then one needs to supply higher coherences relating them, or one can supply some equivalences while merely asserting existence of others in a not-fully-symmetric way.
Detailed discussion can be found in \cite{MR3415698}.},
 an \define{adjunction} is a pair of 1-morphisms $l : a \to b$ (called the \define{left adjoint}) and $r : b \to a$ (called the \define{right adjoint}) together with 2-morphisms $\epsilon : lr \Rightarrow \id_a$ and $\eta : \id_b \Rightarrow rl$ such that the two 2-compositions $(r\epsilon)(\eta r) : r \Rightarrow r$ and $(\epsilon l)(l\eta) : l \Rightarrow l$ are identities.

The \define{walking adjunction} $\Adj$ is the strict $2$-category generated by an adjunction, and we will write $l,r:\globe_1 \to \Adj$ for the strict $2$-functors picking out the generating left and right adjoint morphism, respectively. This strict $2$-category is gaunt, and so we will, without changing names, equally write $\Adj$ for its image under the inclusion $\Gaunt \subset \Cat_{(\infty,\infty)}$.
\end{definition}
 It is shown in~\cite{MR3415698} that $\Adj$ satisfies the correct universal property among $(\infty, 2)$-categories: for any $(\infty, 2)$-category $\cC$, the maps
\[
\Map(\Adj, \cC) \to \Map(\globe_1, \cC)
\] 
induced by precomposing with $r,l:\globe_1 \to \Adj$ are fully faithful (i.e.\ $(-1)$-truncated maps of spaces; full subspace inclusions) and identify $\Map(\Adj, \cC)$ with the full subspace of $\Map(\globe_1, \cC)$ on those arrows in $\cC$ which are right (resp.\ left) adjoints; in particular, both $r,l : \globe_1 \to \Adj$ are epimorphisms (see footnote~\ref{footnote:epi}) in $\Cat_{(\infty,2)}$.
Since for a $1$-morphism in an $(\infty,\infty)$-category $\cC$, ``being an adjunction'' can be characterized in terms of the underlying $(\infty,2)$-category $\iota_2 \cC$, it follows from the adjunction~\eqref{eq:inclusion-and-localization} that $\Adj$ moreover fulfills the correct universal property among $(\infty,\infty)$-categories. 

In the next section we will compute the tensor product $\Adj \otimes \Adj$ as the result of taking $\globe_1 \otimes \globe_1$ and declaring certain cells to be adjunctible. Towards this end, let $\Adj_k := \Sigma^{k-1}\Adj$ denote the walking adjunctible $k$-morphism, and write $l,r : \globe_k \to \Adj_k$ for the suspensions of the maps $l,r : \globe_1 \to \Adj_1$. In
Proposition~\ref{prop:ClaudiaTheo} we will describe
 $\globe_s \otimes \Adj_k$ and $\Adj_s \otimes \globe_k$ as the result of taking $\globe_s \otimes \globe_k$ and declaring certain cells adjunctible.
  Versions and special cases of our Proposition~\ref{prop:ClaudiaTheo} have also appeared in \cite{1511.03589,2002.01037}; the full version is stated\footnote{
The proof in \cite{JFS} contains a gap: the authors of \cite{JFS} try to reduce the statement to the case when $\cC$ is gaunt, and in particular strict, by claiming in their Lemma~7.9 that the canonical functor $\cC \to \Gau(\cC)$ is always conservative; but there are many examples where it is not. We will therefore give a complete and independent proof of Proposition~\ref{prop:ClaudiaTheo}.} in \cite{JFS}.
 
 The details are slightly subtle for the following reason: as observed in \cite[Prop.~7.13]{JFS}, depending on exactly which maps $\globe_k \to \Adj_k$ or $\globe_s \to \Adj_s$ one wants to use, some of the cells that need to be declared adjunctible are not cells from $\globe_s \otimes \globe_k$ but rather are ``mates'' of those cells, which exist only once other cells are declared adjunctible. We review the theory of mates in \S\ref{subsec:mates} below.

Recall from Definition~\ref{defn:walkingcell} and the subsequent discussion that the boundary of the walking $n$-cell $\globe_n$ is glued from two $(n-1)$-cells, namely its \emph{incoming and outgoing boundary $\partial_0 \globe_n$ and $\partial_1\globe_n$}, along their boundary. (This includes the case $n=0$: The $0$-cell $\globe_0$ has empty boundary, which is glued from two $(-1)$-cells $\globe_{-1} = \emptyset$.) The inclusions $\globe_{n-1} = \partial_{i} \globe_n \mono \globe_n$, for $i=0,1$, induce restriction functors $ \mathrm{Fun}^{\mathrm{(op)lax}}(\globe_n, \cC) \to \mathrm{Fun}^{\mathrm{(op)lax}}(\partial_{i} \globe_n, \cC) = \mathrm{Fun}^{\mathrm{(op)lax}}(\globe_{n-1}, \cC)$ that we typically write as $\alpha \mapsto \alpha_{i}$. A functor $\alpha : \globe_m \otimes \globe_n \to \cC$ has a ``top-dimensional filler'' $\alpha_{<}$ given by restricting along the top cell $\globe_{m+n} \to \globe_m \otimes \globe_n$.

\begin{prop}
 \label{prop:ClaudiaTheo} 
Let $\cC$ be an $(\infty, \infty)$-category, and $s\geq 0$ and $k \geq 1$.

A $k$-morphism $\eta$ in $\Funlax(\globe_s, \cC)$  is a left adjoint if and only if its images $\eta_i$ under the restrictions $\Funlax(\globe_s, \cC) \to \Funlax(\partial_i \globe_s, \cC)$ are left adjoints for $i=0,1$ and also its top-dimensional filler $\eta_{<}$ is a left-adjoint $(k+s)$-morphism in $\cC$. 

A $k$-morphism $\mu$ in $\Funoplax(\globe_s, \cC)$  is a right adjoint if and only if its images $\mu_i$ in $\Funoplax(\partial_i \globe_s, \cC)$ are right adjoints for $i=0,1$ and also its top-dimensional filler $\mu_<$ is a right-adjoint $(k+s)$-morphism in $\cC$. 
\end{prop}

We emphasize that Proposition~\ref{prop:ClaudiaTheo} is sensitive to the pairing of ``lax'' with ``left'' and of ``oplax'' with ``right.''

\begin{proof}
The following improvement of our original proof of Proposition~\ref{prop:ClaudiaTheo}  was communicated to us by Naruki Masuda.

The first statement is equivalent to the statement that there exists a factorization (necessarily unique since the left vertical map is an epimorphism), drawn as the bottom dashed arrow, in the following square,  where all the (solid) morphisms between $\globe$'s and $\Adj$'s pick out left adjoints, plus the statement that this square is pushout: 
\[\begin{tikzcd}
\globe_{s+k} \sqcup( \sqcup_{i=0,1} \partial_i \globe_s \otimes \globe_k) \arrow[r] \arrow[d] & \globe_s \otimes \globe_k \arrow[d] \\ \Adj_{s+k}  \sqcup (\sqcup_{i=0,1} \partial_i \globe_s \otimes \Adj_k) \arrow[r, dashed] &\globe_s \otimes \Adj_k.
\end{tikzcd}
\]
The second statement equivalently asserts the existence of a (dashed) factorization making the following square a pushout, where all morphisms between $\globe$'s and $\Adj$'s pick out the right adjoint: 
\[\begin{tikzcd}
\globe_{k+s} \sqcup ( \globe_k \otimes \sqcup_{i=0,1} \partial_i \globe_s )\arrow[r] \arrow[d] & \globe_k\otimes \globe_s \arrow[d] \\ \Adj_{k+s}  \sqcup (\Adj_k \otimes  \sqcup_{i=0,1} \partial_i \globe_s) \arrow[r, dashed] &\Adj_k \otimes \globe_s.\end{tikzcd}
\]
We prove the first (``left'') statement; the second is analogous. 
The statement is tautological when $s=0$, and the $s=k=1$ version follows from the mate calculus of \S\ref{subsec:mates} and can be found in \cite[Lem.~6.1.5]{NarukiThesis}\footnote{Note that \cite{NarukiThesis} uses the opposite convention from ours for the ordering of $\otimes$.}. By induction, increasing $k \leadsto k+1$ can be achieved by tensoring the pushout with $- \otimes \globe_1 $ and using the decomposition $\Sigma X = X \otimes \globe_1  \sqcup_{X \otimes \partial \globe_1} \partial \globe_1$ from Lemma~\ref{lemma:sigmaaspushout}. To increase $s\leadsto s+1$, tensor the pushout with $\globe_1 \otimes -$ and use the decomposition $\Sigma(X^{\co\op}) \simeq \Sigma(X^{\co})^{\co}  \simeq \globe_1 \otimes X \sqcup_{\partial \globe_1 \otimes X} \partial \globe_1$ which follows from Lemma~\ref{lemma:movingSigmapastcoop} and applying $(-)^{\co}$ to Lemma~\ref{lemma:sigmaaspushout} using Lemma~\ref{lem:propertiesofFunoplax:compatwithcoop}. Finally observe that $l^{\co\op}: \globe_1^{\co\op} \to \Adj^{\co\op}$ is equivalent to $l: \globe_1 \to \Adj$ and hence for any $i\geq 1$, it follows from Lemma~\ref{lem:Sigmacoop} that the functor $(\Sigma^i l)^{\co\op}: (\Sigma^i \globe_1)^{\co\op} =: \globe_{i+1}^{\co\op} \to (\Sigma^ i \Adj)^{\co\op} =: \Adj_{i+1}^{\co\op}$ is equivalent to $\Sigma^i l$. 
\end{proof}

A functor $\cC \to \cD$ of $(\infty, \infty)$-categories is called a \emph{subcategory inclusion} if it is a monomorphism\footnote{A morphism $f:a \to b$ in an $\infty$-category $\cE$ is a \emph{monomorphism} if for all objects $c$ the induced map of spaces $\Map(c, a) \to \Map(c,b)$ is a full subspace inclusion.} in $\Cat_{(\infty, \infty)}$. Because the $\globe_n$'s are colimit generators, a functor is a subcategory inclusion if and only if the induced maps $\Map(\globe_n, \cC) \to \Map(\globe_n, \cD)$ between spaces of $n$-cells are full subspace inclusions for all $n \geq 0$, and a subcategory of $\cD$ is uniquely determined by this list of subspaces of the $\Map(\globe_n, \cD)$'s.\footnote{More precisely, let $(\Cat_{(\infty,1)})_{/^{\mathrm{mono}} \cD}$ denote the full subcategory of $(\Cat_{(\infty,1)})_{/\cD}$ on the subcategory inclusions. (Since subcategory inclusions are monomorphisms, this is in fact a poset.) Then, the functor $(\Cat_{(\infty,1)})_{/^{\mathrm{mono}} \cD} \to \prod_{n \in \mathbb{N}_0} \cat{Spaces}_{/^{\mathrm{mono}} \Map(\globe_n, \cD)}$ is fully faithful.}

\begin{prop}\label{prop:FunlaxAdj}
For any $(\infty, \infty)$-category $\cC$, restriction along the $2$-functor $r: \globe_1 \to \Adj$ exhibits $\Funlax(\Adj, \cC)$ as the sub-$(\infty, \infty)$-category of $\Funlax(\globe_1, \cC)$ whose spaces of $n$-morphisms are inductively characterized as follows:
\begin{itemize}
\item Its space of objects (i.e.\ $0$-morphisms) is the full subspace of $\Map(*, \Funlax(\globe_1, \cC)) \simeq \Map(\globe_1, \cC)$ on the right-adjoint $1$-morphisms in $\cC$.
\item For $n\geq 1$, its space of $n$-morphisms is the full subspace of $\Map(\globe_n, \Funlax(\globe_1, \cC))$ on those $n$-morphisms whose source and target $(n-1)$-morphism are in the subcategory $\Funlax(\Adj, \cC) \subseteq \Funlax(\globe_1, \cC)$ and whose top-dimensional filler is a right-adjoint $(n+1)$-morphism in $\cC$.
\end{itemize}
Restriction along the $2$-functor $l : \globe_1 \to \Adj$ exhibits $\Funoplax(\Adj, \cC)$ as the sub-$(\infty, \infty)$-category of $\Funoplax(\globe_1, \cC)$ whose spaces of $n$-morphisms are inductively characterized as follows:
\begin{itemize}
\item Its space of objects (i.e.\ $0$-morphisms) is the full subspace of $\Map(*, \Funoplax(\globe_1, \cC)) \simeq \Map(\globe_1, \cC)$ on the left-adjoint $1$-morphisms in $\cC$.
\item For $n\geq 1$, its space of $n$-morphisms is the full subspace of $\Map(\globe_n, \Funoplax(\globe_1, \cC))$ on those $n$-morphisms whose source and target $(n-1)$-morphism are in the subcategory $\Funoplax(\Adj, \cC) \subseteq \Funoplax(\globe_1, \cC)$ and whose top-dimensional filler is a left-adjoint $(n+1)$-morphism in $\cC$.
\end{itemize}
\end{prop}

\begin{proof}
We address the inclusion $\Funlax(r, \cC) : \Funlax(\Adj, \cC) \to \Funlax(\globe_1, \cC)$; the oplax version is identical. 
First note that $\Funlax(r, \cC)$ is indeed a subcategory inclusion: for any $(\infty,\infty)$-category $\cE$, the induced map $\Map(\cE, \Funlax(\Adj, \cC)) \to \Map(\cE, \Funlax(\globe_1, \cC))$ is a fully faithful inclusion, since it agrees (via currying) with the map $\Map(\Adj, \Funoplax(\cE, \cC)) \to \Map(\globe_1, \Funoplax(\cE, \cC))$.

In particular, the space $\Map(\globe_n, \Funlax(\Adj, \cC))$ of $n$-morphisms in $\Funlax(\Adj, \cC)$ is precisely the full subspace
\[ \Map(\Adj, \Funoplax(\globe_n, \cC)) \subseteq \Map(\globe_1, \Funoplax(\globe_n, \cC))\] which is characterized in Proposition~\ref{prop:ClaudiaTheo}. 
\end{proof}

\begin{example} Recall that a $n$-morphism in an $(\infty, n)$-category is adjunctible if and only if it is invertible. Hence, Proposition~\ref{prop:FunlaxAdj} recovers \cite[Cor.~4.8]{2002.01037}, which states that for $\cC$ an $(\infty,2)$-category, the $(\infty,2)$-category $\Funlax(\Adj, \cC)$ is equivalent to the full subcategory of the (strong!)\ functor category $\Fun(\globe_1, \cC)$ on the right adjoint arrows.
\end{example}

For later use, we record the following immediate consequence of Proposition~\ref{prop:FunlaxAdj}: 
\begin{corollary}\label{cor:adjunctibilitytop}Let $n \geq 0$. 
\begin{enumerate}
\item The functor $\globe_n \otimes \globe_1 \to[\id \otimes l] \globe_n \otimes \Adj$ sends the generating $(n+1)$-morphism to a left-adjoint $(n+1)$-morphism in $\globe_n \otimes \Adj$.
\item The functor $\globe_1 \otimes \globe_n \to[r \otimes \id] \Adj \otimes \globe_n$ sends the generating $(n+1)$-morphism to a right-adjoint $(n+1)$-morphism in $\Adj \otimes \globe_n$.  \qed
\end{enumerate}
\end{corollary}

\subsection{The calculus of mates} \label{subsec:mates}

Proposition~\ref{prop:ClaudiaTheo} answers directly the question of factoring a functor $\alpha : \globe_m \otimes \globe_n \to \cC$ through the epimorphisms $\globe_m \otimes l : \globe_m \otimes \globe_n \to \globe_m \otimes \Adj_n$ and $r \otimes \globe_n : \globe_m \otimes \globe_n \to \Adj_m \otimes \globe_n$. But we will also want to understand the problem of factoring through $\globe_m \otimes r$ and $l \otimes \globe_n$, and the most na\"ive generalization to these cases of Proposition~\ref{prop:ClaudiaTheo} is false. The correct answer in these cases requires the notion of the \define{mates} of $\alpha$. We will only need the case $m=n=1$ in the sequel, and so restrict to that case now.

Suppose we are given a lax square $\alpha$ as follows: 
\begin{equation} \label{eqn:basicsquare}
\begin{tikzcd}[sep=4em]
A \arrow[d, "h", swap] \arrow[r, "f"] & B \arrow[d, "g"]\\
C \arrow[r, "k", swap]  \arrow[ur, Rightarrow, shorten <=7pt, shorten >=7pt, "\alpha" description] & D
\end{tikzcd}
\end{equation}
Suppose furthermore that the (horizontal) 1-morphisms $f : A \to B$ and $k : C \to D$ are left adjoints, with right adjoints $f^R : B \to A$ and $k^R : D \to C$. In this case, $\alpha$ can be ``rotated $90^\circ$ clockwise'' by declaring:
\begin{equation}\label{eqn:rmate}
\begin{tikzcd}[column sep=4em, row sep = tiny]
  & B \arrow[dl, "f^R", swap] \arrow [dddd, "g"] &&   & B \arrow[dl, "f^R", swap] \arrow[dd, equal] \\
A \arrow[dddd, "h", swap] &   && A \arrow[dd, "h", swap] \arrow[dr, "f", swap] \arrow[r, Rightarrow, shorten <=20pt, shorten >=10pt] &  \mbox{}  \\
  &   & := &  & B \arrow[dd, "g"] \\
  & && C \arrow[dr, "k"] \arrow[dd, equal] \arrow[ur, Rightarrow, shorten <=10pt, shorten >=10pt, "\alpha" description] & \\
  & D \arrow[dl, "k^R"] \arrow[uuul, Leftarrow, shorten <=7pt, shorten >=7pt, "\alpha^\rmate" description] && \mbox{} \arrow[r, Rightarrow, shorten <=10pt, shorten >=20pt] & D \arrow[dl, "k^R"] \\
C &   && C &   
\end{tikzcd}
\end{equation}
The triangles are filled with unit and counit data for the adjunctions $f \dashv f^R$ and $k \dashv k^R$. The morphism $\alpha^\rmate$ is called the \define{right mate} of $\alpha$.

Suppose instead that the (vertical) 1-morphisms $h : A \to C$ and $g : B \to D$ are right adjoints, with left adjoints $h^L$ and $g^L$. In this case, $\alpha$ can be ``rotated $90^\circ$ counterclockwise'' by declaring:
\begin{equation}\label{eqn:lmate}
\begin{tikzcd}[row sep=4em, column sep = tiny]
  & A \arrow[rrrr, "f"] \arrow[ddrrr, Leftarrow, shorten <=15pt, shorten >=15pt, "\alpha^\lmate" description] &   & \mbox{}  &   & B & &   & A \arrow[ddr, "h"] \arrow[rr, "f"] &   & B \arrow[rr, equal] \arrow[ddr, "g", swap] & \mbox{}  & B\\
  &   &   &   &   &   &:=&   &   &   &   &   &  \\
C\arrow[uur, "h^L"] \arrow[rrrr, "k"] &   & \mbox{}  &   & D\arrow[uur, "g^L", swap] &   & & C \arrow[uur, "h^L"] \arrow[rr, equal] & \mbox{}\arrow[uu, Rightarrow, shorten <=15pt, shorten >=30pt]  & C \arrow[rr, "k"] \arrow[uur, Rightarrow, shorten <=15pt, shorten >=15pt, "\alpha" description]&   & D \arrow[uur, "g^L", swap] \arrow[uu, Rightarrow, shorten <=30pt, shorten >=15pt] &  
\end{tikzcd}
\end{equation}
The morphism $\alpha^\lmate$ is called the \define{left mate} of $\alpha$.

\begin{remark}
Suppose that the sides of $\alpha$ are sufficiently adjunctible for $\alpha^\rmate$ be defined: the (horizontal) 1-morphisms $f : A \to B$ and $k : C \to D$ are left adjoints. Then $(\alpha^\rmate)^\lmate$ is defined, and the axioms of adjunction imply that, up to coherence isomorophisms,
$$ (\alpha^\rmate)^\lmate = \alpha.$$
Similarly, if $\alpha^\lmate$ is defined, then
$$ (\alpha^\lmate)^\rmate = \alpha,$$
again up to coherence isomorphisms.
\end{remark}

\begin{lemma}\label{lem:adjointsofmates}
  Suppose $\alpha : k \circ h \Rightarrow g \circ f$ is a square as in \eqref{eqn:basicsquare} in which $f$ and $k$ are left adjoints and $g$ and $h$ are right adjoints, so that both $\alpha^\rmate$ and $\alpha^\lmate$ are defined. Then, the $2$-morphism $\alpha^\rmate: h\circ f^R \To k^R \circ g$ is a right adjoint (resp.\ left adjoint) if and only if $\alpha^\lmate:g^L \circ k \To f \circ h^L$ is a left adjoint (resp.\ right adjoint), and there is an equivalence of $2$-morphisms $h^L \circ k^R \To f^R \circ g^L$:
    \[ ((\alpha^\lmate)^R)^\rmate \simeq ((\alpha^\rmate)^L)^\lmate\]
    (resp.\ of $2$-morphisms $((\alpha^\lmate)^L)^\rmate \simeq ((\alpha^\rmate)^R)^\lmate$).
\end{lemma}
\begin{proof}
  We have $\alpha^\lmate = (\alpha^\rmate)^{\lmate\lmate}$. But, on the squares for which it is defined, $(-)^{\lmate\lmate}$ is contravariant for the composition of 2-morphisms and covariant for the composition of 3-morphisms and hence sends right adjoints to left adjoints and left adjoints to right adjoints.
\end{proof}

\begin{remark}\label{rem:verticalandhorizontaladjoints}
Lax squares like \eqref{eqn:basicsquare}, in an $(\infty,\infty)$-category $\cC$, have two natural directions of composition: they are the 2-cells in a double category (more precisely, the $(1;1)$-cells in a double $(\infty,\infty)$-category; compare \cite{JFS}). The horizontal composition (in our diagrams) is the composition of 1-morphisms in $\Funoplax(\globe_1,\cC)$, and the vertical composition (in our diagrams) is the composition of 1-morphisms in $\Funlax(\globe_1,\cC)$.

Upon unpacking the proof, Proposition~\ref{prop:ClaudiaTheo} says what are the adjoints (if they exist) of a square \eqref{eqn:basicsquare} under each of these compositions. The horizontal left and right adjoints are:
\[
  \left(
\begin{tikzcd}[sep=3em]
A \arrow[d, "h", swap] \arrow[r, "f"] & B \arrow[d, "g"]\\
C \arrow[r, "k", swap]  \arrow[ur, Rightarrow, shorten <=4pt, shorten >=4pt, "\alpha" description] & D
\end{tikzcd}  
  \right)^{h L} = 
  \begin{tikzcd}[sep=3em]
B \arrow[d, "g", swap] \arrow[r, "f^L"] & A \arrow[d, "h"]\\
D \arrow[r, "k^L", swap]  \arrow[ur, Rightarrow, shorten <=4pt, shorten >=4pt, "(\alpha^L)^\lmate" description] & C
\end{tikzcd}
,\quad
  \left(
\begin{tikzcd}[sep=3em]
A \arrow[d, "h", swap] \arrow[r, "f"] & B \arrow[d, "g"]\\
C \arrow[r, "k", swap]  \arrow[ur, Rightarrow, shorten <=4pt, shorten >=4pt, "\alpha" description] & D
\end{tikzcd}  
  \right)^{h R} = 
  \begin{tikzcd}[sep=3em]
B \arrow[d, "g", swap] \arrow[r, "f^R"] & A \arrow[d, "h"]\\
D \arrow[r, "k^R", swap]  \arrow[ur, Rightarrow, shorten <=4pt, shorten >=4pt, "(\alpha^\rmate)^R" description] & C
\end{tikzcd}.
\]
The vertical left and right adjoints are:
\[
\left(
\begin{tikzcd}[sep=3em]
A \arrow[d, "h", swap] \arrow[r, "f"] & B \arrow[d, "g"]\\
C \arrow[r, "k", swap]  \arrow[ur, Rightarrow, shorten <=4pt, shorten >=4pt, "\alpha" description] & D
\end{tikzcd}  
\right)^{v L} = 
\begin{tikzcd}[sep=3em]
C \arrow[d, "h^L", swap] \arrow[r, "k"] & D \arrow[d, "g^L"]\\
A \arrow[r, "f", swap]  \arrow[ur, Rightarrow, shorten <=4pt, shorten >=4pt, "(\alpha^\lmate)^L" description] & B
\end{tikzcd}
,\quad
\left(
\begin{tikzcd}[sep=3em]
A \arrow[d, "h", swap] \arrow[r, "f"] & B \arrow[d, "g"]\\
C \arrow[r, "k", swap]  \arrow[ur, Rightarrow, shorten <=4pt, shorten >=4pt, "\alpha" description] & D
\end{tikzcd}  
\right)^{v R} = 
\begin{tikzcd}[sep=3em]
C \arrow[d, "h^R", swap] \arrow[r, "k"] & D \arrow[d, "g^R"]\\
A \arrow[r, "f", swap]  \arrow[ur, Rightarrow, shorten <=4pt, shorten >=4pt, "(\alpha^R)^\rmate" description] & B
\end{tikzcd}.
\]
\end{remark}

\subsection{The lax tensor square of the walking adjunction}\label{sec:adjsquared}

We now apply Propositions~\ref{prop:ClaudiaTheo} and~\ref{prop:FunlaxAdj} to present $\Adj \otimes \Adj$ by starting with $\globe_1 \otimes \globe_1$ and adding adjunctibility.
We record for reference:
\begin{lemma}\label{lem:twohandedadjoint}
  For an $n$-morphism $f$ in an $(\infty,\infty)$-category, the following conditions are equivalent:
  \begin{itemize}
    \item $f$ is a right adjoint, and the unit and counit of the adjunction $f^L \dashv f$ are left adjoints.
    \item $f$ is a left adjoint, and the unit and counit of the adunction $f \dashv f^R$ are right adjoints.
  \end{itemize}
  In this case, the adjoints of unit and counit exhibit any left adjoint $f^L$ also as a right adjoint of $f$; and exhibit any right adjoint $f^R$ also as a left adjoint. 
\end{lemma}
These and similar equivalences, and their applications, are thoroughly explored in \cite{2312.05051}.
\begin{proof}
  Indeed, suppose that $f : X \to Y$ admits a left adjoint $f^L : Y \to X$ with unit $\eta : \id_Y \to ff^L$ and counit $\epsilon : f^L f \to \id_X$, which in turn  admit right adjoints $\eta^R$ and $\epsilon^R$; then $\eta^R$ and $\epsilon^R$ are, respectively, the counit and unit of an adjunction $f \dashv f^L$, witnessing $f^L$ as a right adjoint to $f$.
\end{proof}

\begin{remark}\label{rem:maxingout}
  If $\cC$ is an $(\infty,n+1)$-category, then the conditions in Lemma~\ref{lem:twohandedadjoint} are equivalent to asking that $f$ itself is invertible. Indeed, in this case the unit and counit who are asserted to have adjoints are top-dimensional morphisms, and so they have adjoints only when they are invertible; but an adjunctible morphism with invertible unit and counit is itself invertible.
\end{remark}

\begin{prop}\label{prop:adjFunlaxAdj}
  Consider a $1$-morphism $\eta:s \to t$ in $\Funlax(\Adj, \cC)$, represented via the inclusion $r:\globe_1 \to \Adj$ and Proposition~\ref{prop:FunlaxAdj} by a square in $\cC$ 
\[\begin{tikzcd}[sep=3em]
\bullet \arrow[r, "s"] \arrow[d, "\eta_0"'] &\arrow[d, "\eta_1"] \bullet \\
\bullet \arrow[r, "t"'] \arrow[Rightarrow, ur,"\eta_{<}" description, shorten <=4pt, shorten >=4pt,] & \bullet
\end{tikzcd}\]
for which $s,t$ and $\eta_{<}$ are right-adjoint morphisms in $\cC$. Then, $\eta$ is a \emph{left}-adjoint $1$-morphism in $\Funlax(\Adj, \cC)$ if and only if moreover: 
\begin{itemize}
  \item[{[i]}] $\eta_0$ and $\eta_1$ are left adjoints.
  \item[{[ii]}] 
  $(\eta_{<})^R \simeq (\eta_{<})^L$ satisfies the equivalent conditions of Lemma~\ref{lem:adjointsofmates}: $((\eta_{<})^R)^\rmate$ is a right adjoint; equivalently  $((\eta_{<})^L)^\lmate$ is a left adjoint.
  \item[{[iii]}] 
  $\eta_<$ is satisfies the equivalent conditions of Lemma~\ref{lem:twohandedadjoint}: $\eta_<$ is a right adjoint, and the unit and counit of the adjunction $(\eta_<)^L \dashv \eta_<$ are left adjoints; or equivalently,  $\eta_<$ is a left adjoint, and the unit and counit of the adjunction $\eta_< \dashv (\eta_<)^R$ are right adjoints.
\end{itemize}
\end{prop}
The existence of the isomorphism $(\eta_{<})^R \simeq (\eta_{<})^L$ attested in condition~[ii] follows via Lemma~\ref{lem:twohandedadjoint} from condition~[iii].

\begin{proof}
  Restricting along $r : \globe_1 \to \Adj$, any left adjoint $1$-morphism $\eta$ in $\Funlax(\Adj, \cC)$ is sent to a left adjoint $1$-morphism in $\Funlax(\globe_1, \cC)$ and hence by Proposition~\ref{prop:ClaudiaTheo} to a square for which $\eta_0$ and $\eta_1$ and $\eta_<$ are left adjoints in $\cC$; implying condition~[i] from the statement of the Proposition. Explicitly, the right adjoint to $\eta$ in $\Funlax(\globe_1, \cC)$ is given by the square $((\eta_{<})^R)^\rmate : s \circ \eta_0^R \To \eta_1^R \circ t$.
  
  To lift this adjunction $\eta \dashv \eta^R$ into the subcategory $\Funlax(r, \cC) : \Funlax(\Adj, \cC) \mono \Funlax(\globe_1, \cC)$, it suffices to show that $\eta^R$, and also the unit and counit of the adjunction $\eta \dashv \eta^R$, lift. We will repeatedly use Proposition~\ref{prop:FunlaxAdj}, which says that a morphism lifts against $\Funlax(r, \cC)$ if and only if its source and target lift and also its top-dimensional filler is a right adjoint in $\cC$.

  Thus we find that $\eta$ is a left adjoint in $\Funlax(\Adj, \cC)$ if and only if:
  \begin{itemize}
    \item The source $t$ and target $s$ of $\eta^R$ live in $\Funlax(\Adj, \cC)$. This is already assumed.
    \item The 2-morphism $(\eta^R)_{<} = ((\eta_{<})^R)^\rmate$ is a right adjoint in $\cC$. This is precisely condition~[ii] from the statement of the Proposition.
    \item The source and target of the unit and counit of $\eta \dashv \eta^R$ live in $\Funlax(\Adj, \cC)$. This happens as soon as $\eta$ and $\eta^R$ live in $\Funlax(\Adj, \cC)$, which happens once the previous conditions are satisfied.
    \item The 3-morphisms filling the unit and counit of $\eta \dashv \eta^R$ are right adjoints in $\cC$. 
  \end{itemize}
  
  Only the last bullet point requires any unpacking. The 3-cell filling the unit of $\eta \dashv \eta^R$ is the 1-morphism in $\Hom_{\cC}(t \circ \eta_0 \circ \eta_0^R, t)$ that under the equivalence with $\End_{\cC}(t \circ \eta_0)$ becomes the counit of the adjunction $\eta_{<} \dashv (\eta_{<})^R$. The 3-cell filling the counit of $\eta \dashv \eta^R$ is the $1$-morphism in $\Hom_{\cC}(s, s \circ \eta_1 \circ \eta_1^R)$ that under the equivalence with $\End_{\cC}(s \circ \eta_1)$ becomes the unit of the adjunction $\eta_{<} \dashv (\eta_{<})^R$.  
  Thus these 3-cell fillings are right adjoints if and only if those (co)units are right adjoints, and so we find precisely condition~[iii] from the statement of the Proposition.
\end{proof}

Changing some letters, we find:

\begin{corollary} \label{cor:AdjotimesAdj}
  Let $l, r: \globe_1 \to \Adj$ the inclusion of the generating left, resp.\ right adjoint $1$-morphism in $\Adj$. 
  Then, given an $(\infty,\infty)$-category $\cC$, precomposition with the functor $r\otimes l : \globe_1 \otimes \globe_1 \to \Adj\otimes \Adj$ identifies $\Map(\Adj \otimes \Adj, \cC)$ with the full subspace of $\Map(\globe_1 \otimes \globe_1,\cC)$ on those squares in $\cC$ 
  \[\begin{tikzcd}[sep=3em]
\bullet \arrow[r, "f"] \arrow[d, "h"'] &\arrow[d, "g"] \bullet \\
\bullet \arrow[r, "k"'] \arrow[Rightarrow, ur,"\beta" description, shorten <=4pt, shorten >=4pt,] & \bullet
\end{tikzcd}\]
with the following adjunctibility conditions:
\begin{itemize}
\item[{[i]}] $f$ and $k$ are right adjoints and
 $g$ and $h$ are left adjoints.
\item[{[ii]}] $\beta^R \simeq \beta^L$ satisfies the equivalent conditions of Lemma~\ref{lem:adjointsofmates}: $(\beta^R)^\rmate$ is a right adjoint; equivalently $(\beta^L)^\lmate$ is a left adjoint.
\item[{[iii]}] $\beta$ satisfies the equivalent conditions of Lemma~\ref{lem:twohandedadjoint}: $\beta$ is a right adjoint, and the unit and counit of the adjunction $\beta^L \dashv \beta$ are left adjoints; equivalently  $\beta$ is a left adjoint, and the unit and counit of the adjunction $\beta \dashv \beta^R$ are right adjoints.
 \qed
\end{itemize}
\end{corollary}

In the sequel, we will be most interested in the localization $L_3(\Adj \otimes \Adj)$. Comparing Corollary~\ref{cor:AdjotimesAdj} with Remark~\ref{rem:maxingout} (and recalling Lemma~\ref{lem:laxvstrongsquares}), we find:

\begin{corollary}\label{cor:L3AdjAdj} Let $\cC$ be an $(\infty,3)$-category. Precomposition with the functor $r\otimes l : \globe_1 \otimes \globe_1 \to \Adj\otimes \Adj$ identifies $\Map(\Adj \otimes \Adj, \cC)$ with the full subspace of $\Map(\globe_1\times\globe_1,\cC) \subseteq \Map([1] \otimes [1], \cC)$ on the \emph{strongly commuting} squares 
 \[\begin{tikzcd}[sep=3em]
\bullet \arrow[r, "f"] \arrow[d, "h"'] &\arrow[d, "g"] \bullet \\
\bullet \arrow[r, "k"'] 
\arrow[Rightarrow, ur, phantom, sloped, "\overset\beta{\overset{\textstyle\sim} \Rightarrow}"] 
& \bullet
\end{tikzcd}\]
with the following adjunctibility conditions:
\begin{itemize}
\item[{[i]}] $f$ and $k$ are right adjoints and
 $h$ and $g$ are left adjoints.
\item[{[ii]}] $\beta^{-1}$ satisfies the equivalent conditions of Lemma~\ref{lem:adjointsofmates}: $(\beta^{-1})^\rmate$ is a right adjoint; $(\beta^{-1})^\lmate$ is a left adjoint. \qed
\end{itemize}
\end{corollary}

\subsection{The lax smash squares of the walking adjunction} \label{subsec:smashsquares}

We turn our attention from tensor squares of categories to smash squares of pointed (higher) categories. 

\begin{notation}
We will decide to use \emph{the source of the left adjoint} to point the walking adjunction $\Adj$. This choice is determined by the following criteria: like most mathematicians, we find monads and algebras easier to think about than comonads and coalgebras; the walking monad $\Mnd$ (formally introduced in the next section) is naturally pointed, as it has a single object; we want the canonical inclusion $\Mnd \hookrightarrow \Adj$ to be a map of pointed categories.
\end{notation}

\begin{remark}\label{rem:Adjopcopointed}
  Recall our convention that, for an $(\infty,\infty)$-category, $(-)^\op$ denotes the ``all odds opposite'' and $(-)^\co$ denotes the ``all evens opposite.'' In particular, since $\Adj$ is a $2$-category, $\Adj^\op = \Adj^{\text{1-op}}$ and $\Adj^\co = \Adj^{\text{2-op}}$. 
  There is an interesting equivalence of pointed $2$-categories $\Adj^\op \simeq \Adj$: it is the identity on objects, exchanges the two generating $1$-morphisms, it is the identity on generating $2$-morphisms. There is also an isomorphism $\Adj^\co \simeq \Adj$ of unpointed $2$-categories, but it exchanges the two objects and so does not preserve the choice of pointing.
\end{remark}

\begin{notation} 
The walking arrow $\globe_1$ admits two pointings, which we will distinguish under the names $\globe_1^l$ and $\globe_1^r$, chosen so that the functors picking out the generating left, resp. right adjoint 
$$ l : \globe_1^l \to \Adj \qquad \text{and} \qquad r : \globe_1^r \to \Adj$$
are pointed functors. In other words, $\globe_1^l$ denotes $\globe_1$ pointed at the source object $0 \in \globe_1$, and $\globe_1^r$ denotes $\globe_1$ pointed at the target object $1 \in \globe_1$. Note that $\globe_1^l$ and $\globe_1^r$ are each fixed under $(-)^\co$ and are exchanged under $(-)^\op$.
\end{notation}

\begin{example}\label{eg:retract}
  Let $(\cC,1_\cC)$ be a pointed $(\infty, \infty)$-category. Then, by definition \[\Map_*(\globe_1^r \owedge \globe_1^l, \cC):=\Map(\globe_1 \otimes \globe_1, \cC) \times_{\Map(\globe_1 \otimes \{0\} \cup_{\{1\} \otimes \{0\}} \{1\} \otimes \globe_1, \cC)} \{\mathrm{const}_{1_{\cC}}\}
\] 
which we can informally think of as the space of diagrams in $\cC$ of the following shape:
\[\begin{tikzcd}[sep=3em]
1_{\cC}  \arrow[equal, r] \arrow[d]&  1_{\cC} \arrow[equal, d] \\\bullet \arrow[r] \arrow[Rightarrow, ur, shorten <=4pt, shorten >=4pt,] & 1_{\cC}
\end{tikzcd}
\]
\end{example}

\begin{corollary}\label{cor:adjowedgeadj}
  Let $(\cC, 1_\cC)$ be a pointed $(\infty,\infty)$-category. Precomposition with $r\owedge l : \globe_1^r \owedge \globe_1^l \to \Adj\owedge \Adj$ identifies $\Map_*(\Adj \owedge \Adj, \cC)$ with the full subspace of $\Map(\globe_1^r \owedge \globe_1^l,\cC)$ on those squares in $\cC$ 
  \[\begin{tikzcd}[sep=3em]
1_\cC \arrow[r, equal] \arrow[d, "h"'] &\arrow[d, equal] 1_\cC \\
\bullet \arrow[r, "k"'] \arrow[Rightarrow, ur,"\beta" description, shorten <=4pt, shorten >=4pt,] & 1_\cC
\end{tikzcd}\]
with the following adjunctibility conditions:
\begin{itemize}
\item[{[i]}] 
 $h$ is a left adjoint and
$k$ is a right adjoint.
\item[{[ii]}] $\beta^R \simeq \beta^L$ satisfies the equivalent conditions of Lemma~\ref{lem:adjointsofmates}: $(\beta^R)^\rmate$ is a right adjoint; equivalently  $(\beta^L)^\lmate$ is a left adjoint.
\item[{[iii]}] $\beta$ satisfies the equivalent conditions of Lemma~\ref{lem:twohandedadjoint}: $\beta$ is a right adjoint, and the unit and counit of the adjunction $\beta^L \dashv \beta$ are left adjoints; equivalently $\beta$ is a left adjoint, and the unit and counit of the adjunction $\beta \dashv \beta^R$ are right adjoints.
\end{itemize}
\end{corollary}
\begin{proof}Immediate from Corollary~\ref{cor:AdjotimesAdj}.
\end{proof}

If instead of the lax smash product $\owedge$ we used the  strong smash product $\wedge$, defined from the Cartesian product $\times$ rather than $\otimes$, then we may think of \[\Map_*(\globe_1^r \wedge \globe_1^l, \cC):=\Map(\globe_1 \times \globe_1, \cC) \times_{\Map(\globe_1 \times \{0\} \cup_{\{1\} \times \{0\}} \{1\} \times \globe_1, \cC)} \{\mathrm{const}_{1_{\cC}}\}
\]  as the space of diagrams of shape:
 \begin{equation}\label{eq:strongsmash}
 \begin{tikzcd}[sep=3em]
1_\cC \arrow[r, equal] \arrow[d, ] &\arrow[d, equal] 1_\cC \\
\bullet \arrow[r, ] 
\arrow[Rightarrow, ur, phantom, sloped, "{\simeq}"] 
& 1_\cC
\end{tikzcd} 
\end{equation}
Notice that diagrams of shape~\eqref{eq:strongsmash} are precisely section-retraction pairs in $\cC$. To make this precise, we follow~\cite{HTT} and let $\RetWalking$ denote the \emph{walking retract}, the strict 1-category with two objects $a$ and $b$ and two generating morphisms $s:a\to b$ and $r:b \to a$ satisfying the equation $r\circ s = \id_a$. We will henceforth consider $\RetWalking$ as pointed at the object $a$. 
When computing in $\cat{StrCat}_1$, $\RetWalking$ is definitionally the quotient\footnote{Recall from the introduction to \S\ref{subsec:owedge}  that for a map $f:A \to B$ in a category with pushouts, we write $B/A$ for the pushout $* \cup_{A} B$,  with its canonical pointing.} $[2] / [1]$ of the $2$-simplex $[2]$ by its ``long edge''  $[1] \simeq \{0<2\} \hookrightarrow \{0<1<2\} = [2]$. The discussion of $\RetWalking$ at the start of \cite[Section 4.4.5]{HTT} verifies that $\RetWalking$ remains the quotient $[2]/[1]$ when the latter is computed in $\Cat_{(\infty,1)}$.

\begin{lemma}\label{lem:retract}
The category $\globe_1^r \wedge \globe_1^l$ is equivalent to $\RetWalking$. 
\end{lemma}
\begin{proof}
By definition, $\globe_1^r \wedge \globe_1^l$ is the quotient $\globe_1 \times \globe_1/[2]$ (with its canonical pointing) of the full inclusion of posets $[2]=\{00<10<11\} \hookrightarrow \{0<1\} \times \{0<1\}$.  
The diagram of full poset inclusions  
\[ \begin{tikzcd}
\{00 < 01 <11\} \arrow[hook,r] & \{0<1\} \times \{0<1\} \\
\{00<11\} \arrow[hook,u] \arrow[hook,r] &\{00<10<11\}\arrow[u, hook]
\end{tikzcd}\]
is (after taking nerves) a pushout in $\Fun(\Delta^{\op}, \Set)$ and hence, since all involved maps are monomorphisms, a pushout in $\Fun(\Delta^{\op}, \Spaces)$ and since all involved simplicial spaces are already $(\infty,1)$-categories, it is in fact a pushout in $\Cat_{(\infty,1)}$. Thus, the quotient of the left vertical map $[2]/[1] \simeq \RetWalking$ is equivalent to the quotient of the right vertical map $[1] \times [1] /2 =: \globe_1^r \wedge \globe_1^l$. 
\end{proof}

In the sequel, we will be most interested in the localization $L_3(\Adj \owedge\Adj)$ to an $(\infty, 3)$-category, which can be characterized as follows:

\begin{corollary}\label{cor:smashretract}
Let $(\cC, 1_\cC)$ be a pointed $(\infty,3)$-category. 

Precomposition with $r\owedge l : \globe_1^r \owedge \globe_1^l \to \Adj\owedge \Adj$ identifies $\Map_*(\Adj \owedge \Adj, \cC)$ with the full subspace of $ \Map_*(\RetWalking, \cC) \simeq \Map_*(\globe_1^r \wedge \globe_1^l,\cC) \subseteq \Map_*(\globe_1^r \owedge \globe_1^l, \cC)$ on those section-retraction pairs
 \[\begin{tikzcd}[sep=3em]
1_\cC \arrow[r, equal] \arrow[d, "h"'] &\arrow[d, equal] 1_\cC \\
\bullet \arrow[r, "k"'] 
\arrow[Rightarrow, ur, phantom, sloped, "\overset\beta{\overset{\textstyle\sim} \Rightarrow}"] 
& 1_\cC
\end{tikzcd}\]
such that:
\begin{itemize}
\item[{[i]}]
$h$ is a left adjoint and $k$ is a right adjoint.
\item[{[ii]}] Any of the following equivalent conditions hold:
\begin{itemize}
  \item[{\ensuremath\bullet}] $(\beta^{-1})^\rmate :h^R \To k$ is a right adjoint,
  \item[{\ensuremath\bullet}] $(\beta^{-1})^\lmate: k^L \To h$ is a left adjoint,
  \item[{\ensuremath\bullet}] the whiskering $k\ev_h : khh^R \Rightarrow k$ is a right adjoint,
  \item[{\ensuremath\bullet}] the whiskering $\ev_kh : k^Lkh \Rightarrow h$ is a left adjoint.
\end{itemize}
\end{itemize}
\end{corollary}
\begin{proof} Immediate from Corollary~\ref{cor:L3AdjAdj}, Lemma~\ref{lem:retract}, and the invertibility of $\beta$.
\end{proof}

\begin{definition}\label{def:adjunctibleretract} We will refer to a section-retraction pair $(h, k, \beta: kh \simeq \id_{1_{\cC}})$ satisfying the conditions of Corollary~\ref{cor:smashretract} as an \emph{adjunctible retract}. 
\end{definition}

\begin{remark}\label{rem:adjunctibleretractsinCop}
  Recall from Remark~\ref{rem:Adjopcopointed} that there is an equivalence $\Adj^\op \simeq \Adj$ in $\Cat^*_{(\infty,\infty)}$. Thus, for $\cC$ a pointed $(\infty,3)$-category, using Lemma~\ref{lem:propertiesofFunoplax:compatwithcoop}, we find an isomorphism of spaces
  \begin{multline*}
   \{\text{adjunctible retracts in $\cC^\op$}\} = \maps_*(\Adj \owedge \Adj, \cC^\op) \simeq \maps_*((\Adj \owedge \Adj)^\op, \cC) \\ \simeq \maps_*(\Adj^\op \owedge \Adj^\op, \cC) \simeq \maps_*(\Adj \owedge \Adj, \cC) \simeq \{\text{adjunctible retracts in $\cC$}\}.
   \end{multline*}
\end{remark}

\begin{remark} 
 For $\cC$ a pointed $(\infty,3)$-category, rather than identifying  $\Map_*(\Adj \owedge \Adj, \cC)$ via  $r\owedge l : \globe_1^r \owedge \globe_1^l \to \Adj\owedge \Adj$ with a full subspace of $ \Map_*(\RetWalking, \cC) \simeq \Map_*(\globe_1^r \wedge \globe_1^l,\cC) \subseteq \Map_*(\globe_1^r \owedge \globe_1^l, \cC)$ as in Corollary~\ref{cor:smashretract}, we may also identify it via $l \owedge l: \globe_1^l \owedge \globe_1^l \to \Adj \owedge \Adj$ with the full subspace of $\Map(\globe_1^l \owedge \globe_1^l, \cC)$ on those diagrams 
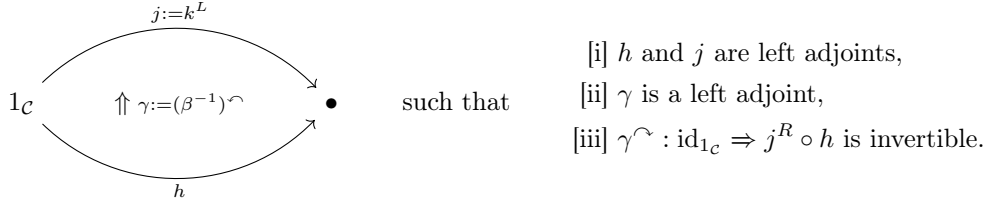
\begin{align*}
\begin{tikzcd}[column sep=.75cm,ampersand replacement=\&]
1_\cC
\arrow[bend left=45]{rr}{j:=k^L}
\arrow[bend right=45,swap]{rr}{h}
\&
\Uparrow {\scriptstyle \gamma := (\beta^{-1})^\lmate}
\&
\bullet
\end{tikzcd}
\qquad
&\text{such that }\qquad
\begin{aligned}
  \text{[i] } &\text{$h$ and $j$ are left adjoints,} \\
  \text{[ii] } &\text{$\gamma$ is a left adjoint,}\\
  \text{[iii] } & \text{$\gamma^\rmate: \id_{1_{\cC}} \To j^R \circ h$ is invertible.}
\end{aligned}
\end{align*}
Alternately, one could use $r\owedge r: \globe_1^r \owedge \globe_1^r \to \Adj \owedge \Adj$ to work entirely with right adjoints.
The presentation in Corollary~\ref{cor:smashretract} in terms of retracts is the most useful for our application.
\end{remark}

\begin{definition}\label{defn:Retlax}
  Given a pointed $(\infty,\infty)$-category $\cC \in \Cat^*_{(\infty,\infty)}$, we define the $(\infty,\infty)$-category  of \define{retracts and oplax morphisms of retracts} in $\cC$  as \[\Ret^\oplax(\cC):=\Funoplax_*(\RetWalking,\cC).\] 
\end{definition}
\begin{example}\label{ex:laxretract}
Using that $\Map(\globe_1, \Funoplax_*(\Ret, \cC)) \simeq \Map_*(\Ret, \Funlax(\globe_1, \cC))$ and the description of $\Funlax(\globe_1, \cC)$ from Example~\ref{exm:funlaxone}, the objects and $1$-morphisms of $\Ret^\oplax(\cC)$ may  be unpacked as follows:
\begin{itemize}
\item an object consists of a tuple $(X, h, k, \beta)$ where $X$ is an object, $1_{\cC} \to[h] X\to[k] 1_{\cC}$ are $1$-morphisms and $\beta: kh\Isom \id_{1_{\cC}}$ is a $2$-isomorphism;
\item a 1-morphism $(X_0, h_0, k_0,  \beta_0) \to (X_1, h_1, k_1,  \beta_1)$ consists of a tuple $(X_{<}, h_{<}, k_{<}, \beta_{<})$ where $X_{<}: X_0 \to X_1$ is a $1$-morphism, $h_{<}: X h_0 \To h_1$ and $k_{<}: k_0 \To k_1 X$  are $2$-morphisms and \[ \beta_{<}: \beta_1 \circ (k_1 h_{<}) \circ (k_{<} h_0) \Iisom   \beta_0\]
is a $3$-isomorphism. 
\end{itemize}
\end{example}

\begin{corollary}\label{cor:retadjissubofretlax}
  For any pointed $(\infty,3)$-category $\cC$,
  precomposition with $r\owedge l : \globe_1^r \wedge \globe_1^l \to L_3(\Adj\owedge \Adj)$
   induces 
 a  sub-$(\infty,3)$-category inclusion 
   \begin{gather*}
   \Funoplax_*(\Adj \owedge \Adj, \cC) \hookrightarrow \Funoplax_*(\RetWalking,\cC) =: \Ret^\oplax(\cC). 
   \end{gather*}
\end{corollary}
\begin{proof}
  Corollary~\ref{cor:smashretract} implies that, after $3$-localization, the functor $r \owedge l$ becomes an epimorphism in $\Cat^*_{(\infty,3)}$. But if $\cA \to \cB$ is an epimorphism in any biclosed monoidal category, then for any $\cC$, the induced map $\underline{\hom}(\cB, \cC) \to \underline{\hom}(\cA, \cC)$ is a monomorphism.
\end{proof}

This justifies:

\begin{definition}\label{defn:retadjcategory}
  Let $\cC$ be a pointed $(\infty,3)$-category. We define the $(\infty,3)$-category of \emph{adjunctible retracts} as the following subcategory 
   of $\Ret^\oplax(\cC)$:
   \begin{gather*}\Ret^\someadj(\cC) := \Funoplax_*(\Adj \owedge \Adj, \cC). 
   \end{gather*}
\end{definition}

By Corollary~\ref{cor:smashretract}, the objects of $\Ret^{\someadj}(\cC)$ are precisely the adjunctible retracts from Definition~\ref{def:adjunctibleretract}; we will now describe the $1$-morphisms of the subcategory $\Ret^{\someadj}(\cC) \subseteq \Ret^{\oplax}(\cC)$:
\begin{prop}\label{prop:retadj} Let $(\cC,1_\cC)$ be a pointed $(\infty,3)$-category. 
An oplax morphism of retracts $(X_{<}, h_{<}, k_{<}, \beta_{<})$ (with notation as in Example~\ref{ex:laxretract}) is  in the subcategory $\Ret^{\someadj}(\cC)$ if and only if its source and target retracts are (i.e.\ satisfy the conditions of the first part of Corollary~\ref{cor:smashretract}) and if moreover the following hold:
 \begin{itemize}
 \item[{[i]}] The $2$-morphism $h_{<}:Xh_0 \To h_1$ is a left adjoint and the $2$-morphism $(k_{<})^{\lmate}: k_1^L \To X k_0^L$ is a right adjoint;
 \item[{[ii]}] The $3$-morphism given by whiskering the unit $\coev_{(k_{<})^{\lmate}}$ as follows is invertible: 
\[
\raisebox{-0.45cm}{
\scalebox{.6}{
\begin{tikzcd}[ampersand replacement=\&]
	1 \&\& {X_0} \&\& 1 \&\& {X_0} \&\& {X_1} \\
	\\
	\&\&\&\&\&\& \phantom{X_0}
	\arrow["{h_0}"{description}, from=1-1, to=1-3]
	\arrow[""{name=0, anchor=center, inner sep=0}, "{h_1}"{description}, curve={height=-90pt}, from=1-1, to=1-9]
	\arrow["{k_0}"{description}, from=1-3, to=1-5]
	\arrow[""{name=1, anchor=center, inner sep=0}, curve={height=-40pt}, equals, from=1-3, to=1-7]
	\arrow["{k_0^L}"{description}, from=1-5, to=1-7]
	\arrow["{X_{<}}"{description}, from=1-7, to=1-9]
	\arrow["{h_{<}}"', pos=0.7, shorten <=30pt, shorten >=4pt, Rightarrow, from=1-5, to=0]
	\arrow["{\ev_{k_0}}"'{pos=0.1}, shorten >=42pt, Rightarrow, from=1-5, to=0]
\end{tikzcd}
}}
~\Rrightarrow~
\raisebox{-0.45cm}{
\scalebox{.6}{
\begin{tikzcd}[ampersand replacement=\&]
	1 \&\& {X_0} \&\& 1 \&\& {X_0} \&\& {X_1} \\
	\\
	\&\&\&\&\&\& {X_0}
	\arrow["{h_0}"{description}, from=1-1, to=1-3]
	\arrow[""{name=0, anchor=center, inner sep=0}, "{h_1}"{description}, curve={height=-90pt}, from=1-1, to=1-9]
	\arrow["{k_0}"{description}, from=1-3, to=1-5]
	\arrow[""{name=1, anchor=center, inner sep=0}, curve={height=-40pt}, equals, from=1-3, to=1-7]
	\arrow["{k_0^L}"{description}, from=1-5, to=1-7]
	\arrow[""{name=2, anchor=center, inner sep=0}, "{k_1^L}"{description}, curve={height=30pt}, from=1-5, to=1-9]
	\arrow["{k_0^L}"{description}, curve={height=24pt}, from=1-5, to=3-7]
	\arrow["{X_{<}}"{description}, from=1-7, to=1-9]
	\arrow["{X_{<}}"{description}, curve={height=24pt}, from=3-7, to=1-9]
	\arrow["{h_{<}}"', pos=0.7, shorten <=30pt, shorten >=4pt, Rightarrow, from=1-5, to=0]
	\arrow["{~k_{<}^{\lmate}}"', pos=0.85,shorten <=26pt, shorten >=-1pt, Rightarrow, from=3-7, to=1-7]
	\arrow["{\ev_{k_0}}"'{pos=0.1}, shorten >=42pt, Rightarrow, from=1-5, to=0]
	\arrow["{(k_{<}^{\lmate})^L}"', pos=0.2, shorten >=22pt, shorten <=2pt, Rightarrow, from=3-7, to=1-7]
\end{tikzcd}
}}\]
 \end{itemize}
\end{prop}

\begin{proof}
Since by definition 
\begin{multline*}
\Map(\globe_1, \Funoplax_*(\Adj \owedge \Adj, \cC)) \simeq \Map_*((\globe_1)_+ \owedge \Adj \owedge \Adj, \cC) \\ \simeq \Map_*(\Adj \owedge \Adj, \Funlax(\globe_1, \cC)),
\end{multline*}
 it suffices to work out the conditions from Corollary~\ref{cor:smashretract} for a retract in $\Funlax(\globe_1, \cC)$. 
 The characterization of adjunctibility and adjoints in $\Funlax(\globe_1, \cC)$ from Proposition~\ref{prop:ClaudiaTheo} and Remark~\ref{rem:verticalandhorizontaladjoints} supplies condition [i] of Corollary~\ref{cor:smashretract} into condition [i] here. Since adjunctible $3$-morphisms in an $(\infty,3)$-category are invertible, condition [ii] of Corollary~\ref{cor:smashretract} translates to condition [ii] here.
\end{proof}
The description of which higher morphisms in $\Ret^\oplax(\cC)$ lift to $\Ret^{\someadj}(\cC)$ 
can be worked out similarly; we will not need it.

\begin{remark}
The conditions labeled [ii] in Proposition~\ref{prop:retadj} unpack the last bulleted versions of the conditions labeled [ii] in Corollary~\ref{cor:smashretract}, asking a whiskering of a (co)unit to be a left adjoint. There is an equivalent condition arising from unpacking the third bullet points in Corollary~\ref{cor:smashretract} which ask a whiskering of (co)unit to be a right adjoint. Since right adjointness in $\Funlax(\globe_1, \cC)$ is characterized in terms of right adjointness of mates of filling cells, the unpacked condition looks more complicated: concretely, in the case of $\Ret^{\someadj}(\cC)$, it will ask a mate (with respect to the adjunctions $(k_i \ev_{h_i})^L \dashv k_i \ev_{h_i}$) of a certain whiskering of the $3$-morphism $\coev_{h_{<}}$ to be invertible. 
\end{remark}

We end this section by recording a useful sufficient condition on an oplax morphism of retracts to be in $\Ret^{\someadj}(\cC)$, which, 
after translating $(X,f,g,\alpha) := (X,h,k,\beta^{-1})$.

\begin{corollary}\label{cor:somesubcategoriesofRetC}
Suppose $\cC$ is a pointed $(\infty,3)$-category and $(X_{<}, h_{<}, k_{<}, \beta_{<})$ (with notation as in Example~\ref{ex:laxretract}) is an oplax morphism of retracts whose source and target retracts satisfy the conditions from Corollary~\ref{cor:smashretract}, i.e.\ are objects of $\Ret^{\someadj}(\cC)$. If $h_{<}$ is a left adjoint, and $k_{<}^{\lmate}$ is invertible, then the oplax map of retracts satisfies the conditions from Proposition~\ref{prop:retadj}, i.e.\ is in the subcategory $\Ret^{\someadj}(\cC)$.
\end{corollary}
\begin{proof}
If $k_{<}^{\lmate}$ is invertible, it is in particular a right adjoint with invertible unit $\coev_{k_{<}^{\lmate}}$. In particular, any whiskering of this unit remains invertible. 
\end{proof}

\subsection{The lax smash square of the walking monad}\label{sec:mndsquared}

\begin{definition}\label{defn:mnd}
The \define{walking monad} $\Mnd$ is the strict $2$-category admitting either of the following equivalent definitions:
\begin{itemize}
  \item $\Mnd$ is the free strict $2$-category with one object $\bullet$ and an associative unital algebra object in the strict monoidal category $\End_\Mnd(\bullet)$.
  \item $\Mnd$ is the full subcategory of $\Adj$ on the source of the generating left adjoint.
\end{itemize}
\end{definition}
The equivalence between these definitions is well-known in the world of strict 2-categories. That the strict $2$-category $\Mnd$ plays the same double role amongst coherent higher categories follows from a result of Haugseng: 
\begin{prop}[{\cite[Cor.~8.9]{2002.01037}}] \label{prop:monadalgebra}
Let $\cC$ be a pointed $(\infty,2)$-category. Then, there is an equivalence, natural in $\cC \in \Cat_{(\infty,2)}^*$, of pointed $(\infty,1)$-categories
\[\Funoplax_*(\Mnd, \cC) \simeq \Alg(\Omega \cC), 
\]
where $\Alg(\Omega \cC)$ denotes the $(\infty,1)$-category of $\EE_1$-algebras and $\EE_1$-algebra morphisms in the monoidal $(\infty,1)$-category $\Omega\cC$ from Corollary~\ref{cor:monoidalstructureonloops}. 
\end{prop}
\begin{proof}
The result is essentially \cite[Cor.~8.9]{2002.01037}, and so we comment only on the translation to our setting. First, by Lemma~\ref{lem:surjectivereduces},  $\Funoplax_*(\Mnd, \cC)$ is an $(\infty,1)$-category and hence equivalent to its underlying $(\infty,1)$-category $\iota_1 \Funoplax_*(\Mnd, \cC)$;  Haugseng shows that the underlying $(\infty,1)$-category $\iota_1 \Funoplax_*(\Mnd, \cC)$ is equivalent to $\Alg_{\EE_1}(\Omega\cC)$. Second, our construction in Corollary~\ref{cor:monoidalstructureonloops} of the monoidal structure on $\Omega \cC$ in terms of the comonoid structure on $\Sone$ agrees with the one considered by Haugseng in~\cite[Def.~8.7]{2002.01037} after  
unstraightening.
\end{proof}

\begin{remark}\label{rem:MndMndop}
The equivalence $(-)^{\op} : \Delta \to \Delta$ induces for any monoidal $(\infty,1)$-category $\cA$ an equivalence \begin{equation}\label{eq:algmop}\Alg(\cA) \simeq \Alg(\cA^{\mop})\end{equation}natural in $\cA$. Using Proposition~\ref{prop:monadalgebra},  this induces for any pointed $(\infty,2)$-category $\cC$ a natural equivalence 
\begin{align*}
\Funoplax_*(\Mnd, \cC)& \overset{\text{Prop.}~\ref{prop:monadalgebra}}{\simeq} \Alg(\Omega \cC) \overset{\text{Eq.}~\eqref{eq:algmop}}{\simeq} \Alg((\Omega \cC)^{\mop}) \overset{\text{Cor.}~\ref{cor:opcoloop}}{\simeq} \Alg(\Omega( \cC^{\op}))\\& \overset{\text{Prop.}~\ref{prop:monadalgebra}}{\simeq} \Funoplax_*(\Mnd, \cC^{\op}) \overset{\text{Lem.}~\ref{lem:propertiesofFunoplax:compatwithcoop}}{\simeq} \Funlax_*(\Mnd^{\op}, \cC)^\op.
\end{align*}Passing to spaces of objects, this becomes an equivalence
\[\Map_*(\Mnd, \cC) \simeq \Map_*(\Mnd^{\op}, \cC)
\]
and hence by Yoneda an equivalence
\begin{equation}\label{eq:MndMndop}\Mnd \simeq \Mnd^{\op}.
\end{equation}
Since the space $\Aut(\Mnd)$ is contractible\footnote{$\Mnd$ is gaunt, hence $\Aut(\Mnd)$ is the \emph{set} of \emph{strict} $2$-functor automorphisms, which is trivial.}, this  is in fact the unique such equivalence.
\end{remark}

At the level of objects, Proposition~\ref{prop:monadalgebra} generalizes to arbitrary $(\infty, \infty)$-categories $\cC$. Indeed, since the ``maximal-$(\infty,1)$-category'' functor $\iota_1: \Cat_{(\infty, \infty)} \to \Cat_{(\infty,1)}$ is a right adjoint, it preserves limits and hence monoid objects. In particular, for any $(\infty, \infty)$-category $\cC$, $\iota_1 \Omega \cC$ is a monoidal $(\infty,1)$-category.

\begin{corollary} \label{cor:monadalgspace}
Let $\cC$ be a pointed $(\infty, \infty)$-category. Then, there is an equivalence of spaces, natural in $\cC \in\Cat_{(\infty,\infty)}^*$:
\[\Map_*(\Mnd, \cC) \simeq  \Alg( \Omega \iota_2 \cC)^{\simeq} \simeq \Alg(\iota_1 \Omega \cC)^{\simeq} .\]
\end{corollary}
\begin{proof}
We note that by adjunction $\Map_*(\Mnd, \cC) \simeq \Map_*(\Mnd, \iota_2 \cC)$. Hence, applying the object-level version of Proposition~\ref{prop:monadalgebra}, this space is equivalent to the space of algebras $\Alg(\Omega \iota_2 \cC)^{\simeq}$. Thus we merely need to identify $\Omega\iota_2\cC$ with $\iota_1\Omega\cC$. Since $\Sone$ is an $(\infty,1)$-category and $\globe_1 \otimes \globe_1$ is an $(\infty,2)$-category, it immediately follows that $\iota_1 \Omega \cC \simeq \iota_1 \Omega \iota_2 \cC$. By Lemma~\ref{lem:surjectivereduces},  $\Omega\iota_2\cC$ is already a $(\infty,1)$-category and hence \[ \iota_1 \Omega \cC \simeq \iota_1 \Omega \iota_2 \cC \simeq \Omega \iota_2 \cC.\qedhere\]
\end{proof}

\begin{remark}\label{rem:alglax}
If $\cC$ is a pointed $(\infty,\infty)$-category, then $\Funoplax_*(\Mnd, \cC)$ may be used as a \emph{definition} of the $(\infty,\infty)$-category \[\Alg^{\oplax}(\Omega\cC):= \Funoplax_*(\Mnd, \cC) \] whose objects are by Corollary~\ref{cor:monadalgspace} the $\EE_1$-algebras in $\Omega\cC$ and whose morphisms are the \define{oplax algebra homomorphisms}. Such an oplax algebra homomorphism between algebras $(A, m_A: A \otimes^{\Omega \cC} A \to A , u_A: 1^{\Omega \cC} \to A, \dots)$ and $(B, m_B: B \otimes^{\Omega \cC} B \to B, u_B: 1^{\Omega \cC} \to B, \dots)$ in $\Omega \cC$ unpacks to a $1$-morphism $f:  A\to B$ in $\Omega \cC$ together with  $2$-morphisms 
 \begin{equation}\label{eqn:oplaxalgmapcomponents}
 m_f: f \circ m_A \Rightarrow m_B \circ (f \otimes^{\Omega \cC} f), \qquad f\circ u_A \To u_B,
 \end{equation}
  and higher coherence morphisms.  (Here $1^{\Omega \cC}$ and $\otimes^{\Omega \cC}$ denote the tensor unit and tensor product on $\Omega \cC$ induced by $1$-morphism composition in $\cC$.) For example, if $A = B = 1^{\Omega\cC}$, then \eqref{eqn:oplaxalgmapcomponents} looks like the data of a \emph{coalgebra} structure on~$f$. We will return to this in Lemma~\ref{lem:loopcoalg}. 
 \end{remark}

Proposition~\ref{prop:monadalgebra} can also be used to characterize categories of coalgebras.

\begin{definition}
The category of \define{coalgebras} in a monoidal $(\infty,1)$-category $\cA$ is defined as
  $$ \coAlg(\cA) := \Alg(\cA^\op)^\op.$$
\end{definition}

\begin{prop}\label{prop:coalgebras}
For any pointed $(\infty,2)$-category $\cC$, there is an equivalence of pointed $(\infty,1)$-categories, natural in $\cC \in \Cat_{(\infty,2)}^*$:
\[\Funoplax_*(\Mnd^{\co}, \cC) \simeq \coAlg(\Omega \cC)
\]
\end{prop}
\begin{proof}
We compose the following natural equivalences:
\begin{align}\label{eqn:seriesofequivalences}
\Funoplax_*(\Mnd^{\co}, \cC) &\simeq \Funoplax_*(\Mnd^{\co\op}, \cC) && \text{(Equation~\eqref{eq:MndMndop})}  \\
 \notag  & \simeq \Funoplax_*(\Mnd, \cC^{\co\op})^{\co\op} && \text{(Lemma~\ref{lem:propertiesofFunoplax:compatwithcoop})} \\
 \notag  & \simeq \Alg(\Omega(\cC^{\co\op}))^{\co\op} && \text{(Proposition~\ref{prop:monadalgebra})} \\
 \notag  & \simeq \Alg((\Omega\cC)^{\mop, \co\op})^{\co\op} && \text{(Corollary~\ref{cor:opcoloop})}\\
 \notag  & \simeq \Alg(\Omega\cC)^{\co\op})^{\co\op} && \text{(Equation~\eqref{eq:algmop})}
\end{align}
Finally use that $(-)^{\co}$ acts trivially on $(\infty,1)$-categories, and hence find that this is equivalent to $\Alg((\Omega\cC)^{\op})^{\op} =: \coAlg(\Omega \cC)$. 
\end{proof}

\begin{definition}\label{def:braidedmonoidal}
A \define{braided monoidal $(\infty,1)$-category} is an object of $\Mon(\Mon(\Cat_{(\infty,1)}))$. Dunn additivity \cite{MR938617} (see also \cite[Thm.~5.1.2.2 and Example~5.1.2.4]{HA}) provides an equivalence of operads $\EE_2 \simeq \EE_1 \otimes \EE_1$, identifying  braided monoidal $(\infty,1)$-categories with  $\EE_2$-algebras in $\Cat_{(\infty,1)}$. \end{definition}
For $(1,1)$-categories, this notion is equivalent to the classical notion of \cite{JoyalStreetMacquarie,JOYALSTREET199320} as sketched in \cite[Ex.~5.1.2.4]{HA}, also see~\cite[Rem.~7.7.7]{2401.02956} for an alternative more detailed proof.

\begin{example}\label{exm:3catbraided}
  The functor $\Omega(-) = \Funoplax_*(\Sone,-) : \Cat_{(\infty,2)}^* \to \Cat_{(\infty,1)}^*$ is a right adjoint and hence preserves products and hence preserves monoid objects: if $\cA$ is a monoidal $(\infty,2)$-category, then $\Omega\cA$ is naturally a braided monoidal $(\infty,1)$-category. In particular, if $\cC$ is a pointed $(\infty,3)$-category, then $\Omega^2\cC$ is naturally a braided monoidal $(\infty,1)$-category.
\end{example}

\begin{notation}\label{notat:verticalhorizontal}
Consider the functor $\Mon(\Mon(\Cat_{(\infty,1)})) \to \Mon(\Cat_{(\infty,1)})$ ``forgetting the outer $\Mon$''; we refer to it as the  \define{(underlying) horizontal monoidal category} of a braided monoidal category. Since the functor $\Mon(\Cat_{(\infty,1)}) \to \Cat_{(\infty,1)}$ preserves limits, there is  another functor $\Mon(\Mon(\Cat_{(\infty,1)}) \to \Cat_{(\infty,1)})$ which ``forgets the inner $\Mon$.'' We refer to it as the \define{(underlying) vertical monoidal category}. Note that the two composites $\Mon(\Mon(\Cat_{(\infty,1)})) \to \Mon(\Cat_{(\infty,1)}) \to \Cat_{(\infty,1)}$ are canonically identified, and hence the horizontal and vertical monoidal category may be thought of as two tensor structures on the same category $\cB$. 
We will often denote the corresponding binary multiplication functors as follows, respectively
\[ A, B \mapsto A | B \hspace{1cm} A, B \mapsto \frac A B.
\]
Indeed, unwinding Example~\ref{exm:3catbraided}, for a pointed $(\infty,3)$-category $\cC$,  the $(\infty,1)$-category $\Omega^2 \cC$ uses composition of $1$-morphisms $f,g \mapsto fg$ as horizontal monoidal structure and the composition of $2$-morphisms $\alpha, \beta \mapsto \alpha \circ \beta$ as the vertical monoidal structure, justifying our terminology.
\end{notation}

\begin{remark}\label{rem:EH}
The Eckman-Hilton argument identifies the vertical and horizontal underlying monoidal categories of a $\cB \in \Mon(\Mon(\Cat_{(\infty,1)}))$, but there are many such identifications and we need to be careful to distinguish them. More precisely, these two underlying monoidal categories arise from restricting along the inclusions $\EE_1 = \EE_1 \otimes * \to \EE_1 \otimes \EE_1$ and $\EE_1 = * \otimes \EE_1 \to \EE_1 \otimes \EE_1$ which, after identifying the space of operad maps $\EE_1 \to \EE_1\otimes \EE_1$ with the space of linear injections $\mathrm{LinInj}(\mathbb{R}^1,\mathbb{R}^1 \oplus \mathbb{R}^1) \sim S^1$, correspond to the inclusions of $\mathbb{R}^1 \oplus \{0\} $  and $\{0\} \oplus \mathbb{R}^1$. We will once and for all choose  a path in the \emph{positive quadrant} to identify them, but warn the reader that this is a choice (with $\Omega S^1 = \bZ$-many options); we return to this in Notation~\ref{not:shearmaps}. In terms of Notation~\ref{notat:verticalhorizontal}, this identifies $A |B \simeq \frac A B$ by rotating ``through the upper left quadrant.'' 
\end{remark}

\begin{remark}\label{rem:visualizebraided}
In terms of Notation~\ref{notat:verticalhorizontal}, we may visualize the equivalence between braided monoidal $(1,1)$-categories in the sense of Definition~\ref{def:braidedmonoidal} and in the classical sense of  \cite{JoyalStreetMacquarie,JOYALSTREET199320} as follows: Given a $\cB \in \Mon(\Mon(\Cat_{(1,1)}))$, we choose its underlying monoidal category to  be its \emph{vertical monoidal category}, i.e.\ we set $A\otimes B:= \frac A B$. We then identify $A |B \simeq \frac A B$ by a ``quarter rotation,'' as explained in Remark~\ref{rem:EH}.  The braiding isomorphism $\br_{A,B}:A \otimes B \to B\otimes A$  is then for example given by the following composite:
\[A|B \simeq \left. \frac 1 A \middle| \frac B 1\right. \simeq \frac{ \vphantom{.}1|B}{A|1\vphantom{.}} \simeq \frac B A
\]
where the middle isomorphism encodes the compatibility between the two monoidal structures in $\Mon(\Mon(\Cat_{(1,1)}))$.
\end{remark}

\begin{lemma}\label{lem:algmonoidal}
The functors $\Alg(-), \coAlg(-): \Mon(\Cat_{(\infty,1)}) \to \Cat_{(\infty,1)}$ preserve limits and hence lift to functors $\Mon(\Mon(\Cat_{(\infty,1)})) \to \Mon(\Cat_{(\infty,1)})$. 
\end{lemma}
\begin{proof}
Pointwise
multiplication makes $\Cat_{(\infty,1)}^{\Delta^{\op}}:=\Fun(\Delta^{\op}, \Cat_{(\infty,1)})$ into a presentable $\Cat_{(\infty,1)}$-module category (i.e.\ an object of $\Mod_{\Cat_{(\infty,1)}}(\mathrm{Pr}^L)$ in the notation of~\cite{HA}).
It hence admits by the adjoint functor theorem an inner hom $\operatorname{Nat}(-,-) : \left(\Cat_{(\infty,1)}^{\Delta^{\op}}\right)^{\op} \times \Cat_{(\infty,1)}^{\Delta^{\op}} \to \Cat_{(\infty,1)}$, which may be thought of as encoding the 2-categorical structure on $\Cat_{(\infty,1)}^{\Delta^{\op}}$.
For monoidal categories $\cC$ and $\cD$, recall that the \emph{category of monoidal functors} may be defined as $\Fun^{\otimes}(\cC, \cD):=\mathrm{Nat}(\cC_{\bullet}, \cD_{\bullet})$, where $\cC_{\bullet}, \cD_{\bullet}: \Delta^{\op} \to \Cat_{(\infty,1)}$ are the respective simplicial categories.
Since the full inclusion $\Mon(\Cat_{(\infty,1)}) \hookrightarrow \Cat_{(\infty,1)}^{\Delta^{\op}}$ is closed under limits (since limits commute with limits) and since for any $F\in\Cat_{(\infty,1)}^{\Delta^{\op}}$, the functor $\mathrm{Nat}(F, -): \Cat_{(\infty,1)}^{\Delta^{\op}} \to \Cat_{(\infty,1)}$ is an inner hom, thus by definition a right adjoint and hence preserves limits, it follows that for any $\cC \in \Mon(\Cat_{(\infty,1)})$, the functor $\Fun^{\otimes}(\cC, -): \Mon(\Cat_{(\infty,1)}) \to \Cat_{(\infty,1)}$ preserves limits.
But both functors $\Alg$ and $\coAlg$ are of the form $\Fun^{\otimes}(\cC,-)$ for appropriate $\cC$: the former by the \emph{walking algebra}, e.g.\ defined equivalently as $\Omega \Mnd$ or as the monoidal envelope of the $\EE_1$-operad \cite[Appendix~A]{MR3345192}; the latter by its opposite monoidal category, the \emph{walking coalgebra}. 
\end{proof}
In other words, if $\cB$ is a braided monoidal $(\infty,1)$-category, then the $(\infty,1)$-categories $\Alg(\cB)$ and $\coAlg(\cB)$ are naturally monoidal, as are the forgetful functors $\Alg(\cB) \to \cB$ and $\coAlg(\cB) \to \cB$.

Recall from Remark~\ref{rem:alglax} that for a pointed $(\infty,3)$-category $\cC$,  $\Funoplax_*(\Mnd, \cC)$ can be thought of as the $(\infty,2)$-category $\Alg^{\oplax}(\Omega \cC)$ of algebras and oplax algebra homomorphisms in the monoidal $(\infty,2)$-category $\Omega \cC$. We next show that an oplax algebra endomorphism of the trivial algebra in $\Omega \cC$ is precisely a coalgebra in $\Omega^2 \cC$:

\begin{lemma}\label{lem:loopcoalg}
Let $\cC$ be a pointed $(\infty,3)$-category. Then, there is an equivalence of monoidal $(\infty,1)$-categories, natural in $\cC \in \Cat_{(\infty,3)}^*$:
\[\Omega \Funoplax_*(\Mnd, \cC)  \simeq \coAlg(\Omega^2 \cC).
\]
Here, $\Omega \Funoplax_*(\Mnd, \cC)$ carries the monoidal structure induced from $\Omega$ as in Corollary~\ref{cor:monoidalstructureonloops} and $\coAlg(\Omega^2 \cC)$ carries the monoidal structure induced from applying the limit preserving functor  $\coAlg(\Omega-)$ to the monoid object $\Omega \cC $.
\end{lemma}
\begin{proof}By Corollary~\ref{cor:movingloopsinsidefun}, there is a natural equivalence of monoidal $(\infty,1)$-categories \[\Omega \Funoplax_*(\Mnd, \cC) \simeq \Funoplax_*(\Mnd^{\co \op}, \Omega \cC) \simeq \Funoplax_*(\Mnd^{\co} , \Omega \cC),\] which by Proposition~\ref{prop:coalgebras} is  equivalent to $\coAlg(\Omega^2\cC)$. This latter equivalence is monoidal since it arises from applying a natural equivalence $\Funoplax_*(\Mnd^{\co\op}, -) \simeq \coAlg(\Omega -)$ between product preserving functors to monoid objects. 
\end{proof}

\begin{definition}\label{defn:bialgebra}
Let $\cB$ be a braided monoidal $(\infty,1)$-category. The $(\infty,1)$-category of \define{bialgebras} in $\cB$ is
 \[\BiAlg(\cB):= \Alg(\coAlg(\cB)).\]
\end{definition}

Using the terminology from Notation~\ref{notat:verticalhorizontal} and unwinding the definition, a bialgebra therefore has an underlying algebra in the horizontal monoidal category of~$\cB$ and an underlying coalgebra in the vertical monoidal category of~$\cB$.  We stress that treating the algebra and coalgebra as living in the same monoidal category (as e.g.\ needed for the definition of a Hopf algebra in \S\ref{subsec:finalproof}) involves a choice of identification between these monoidal directions as  discussed in Remark~\ref{rem:EH}.

\begin{remark}\label{rem:fourshearmaps}
Given $\cB \in \Mon(\Mon(\Cat_{(\infty,1)}))$, (some of) the data of a bialgebra  $B$  may be informally visualized as follows using Notation~\ref{notat:verticalhorizontal}: The underlying coalgebra of $B$ lives by definition in the vertical monoidal category and hence has a binary comultiplication $\Delta_B: B \to \frac B B$. The monoidal structure on the category $\coAlg(\cB)$ is induced by the horizontal monoidal structure on $\cB$. That is, given coalgebras $A$ and $B$, the binary comultiplication on their tensor product $A|B$ is given by 
\[ \Delta_{A|B}:=  A|B \to[\Delta_A | \Delta_B] \left. \frac A A \middle | \frac B B\right. \simeq \frac{ A | B }{ A | B}.
\]

Equipping a coalgebra with an algebra structure in $\coAlg(\cB)$ with respect to the remaining horizontal monoidal direction therefore entails the data of a binary multiplication $m_B : B | B \to B$ which is a coalgebra homomorphism. This compatibility in turn entails an isomorphism between 
\[ B | B \quad \overset{\Delta | \Delta} \longto \quad \left. \frac B B \middle| \frac B B\right. \quad \simeq  \quad \frac{ B|B}{B|B} \quad \overset{\frac m m}\longto \quad \frac B B  \]
and
\[  B  | B \quad \overset{m}\longto \quad B \quad \overset\Delta\longto \quad \frac B B . \]
If $\cB$ is a classical braided monoidal $(1,1)$-category, this  translates via Remark~\ref{rem:visualizebraided} (i.e. where we define  $A\otimes B:= \frac A B$ and identify it with $A|B$ via the quarter rotation) to the familiar  \emph{bialgebra axiom} 
\begin{equation}
\label{eq:strongbialgebra} (m_B \otimes m_B) \circ (\id_B \otimes \br_{B,B} \otimes \id_B )  \circ (\Delta_B \otimes \Delta_B) \simeq \Delta_B \circ m_B.
\end{equation}
\end{remark}

\begin{corollary}\label{cor:mndsquared}
  Let $\cC$ be a pointed $(\infty,3)$-category. Then, there is an equivalence of pointed $(\infty,1)$-categories,  natural in $\cC \in \Cat_{(\infty,3)}^*$:
  $$ \Funoplax_*(\Mnd \owedge \Mnd, \cC) \simeq \BiAlg(\Omega^2\cC).$$
\end{corollary}
\begin{proof}
By definition, $\Funoplax_*(\Mnd \owedge \Mnd, \cC) \simeq \Funoplax_*(\Mnd, \Funoplax_*(\Mnd, \cC))$. By Lemma~\ref{lem:surjectivereduces}, $\Funoplax_*(\Mnd,\cC)$ is an $(\infty,2)$-category and we may hence apply Proposition~\ref{prop:monadalgebra}  to find $\Funoplax_*(\Mnd, \Funoplax_*(\Mnd, \cC)) \simeq \Alg(\Omega \Funoplax_*(\Mnd, \cC))$. The statement then follows from Lemma~\ref{lem:loopcoalg}. 
\end{proof}

Tracing through the proof of Corollary~\ref{cor:mndsquared}, the underlying algebra of a bialgebra is its image under
\begin{equation}\label{eq:underlyingalgebra}
 \Funoplax_*(\Mnd \owedge \Mnd, \cC) \to \Funoplax_*(\Mnd \owedge \Sone, \cC)  \simeq \Funoplax_*(\Mnd, \Omega \cC) 
 \simeq \Alg(\Omega^2 \cC). \end{equation}
 and the underlying coalgebra is its image under 
 \begin{multline}\label{eq:underlyingcoalgebra}
 \Funoplax_*(\Mnd \owedge \Mnd, \cC) \to \Funoplax_*(\Sone \owedge \Mnd, \cC)  
 \simeq \Omega \Funoplax_*(\Mnd, \cC) \\ \overset{\text{Lem.~\ref{lem:loopcoalg}}}\simeq \coAlg(\Omega^2\cC) \end{multline}

\begin{remark}\label{rem:laxbialgebra}
Corollary~\ref{cor:mndsquared} identifies the $3$-localization $L_3(\Mnd \owedge \Mnd)$ with (the twice-delooping of) the \define{walking bialgebra}.
More generally $\Mnd \owedge \Mnd$ may be thought of as the walking \emph{lax bialgebra}, classifying \emph{lax bialgebra objects} in braided monoidal $(\infty,\infty)$-categories, i.e.\ objects in $\Alg^{\oplax}(\coAlg^{\oplax}(\cB))$. When unpacked, a lax bialgebra structure on an object $B \in \cB$ consists of a unital algebra structure $(m_B : B \otimes B \to B, \dots)$, a counital coalgebra structure $(\Delta_B : B \to B \otimes B, \dots)$, and various coherent comparison 2-morphisms (between compositions of the algebra and coalgebra structures). The most interesting of these is the bialgebra axiom~\eqref{eq:strongbialgebra}, which becomes a typically-noninvertible 2-morphism
\begin{equation}
\label{eq:laxbialgebra} 
(m_B \otimes m_B) \circ (\id_B \otimes \br_{B,B} \otimes \id_B )  \circ (\Delta_B \otimes \Delta_B) \Rightarrow \Delta_B \circ m_B.
\end{equation}
 See~\cite[Definifition~3.1.1]{1710.01465} for a bicategorical incarnation of such lax bialgebras\footnote{The paper \cite{1710.01465} weakens the bialgebra axiom~\eqref{eq:strongbialgebra} to a comparison map which goes the other direction from our \eqref{eq:laxbialgebra}, and calls the resulting notion an ``oplax bialgebra.'' This is why we have used the term ``lax bialgebra'' even though it is built entirely out of oplax constructions. Indeed, there is no coherent way to decide which weakening of \eqref{eq:strongbialgebra} should be called ``lax'' and which ``oplax'': if one unpacks $\coAlg^{\oplax}(\Alg^{\oplax}(\cB))$ instead of our $\Alg^{\oplax}(\coAlg^{\oplax}(\cB))$, one does indeed find the other weakening from \eqref{eq:laxbialgebra}.}.
 \end{remark}

\begin{remark}\label{rem:notationforbimnd}
We now explicitly unpack the morphisms in $L_3(\Mnd \owedge\Mnd)$ which corepresent via Corollary~\ref{cor:mndsquared} the binary multiplication and comultiplication of the algebra, resp. coalgebra. See also~\cite{2101.10361} for a similar analysis.
Let $A: \globe_1 \to \Mnd$ denote the generating $1$-morphism and $m: \globe_2 \to \Mnd$ the $2$-morphism corepresenting the binary multiplication. The underlying object of the bialgebra in $L_3(\Mnd \owedge \Mnd)$ is corepresented by the filling $2$-cell $\globe_2 \to \globe_1 \otimes \globe_1 \to[A\otimes A]  \Mnd \otimes \Mnd \to L_3(\Mnd \owedge \Mnd)$. Passing to the gauntification\footnote{At the time of writing, it is not known whether $\Mnd \otimes \Mnd$ is already gaunt, see Warning~\ref{warn:GrayGaunt}. However, see Construction~\ref{cons:universalshear} for how diagrams such as the ones below can be canonically lifted to define cells in $\Mnd \otimes \Mnd$. } this cell may be represented in the graphical calculus from \S\ref{subsec:otimesstrict}) as follows:
\begin{equation}\label{eq:underlyingobject}
\begin{tikzpicture}[baseline=(baseline)]
  \path (0,0) coordinate (middle) +(0,-.125) coordinate (baseline)
  +(-.5,-1) coordinate (sw) 
  +(+.5,-1) coordinate (se) 
  +(-.5,+1) coordinate (nw) 
  +(+.5,+1) coordinate (ne) 
  ;
  \draw[string,gray!50!white, name path=1A] (se) .. controls ++(0,+1) and ++(0,-1) .. (nw);
  \draw[string,gray!50!white, name path=A1] (sw) .. controls ++(0,+1) and ++(0,-1) .. (ne);
  \path[name intersections={of=1A and A1}] (intersection-1) node[dot]{};
\end{tikzpicture}
\end{equation}
Here and below, we have shaded in grey all the cells which become trivial in the quotient $\Mnd \owedge \Mnd$, i.e.\ any cell in $\Mnd \otimes \Mnd$ which is a product (in either order) of the $0$-cell in $\Mnd$ with some other cell. To make up for this shading, we have added a black dot to emphasize that the crossing itself, corresponding to the 2-cell $A \otimes A$, which stays nontrivial in the quotient.

By~\eqref{eq:underlyingalgebra}, the binary multiplication is corepresented by the $3$-cell $\globe_3 \to \globe_2 \otimes \globe_1 \to[m \otimes A] \Mnd \otimes \Mnd \to L_3(\Mnd \owedge \Mnd)$, which looks like:
\begin{equation}\label{eq:underlyingmultiplication}
\begin{tikzpicture}[baseline=(baseline)]
  \path (0,0) coordinate (middle) +(0,-.125) coordinate (baseline)
  +(-.75,-1) coordinate (sw1) 
  +(-.5,-1) coordinate (sw2) 
  +(+.5,-1) coordinate (se) 
  +(-.5,+1) coordinate (nw) 
  +(+.5,+1) coordinate (ne) 
  +(.25,.5) node[draw,circle,inner sep=0.5pt,gray!50!white] (AA) {$\scriptstyle m$}
  ;
  \draw[string,gray!50!white, name path=1A] (se) .. controls ++(0,+1) and ++(0,-1) .. (nw);
  \draw[string,gray!50!white] (AA) .. controls ++(.125,.125) and ++(0,-.125) .. (ne);
  \draw[string,gray!50!white, name path=A12] (sw2) .. controls ++(0,+.5) and ++(-.25,-.5) .. (AA);
  \draw[string,gray!50!white, name path=A11] (sw1) .. controls ++(0,+.5) and ++(-.5,-.25) .. (AA);
  \path[name intersections={of = 1A and A11}] (intersection-1) node[dot]{};
  \path[name intersections={of = 1A and A12}] (intersection-1) node[dot]{};
\end{tikzpicture}
\quad \Rrightarrow \quad
\begin{tikzpicture}[baseline=(baseline)]
  \path (0,0) coordinate (middle) +(0,-.125) coordinate (baseline)
  +(-.75,-1) coordinate (sw1) 
  +(-.5,-1) coordinate (sw2) 
  +(+.5,-1) coordinate (se) 
  +(-.5,+1) coordinate (nw) 
  +(+.5,+1) coordinate (ne) 
  +(-.25,-.5) node[draw,circle,inner sep=0.5pt,gray!50!white] (AA) {$\scriptstyle m$}
  ;
  \draw[string,gray!50!white, name path=1A] (se) .. controls ++(0,+1) and ++(0,-1) .. (nw);
  \draw[string,gray!50!white, name path=A1] (AA) .. controls ++(.5,.5) and ++(0,-.5) .. (ne);
  \draw[string,gray!50!white, name path=A12] (sw2) .. controls ++(0,+.25) and ++(-.125,-.25) .. (AA);
  \draw[string,gray!50!white, name path=A11] (sw1) .. controls ++(0,+.25) and ++(-.25,-.125) .. (AA);
  \path[name intersections={of = 1A and A1}] (intersection-1) node[dot]{};
\end{tikzpicture}
\end{equation}
Tracing through \eqref{eq:underlyingcoalgebra}, the binary comultiplication is corepresented by the $3$-cell 
$\globe_3 \to \globe_1 \otimes \globe_2 \to[A \otimes m] \Mnd \otimes \Mnd \to L_3(\Mnd \owedge \Mnd)$, which looks like: \begin{equation}\label{eq:underlyingcomultiplication}
\begin{tikzpicture}[baseline=(baseline),xscale=-1]
  \path (0,0) coordinate (middle) +(0,-.125) coordinate (baseline)
  +(-.75,-1) coordinate (sw1) 
  +(-.5,-1) coordinate (sw2) 
  +(+.5,-1) coordinate (se) 
  +(-.5,+1) coordinate (nw) 
  +(+.5,+1) coordinate (ne) 
  +(-.25,-.5) node[draw,circle,inner sep=0.5pt,gray!50!white] (AA) {$\scriptstyle m$}
  ;
  \draw[string,gray!50!white, name path=1A] (se) .. controls ++(0,+1) and ++(0,-1) .. (nw);
  \draw[string,gray!50!white, name path=A1] (AA) .. controls ++(.5,.5) and ++(0,-.5) .. (ne);
  \draw[string,gray!50!white, name path=A12] (sw2) .. controls ++(0,+.25) and ++(-.125,-.25) .. (AA);
  \draw[string,gray!50!white, name path=A11] (sw1) .. controls ++(0,+.25) and ++(-.25,-.125) .. (AA);
  \path[name intersections={of = 1A and A1}] (intersection-1) node[dot]{};
\end{tikzpicture}
\quad \Rrightarrow \quad
\begin{tikzpicture}[baseline=(baseline),xscale=-1]
  \path (0,0) coordinate (middle) +(0,-.125) coordinate (baseline)
  +(-.75,-1) coordinate (sw1) 
  +(-.5,-1) coordinate (sw2) 
  +(+.5,-1) coordinate (se) 
  +(-.5,+1) coordinate (nw) 
  +(+.5,+1) coordinate (ne) 
  +(.25,.5) node[draw,circle,inner sep=0.5pt,gray!50!white] (AA) {$\scriptstyle m$}
  ;
  \draw[string,gray!50!white, name path=1A] (se) .. controls ++(0,+1) and ++(0,-1) .. (nw);
  \draw[string,gray!50!white] (AA) .. controls ++(.125,.125) and ++(0,-.125) .. (ne);
  \draw[string,gray!50!white, name path=A12] (sw2) .. controls ++(0,+.5) and ++(-.25,-.5) .. (AA);
  \draw[string,gray!50!white, name path=A11] (sw1) .. controls ++(0,+.5) and ++(-.5,-.25) .. (AA);
  \path[name intersections={of = 1A and A11}] (intersection-1) node[dot]{};
  \path[name intersections={of = 1A and A12}] (intersection-1) node[dot]{};
\end{tikzpicture}
\end{equation}
Hence, one may think of the algebra as multiplying along the southeast-to-northwest axis and the coalgebra as comultiplying along the southwest-to-northeast axis. 
The identification in Corollary~\ref{cor:mndsquared} of $L_3(\Mnd \owedge \Mnd)$ with the (twice delooping of the) walking bialgebra in which by definition the coalgebra is in the vertical monoidal category and the algebra is in the horizontal monoidal category therefore intuitively involves a 45-degree rotation to the left. This is ultimately the same 45-degree rotation between the direction of a $2$-morphism in a lax square and the vertical axis of $2$-morphism composition. \end{remark}

Combining the work of this section and the previous one gives the first part of Theorem~\ref{thm:maintheoremfunctorial}:

\begin{corollary}
\label{cor:part2ofmaintheorem}
  Restriction along the canonical inclusions $\Mnd \to \Adj$ defines a functor
  \begin{multline*}
   H : \Ret^\someadj(\cC) \overset{\text{Def.~\ref{defn:retadjcategory}}}\simeq \Funoplax_*(\Adj \owedge \Adj, \cC) \\ \longto \Funoplax_*(\Mnd \owedge \Mnd, \cC) \overset{\text{Cor.~\ref{cor:mndsquared}}}\simeq \BiAlg(\Omega^2\cC) . \qed \hspace{-1em}\end{multline*}
\end{corollary}
We will unpack how the functor $H$ acts on objects in \S\ref{subsec:unpacked}.

\begin{remark}
  The domain of $H$ is a properly-$(\infty,3)$-category: it has noninvertible 2- and 3-morphisms in general. The codomain is an $(\infty,1)$-category. Thus $H$ factors through the $(\infty,1)$-localiztion $L_1(\Ret^\someadj(\cC))$ of the domain. We do not have a good description of that localization.
\end{remark}

\subsection{The Hopf axiom} \label{subsec:finalproof}
It remains only to show that the bialgebras in the image of the functor $H$ from Corollary~\ref{cor:part2ofmaintheorem} are (co)Hopf.

Classically, a bialgebra $B$ (with binary multiplication denoted $m_B: B \otimes B \to B$ and comultiplication denoted $\Delta_B:B \to B \otimes B$) in a braided monoidal $(1,1)$-category is  a \emph{Hopf algebra} if it admits an antipode, see e.g.\ 
\cite{q-alg/9509023}.  It is a straight-forward exercise\footnote{Write $1_\cB$ for the unit object in~$\cB$, and $1_B : 1_\cB \to B$ and $\epsilon_B : B \to 1_\cB$ for the unit and counit in a bialgebra~$B$. If~$\sh_B$ is invertible, then the map $S_B : (\epsilon_B \otimes \id_B) \circ (\sh_B)^{-1} \circ (\id_B \otimes 1_B) : B \to B$ is an antipode. If $B$ admits an antipode $S_B$, then $(\id_B \otimes m_B) \circ (\id_B \otimes S_B \otimes \id_B) \circ (\Delta_B \otimes \id_B)$ is an inverse to $\sh_B$.} to show that this condition is equivalent to the invertibility of the following \emph{shear map} for $B$: 
  \begin{equation}\label{eq:shear}\sh_B:= (\id_B \otimes m_B) \circ (\Delta_B \otimes \id_B): B \otimes B \to B \otimes B\end{equation}
  
We will generalize this latter condition to the $\infty$-categorical context. 
\begin{definition} \label{def:Hopf} A bialgebra $B$ in a braided monoidal $(\infty,1)$-category $\cB \in \Mon(\Mon(\Cat_{(\infty,1)}))$ is Hopf if its shear map~\eqref{eq:shear} is invertible. We write $\Hopf(\cB) \subset \BiAlg(\cB)$ for the full subcategory on the Hopf algebras.

\end{definition}
The inclusion  $\Cat_{(1,1)} \hookrightarrow \Cat_{(\infty,1)}$ has a product preserving (and hence symmetric monoidal) left adjoint $\operatorname{h}_1 : \Cat_{(\infty,1)} \to \Cat_{(1,1)}$ which takes an $(\infty,1)$-category to its \define{homotopy category}. Thus, for $\cB \in \Mon(\Mon(\Cat_{(\infty,1)}))$, also $\operatorname{h}_1 \cB \in \Mon(\Mon(\Cat_{(1,1)}))$ and hence has the structure of a braided monoidal $1$-category in the sense of Joyal--Street  (see \cite{HA,2401.02956})
and any  bialgebra in $\cB$ determines a bialgebra in the classical sense in $\operatorname{h}_1\cB$. Finally, conservativity of $\cB \to \operatorname{h}_1 \cB$ implies:
\begin{lemma} A bialgebra $B$ in a braided monoidal $(\infty,1)$-category $\cB \in \Mon(\Mon(\Cat_{(\infty,1)}))$ is Hopf if and only if $B$ is Hopf in the classical sense in $\operatorname{h}_1\cB$. \qed
\end{lemma}

\begin{warn}
Recall from Remark~ that the underlying algebra and coalgebra of a bialgebra are defined with respect to two a  priori different monoidal structures. Hence, writing~\eqref{eq:shear} requires, as explained in Remark~\ref{rem:EH},  a choice of identification between these two monoidal directions; and such choices are not unique.
\end{warn}

Without making such a choice, there are four equally valid candidates for a shear map:
\begin{notation} \label{not:shearmaps}
Let $B$ be a bialgebra in a braided monoidal $(\infty,1)$-category $\cB$. Denotes its underlying horizontal and vertical monoidal structures  from Notation~\ref{notat:verticalhorizontal} by $A, B \mapsto A |B$ and $A, B \mapsto \frac{A}{B}$, respectively. Define\footnote{More precisely, these morphisms are well defined in the homotopy category $\operatorname{h}_1\cB$ and we choose an arbitrary lift of each to a morphism in $\cB$. }
\begin{gather*}
 \sh_B^\nwarrow    := \left[  B|B \quad\overset{B|\Delta}\longto\quad  B {\left|\frac{B}{B}\right.} \simeq  \frac{B|B}B    \quad\overset{ \frac m B}\longto\quad \frac B B \right] \\
 \sh_B^\nearrow:= \left[ B|B \quad\overset{\Delta | B }\longto\quad {\left.\frac{B}{B}\right|}B \simeq  \frac{B|B}B     \quad\overset{\frac m B}\longto\quad \frac B B \right] \\
 \sh_B^\swarrow  := \left[ B|B \quad\overset{B| \Delta}\longto\quad  B {\left|\frac{B}{B}\right.} \simeq \frac B{B|B}     \quad\overset{\frac B m}\longto\quad \frac B B \right] \\
 \sh_B^\searrow  := \left[ B|B \quad\overset{\Delta | B}\longto\quad  {\left.\frac{B}{B}\right|}B   \simeq \frac B{B|B}  \quad\overset{\frac B m}\longto\quad \frac B B \right] 
\end{gather*}
where the middle isomorphism is a choice of ``interchange isomorphism'' which is part of the data of $\cB$. 
The superscript arrows indicate the direction the exchanged ``$B$'' travels.
\end{notation}

\begin{lemma} \label{lem:reversingashearmap}
For a bialgebra $B$ in a braided monoidal $(\infty,1)$-category,  $\sh_B^\nwarrow$  is invertible if and only if  $\sh_B^\searrow$ is, and $ \sh_B^\nearrow$ is invertible if and only if $ \sh_B^\swarrow$ is. \end{lemma}
\begin{proof} By conservativity of $\cB \to \operatorname{h}_1 \cB$, these conditions can be checked in the homotopy category, where it is a straight-forward exercise. 
\end{proof}
After identifying $A|B \simeq \frac A B$ via a 90 degree rotation (see Remark~\ref{rem:EH}), the shear map $\sh_B^{\searrow}$ becomes the shear map~\eqref{eq:shear}, and hence a bialgebra is Hopf if $\sh_B^{\searrow}$, or equivalently $\sh_B^{\nwarrow}$ are invertible. The remaining case is therefore:

\begin{definition}A bialgebra $B \in \BiAlg(\cB)$ in a braided monoidal $(\infty,1)$-category $\cB \in \Mon(\Mon(\Cat_{(\infty,1)}))$ is \emph{coHopf} if $\sh_B^{\nearrow}$ (or equivalently $\sh_B^{\swarrow}$) is invertible.  We write $\coHopf(\cB) \subset \BiAlg(\cB)$ for the full subcategory on the coHopf algebras.
\end{definition}

\begin{remark}\label{rem:coshear} Identifying $A|B \simeq \frac A B$ as before, these \emph{coshear maps} become
\begin{gather*}
 \sh_B^{\nearrow} = (m_B \otimes \id_B) \circ (\id_B \otimes \br_{B,B}) \circ (\Delta_B \otimes \id_B) =: \sh_B^{\co}.
\\
 \sh_B^{\swarrow} = ( \id_B \otimes m_B) \circ (\br_{B,B} \otimes \id_B) \circ (\id_B \otimes \Delta_B),
\end{gather*}
where $\br_{B,B}$ is defined in Remark~\ref{rem:visualizebraided}.
\end{remark}

\begin{remark}
Let $(-)^{\vop}$ and $(-)^{\hop}$ denote the automorphisms of $\Mon(\Mon(\Cat_{(\infty,1)}))$ reversing the vertical, resp.\ horizontal monoidal directions. Then, the equivalence~\eqref{eq:algmop} induces canonical equivalences $\BiAlg(\cB) \simeq \BiAlg(\cB^{\vop}) \simeq \BiAlg(\cB^{\hop}) \simeq \BiAlg(\cB^{\vop, \hop})$. Since vertical and horizontal reflection exchanges shear maps, these equivalences restrict to equivalences between the subcategories
\[ \Hopf(\cB) \simeq \coHopf(\cB^{\vop}) \simeq \coHopf(\cB^{\hop}) \simeq \Hopf(\cB^{\vop \hop}).
\]
In particular, a bialgebra $B \in \BiAlg(\cB)$ is coHopf if and only if the corresponding bialgebra in $\BiAlg(\cB^{\hop})$ is Hopf and vice versa.

For $\cB$ a braided monoidal $(1,1)$-category in the classical sense, we may identify $\cB^{\vop}$ with its \emph{monoidal opposite}  $\cB^{\mop}$ and $\cB^{\hop}$ with $\cB^{\rev}$, defined as the same monoidal category equipped with the inverse braiding.
In these terms, the equivalence  $\BiAlg(\cB) \simeq \BiAlg(\cB^{\rev})$ sends a bialgebra $B$ to the bialgebra  $B^{\co}$ with the same multiplication $m_B: B \otimes B\to B$ but with comultiplication $\br_{B, B} \circ \Delta_B$. This will indeed satisfy the bialgebra axiom for the reversed braiding. Moreover,  $\sh_{B^\co} = \br_{B,B}^{-1} \circ \sh_B^\co$ and hence $B^{\co}$ will be coHopf if and only if $B$ is Hopf and vice versa. 
\end{remark}

\begin{remark}
Another straightforward exercise\footnote{The inverse to $\sh_{B}^\co$ is $\bigl(\id_B \otimes (m_B\br_{B,B}^{-1})\bigr) \circ \bigl(\br_{B,B}^{-1} \otimes \id_B\bigr) \circ \bigl(\id_B \otimes (\br_{B,B}^{-1}\Delta_B)\bigr)$.}: if $B$ is Hopf, then it is also coHopf if and only its antipode $S_B$ is invertible.  In particular, since already in $\cat{Vec}$ there do exist Hopf algebras with noninvertible antipode \cite{MR292876,MR1761130}, the Hopf and coHopf axioms are not equivalent: $$\Hopf(\cB) \neq \coHopf(\cB).$$
\end{remark}

\begin{definition}\label{def:strictshear} Recall the graphical calculus of~\S\ref{subsec:otimesstrict}, and its implementation for $\Mnd \otimes^{\strict} \Mnd$ in
Remark~\ref{rem:notationforbimnd}.\footnote{In the diagrams in Remark~\ref{rem:notationforbimnd}, we grayed out some cells in order to focus on the quotient $\Mnd \owedge^{\strict} \Mnd$. Here we care about the unquotiented $\Mnd \otimes^{\strict} \Mnd$, so we will not gray anything and we will not emphasize the intersections with bullets.} In this notation, the \define{strict universal shear map} is the following $3$-morphism in the  strict $4$-category $\Mnd \otimes^{\strict} \Mnd$:
\begin{equation}\label{eq:shearpicture}
\widetilde{\sh}:= 
\qquad
\begin{tikzpicture}[baseline=(baseline)]
  \path (0,0) coordinate (middle) +(0,-.125) coordinate (baseline)
  +(-1,+1) coordinate (nw) 
  +(+,+1) coordinate (ne) 
  +(-1,-1) coordinate (a) 
  +(-.25,-1) coordinate (b) 
  +(+.25,-1) coordinate (c) 
  +(+1,-1) coordinate (d) 
  ;
  \path (middle)
  +(-.25,-.25) node[draw,circle,inner sep=0.5pt] (mw) {$\scriptstyle m$}
  +(+.5,+.5) node[draw,circle,inner sep=0.5pt] (me) {$\scriptstyle m$}
  ;
  \draw[string] (mw) -- (nw);
  \draw[string] (me) -- (ne);
  \draw[string] (a) .. controls +(.25,1) and +(-1,-.5) .. (me);
  \draw[string] (c) .. controls +(.25,1) and +(.25,-.5) .. (me);
  \draw[string] (b) .. controls +(-.25,.25) and +(-.125,-.25) .. (mw);
  \draw[string] (d) .. controls +(-.25,.25) and +(+1,-.25) .. (mw);
\end{tikzpicture}
\Rrightarrow
\begin{tikzpicture}[baseline=(baseline)]
  \path (0,0) coordinate (middle) +(0,-.125) coordinate (baseline)
  +(-1,+1) coordinate (nw) 
  +(+,+1) coordinate (ne) 
  +(-1,-1) coordinate (a) 
  +(-.25,-1) coordinate (b) 
  +(+.25,-1) coordinate (c) 
  +(+1,-1) coordinate (d) 
  ;
  \path (middle)
  +(-.5,+.5) node[draw,circle,inner sep=0.5pt] (mw) {$\scriptstyle m$}
  +(+.5,+.5) node[draw,circle,inner sep=0.5pt] (me) {$\scriptstyle m$}
  ;
  \draw[string] (mw) -- (nw);
  \draw[string] (me) -- (ne);
  \draw[string] (a) .. controls +(.25,1) and +(-1,-.5) .. (me);
  \draw[string] (c) .. controls +(.25,1) and +(.25,-.5) .. (me);
  \draw[string] (b) .. controls +(-.25,1) and +(-.25,-.5) .. (mw);
  \draw[string] (d) .. controls +(-.25,1) and +(+1,-.5) .. (mw);
\end{tikzpicture}
\Rrightarrow
\begin{tikzpicture}[baseline=(baseline)]
  \path (0,0) coordinate (middle) +(0,-.125) coordinate (baseline)
  +(-1,+1) coordinate (nw) 
  +(+,+1) coordinate (ne) 
  +(-1,-1) coordinate (a) 
  +(-.25,-1) coordinate (b) 
  +(+.25,-1) coordinate (c) 
  +(+1,-1) coordinate (d) 
  ;
  \path (middle)
  +(-.5,+.5) node[draw,circle,inner sep=0.5pt] (mw) {$\scriptstyle m$}
  +(+.25,-.25) node[draw,circle,inner sep=0.5pt] (me) {$\scriptstyle m$}
  ;
  \draw[string] (mw) -- (nw);
  \draw[string] (me) -- (ne);
  \draw[string] (a) .. controls +(.25,.25) and +(-1,-.25) .. (me);
  \draw[string] (c) .. controls +(.25,.25) and +(.125,-.25) .. (me);
  \draw[string] (b) .. controls +(-.25,1) and +(-.25,-.5) .. (mw);
  \draw[string] (d) .. controls +(-.25,1) and +(+1,-.5) .. (mw);
\end{tikzpicture}
\end{equation}
\end{definition}

At the time of writing, it is unknown whether the gauntification $\Mnd \otimes \Mnd \to \Mnd \otimes^{\strict} \Mnd$ is an equivalence (see Warning~\ref{warn:GrayGaunt}). We will therefore now construct ``by hand'' a lift of this universal shear map  to a 3-morphism in the (possibly weak) $(\infty,4)$-category $\Mnd \otimes \Mnd$.

\begin{construction}
\label{cons:universalshear}
Let $\OO^2$ denote Street's second oriental, i.e.\ the strict $2$-category generated by the following $0$-, $1$- and $2$-morphisms:
\begin{equation}\label{eq:oriental}
\OO^2 := \quad 
\begin{tikzcd}
\bullet  &  \arrow[d, Leftarrow, shorten >=4pt] & \bullet \arrow[ll] \arrow[dl]  \\
& \bullet \arrow[ul] 
\end{tikzcd}
\end{equation}
There is a strict $2$-functor $\OO^2 \to \Mnd$ which sends the three generating $1$-cells of $\OO^2$ to the single generating $1$-cell of $\Mnd$ and sends the generating $2$-cell to the ``multiplication'' in $\Mnd$. Anticipating this $2$-functor, we graphically denote the generator of $\OO^2$ by
\[ 
\begin{tikzpicture}
  \path (0,0) node[dot] (m) {}  +(0,-.125) coordinate (baseline)
  +(0,+.75) coordinate (n)
  +(-.5,-1) coordinate (sw)
  +(+.5,-1) coordinate (se)
  ;
  \path (m)
  +(0,-.5) node[dot, gray!50!white] (b) {}
  +(-.75,+.25) node[dot, gray!50!white] (c) {}
  +(+.75,+.25) node[dot, gray!50!white] (a) {}
  ;
  \draw[grayarrow] (a) -- (b);
  \draw[grayarrow] (b) -- (c);
  \draw[grayarrow] (a) -- (c);
  \draw[string] (m) -- (n);
  \draw[string] (se) .. controls +(0,.5) and +(+.5,-.5) .. (m);
  \draw[string] (sw) .. controls +(0,.5) and +(-.5,-.5) .. (m);
  \path (m) node[gray!50!white] {$\Uparrow$};
\end{tikzpicture}
\] 
In particular, it is immediate that the strict universal shear map in $\Mnd\otimes^{\strict}\Mnd$ from Definition~\ref{def:strictshear} lifts through $\OO^2 \otimes^{\strict} \OO^2$. Since $\OO^2$ is in the idempotent completion of Campion's category of Gray cubes~$\fancysquare$ (indeed, it is a retract of $\globe_1 \otimes^{\strict} \globe_1$), it follows by the construction in~\cite{2311.00205} that the comparison functor $\OO^2 \otimes \OO^2 \to \OO^2 \otimes^{\strict} \OO^2$ is an equivalence. Since gauntification is natural and hence the diagram 
\[\begin{tikzcd}
\OO^2 \otimes \OO^2 \arrow[r, "\simeq"] \arrow[d] & \OO^2\otimes^{\strict} \OO^2 \arrow[d] \\ \Mnd \otimes \Mnd \arrow[r] & \Mnd \otimes^{\strict} \Mnd
\end{tikzcd}
\] commutes, this lifts the shear map from Definition~\ref{def:strictshear} to a $3$-morphism in $\Mnd \otimes \Mnd$, which we call the \emph{universal shear map}. We will also refer to its image in $L_3(\Mnd \otimes \Mnd)$ and $L_3(\Mnd \owedge \Mnd)$ as the universal shear map. 
\end{construction}

\begin{lemma}\label{lemma:universalshearrepresents}
Let $\cC$ be a pointed $(\infty,3)$-category and $B$ a bialgebra in $\Omega^2\cC$, identified via Corollary~\ref{cor:mndsquared} with a pointed functor $L_3(\Mnd \owedge \Mnd) \to[B] \cC$. Then the composition $\Mnd \otimes \Mnd \to \Mnd \owedge \Mnd \to L_3(\Mnd \owedge \Mnd) \to[B] \cC$ sends the universal shear map from Construction~\ref{cons:universalshear} to the coshear map $\sh_B^\nearrow$ from Notation~\ref{not:shearmaps} and Remark~\ref{rem:coshear}.
\end{lemma}

\begin{proof}
For better readability, we follow the notation of Remark~\ref{rem:notationforbimnd} and gray out the morphisms in $\widetilde{\sh}$ that become trivial when mapped to the quotient $\Mnd \owedge \Mnd$ and add a bullet to emphasize the (nontrivial) crossing. 
\begin{equation}
\widetilde{\sh}:= 
\qquad
\begin{tikzpicture}[baseline=(baseline)]
  \path (0,0) coordinate (middle) +(0,-.125) coordinate (baseline)
  +(-1,+1) coordinate (nw) 
  +(+,+1) coordinate (ne) 
  +(-1,-1) coordinate (a) 
  +(-.25,-1) coordinate (b) 
  +(+.25,-1) coordinate (c) 
  +(+1,-1) coordinate (d) 
  ;
  \path (middle)
  +(-.25,-.25) node[draw,circle,inner sep=0.5pt,gray!50!white] (mw) {$\scriptstyle m$}
  +(+.5,+.5) node[draw,circle,inner sep=0.5pt,gray!50!white] (me) {$\scriptstyle m$}
  ;
  \draw[string,gray!50!white, name path=X] (mw) -- (nw);
  \draw[string,gray!50!white] (me) -- (ne);
  \draw[string,gray!50!white, name path=Y] (a) .. controls +(.25,1) and +(-1,-.5) .. (me);
  \draw[string,gray!50!white, name path=Z] (c) .. controls +(.25,1) and +(.25,-.5) .. (me);
  \draw[string,gray!50!white] (b) .. controls +(-.25,.25) and +(-.125,-.25) .. (mw);
  \draw[string,gray!50!white, name path=W] (d) .. controls +(-.25,.25) and +(+1,-.25) .. (mw);
     \path[name intersections={of=X and Y}] (intersection-1) node[dot]{};
    \path[name intersections={of=Z and W}] (intersection-1) node[dot]{};  
\end{tikzpicture}
\Rrightarrow
\begin{tikzpicture}[baseline=(baseline)]
  \path (0,0) coordinate (middle) +(0,-.125) coordinate (baseline)
  +(-1,+1) coordinate (nw) 
  +(+,+1) coordinate (ne) 
  +(-1,-1) coordinate (a) 
  +(-.25,-1) coordinate (b) 
  +(+.25,-1) coordinate (c) 
  +(+1,-1) coordinate (d) 
  ;
  \path (middle)
  +(-.5,+.5) node[draw,circle,inner sep=0.5pt,gray!50!white] (mw) {$\scriptstyle m$}
  +(+.5,+.5) node[draw,circle,inner sep=0.5pt,gray!50!white] (me) {$\scriptstyle m$}
  ;
  \draw[string,gray!50!white] (mw) -- (nw);
  \draw[string,gray!50!white] (me) -- (ne);
  \draw[string,gray!50!white, name path=X] (a) .. controls +(.25,1) and +(-1,-.5) .. (me);
  \draw[string,gray!50!white, name path=Y] (c) .. controls +(.25,1) and +(.25,-.5) .. (me);
  \draw[string,gray!50!white, name path=Z] (b) .. controls +(-.25,1) and +(-.25,-.5) .. (mw);
  \draw[string,gray!50!white, name path=W] (d) .. controls +(-.25,1) and +(+1,-.5) .. (mw);
      \path[name intersections={of=X and W}] (intersection-1) node[dot]{};
    \path[name intersections={of=Y and W}] (intersection-1) node[dot]{};
    \path[name intersections={of=X and Z}] (intersection-1) node[dot]{};
\end{tikzpicture}
\Rrightarrow
\begin{tikzpicture}[baseline=(baseline)]
  \path (0,0) coordinate (middle) +(0,-.125) coordinate (baseline)
  +(-1,+1) coordinate (nw) 
  +(+,+1) coordinate (ne) 
  +(-1,-1) coordinate (a) 
  +(-.25,-1) coordinate (b) 
  +(+.25,-1) coordinate (c) 
  +(+1,-1) coordinate (d) 
  ;
  \path (middle)
  +(-.5,+.5) node[draw,circle,inner sep=0.5pt,gray!50!white] (mw) {$\scriptstyle m$}
  +(+.25,-.25) node[draw,circle,inner sep=0.5pt,gray!50!white] (me) {$\scriptstyle m$}
  ;
  \draw[string,gray!50!white] (mw) -- (nw);
  \draw[string,gray!50!white, name path=W] (me) -- (ne);
  \draw[string,gray!50!white, name path=Y] (a) .. controls +(.25,.25) and +(-1,-.25) .. (me);
  \draw[string,gray!50!white] (c) .. controls +(.25,.25) and +(.125,-.25) .. (me);
  \draw[string,gray!50!white, name path=X] (b) .. controls +(-.25,1) and +(-.25,-.5) .. (mw);
  \draw[string,gray!50!white, name path=Z] (d) .. controls +(-.25,1) and +(+1,-.5) .. (mw);
    \path[name intersections={of=X and Y}] (intersection-1) node[dot]{};
    \path[name intersections={of=Z and W}] (intersection-1) node[dot]{};
\end{tikzpicture}
\end{equation}

 We can now unpack these pictures in terms of the dictionary from Remark~\ref{rem:notationforbimnd}: the solid black bullets at the crossings correspond to the underlying object of the walking bialgebra $B$, and the comultiplication $\Delta_B$ and multiplication $m_B$ correspond respectively to the moves that split a bullet into two or merge two bullets into one, as in (\ref{eq:underlyingmultiplication},\ref{eq:underlyingcomultiplication}).
 As explained there, in terms of these pictures,  the equivalence  $\Funoplax_*(\Mnd \otimes \Mnd, \cC) \simeq \BiAlg(\Omega^2 \cC)$ implements a ``$45^\circ$-rotation to the left,'' aligning the comultiplication of the coalgebra with the vertical axis and the multiplication of the algebra with the horizontal axis. After this ``rotation'', comparing with Notation~\ref{not:shearmaps},  the image of $\widetilde{\sh}$ becomes $\sh^{\nearrow}$. 
\end{proof}

The core of the proof of our main Theorem~\ref{thm:mainHopftheorem} is the following remarkable result:

\begin{theorem} \label{thm:invertibleshear}Under $\Mnd \otimes \Mnd \to \Adj \otimes \Adj \to L_3(\Adj \otimes \Adj)$, the universal shear $3$-morphism from Construction~\ref{cons:universalshear} is sent to an invertible $3$-morphism. 
\end{theorem}
\begin{proof}
Since we currently do not know whether $\Mnd \otimes \Mnd$ and $\Adj \otimes \Adj$ are strict, we have to resort to a similar trick as in Construction~\ref{cons:universalshear}.
Let $\OO^2$ denote the strict $2$-category~\eqref{eq:oriental} and $e\OO^2$ its whiskering with two $1$-cells: 
\begin{equation*} 
e\OO^2 := \quad
\begin{tikzcd}
\bullet&  \arrow[l]\bullet &\arrow[d, Leftarrow, shorten >=4pt]&  \arrow[ll] \arrow[dl] \bullet &  \arrow[l]\bullet\\
&& \bullet \arrow[ul] 
\end{tikzcd}
\end{equation*}
Construct the following commutative diagram
\begin{equation}\label{eqn:commutativesquareoforientals}
\begin{tikzcd}
 \OO^2  \arrow[d]\arrow[r] & e \OO^2 \arrow[d] \\
 \Mnd \arrow[r] & \Adj
\end{tikzcd}
\end{equation}
of strict $2$-categories:
\begin{itemize}
  \item The strict $2$-functor $\Mnd \to \Adj$ is the standard inclusion.
  \item The  strict $2$-functor $\OO^2 \to e\OO^2$ sends the $2$-arrow in $\OO^2$ to its obvious whiskering; in particular, it sends the  ``long'' $1$-arrow to the composite of the top path of $1$-arrows, and it sends the short $1$-arrows to the 
 composite of the left, resp.\ right half of the bottom path of $1$-arrows. 
 \item The strict $2$-functor $\OO^2 \to \Mnd$ sends all generating $1$-cells of $\OO^2$ to the generating $1$-cell of $\Mnd$ and sends the generating $2$-cell to the ``multiplication'' in $\Mnd$.
 \item The strict $2$-functor $e\OO^2 \to \Adj$ sends the generating cells to the following morphisms in $\Adj$ (where~$l$ and~$r$ denote the generating left and right adjoints, and~$\ev$ denotes the counit of the adjunction):
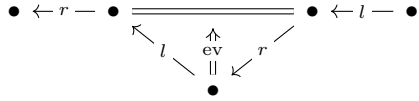
\[
\begin{tikzcd}
\bullet\arrow[r, leftarrow,"r" description] & \bullet \arrow[dr, leftarrow,"l" description]  \arrow[rr, equal] &\arrow[d, Leftarrow, "\ev" description]& \bullet  \arrow[r,leftarrow, "l" description] & \bullet\\
&& \bullet \arrow[ur, leftarrow,"r" description] 
\end{tikzcd}
\]
\end{itemize}
The commutative square \eqref{eqn:commutativesquareoforientals} induces a commutative diagram of $(\infty,\infty)$-categories:
\[\begin{tikzcd}
 \OO^2 \otimes \OO^2 \arrow[d] \arrow[r] & e \OO^2 \otimes e\OO^2 \arrow[d] \\
 \Mnd \otimes \Mnd \arrow[r] & \Adj \otimes \Adj
\end{tikzcd}.
\]
By definition, the universal shear $3$-morphism lifts through $\OO^2 \otimes \OO^2$ and hence its image in $\Adj \otimes \Adj$ lifts through $e\OO^2 \otimes e\OO^2$.
Both $\OO^2$ and $e\OO^2$ are in the idempotent completion of Campion's category~$\fancysquare$,
and so by~\cite{2311.00205} the comparison functors $\OO^2 \otimes \OO^2 \to \OO^2 \otimes^{\strict} \OO^2$ and $e\OO^2 \otimes e\OO^2 \to e\OO^2 \otimes^{\strict} e\OO^2$ to the strict tensor products are equivalences; thus we may compute with the strict Gray tensor product.

For the remainder of the proof, we switch notation to the graphical calculus for~$\otimes^\strict$ described in~\S\ref{subsec:otimesstrict}. Our notation for the generators of $\OO^2$ and $e\OO^2$ anticipates  the functors 
$\OO^2 \to \Mnd$ and $e\OO^2 \to \Adj$; in particular, 
we will suppress the labels for objects (they can always be worked out from context). 
We denote the generators of $\OO^2$ and $e\OO^2$ as follows:
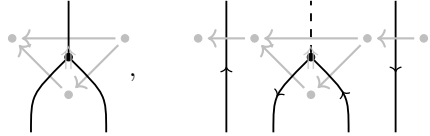
\[ 
\begin{tikzpicture}[baseline=(baseline)]
  \path (0,0) node[dot] (m) {}  +(0,-.25) coordinate (baseline)
  +(0,+.75) coordinate (n)
  +(-.5,-1) coordinate (sw)
  +(+.5,-1) coordinate (se)
  ;
  \path (m)
  +(0,-.5) node[dot, gray!50!white] (b) {}
  +(-.75,+.25) node[dot, gray!50!white] (c) {}
  +(+.75,+.25) node[dot, gray!50!white] (a) {}
  ;
  \draw[grayarrow] (a) -- (b);
  \draw[grayarrow] (b) -- (c);
  \draw[grayarrow] (a) -- (c);
  \draw[string] (m) -- (n);
  \draw[string] (se) .. controls +(0,.5) and +(+.5,-.5) .. (m);
  \draw[string] (sw) .. controls +(0,.5) and +(-.5,-.5) .. (m);
  \path (m) node[gray!50!white] {$\Uparrow$};
\end{tikzpicture}
,\qquad
\begin{tikzpicture}[baseline=(baseline)]
  \path (0,0) node[dot] (m) {}  +(0,-.25) coordinate (baseline)
  +(0,+.75) coordinate (n)
  +(-.5,-1) coordinate (sw)
  +(+.5,-1) coordinate (se)
  +(-1.125,-1) coordinate (farsw)
  +(-1.125,+.75) coordinate (farnw)
  +(+1.125,-1) coordinate (farse)
  +(+1.125,+.75) coordinate (farne)
  ;
  \path (m)
  +(0,-.5) node[dot, gray!50!white] (b) {}
  +(-.75,+.25) node[dot, gray!50!white] (c) {}
  +(+.75,+.25) node[dot, gray!50!white] (a) {}
  +(+1.5,+.25) node[dot, gray!50!white] (z) {}
  +(-1.5,+.25) node[dot, gray!50!white] (d) {}
  ;
  \draw[grayarrow] (a) -- (b);
  \draw[grayarrow] (b) -- (c);
  \draw[grayarrow] (a) -- (c);
  \draw[string,dashed] (m) -- (n);
  \draw[string, arrow data={.5}{>}] (se) .. controls +(0,.5) and +(+.5,-.5) .. (m);
  \draw[string, arrow data={.5}{<}] (sw) .. controls +(0,.5) and +(-.5,-.5) .. (m);
  \path (m) node[gray!50!white] {$\Uparrow$};
  \draw[string, arrow data={.5}{<}] (farse) -- (farne);
  \draw[string, arrow data={.5}{>}] (farsw) -- (farnw);
  \draw[grayarrow] (z) -- (a);
  \draw[grayarrow] (c) -- (d);
\end{tikzpicture}
\]

Recall that the universal shear map lifts to a $3$-morphism  $\widetilde{\sh}$ in $\OO^2 \otimes \OO^2 \simeq \OO^2 \otimes^{\strict} \OO^2$ which in the graphical calculus is defined as~\eqref{eq:shearpicture}. 
Under the functor $\OO^2 \otimes \OO^2 \to e\OO^2 \otimes e\OO^2$, this $3$-morphism $\widetilde{\sh}$ is mapped to the following composite (where each of the $3$-morphism below is a composite of two generating interchange $3$-morphisms):
\begin{equation}\label{eq:imagesh}
\begin{tikzpicture}[baseline=(baseline)]
  \path (0,0) coordinate (middle) +(0,-.125) coordinate (baseline)
  +(-1,+1) coordinate (nw) +(-1.25,+1) coordinate (nwp) +(-.75,+1) coordinate (nwq)
  +(+,+1) coordinate (ne) +(+1.25,+1) coordinate (nep) +(+.75,+1) coordinate (neq)
  +(-1,-1) coordinate (a) +(-1.25,-1) coordinate (ap)
  +(-.25,-1) coordinate (b) +(-.5,-1) coordinate (bp)
  +(+.25,-1) coordinate (c) +(.5,-1) coordinate (cp)
  +(+1,-1) coordinate (d) +(+1.25,-1) coordinate (dp)
  ;
  \path (middle)
  +(-.25,-.25) node[dot] (mw) {}
  +(+.5,+.5) node[dot] (me) {}
  ;
  \draw[string,dashed] (mw) -- (nw);
  \draw[string,dashed] (me) -- (ne);
  \draw[string, arrow data={.5}{<}] (a) .. controls +(.25,1) and +(-1,-.5) .. (me);
  \draw[string, arrow data={.5}{>}] (c) .. controls +(.25,1) and +(.25,-.5) .. (me);
  \draw[string, arrow data={.5}{<}] (b) .. controls +(-.25,.25) and +(-.125,-.25) .. (mw);
  \draw[string, arrow data={.5}{>}] (d) .. controls +(-.25,.25) and +(+1,-.25) .. (mw);
  \draw[string, arrow data={.6}{>}] (ap) .. controls +(.25,1.5) and +(-.5,-.5) .. (neq);
  \draw[string, arrow data={.6}{>}] (bp) .. controls +(-.5,.5) and +(+.5,-.5) .. (nwp);
  \draw[string, arrow data={.55}{<}] (dp) .. controls +(-.5,1) and +(+.5,-.5) .. (nwq);
  \draw[string, arrow data={.6}{<}] (cp) .. controls +(.25,.5) and +(-.5,-.5) .. (nep);
\end{tikzpicture}
\Rrightarrow
\begin{tikzpicture}[baseline=(baseline)]
  \path (0,0) coordinate (middle) +(0,-.125) coordinate (baseline)
  +(-1,+1) coordinate (nw) +(-1.25,+1) coordinate (nwp) +(-.75,+1) coordinate (nwq)
  +(+,+1) coordinate (ne) +(+1.25,+1) coordinate (nep) +(+.75,+1) coordinate (neq)
  +(-1,-1) coordinate (a) +(-1.25,-1) coordinate (ap)
  +(-.25,-1) coordinate (b) +(-.5,-1) coordinate (bp)
  +(+.25,-1) coordinate (c) +(.5,-1) coordinate (cp)
  +(+1,-1) coordinate (d) +(+1.25,-1) coordinate (dp)
  ;
  \path (middle)
  +(-.5,+.5) node[dot] (mw) {}
  +(+.5,+.5) node[dot] (me) {}
  ;
  \draw[string,dashed] (mw) -- (nw);
  \draw[string,dashed] (me) -- (ne);
  \draw[string, arrow data={.6}{<}] (a) .. controls +(.25,1) and +(-1,-.5) .. (me);
  \draw[string, arrow data={.5}{>}] (c) .. controls +(.25,1) and +(.25,-.5) .. (me);
  \draw[string, arrow data={.5}{<}] (b) .. controls +(-.25,1) and +(-.25,-.5) .. (mw);
  \draw[string, arrow data={.6}{>}] (d) .. controls +(-.25,1) and +(+1,-.5) .. (mw);
  \draw[string, arrow data={.55}{>}] (ap) .. controls +(.25,1.5) and +(-.5,-.5) .. (neq);
  \draw[string, arrow data={.6}{>}] (bp) .. controls +(-.125,.5) and +(+.5,-.5) .. (nwp);
  \draw[string, arrow data={.55}{<}] (dp) .. controls +(-.25,1.5) and +(+.5,-.5) .. (nwq);
  \draw[string, arrow data={.6}{<}] (cp) .. controls +(.125,.5) and +(-.5,-.5) .. (nep);
\end{tikzpicture}
\Rrightarrow
\begin{tikzpicture}[baseline=(baseline)]
  \path (0,0) coordinate (middle) +(0,-.125) coordinate (baseline)
  +(-1,+1) coordinate (nw) 
  +(+,+1) coordinate (ne) 
  +(-1,-1) coordinate (a) 
  +(-.25,-1) coordinate (b) 
  +(+.25,-1) coordinate (c) 
  +(+1,-1) coordinate (d) 
  ;
  \path (middle)
  +(-.5,+.5) node[dot] (mw) {}
  +(+.25,-.25) node[dot] (me) {}
  ;
  \draw[string,dashed] (mw) -- (nw);
  \draw[string,dashed] (me) -- (ne);
  \draw[string, arrow data={.5}{<}] (a) .. controls +(.25,.25) and +(-1,-.25) .. (me);
  \draw[string, arrow data={.5}{>}] (c) .. controls +(.25,.25) and +(.125,-.25) .. (me);
  \draw[string, arrow data={.5}{<}] (b) .. controls +(-.25,1) and +(-.25,-.5) .. (mw);
  \draw[string, arrow data={.5}{>}] (d) .. controls +(-.25,1) and +(+1,-.5) .. (mw);
  \draw[string, arrow data={.55}{>}] (ap) .. controls +(.5,1) and +(-.5,-.5) .. (neq);
  \draw[string, arrow data={.6}{>}] (bp) .. controls +(-.25,.5) and +(+.5,-.5) .. (nwp);
  \draw[string, arrow data={.6}{<}] (dp) .. controls +(-.25,1.5) and +(+.5,-.5) .. (nwq);
  \draw[string, arrow data={.6}{<}] (cp) .. controls +(.5,.5) and +(-.5,-.5) .. (nep);
\end{tikzpicture}
\end{equation}
The 3-morphism \eqref{eq:imagesh} is a whiskering of the following composite in $e \OO^2 \otimes e\OO^2$:
\begin{equation}\label{eq:longcomposite}
\begin{tikzpicture}[baseline=(baseline),scale=.75]
  \path (0,0) coordinate (middle) +(0,-.125) coordinate (baseline)
  +(-1,+1) coordinate (nw) +(-1.25,+1) coordinate (nwp) +(-.75,+1) coordinate (nwq)
  +(+,+1) coordinate (ne) +(+1.25,+1) coordinate (nep) +(+.75,+1) coordinate (neq)
  +(-1,-1) coordinate (a) +(-1.25,-1) coordinate (ap)
  +(-.25,-1) coordinate (b) +(-.5,-1) coordinate (bp)
  +(+.25,-1) coordinate (c) +(.5,-1) coordinate (cp)
  +(+1,-1) coordinate (d) +(+1.25,-1) coordinate (dp)
  ;
  \path (middle)
  +(-.25,-.25) node[dot] (mw) {}
  +(+.5,+.5) node[dot] (me) {}
  ;
  \draw[string,dashed] (mw) -- (nw);
  \draw[string,dashed] (me) -- (ne);
  \draw[string, arrow data={.5}{<}] (a) .. controls +(.25,1) and +(-1,-.5) .. (me);
  \draw[string, arrow data={.5}{>}] (c) .. controls +(.25,1) and +(.25,-.5) .. (me);
  \draw[string, arrow data={.5}{<}] (b) .. controls +(-.25,.25) and +(-.125,-.25) .. (mw);
  \draw[string, arrow data={.5}{>}] (d) .. controls +(-.25,.25) and +(+1,-.25) .. (mw);
  \draw[string, arrow data={.6}{>}] (ap) .. controls +(.25,1.5) and +(-.5,-.5) .. (neq);
  \draw[string, arrow data={.55}{<}] (dp) .. controls +(-.5,1) and +(+.5,-.5) .. (nwq);
\end{tikzpicture}
\Rrightarrow
\begin{tikzpicture}[baseline=(baseline),scale=.75]
  \path (0,0) coordinate (middle) +(0,-.125) coordinate (baseline)
  +(-1,+1) coordinate (nw) +(-1.25,+1) coordinate (nwp) +(-.75,+1) coordinate (nwq)
  +(+,+1) coordinate (ne) +(+1.25,+1) coordinate (nep) +(+.75,+1) coordinate (neq)
  +(-1,-1) coordinate (a) +(-1.25,-1) coordinate (ap)
  +(-.25,-1) coordinate (b) +(-.5,-1) coordinate (bp)
  +(+.25,-1) coordinate (c) +(.5,-1) coordinate (cp)
  +(+1,-1) coordinate (d) +(+1.25,-1) coordinate (dp)
  ;
  \path (middle)
  +(-.5,0) node[dot] (mw) {}
  +(+.5,+.5) node[dot] (me) {}
  ;
  \draw[string,dashed] (mw) -- (nw);
  \draw[string,dashed] (me) -- (ne);
  \draw[string, arrow data={.5}{<}] (a) .. controls +(.25,.5) and +(-.5,-.5) .. (me);
  \draw[string, arrow data={.5}{>}] (c) .. controls +(.25,1) and +(.25,-.5) .. (me);
  \draw[string, arrow data={.5}{<}] (b) .. controls +(-.25,.25) and +(-.125,-.25) .. (mw);
  \draw[string, arrow data={.6}{>}] (d) .. controls +(-.25,.25) and +(+1,-.25) .. (mw);
  \draw[string, arrow data={.6}{>}] (ap) .. controls +(.25,1.5) and +(-.5,-.5) .. (neq);
  \draw[string, arrow data={.5}{<}] (dp) .. controls +(-.5,1) and +(+.5,-.5) .. (nwq);
\end{tikzpicture}
\Rrightarrow
\begin{tikzpicture}[baseline=(baseline),scale=.75]
  \path (0,0) coordinate (middle) +(0,-.125) coordinate (baseline)
  +(-1,+1) coordinate (nw) +(-1.25,+1) coordinate (nwp) +(-.75,+1) coordinate (nwq)
  +(+,+1) coordinate (ne) +(+1.25,+1) coordinate (nep) +(+.75,+1) coordinate (neq)
  +(-1,-1) coordinate (a) +(-1.25,-1) coordinate (ap)
  +(-.25,-1) coordinate (b) +(-.5,-1) coordinate (bp)
  +(+.25,-1) coordinate (c) +(.5,-1) coordinate (cp)
  +(+1,-1) coordinate (d) +(+1.25,-1) coordinate (dp)
  ;
  \path (middle)
  +(-.5,+.5) node[dot] (mw) {}
  +(+.5,+.5) node[dot] (me) {}
  ;
  \draw[string,dashed] (mw) -- (nw);
  \draw[string,dashed] (me) -- (ne);
  \draw[string, arrow data={.6}{<}] (a) .. controls +(.25,1) and +(-1,-.5) .. (me);
  \draw[string, arrow data={.5}{>}] (c) .. controls +(.25,1) and +(.25,-.5) .. (me);
  \draw[string, arrow data={.5}{<}] (b) .. controls +(-.25,1) and +(-.25,-.5) .. (mw);
  \draw[string, arrow data={.6}{>}] (d) .. controls +(-.25,1) and +(+1,-.5) .. (mw);
  \draw[string, arrow data={.55}{>}] (ap) .. controls +(.25,1.5) and +(-.5,-.5) .. (neq);
  \draw[string, arrow data={.55}{<}] (dp) .. controls +(-.25,1.5) and +(+.5,-.5) .. (nwq);
\end{tikzpicture}
\Rrightarrow
\begin{tikzpicture}[baseline=(baseline),scale=.75]
  \path (0,0) coordinate (middle) +(0,-.125) coordinate (baseline)
  +(-1,+1) coordinate (nw) 
  +(+,+1) coordinate (ne) 
  +(-1,-1) coordinate (a) 
  +(-.25,-1) coordinate (b) 
  +(+.25,-1) coordinate (c) 
  +(+1,-1) coordinate (d) 
  ;
  \path (middle)
  +(-.5,+.5) node[dot] (mw) {}
  +(+.5,0) node[dot] (me) {}
  ;
  \draw[string,dashed] (mw) -- (nw);
  \draw[string,dashed] (me) -- (ne);
  \draw[string, arrow data={.6}{<}] (a) .. controls +(.25,.25) and +(-1,-.25) .. (me);
  \draw[string, arrow data={.5}{>}] (c) .. controls +(.25,.25) and +(.125,-.25) .. (me);
  \draw[string, arrow data={.5}{<}] (b) .. controls +(-.25,1) and +(-.25,-.5) .. (mw);
  \draw[string, arrow data={.5}{>}] (d) .. controls +(-.25,.5) and +(+.5,-.5) .. (mw);
  \draw[string, arrow data={.5}{>}] (ap) .. controls +(.5,1) and +(-.5,-.5) .. (neq);
  \draw[string, arrow data={.6}{<}] (dp) .. controls +(-.25,1.5) and +(+.5,-.5) .. (nwq);
\end{tikzpicture}
\Rrightarrow
\begin{tikzpicture}[baseline=(baseline),scale=.75]
  \path (0,0) coordinate (middle) +(0,-.125) coordinate (baseline)
  +(-1,+1) coordinate (nw) 
  +(+,+1) coordinate (ne) 
  +(-1,-1) coordinate (a) 
  +(-.25,-1) coordinate (b) 
  +(+.25,-1) coordinate (c) 
  +(+1,-1) coordinate (d) 
  ;
  \path (middle)
  +(-.5,+.5) node[dot] (mw) {}
  +(+.25,-.25) node[dot] (me) {}
  ;
  \draw[string,dashed] (mw) -- (nw);
  \draw[string,dashed] (me) -- (ne);
  \draw[string, arrow data={.5}{<}] (a) .. controls +(.25,.25) and +(-1,-.25) .. (me);
  \draw[string, arrow data={.5}{>}] (c) .. controls +(.25,.25) and +(.125,-.25) .. (me);
  \draw[string, arrow data={.5}{<}] (b) .. controls +(-.25,1) and +(-.25,-.5) .. (mw);
  \draw[string, arrow data={.5}{>}] (d) .. controls +(-.25,1) and +(+1,-.5) .. (mw);
  \draw[string, arrow data={.55}{>}] (ap) .. controls +(.5,1) and +(-.5,-.5) .. (neq);
  \draw[string, arrow data={.6}{<}] (dp) .. controls +(-.25,1.5) and +(+.5,-.5) .. (nwq);
\end{tikzpicture}
\end{equation}
Thus to show that \eqref{eq:imagesh} becomes invertible under $e\OO^2 \otimes e\OO^2 \to \Adj \otimes \Adj \to L_3(\Adj \otimes \Adj)$, it suffices to show that \eqref{eq:longcomposite} becomes invertible.

In the notation of Corollary~\ref{cor:L3AdjAdj}, the $2$-morphism
\tikz[baseline=(baseline)]{ \path (0,0) coordinate (baseline); \draw[string, arrow data={.85}{>}] (0,-.5ex) -- (2ex,2ex); \draw[string, arrow data={.8}{<}] (2ex,-.5ex) -- (0,2ex); }
corepresents the filler $\beta$ of the square $kh \To gf$; that Corollary says that this $2$-morphism becomes invertible in $L_3(\Adj \otimes \Adj)$. Thus to prove that~\eqref{eq:longcomposite} becomes invertible, it suffices to prove that it becomes invertible after whiskering with \tikz[baseline=(baseline)]{ \path (0,0) coordinate (baseline); \draw[string, arrow data={.85}{>}] (0,-.5ex) -- (2ex,2ex); \draw[string, arrow data={.8}{<}] (2ex,-.5ex) -- (0,2ex); }. This whiskering becomes the top horizontal composite in the following diagram of $3$- and $4$-morphisms in $e\OO^2 \otimes^{\strict} e\OO^2$:
\begin{equation}\label{eqn:bigdiagram}\hspace{-1in}
\begin{array}{ccccccccc}
\begin{tikzpicture}[baseline=(baseline),scale=.75]
  \path (0,0) coordinate (middle) +(0,-.125) coordinate (baseline)
  +(-1,+1) coordinate (nw) +(-1.25,+1) coordinate (nwp) +(-.75,+1) coordinate (nwq)
  +(+,+1) coordinate (ne) +(+1.25,+1) coordinate (nep) +(+.75,+1) coordinate (neq)
  +(-1,-1) coordinate (a) +(-1.25,-1) coordinate (ap)
  +(+.25,-1) coordinate (b) +(-.5,-1) coordinate (bp)
  +(-.25,-1) coordinate (c) +(.5,-1) coordinate (cp)
  +(+1,-1) coordinate (d) +(+1.25,-1) coordinate (dp)
  ;
  \path (middle)
  +(-.25,-.25) node[dot] (mw) {}
  +(+.5,+.5) node[dot] (me) {}
  ;
  \draw[string,dashed] (mw) -- (nw);
  \draw[string,dashed] (me) -- (ne);
  \draw[string, arrow data={.5}{<}] (a) .. controls +(.25,1) and +(-1,-.5) .. (me);
  \draw[string, arrow data={.6}{>}] (c) .. controls +(.25,1) and +(.25,-.5) .. (me);
  \draw[string, arrow data={.4}{<}] (b) .. controls +(-.25,.25) and +(-.125,-.25) .. (mw);
  \draw[string, arrow data={.5}{>}] (d) .. controls +(-.25,.25) and +(+1,-.25) .. (mw);
  \draw[string, arrow data={.6}{>}] (ap) .. controls +(.25,1.5) and +(-.5,-.5) .. (neq);
  \draw[string, arrow data={.55}{<}] (dp) .. controls +(-.5,1) and +(+.5,-.5) .. (nwq);
\end{tikzpicture}
& \Rrightarrow &
\begin{tikzpicture}[baseline=(baseline),scale=.75]
  \path (0,0) coordinate (middle) +(0,-.125) coordinate (baseline)
  +(-1,+1) coordinate (nw) +(-1.25,+1) coordinate (nwp) +(-.75,+1) coordinate (nwq)
  +(+,+1) coordinate (ne) +(+1.25,+1) coordinate (nep) +(+.75,+1) coordinate (neq)
  +(-1,-1) coordinate (a) +(-1.25,-1) coordinate (ap)
  +(+.25,-1) coordinate (b) +(-.5,-1) coordinate (bp)
  +(-.25,-1) coordinate (c) +(.5,-1) coordinate (cp)
  +(+1,-1) coordinate (d) +(+1.25,-1) coordinate (dp)
  ;
  \path (middle)
  +(-.5,0) node[dot] (mw) {}
  +(+.5,+.5) node[dot] (me) {}
  ;
  \draw[string,dashed] (mw) -- (nw);
  \draw[string,dashed] (me) -- (ne);
  \draw[string, arrow data={.5}{<}] (a) .. controls +(.25,.5) and +(-.5,-.5) .. (me);
  \draw[string, arrow data={.5}{>}] (c) .. controls +(.25,.5) and +(.25,-.5) .. (me);
  \draw[string, arrow data={.5}{<}] (b) .. controls +(-.25,.25) and +(-.125,-.25) .. (mw);
  \draw[string, arrow data={.5}{>}] (d) .. controls +(-.25,.25) and +(+1,-.25) .. (mw);
  \draw[string, arrow data={.6}{>}] (ap) .. controls +(.25,1.5) and +(-.5,-.5) .. (neq);
  \draw[string, arrow data={.5}{<}] (dp) .. controls +(-.5,1) and +(+.5,-.5) .. (nwq);
\end{tikzpicture}
& \overset R \Rrightarrow &
\begin{tikzpicture}[baseline=(baseline),scale=.75]
  \path (0,0) coordinate (middle) +(0,-.125) coordinate (baseline)
  +(-1,+1) coordinate (nw) +(-1.25,+1) coordinate (nwp) +(-.75,+1) coordinate (nwq)
  +(+,+1) coordinate (ne) +(+1.25,+1) coordinate (nep) +(+.75,+1) coordinate (neq)
  +(-1,-1) coordinate (a) +(-1.25,-1) coordinate (ap)
  +(+.25,-1) coordinate (b) +(-.5,-1) coordinate (bp)
  +(-.25,-1) coordinate (c) +(.5,-1) coordinate (cp)
  +(+1,-1) coordinate (d) +(+1.25,-1) coordinate (dp)
  ;
  \path (middle)
  +(-.5,+.5) node[dot] (mw) {}
  +(+.5,+.5) node[dot] (me) {}
  ;
  \draw[string,dashed] (mw) -- (nw);
  \draw[string,dashed] (me) -- (ne);
  \draw[string, arrow data={.6}{<}] (a) .. controls +(.25,1) and +(-1,-.5) .. (me);
  \draw[string, arrow data={.5}{>}] (c) .. controls +(.25,.5) and +(.5,-.5) .. (me);
  \draw[string, arrow data={.5}{<}] (b) .. controls +(-.25,.5) and +(-.5,-.5) .. (mw);
  \draw[string, arrow data={.6}{>}] (d) .. controls +(-.25,1) and +(+1,-.5) .. (mw);
  \draw[string, arrow data={.55}{>}] (ap) .. controls +(.25,1.5) and +(-.5,-.5) .. (neq);
  \draw[string, arrow data={.55}{<}] (dp) .. controls +(-.25,1.5) and +(+.5,-.5) .. (nwq);
\end{tikzpicture}
& \overset L \Rrightarrow &
\begin{tikzpicture}[baseline=(baseline),scale=.75]
  \path (0,0) coordinate (middle) +(0,-.125) coordinate (baseline)
  +(-1,+1) coordinate (nw) 
  +(+,+1) coordinate (ne) 
  +(-1,-1) coordinate (a) 
  +(+.25,-1) coordinate (b) 
  +(-.25,-1) coordinate (c) 
  +(+1,-1) coordinate (d) 
  ;
  \path (middle)
  +(-.5,+.5) node[dot] (mw) {}
  +(+.5,0) node[dot] (me) {}
  ;
  \draw[string,dashed] (mw) -- (nw);
  \draw[string,dashed] (me) -- (ne);
  \draw[string, arrow data={.5}{<}] (a) .. controls +(.25,.25) and +(-1,-.25) .. (me);
  \draw[string, arrow data={.5}{>}] (c) .. controls +(.25,.25) and +(.125,-.25) .. (me);
  \draw[string, arrow data={.4}{<}] (b) .. controls +(-.25,.5) and +(-.25,-.5) .. (mw);
  \draw[string, arrow data={.55}{>}] (d) .. controls +(-.25,.5) and +(+.5,-.5) .. (mw);
  \draw[string, arrow data={.5}{>}] (ap) .. controls +(.5,1) and +(-.5,-.5) .. (neq);
  \draw[string, arrow data={.6}{<}] (dp) .. controls +(-.25,1.5) and +(+.5,-.5) .. (nwq);
\end{tikzpicture}
& \Rrightarrow &
\begin{tikzpicture}[baseline=(baseline),scale=.75]
  \path (0,0) coordinate (middle) +(0,-.125) coordinate (baseline)
  +(-1,+1) coordinate (nw) 
  +(+,+1) coordinate (ne) 
  +(-1,-1) coordinate (a) 
  +(+.25,-1) coordinate (b) 
  +(-.25,-1) coordinate (c) 
  +(+1,-1) coordinate (d) 
  ;
  \path (middle)
  +(-.5,+.5) node[dot] (mw) {}
  +(+.25,-.25) node[dot] (me) {}
  ;
  \draw[string,dashed] (mw) -- (nw);
  \draw[string,dashed] (me) -- (ne);
  \draw[string, arrow data={.5}{<}] (a) .. controls +(.25,.25) and +(-1,-.25) .. (me);
  \draw[string, arrow data={.4}{>}] (c) .. controls +(.25,.25) and +(.125,-.25) .. (me);
  \draw[string, arrow data={.6}{<}] (b) .. controls +(-.25,1) and +(-.25,-.5) .. (mw);
  \draw[string, arrow data={.5}{>}] (d) .. controls +(-.25,1) and +(+1,-.5) .. (mw);
  \draw[string, arrow data={.55}{>}] (ap) .. controls +(.5,1) and +(-.5,-.5) .. (neq);
  \draw[string, arrow data={.6}{<}] (dp) .. controls +(-.25,1.5) and +(+.5,-.5) .. (nwq);
\end{tikzpicture}
\\[8ex]
& \raisebox{-1ex}{$\scriptstyle L$} \raisebox{+1ex}{\rotatebox{-45}{$\Rrightarrow$}}
& \circ
& \raisebox{-1ex}{$\scriptstyle L$} \raisebox{+1ex}{\rotatebox{-45}{$\Rrightarrow$}}
& \circ
& \raisebox{+1ex}{$\scriptstyle R$} \raisebox{-1ex}{\rotatebox{+45}{$\Rrightarrow$}}
& \circ
& \raisebox{+1ex}{$\scriptstyle R$} \raisebox{-1ex}{\rotatebox{+45}{$\Rrightarrow$}}
\\[4ex]
&& 
\begin{tikzpicture}[baseline=(baseline),scale=.75]
  \path (0,0) coordinate (middle) +(0,-.125) coordinate (baseline)
  +(-1,+1) coordinate (nw) +(-1.25,+1) coordinate (nwp) +(-.75,+1) coordinate (nwq)
  +(+,+1) coordinate (ne) +(+1.25,+1) coordinate (nep) +(+.75,+1) coordinate (neq)
  +(-1,-1) coordinate (a) +(-1.25,-1) coordinate (ap)
  +(+.25,-1) coordinate (b) +(-.5,-1) coordinate (bp)
  +(-.25,-1) coordinate (c) +(.5,-1) coordinate (cp)
  +(+1,-1) coordinate (d) +(+1.25,-1) coordinate (dp)
  ;
  \path (middle)
  +(-.25,-.25) node[dot] (mw) {}
  +(+.5,0) node[dot] (me) {}
  ;
  \draw[string,dashed] (mw) -- (nw);
  \draw[string,dashed] (me) -- (ne);
  \draw[string, arrow data={.7}{<}] (a) .. controls +(.25,1) and +(-1,0) .. (me);
  \draw[string, arrow data={.5}{>}] (c) .. controls +(.25,.5) and +(.5,-.5) .. (me);
  \draw[string, arrow data={.6}{<}] (b) .. controls +(-.25,.5) and +(-.5,-.5) .. (mw);
  \draw[string, arrow data={.4}{>}] (d) .. controls +(-.25,.5) and +(+1,0) .. (mw);
  \draw[string, arrow data={.55}{>}] (ap) .. controls +(.25,1.5) and +(-.5,-.5) .. (neq);
  \draw[string, arrow data={.55}{<}] (dp) .. controls +(-.25,1.5) and +(+.5,-.5) .. (nwq);
\end{tikzpicture}   
& \Rrightarrow &
\begin{tikzpicture}[baseline=(baseline),scale=.75]
  \path (0,0) coordinate (middle) +(0,-.125) coordinate (baseline)
  +(-1,+1) coordinate (nw) +(-1.25,+1) coordinate (nwp) +(-.75,+1) coordinate (nwq)
  +(+,+1) coordinate (ne) +(+1.25,+1) coordinate (nep) +(+.75,+1) coordinate (neq)
  +(-1,-1) coordinate (a) +(-1.25,-1) coordinate (ap)
  +(+.25,-1) coordinate (b) +(-.5,-1) coordinate (bp)
  +(-.25,-1) coordinate (c) +(.5,-1) coordinate (cp)
  +(+1,-1) coordinate (d) +(+1.25,-1) coordinate (dp)
  ;
  \path (middle)
  +(-.5,0) node[dot] (mw) {}
  +(+.5,0) node[dot] (me) {}
  ;
  \draw[string,dashed] (mw) -- (nw);
  \draw[string,dashed] (me) -- (ne);
  \draw[string, arrow data={.6}{<}] (a) .. controls +(.25,1) and +(-1,-.5) .. (me);
  \draw[string, arrow data={.5}{>}] (c) .. controls +(.25,.5) and +(.5,-.5) .. (me);
  \draw[string, arrow data={.5}{<}] (b) .. controls +(-.25,.5) and +(-.5,-.5) .. (mw);
  \draw[string, arrow data={.6}{>}] (d) .. controls +(-.25,1) and +(+1,-.5) .. (mw);
  \draw[string, arrow data={.55}{>}] (ap) .. controls +(.25,1.5) and +(-.5,-.5) .. (neq);
  \draw[string, arrow data={.55}{<}] (dp) .. controls +(-.25,1.5) and +(+.5,-.5) .. (nwq);
\end{tikzpicture}
& \Rrightarrow &
\begin{tikzpicture}[baseline=(baseline),scale=.75]
  \path (0,0) coordinate (middle) +(0,-.125) coordinate (baseline)
  +(-1,+1) coordinate (nw) +(-1.25,+1) coordinate (nwp) +(-.75,+1) coordinate (nwq)
  +(+,+1) coordinate (ne) +(+1.25,+1) coordinate (nep) +(+.75,+1) coordinate (neq)
  +(-1,-1) coordinate (a) +(-1.25,-1) coordinate (ap)
  +(+.25,-1) coordinate (b) +(-.5,-1) coordinate (bp)
  +(-.25,-1) coordinate (c) +(.5,-1) coordinate (cp)
  +(+1,-1) coordinate (d) +(+1.25,-1) coordinate (dp)
  ;
  \path (middle)
  +(-.5,0) node[dot] (mw) {}
  +(+.25,-.25) node[dot] (me) {}
  ;
  \draw[string,dashed] (mw) -- (nw);
  \draw[string,dashed] (me) -- (ne);
  \draw[string, arrow data={.6}{<}] (a) .. controls +(.25,.5) and +(-1,0) .. (me);
  \draw[string, arrow data={.6}{>}] (c) .. controls +(.25,.5) and +(.5,-.5) .. (me);
  \draw[string, arrow data={.5}{<}] (b) .. controls +(-.25,.5) and +(-.5,-.5) .. (mw);
  \draw[string, arrow data={.7}{>}] (d) .. controls +(-.25,1) and +(+1,0) .. (mw);
  \draw[string, arrow data={.55}{>}] (ap) .. controls +(.25,1.5) and +(-.5,-.5) .. (neq);
  \draw[string, arrow data={.55}{<}] (dp) .. controls +(-.25,1.5) and +(+.5,-.5) .. (nwq);
\end{tikzpicture}
\\[8ex]
& \raisebox{1.5ex}{$\scriptstyle R$}\raisebox{-.5ex}{\rotatebox{45}{$\Rrightarrow$}}
&&& \raisebox{2.5ex}{\rotatebox{-90}{\fourarrow}}
&&& \raisebox{1.5ex}{\rotatebox{-45}{$\Rrightarrow$}}\raisebox{1.5ex}{$\scriptstyle L$}
\\[4ex]
\begin{tikzpicture}[baseline=(baseline),scale=.75]
  \path (0,0) coordinate (middle) +(0,-.125) coordinate (baseline)
  +(-1,+1) coordinate (nw) +(-1.25,+1) coordinate (nwp) +(-.75,+1) coordinate (nwq)
  +(+,+1) coordinate (ne) +(+1.25,+1) coordinate (nep) +(+.75,+1) coordinate (neq)
  +(-1,-1) coordinate (a) +(-1.25,-1) coordinate (ap)
  +(+.25,-1) coordinate (b) +(-.5,-1) coordinate (bp)
  +(-.25,-1) coordinate (c) +(.5,-1) coordinate (cp)
  +(+1,-1) coordinate (d) +(+1.25,-1) coordinate (dp)
  ;
  \path (middle)
  +(.5,-.25) node[dot] (mw) {}
  +(0,.25) node[dot] (me) {}
  ;
  \draw[string,dashed] (mw) .. controls +(-.5,+.25) and +(+.25,-.5) .. (nw);
  \draw[string,dashed] (me) .. controls +(+.5,+.5) and +(-.5,-.5) .. (ne);
  \draw[string, arrow data={.6}{<}] (a) .. controls +(0,.25) and +(-.5,-.5) .. (me);
  \draw[string, arrow data={.5}{>}] (c) .. controls +(0,.5) and +(.25,-.5) .. (me);
  \draw[string, arrow data={.5}{<}] (b) .. controls +(0,.25) and +(0,-.25) .. (mw);
  \draw[string, arrow data={.6}{>}] (d) .. controls +(0,.25) and +(+.5,-.25) .. (mw);
  \draw[string, arrow data={.5}{>}] (ap) .. controls +(.25,1.5) and +(-.5,-.5) .. (neq);
  \draw[string, arrow data={.5}{<}] (dp) .. controls +(-.25,1.5) and +(+.5,-.5) .. (nwq);
\end{tikzpicture}
&& \overset{i}{\Rrightarrow} &&
\begin{tikzpicture}[baseline=(baseline),scale=.75]
  \path (0,0) coordinate (middle) +(0,-.125) coordinate (baseline)
  +(-1,+1) coordinate (nw) +(-1.25,+1) coordinate (nwp) +(-.75,+1) coordinate (nwq)
  +(+,+1) coordinate (ne) +(+1.25,+1) coordinate (nep) +(+.75,+1) coordinate (neq)
  +(-1,-1) coordinate (a) +(-1.25,-1) coordinate (ap)
  +(+.25,-1) coordinate (b) +(-.5,-1) coordinate (bp)
  +(-.25,-1) coordinate (c) +(.5,-1) coordinate (cp)
  +(+1,-1) coordinate (d) +(+1.25,-1) coordinate (dp)
  ;
  \path (middle)
  +(+.5,0) node[dot] (mw) {}
  +(-.5,0) node[dot] (me) {}
  ;
  \draw[string,dashed] (mw) .. controls +(-.5,+.5) and +(+.5,-.5) .. (nw);
  \draw[string,dashed] (me) .. controls +(+.5,+.5) and +(-.5,-.5) .. (ne);
  \draw[string, arrow data={.6}{<}] (a) .. controls +(0,.25) and +(-.5,-.5) .. (me);
  \draw[string, arrow data={.5}{>}] (c) .. controls +(0,.5) and +(.25,-.5) .. (me);
  \draw[string, arrow data={.5}{<}] (b) .. controls +(0,.5) and +(-.25,-.5) .. (mw);
  \draw[string, arrow data={.6}{>}] (d) .. controls +(0,.25) and +(+.5,-.5) .. (mw);
  \draw[string, arrow data={.5}{>}] (ap) .. controls +(.25,1.5) and +(-.5,-.5) .. (neq);
  \draw[string, arrow data={.5}{<}] (dp) .. controls +(-.25,1.5) and +(+.5,-.5) .. (nwq);
\end{tikzpicture}
&& \overset{i}{\Rrightarrow} &&
\begin{tikzpicture}[baseline=(baseline),scale=.75]
  \path (0,0) coordinate (middle) +(0,-.125) coordinate (baseline)
  +(-1,+1) coordinate (nw) +(-1.25,+1) coordinate (nwp) +(-.75,+1) coordinate (nwq)
  +(+,+1) coordinate (ne) +(+1.25,+1) coordinate (nep) +(+.75,+1) coordinate (neq)
  +(-1,-1) coordinate (a) +(-1.25,-1) coordinate (ap)
  +(+.25,-1) coordinate (b) +(-.5,-1) coordinate (bp)
  +(-.25,-1) coordinate (c) +(.5,-1) coordinate (cp)
  +(+1,-1) coordinate (d) +(+1.25,-1) coordinate (dp)
  ;
  \path (middle)
  +(0,.25) node[dot] (mw) {}
  +(-.5,-.25) node[dot] (me) {}
  ;
  \draw[string,dashed] (mw) .. controls +(-.5,+.5) and +(+.5,-.5) .. (nw);
  \draw[string,dashed] (me) .. controls +(+.5,+.25) and +(-.25,-.5) .. (ne);
  \draw[string, arrow data={.6}{<}] (a) .. controls +(0,.25) and +(-.5,-.25) .. (me);
  \draw[string, arrow data={.5}{>}] (c) .. controls +(0,.25) and +(0,-.25) .. (me);
  \draw[string, arrow data={.5}{<}] (b) .. controls +(0,.5) and +(-.25,-.5) .. (mw);
  \draw[string, arrow data={.6}{>}] (d) .. controls +(0,.25) and +(+.5,-.5) .. (mw);
  \draw[string, arrow data={.5}{>}] (ap) .. controls +(.25,1.5) and +(-.5,-.5) .. (neq);
  \draw[string, arrow data={.5}{<}] (dp) .. controls +(-.25,1.5) and +(+.5,-.5) .. (nwq);
\end{tikzpicture}
\end{array}
\hspace{-1in}\end{equation}
The three top squares in \eqref{eqn:bigdiagram}, i.e.\ those marked with ``$\circ$,'' strictly commute by  the axioms of composition in a strict $3$-category. The bottom hexagon in \eqref{eqn:bigdiagram} is filled with the $4$-cell filler of $\globe_2 \otimes \globe_2 \to e\OO^2 \otimes e\OO^2$, where $\globe_2 \to e\OO^2$ is the generating $2$-cell; being a $4$-cell, the bottom hexagon definitionally becomes invertible in $L_3(\Adj \otimes \Adj)$. 
The bottom horizontal $3$-morphisms labeled ``$i$'' along the bottom row of  \eqref{eqn:bigdiagram} become invertible in $\Adj \otimes \Adj$ since the dashed $1$-morphism is sent to an identity under $e\OO^2 \to \Adj$.

To complete the proof, it suffices to show that the $3$-cells labeled ``$R$'' and ``$L$''  become invertible in $L_3(\Adj \otimes \Adj)$;
indeed, then the top composition would be isomorphic, in $L_3(\Adj \otimes \Adj)$, to the composition around the perimeter of \eqref{eqn:bigdiagram} of invertible $3$-morphisms.
By Corollary~\ref{cor:adjunctibilitytop}, the generating $3$-cell $\globe_3 \to \globe_1 \otimes \globe_2$ is sent under $r \otimes \id$ to a right adjoint in $\Adj \otimes \globe_2$. Thus for any $2$-cell $\alpha : \globe_2 \to \Adj$, the $3$-cell $\globe_3 \to \globe_1 \otimes \globe_2  \to[ r \otimes \alpha] \Adj \otimes \Adj$ is a right adjoint. Such $3$-cells are represented graphically by ``pulling a $2$-cell southeastward across a northeast-pointing line.'' By the same token, any $3$-cell of the form $\globe_3 \to \globe_2 \otimes \globe_1 \to[\alpha \otimes l] \Adj \otimes \Adj$ --- represented graphically by ``pulling a $2$-cell northeastward across a southeast-pointing line'' --- is a left adjoint. These are, respectively, the $3$-cells labeled ``$R$'' and ``$L$.''
 But adjunctible $3$-cells in a $3$-category are invertible, and so these cells get mapped to invertibles in $L_3(\Adj \otimes \Adj)$.
\end{proof}

\begin{corollary}
\label{thm:adjinduceshopf} 
The functor 
$H : \Ret^{\someadj}(\cC) \to \BiAlg(\Omega^2\cC)$ from Corollary~\ref{cor:part2ofmaintheorem}
 factors through the full subcategory $\cat{coHopf}(\Omega^2\cC) \subset \BiAlg(\Omega^2\cC)$. 
\end{corollary}

\begin{proof}
We must show that the coshear map of any bialgebra in the image of 
$$\Ret^{\someadj}(\cC) = \Funoplax_*(\Adj \owedge \Adj, \cC) \to \Funoplax_*(\Mnd \owedge \Mnd, \cC)  = \BiAlg(\Omega^2\cC) $$
is invertible.
 By Lemma~\ref{lemma:universalshearrepresents}, the coshear map is represented by a universal shear $3$-morphism in $\Mnd \otimes \Mnd$. Since the diagram 
\[\begin{tikzcd}
L_3(\Mnd \otimes \Mnd) \arrow[r] \arrow[d] & L_3(\Adj \otimes \Adj) \arrow[d]\\
L_3(\Mnd \owedge \Mnd) \arrow[r] & L_3(\Adj \owedge \Adj)
\end{tikzcd}
\]
commutes, the statement follows from Theorem~\ref{thm:invertibleshear}.
\end{proof}

\section{Special cases and generalizations}

We end with some comments on the constructions in this paper. 
In \S\ref{subsec:hopmonad} we extend our main theorem from a construction of Hopf algebras to a construction of Hopf monads a la \cite{MR2355605,MR2793022} --- this is straightforward because, as we explain, the walking bimonad is an intermediate quotient between $L_3(\Mnd \otimes \Mnd)$ and $L_3(\Mnd \owedge \Mnd)$. We then unpack our construction of a Hopf algebra from an adjunctible retract in \S\ref{subsec:unpacked}, and in \S\ref{subsec:extradualizability} we explore what happens when the input, and hence the output, admits further dualizability. 
 Finally, \S\ref{subsec:tannakian} works out the details of how classical Tannakian reconstruction arises from our main Theorem~\ref{thm:mainHopftheorem}, and in particular shows that our construction can give rise to Hopf algebras without invertible antipodes.

\subsection{Hopf monads}\label{subsec:hopmonad}
In this section, we discuss how \define{opmonoidal monads} aka \define{bimonads}\footnote{These are called ``Hopf monads'' in \cite{MR1887157}, but, following \cite{MR2355605} and later references, we will reserve that term for bimonads with a Hopf-type condition. The term ``opmonoidal monad'' is more descriptive, but ``bimonad'' is shorter and established in the literature.}  as defined and studied in~\cite{MR1887157,MR2355605,MR2500058,MR2793022} fit neatly into our  framework and how our main construction immediately generalize to a construction of \define{Hopf monads}. 

Although \emph{op.\ cit.}\ mostly focus on bimonads in the $2$-category $\Cat_{(1,1)}$ of strict $1$-categories, they can be defined more generally inside an arbitrary monoidal $(\infty,2)$-category as follows.

\begin{definition} Let $\cD$ be a monoidal $(\infty,2)$-category\footnote{This definition parses for monoidal $(\infty,\infty)$-categories, but recovers ``lax bimonads'' in which certain compatibility $3$-morphisms in $\cD$ (coming from the $4$-cells in $\Mnd_+ \owedge \Mnd$) are not demanded to be an isomorphism; compare Remark~\ref{rem:laxbialgebra}.}, and write $\rB \cD$ for its delooped $(\infty,3)$-category. Recall from the discussion around~\eqref{eqn:oplaxalgmapcomponents} that the $(\infty,2)$-category of algebras and oplax algebra homomorphisms in $\cD$ is
\[\Alg^{\oplax}(\cD):= \Funoplax_*(\Mnd, \rB\cD).\]
A \define{bimonad} in $\cD$ is a monad in $\Alg^{\oplax}(\cD)$.
\end{definition}

Recall from \S\ref{subsec:owedge} that we write $(-)_+ : \Cat_{(\infty,\infty)} \to \Cat_{(\infty,\infty)}^*$ for the left adjoint to the forgetful map freely adjoining a basepoint. The following is immediately analogous to Corollary~\ref{cor:mndsquared}:

\begin{lemma} Let $\cC$ be a pointed $(\infty,3)$-category. Then, the equivalences 
\begin{multline}\label{eqn:mnd+mnd}
  \maps_*(\Mnd_+ \owedge \Mnd, \cC)  \simeq \maps_*( \Mnd_+ , \Funoplax_*(\Mnd, \cC))  \\  \simeq \maps(\Mnd, \Funoplax_*(\Mnd, \cC)) \simeq \maps(\Mnd, \Alg^{\oplax}(\Omega\cC)). 
\end{multline}
identify the left hand side with the space of bimonads in $\Omega \cC$. \qed
\end{lemma}

\begin{remark}
Let us unpack this equivalence \eqref{eqn:mnd+mnd} between $L_3(\Mnd_+ \owedge \Mnd)$ and (the delooping of) the walking bimonad following the notation of Remark~\ref{rem:notationforbimnd}. We encourage the reader to compare our ``two plus one''-dimensional diagrams to the three-dimensional diagrams from \cite{MR2500058}.
Recall from \eqref{eq:plussmash} that $\Mnd_+ \owedge \Mnd \simeq (\Mnd \otimes \Mnd) / (\Mnd \otimes *)$ is a quotient of $\Mnd \otimes \Mnd$, and we continue the convention of shading out the cells that are trivialized in this quotient. For example, if we denote by $A\in \Mnd$ the generating $1$-cell, then the 1-cell $A \otimes * \in \Mnd \otimes \Mnd$ trivializes in $\Mnd_+ \owedge \Mnd$ and so will be drawn in grey, but $* \otimes A$ does not trivialize and rather supplies the unique generating 1-cell:
$$
{\color{gray!50!white} A \otimes * \quad =\quad \begin{tikzpicture}[baseline=(baseline)]
  \path (0,0) coordinate (middle) +(0,-.125) coordinate (baseline)
  +(-.5,-1) coordinate (sw) 
  +(+.5,-1) coordinate (se) 
  +(-.5,+1) coordinate (nw) 
  +(+.5,+1) coordinate (ne) 
  ;
  \draw[string,gray!50!white, name path=A1] (sw) .. controls ++(0,+1) and ++(0,-1) .. (ne);
\end{tikzpicture}
\quad,}
\qquad\qquad
 * \otimes A \quad =\quad \begin{tikzpicture}[baseline=(baseline)]
  \path (0,0) coordinate (middle) +(0,-.125) coordinate (baseline)
  +(-.5,-1) coordinate (sw) 
  +(+.5,-1) coordinate (se) 
  +(-.5,+1) coordinate (nw) 
  +(+.5,+1) coordinate (ne) 
  ;
  \draw[string, name path=1A] (se) .. controls ++(0,+1) and ++(0,-1) .. (nw);
\end{tikzpicture}
\quad.
$$
Indeed, all of $* \otimes \Mnd$ survives: this 1-cell carries a monoid structure.
  The binary multiplication is the 2-cell $* \otimes m$:
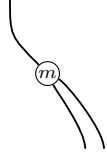
$$
\begin{tikzpicture}[baseline=(baseline),xscale=-1]
  \path (0,0) coordinate (middle) +(0,-.125) coordinate (baseline)
  +(-.75,-1) coordinate (sw1) 
  +(-.5,-1) coordinate (sw2) 
  +(+.5,-1) coordinate (se) 
  +(-.5,+1) coordinate (nw) 
  +(+.5,+1) coordinate (ne) 
  +(0,0) node[draw,circle,inner sep=0.5pt] (AA) {$\scriptstyle m$}
  ;
  \draw[string,name path=A1] (AA) .. controls ++(.5,.5) and ++(0,-.5) .. (ne);
  \draw[string,name path=A12] (sw2) .. controls ++(0,+.25) and ++(-.125,-.25) .. (AA);
  \draw[string,name path=A11] (sw1) .. controls ++(0,+.25) and ++(-.25,-.125) .. (AA);
\end{tikzpicture}
$$
Together with the unit $*\otimes u$, these corepresent (the binary multiplication and unit of) a monoid in $\Omega \cC$. 
In addition,  there is one other generating $2$-cell in $\Mnd_+ \owedge \Mnd$, namely $A \otimes A$:
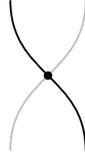
$$ 
\begin{tikzpicture}[baseline=(baseline)]
  \path (0,0) coordinate (middle) +(0,-.125) coordinate (baseline)
  +(-.5,-1) coordinate (sw) 
  +(+.5,-1) coordinate (se) 
  +(-.5,+1) coordinate (nw) 
  +(+.5,+1) coordinate (ne) 
  ;
  \draw[string, name path=1A] (se) .. controls ++(0,+1) and ++(0,-1) .. (nw);
  \draw[string,gray!50!white, name path=A1] (sw) .. controls ++(0,+1) and ++(0,-1) .. (ne);
  \path[name intersections={of=1A and A1}] (intersection-1) node[dot]{};
\end{tikzpicture}
$$
In the translation to bimonads, this 2-cell becomes the underlying endofunctor $T$ of the bimonad. Its oplax monoidality  is given by the 3-cells $A \otimes m $ (and a similar, simpler 3-cell involving units):
$$ \Delta_T \quad := \quad
\begin{tikzpicture}[baseline=(baseline),xscale=-1]
  \path (0,0) coordinate (middle) +(0,-.125) coordinate (baseline)
  +(-.75,-1) coordinate (sw1) 
  +(-.5,-1) coordinate (sw2) 
  +(+.5,-1) coordinate (se) 
  +(-.5,+1) coordinate (nw) 
  +(+.5,+1) coordinate (ne) 
  +(-.25,-.5) node[draw,circle,inner sep=0.5pt] (AA) {$\scriptstyle m$}
  ;
  \draw[string,gray!50!white, name path=1A] (se) .. controls ++(0,+1) and ++(0,-1) .. (nw);
  \draw[string, name path=A1] (AA) .. controls ++(.5,.5) and ++(0,-.5) .. (ne);
  \draw[string, name path=A12] (sw2) .. controls ++(0,+.25) and ++(-.125,-.25) .. (AA);
  \draw[string, name path=A11] (sw1) .. controls ++(0,+.25) and ++(-.25,-.125) .. (AA);
  \path[name intersections={of = 1A and A1}] (intersection-1) node[dot]{};
\end{tikzpicture} \quad \Rrightarrow \quad
\begin{tikzpicture}[baseline=(baseline),xscale=-1]
  \path (0,0) coordinate (middle) +(0,-.125) coordinate (baseline)
  +(-.75,-1) coordinate (sw1) 
  +(-.5,-1) coordinate (sw2) 
  +(+.5,-1) coordinate (se) 
  +(-.5,+1) coordinate (nw) 
  +(+.5,+1) coordinate (ne) 
  +(.25,.5) node[draw,circle,inner sep=0.5pt] (AA) {$\scriptstyle m$}
  ;
  \draw[string,gray!50!white, name path=1A] (se) .. controls ++(0,+1) and ++(0,-1) .. (nw);
  \draw[string] (AA) .. controls ++(.125,.125) and ++(0,-.125) .. (ne);
  \draw[string, name path=A12] (sw2) .. controls ++(0,+.5) and ++(-.25,-.5) .. (AA);
  \draw[string, name path=A11] (sw1) .. controls ++(0,+.5) and ++(-.5,-.25) .. (AA);
  \path[name intersections={of = 1A and A11}] (intersection-1) node[dot]{};
  \path[name intersections={of = 1A and A12}] (intersection-1) node[dot]{};
\end{tikzpicture},$$
whereas its ``monad-ality'' is given by $m \otimes A$ (and a similar, simpler 3-cell involving units):
\[ \mu_T \quad := \quad
\begin{tikzpicture}[baseline=(baseline)]
  \path (0,0) coordinate (middle) +(0,-.125) coordinate (baseline)
  +(-.75,-1) coordinate (sw1) 
  +(-.5,-1) coordinate (sw2) 
  +(+.5,-1) coordinate (se) 
  +(-.5,+1) coordinate (nw) 
  +(+.5,+1) coordinate (ne) 
  +(.25,.5) node[draw,circle,inner sep=0.5pt,gray!50!white] (AA) {$\scriptstyle m$}
  ;
  \draw[string, name path=1A] (se) .. controls ++(0,+1) and ++(0,-1) .. (nw);
  \draw[string,gray!50!white] (AA) .. controls ++(.125,.125) and ++(0,-.125) .. (ne);
  \draw[string,gray!50!white, name path=A12] (sw2) .. controls ++(0,+.5) and ++(-.25,-.5) .. (AA);
  \draw[string,gray!50!white, name path=A11] (sw1) .. controls ++(0,+.5) and ++(-.5,-.25) .. (AA);
  \path[name intersections={of = 1A and A11}] (intersection-1) node[dot]{};
  \path[name intersections={of = 1A and A12}] (intersection-1) node[dot]{};
\end{tikzpicture}
\quad \Rrightarrow \quad
\begin{tikzpicture}[baseline=(baseline)]
  \path (0,0) coordinate (middle) +(0,-.125) coordinate (baseline)
  +(-.75,-1) coordinate (sw1) 
  +(-.5,-1) coordinate (sw2) 
  +(+.5,-1) coordinate (se) 
  +(-.5,+1) coordinate (nw) 
  +(+.5,+1) coordinate (ne) 
  +(-.25,-.5) node[draw,circle,inner sep=0.5pt,gray!50!white] (AA) {$\scriptstyle m$}
  ;
  \draw[string, name path=1A] (se) .. controls ++(0,+1) and ++(0,-1) .. (nw);
  \draw[string,gray!50!white, name path=A1] (AA) .. controls ++(.5,.5) and ++(0,-.5) .. (ne);
  \draw[string,gray!50!white, name path=A12] (sw2) .. controls ++(0,+.25) and ++(-.125,-.25) .. (AA);
  \draw[string,gray!50!white, name path=A11] (sw1) .. controls ++(0,+.25) and ++(-.25,-.125) .. (AA);
  \path[name intersections={of = 1A and A1}] (intersection-1) node[dot]{};
\end{tikzpicture}\]
\end{remark}

Which bimonads deserve to be called ``Hopf''? This is answered in \cite{MR2355605,MR2500058} for bimonads living on monoidal categories with duals; the reason those papers ask for duals is because of the relationship between antipodes and dual modules. The existence-of-duals requirement is dropped in \cite{MR2793022}, which reexpresses Hopf-ness in terms of the invertibility of a shear map (called a ``fusion map'' therein). Specifically, \S2.6 of that paper introduces two shear maps, denoted $H^l$ and $H^r$ therein. They correspond, respectively, to the shear maps denoted $\sh^\nearrow$ and $\sh^\searrow$ in 
Notation~\ref{not:shearmaps}.

The first main result of \cite{MR2793022} is that if the underlying monoidal category of a bimonad admits left resp.\ right duals, then $H^l$ resp.\ $H^r$ is invertible if and only if the bimonad admits a left resp.\ right antipode in the sense of \cite{MR2355605}.
 The perspective of \cite{MR2355605,MR2793022} is that a bialgebra is properly considered ``Hopf'' only when 
 both of the shear maps  are invertible; in other words, those authors use the term ``Hopf algebra'' for what we would call ``Hopf algebra with invertible antipode.'' Continuing this perspective, \cite{MR2355605,MR2793022} define a bimonad to be ``Hopf'' only when both $H^l$ and $H^r$ are invertible.
 We reject this perspective, since we believe that Hopf algebras with noninvertible antipode are of great interest, and instead declare:
\begin{definition}
  A bimonad $T$ is a \define{Hopf monad} when the shear map 
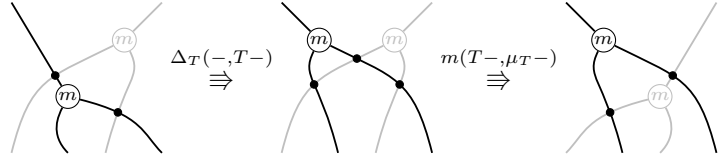
\[ \begin{tikzpicture}[baseline=(baseline)]
  \path (0,0) coordinate (middle) +(0,-.125) coordinate (baseline)
  +(-1,+1) coordinate (nw) 
  +(+,+1) coordinate (ne) 
  +(-1,-1) coordinate (a) 
  +(-.25,-1) coordinate (b) 
  +(+.25,-1) coordinate (c) 
  +(+1,-1) coordinate (d) 
  ;
  \path (middle)
  +(-.25,-.25) node[draw,circle,inner sep=0.5pt] (mw) {$\scriptstyle m$}
  +(+.5,+.5) node[draw,circle,inner sep=0.5pt,gray!50!white] (me) {$\scriptstyle m$}
  ;
  \draw[string, name path=X] (mw) -- (nw);
  \draw[string,gray!50!white] (me) -- (ne);
  \draw[string,gray!50!white, name path=Y] (a) .. controls +(.25,1) and +(-1,-.5) .. (me);
  \draw[string,gray!50!white, name path=Z] (c) .. controls +(.25,1) and +(.25,-.5) .. (me);
  \draw[string] (b) .. controls +(-.25,.25) and +(-.125,-.25) .. (mw);
  \draw[string, name path=W] (d) .. controls +(-.25,.25) and +(+1,-.25) .. (mw);
    \path[name intersections={of=X and Y}] (intersection-1) node[dot]{};
    \path[name intersections={of=Z and W}] (intersection-1) node[dot]{};
\end{tikzpicture}
\overset{\Delta_T(-,T-)}\Rrightarrow
\begin{tikzpicture}[baseline=(baseline)]
  \path (0,0) coordinate (middle) +(0,-.125) coordinate (baseline)
  +(-1,+1) coordinate (nw) 
  +(+,+1) coordinate (ne) 
  +(-1,-1) coordinate (a) 
  +(-.25,-1) coordinate (b) 
  +(+.25,-1) coordinate (c) 
  +(+1,-1) coordinate (d) 
  ;
  \path (middle)
  +(-.5,+.5) node[draw,circle,inner sep=0.5pt] (mw) {$\scriptstyle m$}
  +(+.5,+.5) node[draw,circle,inner sep=0.5pt,gray!50!white] (me) {$\scriptstyle m$}
  ;
  \draw[string] (mw) -- (nw);
  \draw[string,gray!50!white] (me) -- (ne);
  \draw[string,gray!50!white, name path=X] (a) .. controls +(.25,1) and +(-1,-.5) .. (me);
  \draw[string,gray!50!white, name path=Y] (c) .. controls +(.25,1) and +(.25,-.5) .. (me);
  \draw[string, name path=Z] (b) .. controls +(-.25,1) and +(-.25,-.5) .. (mw);
  \draw[string, name path=W] (d) .. controls +(-.25,1) and +(+1,-.5) .. (mw);
    \path[name intersections={of=X and W}] (intersection-1) node[dot]{};
    \path[name intersections={of=Y and W}] (intersection-1) node[dot]{};
    \path[name intersections={of=X and Z}] (intersection-1) node[dot]{};
\end{tikzpicture}
\overset{m(T-, \mu_T-)}\Rrightarrow
\begin{tikzpicture}[baseline=(baseline)]
  \path (0,0) coordinate (middle) +(0,-.125) coordinate (baseline)
  +(-1,+1) coordinate (nw) 
  +(+,+1) coordinate (ne) 
  +(-1,-1) coordinate (a) 
  +(-.25,-1) coordinate (b) 
  +(+.25,-1) coordinate (c) 
  +(+1,-1) coordinate (d) 
  ;
  \path (middle)
  +(-.5,+.5) node[draw,circle,inner sep=0.5pt] (mw) {$\scriptstyle m$}
  +(+.25,-.25) node[draw,circle,inner sep=0.5pt,gray!50!white] (me) {$\scriptstyle m$}
  ;
  \draw[string] (mw) -- (nw);
  \draw[string,gray!50!white, name path=W] (me) -- (ne);
  \draw[string,gray!50!white, name path=Y] (a) .. controls +(.25,.25) and +(-1,-.25) .. (me);
  \draw[string,gray!50!white] (c) .. controls +(.25,.25) and +(.125,-.25) .. (me);
  \draw[string, name path=X] (b) .. controls +(-.25,1) and +(-.25,-.5) .. (mw);
  \draw[string, name path=Z] (d) .. controls +(-.25,1) and +(+1,-.5) .. (mw);
    \path[name intersections={of=X and Y}] (intersection-1) node[dot]{};
    \path[name intersections={of=Z and W}] (intersection-1) node[dot]{};
\end{tikzpicture}
 \]
is invertible.
\end{definition}
\begin{remark}
As indicated by the presence of the shaded-out edges, we have selected this shear map as it is the image in $\Mnd_+ \owedge \Mnd$ of the universal shear map from Definition~\ref{def:strictshear}.
Comparing with \cite{MR2793022}, this recovers precisely the map called $H^l$ therein: our ``Hopf monads'' are their ``left Hopf monads.'' The map called $H^r$ therein does not (as far as we can tell) lift to $\Mnd \otimes \Mnd$. 
\end{remark}

Our Theorem~\ref{thm:invertibleshear} applies already to the map $\Mnd \otimes \Mnd \to \Adj \otimes \Adj$, and so we immediately conclude:
\begin{corollary}\label{cor:Hopfmonad}
  Let $\cC$ be a pointed $(\infty,3)$-category.  The restriction functor along $\Mnd \to \Adj$ 
  $$H_+ : \Funoplax_*(\Adj_+ \owedge \Adj, \cC) \to \Funoplax_*(\Mnd_+ \owedge \Mnd, \cC) \simeq \cat{BiMnd}(\Omega\cC)$$
  factors through the full sub-$(\infty,2)$-category on the Hopf monads. \qed
\end{corollary}

The reader can easily work out a presentation of the 3-localization of $\Adj_+ \owedge \Adj \simeq (\Adj \otimes \Adj) / (\Adj \otimes *)$ using Corollary~\ref{cor:L3AdjAdj}.

\begin{remark}
Let $\cD$ be a monoidal $(\infty,2)$-category. A \define{comonoidal adjunction} in $\cD$ is an adjunction in $\Alg^{\oplax}(\cD)$ --- in other words, (the delooping of) \define{the walking comonoidal adjunction} is $L_3(\Adj_+ \owedge \Mnd)$. Restriction along the map $\Mnd_+ \owedge \Mnd \to \Adj_+ \owedge \Mnd$ produces a bimonad from any comonoidal adjunction.

The paper \cite{MR2793022} defines certain shear maps in  $L_3(\Adj_+ \owedge \Mnd)$, denoted $\mathbb{H}^l$ and $\mathbb{H}^r$ therein, and calls a
comonoidal adjuction a \define{left} resp.\ \define{right Hopf adjunction} when $\mathbb{H}^l$ resp.\  $\mathbb{H}^r$ is invertible.

Of the two shear maps from \cite{MR2793022}, $\mathbb{H}^l$ lifts to $\Adj \otimes \Mnd$ (and $\mathbb{H}^r$ does not, as far as we can tell). Namely, extending the notation from the proof of Theorem~\ref{thm:invertibleshear}, the lift is:
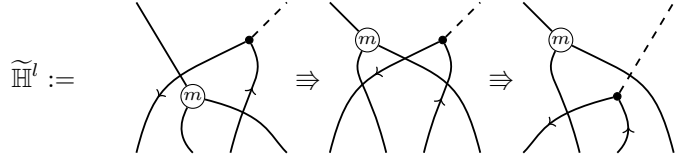
\[ 
\widetilde{\mathbb{H}}^l := 
\qquad 
\begin{tikzpicture}[baseline=(baseline)]
  \path (0,0) coordinate (middle) +(0,-.125) coordinate (baseline)
  +(-1,+1) coordinate (nw) 
  +(+,+1) coordinate (ne) 
  +(-1,-1) coordinate (a) 
  +(-.25,-1) coordinate (b) 
  +(+.25,-1) coordinate (c) 
  +(+1,-1) coordinate (d) 
  ;
  \path (middle)
  +(-.25,-.25) node[draw,circle,inner sep=0.5pt] (mw) {$\scriptstyle m$}
  +(+.5,+.5) node[dot] (me) {}
  ;
  \draw[string] (mw) -- (nw);
  \draw[string,dashed] (me) -- (ne);
  \draw[string, arrow data={.4}{<}] (a) .. controls +(.25,1) and +(-1,-.5) .. (me);
  \draw[string, arrow data={.6}{>}] (c) .. controls +(.25,1) and +(.25,-.5) .. (me);
  \draw[string] (b) .. controls +(-.25,.25) and +(-.125,-.25) .. (mw);
  \draw[string] (d) .. controls +(-.25,.25) and +(+1,-.25) .. (mw);
\end{tikzpicture}
\Rrightarrow
\begin{tikzpicture}[baseline=(baseline)]
  \path (0,0) coordinate (middle) +(0,-.125) coordinate (baseline)
  +(-1,+1) coordinate (nw) 
  +(+,+1) coordinate (ne) 
  +(-1,-1) coordinate (a) 
  +(-.25,-1) coordinate (b) 
  +(+.25,-1) coordinate (c) 
  +(+1,-1) coordinate (d) 
  ;
  \path (middle)
  +(-.5,+.5) node[draw,circle,inner sep=0.5pt] (mw) {$\scriptstyle m$}
  +(+.5,+.5) node[dot] (me) {}
  ;
  \draw[string] (mw) -- (nw);
  \draw[string,dashed] (me) -- (ne);
  \draw[string, arrow data={.6}{<}] (a) .. controls +(.25,1) and +(-1,-.5) .. (me);
  \draw[string, arrow data={.5}{>}] (c) .. controls +(.25,1) and +(.25,-.5) .. (me);
  \draw[string] (b) .. controls +(-.25,1) and +(-.25,-.5) .. (mw);
  \draw[string] (d) .. controls +(-.25,1) and +(+1,-.5) .. (mw);
\end{tikzpicture}
\Rrightarrow
\begin{tikzpicture}[baseline=(baseline)]
  \path (0,0) coordinate (middle) +(0,-.125) coordinate (baseline)
  +(-1,+1) coordinate (nw) 
  +(+,+1) coordinate (ne) 
  +(-1,-1) coordinate (a) 
  +(-.25,-1) coordinate (b) 
  +(+.25,-1) coordinate (c) 
  +(+1,-1) coordinate (d) 
  ;
  \path (middle)
  +(-.5,+.5) node[draw,circle,inner sep=0.5pt] (mw) {$\scriptstyle m$}
  +(+.25,-.25) node[dot] (me) {}
  ;
  \draw[string] (mw) -- (nw);
  \draw[string,dashed] (me) -- (ne);
  \draw[string, arrow data={.4}{<}] (a) .. controls +(.25,.25) and +(-1,-.25) .. (me);
  \draw[string, arrow data={.5}{>}] (c) .. controls +(.25,.25) and +(.125,-.25) .. (me);
  \draw[string] (b) .. controls +(-.25,1) and +(-.25,-.5) .. (mw);
  \draw[string] (d) .. controls +(-.25,1) and +(+1,-.5) .. (mw);
\end{tikzpicture}
\]
Take the universal shear map $\widetilde{\sh} \in \Mnd \otimes \Mnd$ and push it forward to $L_3(\Adj \otimes \Mnd)$: one finds a certain whiskering of $\widetilde{\mathbb{H}}^l$ composed with the oplax monoidality for the left adjoint. It is a standard exercise that the left adjoint in a comonoidal adjunction is not merely oplax monoidal but in fact strong monoidal, and so this composition does not affect questions of invertibility. This immediately implies (a monoidal $(\infty,2)$-category generalization of)  Proposition~2.14 of \cite{MR2793022}: the bimonad produced from a Hopf adjunction is a Hopf monad.

On the other hand, our proof of  Theorem~\ref{thm:invertibleshear} showed invertibility of a further ``unwhiskering'' of this shear map (after pushing it along the map to $\Adj \otimes \Adj$). Thus we can factor Corollary~\ref{cor:Hopfmonad}: the restriction functor $\Funoplax_*(\Adj_+ \owedge \Adj, \cC) \to \Funoplax_*(\Adj_+ \owedge \Mnd, \cC)$ factors through the full subcategory of Hopf adjunctions.
\end{remark}

\subsection{The construction unpacked} \label{subsec:unpacked} 

Our main construction in Section~\ref{sec:mainconstruction} takes an adjunctible retract $(X,h, k, \beta)$ in a pointed $(\infty,3)$-category $(\cC, 1_{\cC})$ as in Definition~\ref{def:adjunctibleretract} and produces from it  a Hopf algebra in $\Omega^2 \cC$. We will now explicitly unpack the underlying algebra and coalgebra structure of this Hopf algebra. In order to write fewer $^{-1}$s, in the remainder of the paper we  rename  \begin{equation}\label{eq:renaming} \alpha:= \beta^{-1}, \hspace{0.5cm} f:= h , \hspace{0.5cm} g:= k,\end{equation} matching the conventions of our main Theorem~\ref{thm:mainHopftheorem}.

\begin{lemma}\label{lem:unpacked}
Let $(\cC, 1_{\cC})$ be a pointed $(\infty,3)$-category and given an adjunctible retract therein, i.e.\ a diagram
 \[\begin{tikzcd}[sep=3em]
1_\cC \arrow[r, equal] \arrow[d, "f"'] &\arrow[d, equal] 1_\cC \\
X \arrow[r, "g"'] 
\arrow[ur, Leftarrow, "\sim" sloped, "\alpha" '] 
& 1_\cC
\end{tikzcd}\]
such that $f$ admits a right adjoint $f^R$, $g$ admits a left adjoint $g^L$, and $\alpha^\rmate$ admits a left adjoint adjoint $(\alpha^\rmate)^L$ or equivalently $\alpha^\lmate$ admits a right adjoint $(\alpha^\lmate)^R$ (note the relabelling~\eqref{eq:renaming} from Definition~\ref{def:adjunctibleretract}).

Then, the underlying object of the Hopf algebra $H(C, f,g,\alpha)$  constructed in Corollary~\ref{cor:part2ofmaintheorem} is given by the following composite: 
\begin{equation}\label{eqn:unpackedpasting}
\begin{tikzcd}[sep=5em]
1_\cC\arrow[r, equal] \arrow[d, equal]& 1_\cC \arrow[r, equal] \arrow[d, "f" description]& 1_\cC \arrow[d, equal] \\
1_\cC  \arrow[ur, Rightarrow, shorten <=7pt, shorten >=7pt, "\alpha^{\lmate}" description]\arrow[r, "g^L" description] \arrow[d, equal] & X  \arrow[ur, Rightarrow, shorten <=7pt, shorten >=7pt,
"\sim" sloped, "\alpha^{-1}" description] \arrow[r, "g" description] \arrow[d, "f^R"description] & 1_\cC\arrow[d, equal]\\
1_\cC \arrow[ur, Rightarrow, shorten <=5pt, shorten >=5pt, "\tiny\substack{((\alpha^{\rmate})^L)^{\lmate} \\ \rotatebox[origin=c]{270}{\ensuremath\simeq} \\ ((\alpha^{\lmate})^R)^{\rmate}}" description]  \arrow[r, equal] & 1_\cC \arrow[ur, Rightarrow, shorten <=7pt, shorten >=7pt, "\alpha^{\rmate}" description] \arrow[r, equal] & 1_\cC
\end{tikzcd}
\end{equation}
\end{lemma}
\begin{proof}
By definition, the bialgebra $H(X,f,g,\alpha) \in \Omega^2 \cC$ is given by restricting along  $\Mnd \owedge \Mnd \to \Adj \owedge \Adj$, its underlying object is the further restriction along $\Sone \owedge \Sone \to \Mnd \owedge \Mnd$. Tracing through the definitions we find the above composite. \end{proof}

\begin{remark} \label{rem:simplesquare}As in condition [ii] of Corollary~\ref{cor:smashretract}, invertibility of $\alpha$ implies that   $\alpha^{\rmate}$ admits a left adjoint (equivalently $\alpha^{\lmate}$ admits a right adjoint) if and only if  the whiskering of the counit $g \ev_f :g ff^R \To g$ admits a left adjoint (equivalently $\ev_g f: g^Lg f \To f$ admits a right adjoint). In particular,  
\[( \alpha^{\rmate})^L \simeq (\alpha^{-1}  f^R) \circ (g\, \ev_f)^{L}  \hspace{1cm} (\alpha^{\lmate})^R \simeq (g^L   \alpha^{-1}) \circ (\ev_g \,  f)^R.
\]
Using functoriality of $(-)^{\lmate}$ and $(-)^{\rmate}$, respectively, the composite \eqref{eqn:unpackedpasting} defining $H(X,f,g,\alpha)$ is therefore equivalent to both the composites 
\begin{equation} \label{eq:simplification1} \id_{1_{\cC}} \To[\alpha] g f \To[(g \, \ev_f)^L f] gff^R f \To[g \,\ev_f\, f] gf \To[\alpha^{-1}] \id_{1_{\cC}}
\end{equation}
and
\begin{equation}
\label{eq:simplification2}
\id_{1_{\cC}} \To[\alpha] gf \To[g (\ev_g \,f)^R] gg^L gf \To[g\, \ev_g\, f] gf \To[\alpha^{-1}] \id_{1_{\cC}}.
\end{equation}
The algebra structure on~\eqref{eqn:unpackedpasting} is induced from writing it via~\eqref{eq:simplification1} as a composite $\gamma \circ \gamma^L$ (where $\gamma = \alpha^{-1} \circ (g\,\ev_f \,f) : gff^Rf \To \id_{1_{\cC}}$), while the coalgebra structure is induced from writing it via~\eqref{eq:simplification2} as a composite $\delta \circ \delta^R$ (where $\delta = \alpha^{-1} \circ (g\,\ev_g \,f): gg^Lgf \To \id_{1_{\cC}}$). 

From this perspective, any $2$-morphism $\mu: \id_{1_{\cC}} \To gff^Rf$ defines a canonical (and coherent) left module $\gamma \circ \mu$ while any $2$-morphism $\nu:\id_{1_{\cC}} \To gg^Lg f $ defines a canonical left comodule $\delta \circ \nu$. In particular, $\mu:= (gf \coev_f) \circ \alpha$ defines a left module structure on $\gamma \circ \mu \simeq \id_{\id_{1_{\cC}}}$, i.e.\ an algebra homomorphism $\gamma  \circ \gamma^L \To \id_{1_{\cC}}$ --- this is precisely the counit of the Hopf algebra. Similarly, $\nu:=(\coev_g gf) \circ \alpha $ defines a left comodule structure on $\delta \circ \nu \simeq \id_{\id_{1_{\cC}}}$, i.e.\ a coalgebra homomorphism $\id_{1_{\cC}} \To \delta \circ \delta^R$ --- this is precisely the unit of the Hopf algebra.
\end{remark}

We find it instructive to give an equivalent description of the (co)algebra structure directly in terms of the more symmetric composite~\eqref{eqn:unpackedpasting}. 
We will henceforth abbreviate $\alpha^\sharp := ((\alpha^{\rmate})^L)^{\lmate} \simeq ((\alpha^{\lmate})^R)^{\rmate}$.
String diagramatically (and recalling our convention that $1$-morphisms in string diagrams compose from right to left, and $2$-morphisms compose from bottom to top), we can equally write $H(X,f,g,\alpha)$ as:
\begin{equation}\label{eqn:unpackedstring}
H(X,f,g,\alpha) \quad \simeq \quad 
\begin{tikzpicture}[scale=1.5,baseline=(M.base)]
\path[fill=gray,opacity=0.25] (-1,0) -- (0,-1) -- (+1,0) -- (0,+1) -- cycle;
\path 
(-1,0) node[draw, ellipse, inner sep=0.5pt,fill=white] (W) {$\scriptstyle \alpha^\rmate$}
(+1,0) node[draw, ellipse, inner sep=0.5pt,fill=white] (E) {$\scriptstyle \alpha^\lmate$}
(0,+1) node[draw, ellipse, inner sep=0.5pt,fill=white] (N) {$\scriptstyle \alpha^{-1}$}
(0,-1) node[draw, ellipse, inner sep=0.5pt,fill=white] (S) {$\scriptstyle \alpha^\sharp$}
(0,0) node (M) {$X$}
;
\draw[string, black!50!red] (S) -- node[auto] {$\scriptstyle f^R$} (W);
\draw[string, black!50!blue] (S) -- node[auto,swap] {$\scriptstyle g^L$} (E);
\draw[string, black!50!blue] (W) -- node[auto] {$\scriptstyle g$} (N);
\draw[string, black!50!red] (E) -- node[auto,swap] {$\scriptstyle f$} (N);
\end{tikzpicture}
\end{equation} 

The algebra and coalgebra structures on this object are described abstractly in Corollary~\ref{cor:part2ofmaintheorem} and Remark~\ref{rem:notationforbimnd}.
Concretely, the algebra structure arises as follows. Since $\alpha^\sharp = ((\alpha^{\rmate})^L)^{\lmate}$, unpacking the mates in the south and east corners gives:
\begin{equation}\label{eqn:Hasanalgebra}
H(X,f,g,\alpha) \quad \simeq \quad 
\begin{tikzpicture}[scale=1.5,baseline=(M.base)]
\path 
(-1,0) coordinate (W)
(+1,0) coordinate (E)
(0,+1) coordinate (N)
(0,-1) coordinate (S)
(0,0) node (M) {$X$}
;
\path[fill=gray,opacity=0.25] (W) -- (S) .. controls +(1,-1) and +(-.5,.5) .. (E) .. controls +(0,.5) and +(.5,-.5) .. (N) -- (W);
\draw[string, black!50!red] (S) -- node[auto] {$\scriptstyle f^R$} (W);
\draw[string, black!50!blue] (S) .. controls +(1,-1) and +(-.5,.5) .. node[auto,swap, pos=0.6] {$\scriptstyle g^L$} (E);
\draw[string, black!50!blue] (W) -- node[auto] {$\scriptstyle g$} (N);
\draw[string, black!50!red] (E) .. controls +(0,.5) and +(.5,-.5) .. node[auto,swap] {$\scriptstyle f$} (N);
\path 
(W) node[draw, ellipse, inner sep=0.5pt,fill=white] {$\scriptstyle \alpha^\rmate$}
(E) node[draw, ellipse, inner sep=0.5pt,fill=white] {$\scriptstyle \alpha$}
(N) node[draw, ellipse, inner sep=0.5pt,fill=white] {$\scriptstyle \alpha^{-1}$}
(S) node[draw, ellipse, inner sep=0.5pt,fill=white] {$\scriptstyle \alpha^{\rmate L}$}
;
\end{tikzpicture}
\quad \simeq \quad 
\begin{tikzpicture}[scale=1.5,baseline=(M.base)]
\path[fill=gray,opacity=0.25] (-.5,.5) -- (.5,+1) -- (+.5,-1) -- (-.5,-.5) -- cycle;
\path 
(-.5,.5) node[draw, ellipse, inner sep=0.5pt,fill=white] (W) {$\scriptstyle \alpha^\rmate$}
(+.5,-1) node[draw, ellipse, inner sep=0.5pt,fill=white] (E) {$\scriptstyle \alpha$}
(.5,+1) node[draw, ellipse, inner sep=0.5pt,fill=white] (N) {$\scriptstyle \alpha^{-1}$}
(-.5,-.5) node[draw, ellipse, inner sep=0.5pt,fill=white] (S) {$\scriptstyle \alpha^{\rmate L}$}
(0,0) node (M) {$X$}
;
\draw[string, black!50!red] (S) -- node[auto] {$\scriptstyle f^R$} (W);
\draw[string, black!50!blue] (S) -- node[auto,swap] {$\scriptstyle g$} (E);
\draw[string, black!50!blue] (W) -- node[auto] {$\scriptstyle g$} (N);
\draw[string, black!50!red] (E) -- node[auto,swap] {$\scriptstyle f$} (N);
\end{tikzpicture}\vspace{-24pt}
\end{equation}
But this is manifestly a (vertical) composition of the form $\beta \circ \beta^L$, and hence an algebra. On the other hand, writing $\alpha^\sharp \simeq ((\alpha^{\lmate})^R)^{\rmate}$, we can instead unpack the south and west mates to give: 
\begin{equation}\label{eqn:Hasacoalgebra}
H(X,f,g,\alpha) \quad \simeq \quad 
\begin{tikzpicture}[scale=1.5,baseline=(M.base)]
\path 
(-1,0) coordinate (W)
(+1,0) coordinate (E)
(0,+1) coordinate (N)
(0,-1) coordinate (S)
(0,0) node (M) {$X$}
;
\path[fill=gray,opacity=0.25] (S) .. controls +(-1,-1) and +(.5,.5) .. (W) .. controls +(0,.5) and +(-.5,-.5) .. (N) -- (E) -- (S);
\draw[string, black!50!red] (S) .. controls +(-1,-1) and +(.5,.5) .. node[auto, pos=0.6] {$\scriptstyle f^R$} (W);
\draw[string, black!50!blue] (S) -- node[auto,swap] {$\scriptstyle g^L$} (E);
\draw[string, black!50!blue] (W) .. controls +(0,.5) and +(-.5,-.5) .. node[auto] {$\scriptstyle g$} (N);
\draw[string, black!50!red] (E) -- node[auto,swap] {$\scriptstyle f$} (N);
\path 
(W) node[draw, ellipse, inner sep=0.5pt,fill=white] {$\scriptstyle \alpha$}
(E) node[draw, ellipse, inner sep=0.5pt,fill=white] {$\scriptstyle \alpha^\lmate$}
(N) node[draw, ellipse, inner sep=0.5pt,fill=white] {$\scriptstyle \alpha^{-1}$}
(S) node[draw, ellipse, inner sep=0.5pt,fill=white] {$\scriptstyle \alpha^{\lmate R}$}
;
\end{tikzpicture}
\quad \simeq \quad 
\begin{tikzpicture}[scale=1.5,baseline=(M.base)]
\path[fill=gray,opacity=0.25] (-.5,-1) -- (-.5,+1) -- (+.5,.5) -- (.5,-.5) -- cycle;
\path 
(-.5,-1) node[draw, ellipse, inner sep=0.5pt,fill=white] (W) {$\scriptstyle \alpha$}
(+.5,.5) node[draw, ellipse, inner sep=0.5pt,fill=white] (E) {$\scriptstyle \alpha^\lmate$}
(-.5,+1) node[draw, ellipse, inner sep=0.5pt,fill=white] (N) {$\scriptstyle \alpha^{-1}$}
(.5,-.5) node[draw, ellipse, inner sep=0.5pt,fill=white] (S) {$\scriptstyle \alpha^{\lmate R}$}
(0,0) node (M) {$X$}
;
\draw[string, black!50!red] (S) -- node[auto] {$\scriptstyle f$} (W);
\draw[string, black!50!blue] (S) -- node[auto,swap] {$\scriptstyle g^L$} (E);
\draw[string, black!50!blue] (W) -- node[auto] {$\scriptstyle g$} (N);
\draw[string, black!50!red] (E) -- node[auto,swap] {$\scriptstyle f$} (N);
\end{tikzpicture}
\vspace{-24pt}\end{equation}
This is a composition of the form $\gamma \circ \gamma^R$ and hence a coalgebra.

In other words, the multiplication, unit, counit, and comultiplication can be pictured respectively as the following 3-dimensional diagrams:
$$
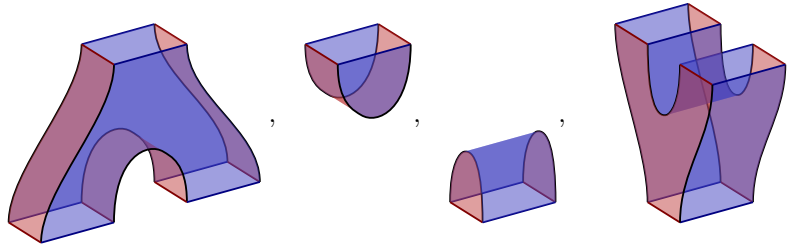
\hspace{-1in}
\begin{tikzpicture}[scale=0.7]
  \path 
  (0,0,0)
  +(-1,0,0)  coordinate (Wm)
  +(+1,0,0)  coordinate (Em)
  +(0,0,+1)  coordinate (Sm)
  +(0,0,-1)  coordinate (Nm)
  ++(+.25,-1.85,0)
  +(-.5,0,-.5) coordinate (pi)
  +(+.5,0,+.5) coordinate (pii)
  (-1,-3,+1)
  +(-1,0,0)  coordinate (Wm1) 
  +(+1,0,0)  coordinate (Em1) 
  +(0,0,+1)  coordinate (Sm1) 
  +(0,0,-1)  coordinate (Nm1) 
  (+1,-3,-1)
  +(-1,0,0)  coordinate (Wm2) 
  +(+1,0,0)  coordinate (Em2) 
  +(0,0,+1)  coordinate (Sm2) 
  +(0,0,-1)  coordinate (Nm2) 
  ;
  \draw[string, black!50!red] (Em) -- (Nm);
  \draw[string, black!50!blue] (Nm) -- (Wm);
  \draw[string, black!50!red] (Em1) -- (Nm1);
  \draw[string, black!50!blue] (Nm1) -- (Wm1);
  \draw[string, black!50!red] (Em2) -- (Nm2);
  \draw[string, black!50!blue] (Nm2) -- (Wm2);
  \draw[string] (Nm2) .. controls +(0,+1,0) and +(0,-1,0) .. (Nm);
  \draw[string] (Nm1) .. controls +(0,+1.25,0) and +(-.5,0,-.5) .. (pi) .. controls +(+.2,0,+.2) and +(0,+.75,0)  .. (Wm2);
  \path[fill=gray!50!blue,opacity=0.5] (Wm1) .. controls +(0,+1,0) and +(0,-1,0) .. (Wm) -- (Nm) .. controls +(0,-1,0) and +(0,+1,0) .. (Nm2) -- (Wm2) 
  .. controls +(0,+.75,0) and +(+.2,0,+.2) .. (pi)  .. controls +(-.5,0,-.5) and +(0,+1.25,0)  ..  (Nm1) -- (Wm1);
  \path[fill=gray!50!red,opacity=0.5] (Nm2) .. controls +(0,+1,0) and +(0,-1,0) .. (Nm) -- (Em) .. controls +(0,-1,0) and +(0,+1,0) .. (Em2) -- (Nm2);
  \draw[string] (Wm1) .. controls +(0,+1,0) and +(0,-1,0) .. (Wm);
  \draw[string] (Em2) .. controls +(0,+1,0) and +(0,-1,0) .. (Em);
  \path[fill=gray!50!red,opacity=0.5] (Wm1) .. controls +(0,+1,0) and +(0,-1,0) .. (Wm) -- (Sm) .. controls +(0,-1,0) and +(0,+1,0) .. (Sm1) -- (Wm1);
  \path[fill=gray!50!red,opacity=0.5] (Nm1) .. controls +(0,+1.25,0) and +(-.5,0,-.5) .. (pi) -- (pii) .. controls +(-.5,0,-.5) and +(0,+1.25,0)  ..  (Em1) -- (Nm1);
  \path[fill=gray!50!red,opacity=0.5] (Sm2) .. controls +(0,+.75,0) and +(+.2,0,+.2) .. (pii) -- (pi) .. controls +(+.2,0,+.2) and +(0,+.75,0)  .. (Wm2) -- (Sm2);
  \path[fill=gray!50!blue,opacity=0.5] (Sm1) .. controls +(0,+1,0) and +(0,-1,0) .. (Sm) -- (Em) .. controls +(0,-1,0) and +(0,+1,0) .. (Em2) -- (Sm2) .. controls +(0,+.75,0) and +(+.2,0,+.2) .. (pii)  .. controls +(-.5,0,-.5) and +(0,+1.25,0)  ..  (Em1) -- (Sm1);
  \draw[string] (Sm1) .. controls +(0,+1,0) and +(0,-1,0) .. (Sm);
  \draw[string] (Em1) .. controls +(0,+1.25,0) and +(-.5,0,-.5) .. (pii) .. controls +(+.2,0,+.2) and +(0,+.75,0)  .. (Sm2);
  \draw[string, black!50!red] (Wm1) -- (Sm1);
  \draw[string, black!50!blue] (Sm1) -- (Em1);
  \draw[string, black!50!red] (Wm2) -- (Sm2);
  \draw[string, black!50!blue] (Sm2) -- (Em2);
  \draw[string, black!50!red] (Wm) -- (Sm);
  \draw[string, black!50!blue] (Sm) -- (Em);
  \path (2.625,-1.5) node {,};
  \path 
  (4.25,0,0)
  +(-1,0,0)  coordinate (Wu)
  +(+1,0,0)  coordinate (Eu)
  +(0,0,+1)  coordinate (Su)
  +(0,0,-1)  coordinate (Nu)
  ++(-.25,-1.15,0)
  +(-.5,0,-.5) coordinate (boti)
  +(+.5,0,+.5) coordinate (botii)
  ;
  \draw[string, black!50!red]  (Eu) -- (Nu);
  \draw[string, black!50!blue] (Nu) -- (Wu);
  \draw[string] (Wu) .. controls +(0,-.75,0) and +(-.2,0,-.2) .. (boti) .. controls +(+.5,0,+.5) and +(0,-1.25,0) .. (Nu);
  \path[fill=gray!50!blue,opacity=0.5] (Wu) .. controls +(0,-.75,0) and +(-.2,0,-.2) .. (boti) .. controls +(+.5,0,+.5) and +(0,-1.25,0) .. (Nu) -- (Wu);
  \path[fill=gray!50!red,opacity=0.5] (boti) .. controls +(+.5,0,+.5) and +(0,-1.25,0) .. (Nu) -- (Eu) .. controls +(0,-1.25,0) and +(+.5,0,+.5) .. (botii) -- (boti);
  \path[fill=gray!50!red,opacity=0.5] (Wu) .. controls +(0,-.75,0) and +(-.2,0,-.2) .. (boti) -- (botii) .. controls +(-.2,0,-.2) and +(0,-.75,0)  .. (Su) -- (Wu);
  \path[fill=gray!50!blue,opacity=0.5] (Su) .. controls +(0,-.75,0) and +(-.2,0,-.2) .. (botii) .. controls +(+.5,0,+.5) and +(0,-1.25,0) .. (Eu) -- (Su);
  \draw[string, black!50!red]  (Wu) -- (Su);
  \draw[string, black!50!blue] (Su) -- (Eu);
  \draw[string] (Su) .. controls +(0,-.75,0) and +(-.2,0,-.2) .. (botii) .. controls +(+.5,0,+.5) and +(0,-1.25,0) .. (Eu);
  \path (5.375,-1.5) node {,};
  \path 
  (7,-3,0)
  +(-1,0,0)  coordinate (Wcu)
  +(+1,0,0)  coordinate (Ecu)
  +(0,0,+1)  coordinate (Scu)
  +(0,0,-1)  coordinate (Ncu)
  ++(-.05,1.15,0)
  +(-.5,0,+.5) coordinate (topi)
  +(+.5,0,-.5) coordinate (topii)
  ;
  \draw[string, black!50!red]  (Ecu) -- (Ncu);
  \draw[string, black!50!blue] (Ncu) -- (Wcu);
  \draw[string] (Wcu) .. controls +(0,+.75,0) and +(-.05,0,+.05) .. (topi) .. controls +(+.15,0,-.15) and +(0,+1,0) .. (Scu);
  \draw[string] (Ncu) .. controls +(0,+.75,0) and +(-.05,0,+.05) .. (topii) .. controls +(+.15,0,-.15) and +(0,+1,0) .. (Ecu);
  \path[fill=gray!50!red,opacity=0.5] (Ncu) .. controls +(0,+.75,0) and +(-.05,0,+.05) .. (topii) .. controls +(+.15,0,-.15) and +(0,+1,0) .. (Ecu) -- (Ncu);
  \path[fill=gray!50!blue,opacity=0.5] (Wcu) .. controls +(0,+.75,0) and +(-.05,0,+.05) .. (topi) -- (topii) .. controls +(-.05,0,+.05) and +(0,+.75,0) .. (Ncu) -- (Wcu);
  \path[fill=gray!50!red,opacity=0.5] (Wcu) .. controls +(0,+.75,0) and +(-.05,0,+.05) .. (topi) .. controls +(+.15,0,-.15) and +(0,+1,0) .. (Scu) -- (Wcu);
  \draw[string, black!50!red]  (Wcu) -- (Scu);
  \draw[string, black!50!blue] (Scu) -- (Ecu);
  \path[fill=gray!50!blue,opacity=0.5] (topi) .. controls +(+.15,0,-.15) and +(0,+1,0) .. (Scu) -- (Ecu) .. controls +(0,+1,0)  and +(+.15,0,-.15) .. (topii) -- (topi);
  \path (8.125,-1.5) node {,};
  \path
  (10.75,-3,0)
  +(-1,0,0)  coordinate (Wc)
  +(+1,0,0)  coordinate (Ec)
  +(0,0,+1)  coordinate (Sc)
  +(0,0,-1)  coordinate (Nc)
  ++(+.05,1.85,0)
  +(-.5,0,+.5) coordinate (qi)
  +(+.5,0,-.5) coordinate (qii)
  (9.75,0,-1)
  +(-1,0,0)  coordinate (Wc1) 
  +(+1,0,0)  coordinate (Ec1) 
  +(0,0,+1)  coordinate (Sc1) 
  +(0,0,-1)  coordinate (Nc1) 
  (11.75,0,+1)
  +(-1,0,0)  coordinate (Wc2) 
  +(+1,0,0)  coordinate (Ec2) 
  +(0,0,+1)  coordinate (Sc2) 
  +(0,0,-1)  coordinate (Nc2) 
  ;
  \draw[string, black!50!red]  (Ec) -- (Nc);
  \draw[string, black!50!blue] (Nc) -- (Wc);
  \draw[string, black!50!red]  (Ec1) -- (Nc1);
  \draw[string, black!50!blue] (Nc1) -- (Wc1);
  \draw[string, black!50!red]  (Ec2) -- (Nc2);
  \draw[string, black!50!blue] (Nc2) -- (Wc2);
  \draw[string] (Wc1) .. controls +(0,-1,0) and +(0,+1,0) .. (Wc);
  \draw[string] (Nc1) .. controls +(0,-1,0) and +(0,+1,0) .. (Nc);
  \draw[string] (Ec2) .. controls +(0,-1,0) and +(0,+1,0) .. (Ec);
  \draw[string] (Nc2) .. controls +(0,-.75,0) and +(+.05,0,-.05) .. (qii) .. controls +(-.15,0,+.15) and +(0,-1,0) .. (Ec1);
  \path[fill=gray!50!blue,opacity=0.5] (Wc1) .. controls +(0,-1,0) and +(0,+1,0) .. (Wc) -- (Nc) .. controls +(0,+1,0) and +(0,-1,0) .. (Nc1) -- (Wc1);
  \path[fill=gray!50!red,opacity=0.5] (Nc1) .. controls +(0,-1,0) and +(0,+1,0) .. (Nc) -- (Ec) .. controls +(0,+1,0) and +(0,-1,0) .. (Ec2) -- (Nc2) .. controls +(0,-.75,0) and +(+.05,0,-.05) .. (qii) .. controls +(-.15,0,+.15) and +(0,-1,0) .. (Ec1) -- (Nc1);
  \path[fill=gray!50!blue,opacity=0.5] (qii) .. controls +(-.15,0,+.15) and +(0,-1,0) .. (Ec1) -- (Sc1) .. controls +(0,-1,0) and +(-.15,0,+.15) .. (qi) -- (qii);
  \draw[string] (Wc2) .. controls +(0,-.75,0) and +(+.05,0,-.05) .. (qi) .. controls +(-.15,0,+.15) and +(0,-1,0) .. (Sc1);
  \path[fill=gray!50!blue,opacity=0.5] (qi) .. controls +(+.05,0,-.05) and +(0,-.75,0)  .. (Wc2) -- (Nc2) .. controls +(0,-.75,0) and +(+.05,0,-.05) .. (qii) -- (qi);
  \path[fill=gray!50!red,opacity=0.5] (Wc1) .. controls +(0,-1,0) and +(0,+1,0) .. (Wc) -- (Sc) .. controls +(0,+1,0) and +(0,-1,0) .. (Sc2) -- (Wc2) .. controls +(0,-.75,0) and +(+.05,0,-.05) .. (qi) .. controls +(-.15,0,+.15) and +(0,-1,0) .. (Sc1) -- (Wc1);
  \path[fill=gray!50!blue,opacity=0.5] (Sc2) .. controls +(0,-1,0) and +(0,+1,0) .. (Sc) -- (Ec) .. controls +(0,+1,0) and +(0,-1,0) .. (Ec2) -- (Sc2);
  \draw[string, black!50!red]  (Wc) -- (Sc);
  \draw[string, black!50!blue] (Sc) -- (Ec);
  \draw[string, black!50!red]  (Wc1) -- (Sc1);
  \draw[string, black!50!blue] (Sc1) -- (Ec1);
  \draw[string, black!50!red]  (Wc2) -- (Sc2);
  \draw[string, black!50!blue] (Sc2) -- (Ec2);
  \draw[string] (Sc2) .. controls +(0,-1,0) and +(0,+1,0) .. (Sc);
\end{tikzpicture}
\hspace{-1in}
$$

The bialgebra axioms asserting compatibility between the algebra and coalgebra structures hold coherently for abstract reasons by 
Corollaries~\ref{cor:mndsquared} and~\ref{cor:part2ofmaintheorem}.
In terms of the three-dimensional diagrams, these laws assert: 
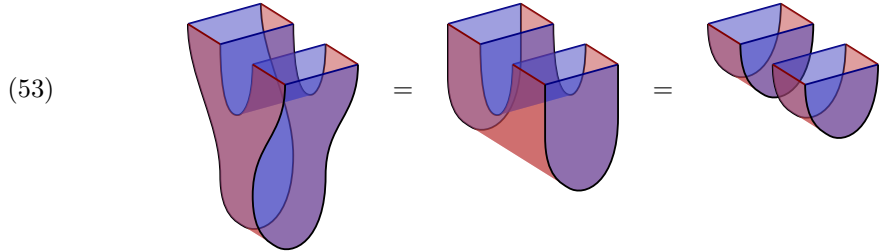
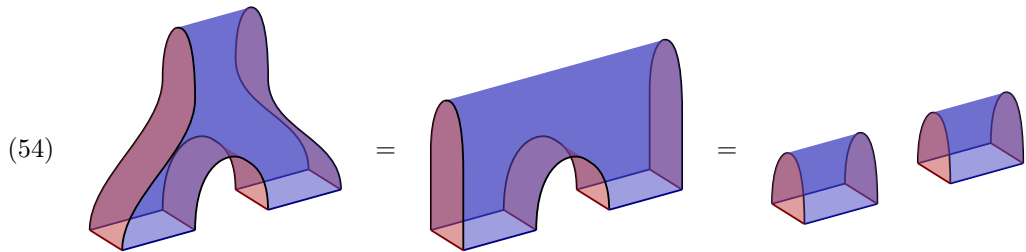
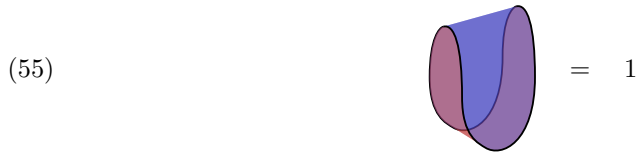
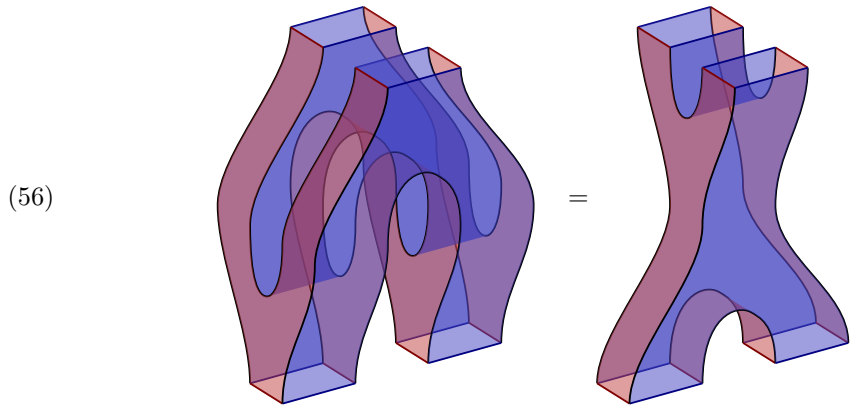
\begin{gather}
\label{eqn:coaug}
\begin{tikzpicture}[baseline=(middle), scale=0.7]
  \path (0,-1,0) coordinate (middle);
  \path
  (0,-2.5,0)
  +(-1,0,0)  coordinate (Wc)
  +(+1,0,0)  coordinate (Ec)
  +(0,0,+1)  coordinate (Sc)
  +(0,0,-1)  coordinate (Nc)
  ++(-.25,-1.15,0)
  +(-.5,0,-.5) coordinate (boti)
  +(+.5,0,+.5) coordinate (botii)
  (0,-3,0)
  ++(+.05,1.85,0)
  +(-.5,0,+.5) coordinate (qi)
  +(+.5,0,-.5) coordinate (qii)
  (-1,0,-1)
  +(-1,0,0)  coordinate (Wc1) 
  +(+1,0,0)  coordinate (Ec1) 
  +(0,0,+1)  coordinate (Sc1) 
  +(0,0,-1)  coordinate (Nc1) 
  (+1,0,+1)
  +(-1,0,0)  coordinate (Wc2) 
  +(+1,0,0)  coordinate (Ec2) 
  +(0,0,+1)  coordinate (Sc2) 
  +(0,0,-1)  coordinate (Nc2) 
  ;
  \draw[string, black!50!red]  (Ec1) -- (Nc1);
  \draw[string, black!50!blue] (Nc1) -- (Wc1);
  \draw[string, black!50!red]  (Ec2) -- (Nc2);
  \draw[string, black!50!blue] (Nc2) -- (Wc2);
  \draw[string] (Wc1) .. controls +(0,-1,0) and +(0,+1,0) .. (Wc) .. controls +(0,-.75,0) and +(-.2,0,-.2) .. (boti) .. controls +(+.5,0,+.5) and +(0,-1.25,0) .. (Nc)  .. controls +(0,+1,0) and +(0,-1,0) .. (Nc1);
  \draw[string] (Wc2) .. controls +(0,-.75,0) and +(+.05,0,-.05) .. (qi) .. controls +(-.15,0,+.15) and +(0,-1,0) .. (Sc1);
  \draw[string] (Nc2) .. controls +(0,-.75,0) and +(+.05,0,-.05) .. (qii) .. controls +(-.15,0,+.15) and +(0,-1,0) .. (Ec1);
  \path[fill=gray!50!blue,opacity=0.5] (Wc1) .. controls +(0,-1,0) and +(0,+1,0) .. (Wc) .. controls +(0,-.75,0) and +(-.2,0,-.2) .. (boti) .. controls +(+.5,0,+.5) and +(0,-1.25,0) .. (Nc) .. controls +(0,+1,0) and +(0,-1,0) .. (Nc1) -- (Wc1);
  \path[fill=gray!50!red,opacity=0.5] (Nc1) .. controls +(0,-1,0) and +(0,+1,0) .. (Nc) .. controls +(0,-1.25,0) and +(+.5,0,+.5) .. (boti) -- (botii) .. controls +(+.5,0,+.5) and +(0,-1.25,0) .. (Ec) .. controls +(0,+1,0) and +(0,-1,0) .. (Ec2) -- (Nc2) .. controls +(0,-.75,0) and +(+.05,0,-.05) .. (qii) .. controls +(-.15,0,+.15) and +(0,-1,0) .. (Ec1) -- (Nc1);
  \path[fill=gray!50!blue,opacity=0.5] (qii) .. controls +(-.15,0,+.15) and +(0,-1,0) .. (Ec1) -- (Sc1) .. controls +(0,-1,0) and +(-.15,0,+.15) .. (qi) -- (qii);
  \path[fill=gray!50!blue,opacity=0.5] (qi) .. controls +(+.05,0,-.05) and +(0,-.75,0)  .. (Wc2) -- (Nc2) .. controls +(0,-.75,0) and +(+.05,0,-.05) .. (qii) -- (qi);
  \path[fill=gray!50!red,opacity=0.5] (Wc1) .. controls +(0,-1,0) and +(0,+1,0) .. (Wc) .. controls +(0,-.75,0) and +(-.2,0,-.2) .. (boti) -- (botii) .. controls +(-.2,0,-.2) and +(0,-.75,0)  .. (Sc) .. controls +(0,+1,0) and +(0,-1,0) .. (Sc2) -- (Wc2) .. controls +(0,-.75,0) and +(+.05,0,-.05) .. (qi) .. controls +(-.15,0,+.15) and +(0,-1,0) .. (Sc1) -- (Wc1);
  \path[fill=gray!50!blue,opacity=0.5] (Sc2) .. controls +(0,-1,0) and +(0,+1,0) .. (Sc) .. controls +(0,-.75,0) and +(-.2,0,-.2) .. (botii) .. controls +(+.5,0,+.5) and +(0,-1.25,0) .. (Ec) .. controls +(0,+1,0) and +(0,-1,0) .. (Ec2) -- (Sc2);
  \draw[string, black!50!red]  (Wc1) -- (Sc1);
  \draw[string, black!50!blue] (Sc1) -- (Ec1);
  \draw[string, black!50!red]  (Wc2) -- (Sc2);
  \draw[string, black!50!blue] (Sc2) -- (Ec2);
  \draw[string] (Sc2) .. controls +(0,-1,0) and +(0,+1,0) .. (Sc) .. controls +(0,-.75,0) and +(-.2,0,-.2) .. (botii) .. controls +(+.5,0,+.5) and +(0,-1.25,0) .. (Ec) .. controls +(0,+1,0) and +(0,-1,0) .. (Ec2);
\end{tikzpicture}
\quad = \quad
\begin{tikzpicture}[baseline=(middle),scale=0.7]
  \path (0,-1,0) coordinate (middle);
  \path
  (0,-1,0)
  +(-2,0,-1)  coordinate (Wc)
  +(+2,0,1)  coordinate (Ec)
  +(1,0,+2)  coordinate (Sc)
  +(-1,0,-2)  coordinate (Nc)
  ++(-.25,-1.15,0)
  +(-1.5,0,-1.5) coordinate (boti)
  +(+1.5,0,+1.5) coordinate (botii)
  (0,-3,0)
  ++(+.05,1.85,0)
  +(-.5,0,+.5) coordinate (qi)
  +(+.5,0,-.5) coordinate (qii)
  (-1,0,-1)
  +(-1,0,0)  coordinate (Wc1) 
  +(+1,0,0)  coordinate (Ec1) 
  +(0,0,+1)  coordinate (Sc1) 
  +(0,0,-1)  coordinate (Nc1) 
  (+1,0,+1)
  +(-1,0,0)  coordinate (Wc2) 
  +(+1,0,0)  coordinate (Ec2) 
  +(0,0,+1)  coordinate (Sc2) 
  +(0,0,-1)  coordinate (Nc2) 
  ;
  \draw[string, black!50!red]  (Ec1) -- (Nc1);
  \draw[string, black!50!blue] (Nc1) -- (Wc1);
  \draw[string, black!50!red]  (Ec2) -- (Nc2);
  \draw[string, black!50!blue] (Nc2) -- (Wc2);
  \draw[string] (Wc1) .. controls +(0,-1,0) and +(0,+1,0) .. (Wc) .. controls +(0,-.75,0) and +(-.2,0,-.2) .. (boti) .. controls +(+.5,0,+.5) and +(0,-1.25,0) .. (Nc)  .. controls +(0,+1,0) and +(0,-1,0) .. (Nc1);
  \draw[string] (Wc2) .. controls +(0,-.75,0) and +(+.05,0,-.05) .. (qi) .. controls +(-.15,0,+.15) and +(0,-1,0) .. (Sc1);
  \draw[string] (Nc2) .. controls +(0,-.75,0) and +(+.05,0,-.05) .. (qii) .. controls +(-.15,0,+.15) and +(0,-1,0) .. (Ec1);
  \path[fill=gray!50!blue,opacity=0.5] (Wc1) .. controls +(0,-1,0) and +(0,+1,0) .. (Wc) .. controls +(0,-.75,0) and +(-.2,0,-.2) .. (boti) .. controls +(+.5,0,+.5) and +(0,-1.25,0) .. (Nc) .. controls +(0,+1,0) and +(0,-1,0) .. (Nc1) -- (Wc1);
  \path[fill=gray!50!red,opacity=0.5] (Nc1) .. controls +(0,-1,0) and +(0,+1,0) .. (Nc) .. controls +(0,-1.25,0) and +(+.5,0,+.5) .. (boti) -- (botii) .. controls +(+.5,0,+.5) and +(0,-1.25,0) .. (Ec) .. controls +(0,+1,0) and +(0,-1,0) .. (Ec2) -- (Nc2) .. controls +(0,-.75,0) and +(+.05,0,-.05) .. (qii) .. controls +(-.15,0,+.15) and +(0,-1,0) .. (Ec1) -- (Nc1);
  \path[fill=gray!50!blue,opacity=0.5] (qii) .. controls +(-.15,0,+.15) and +(0,-1,0) .. (Ec1) -- (Sc1) .. controls +(0,-1,0) and +(-.15,0,+.15) .. (qi) -- (qii);
  \path[fill=gray!50!blue,opacity=0.5] (qi) .. controls +(+.05,0,-.05) and +(0,-.75,0)  .. (Wc2) -- (Nc2) .. controls +(0,-.75,0) and +(+.05,0,-.05) .. (qii) -- (qi);
  \path[fill=gray!50!red,opacity=0.5] (Wc1) .. controls +(0,-1,0) and +(0,+1,0) .. (Wc) .. controls +(0,-.75,0) and +(-.2,0,-.2) .. (boti) -- (botii) .. controls +(-.2,0,-.2) and +(0,-.75,0)  .. (Sc) .. controls +(0,+1,0) and +(0,-1,0) .. (Sc2) -- (Wc2) .. controls +(0,-.75,0) and +(+.05,0,-.05) .. (qi) .. controls +(-.15,0,+.15) and +(0,-1,0) .. (Sc1) -- (Wc1);
  \path[fill=gray!50!blue,opacity=0.5] (Sc2) .. controls +(0,-1,0) and +(0,+1,0) .. (Sc) .. controls +(0,-.75,0) and +(-.2,0,-.2) .. (botii) .. controls +(+.5,0,+.5) and +(0,-1.25,0) .. (Ec) .. controls +(0,+1,0) and +(0,-1,0) .. (Ec2) -- (Sc2);
  \draw[string, black!50!red]  (Wc1) -- (Sc1);
  \draw[string, black!50!blue] (Sc1) -- (Ec1);
  \draw[string, black!50!red]  (Wc2) -- (Sc2);
  \draw[string, black!50!blue] (Sc2) -- (Ec2);
  \draw[string] (Sc2) .. controls +(0,-1,0) and +(0,+1,0) .. (Sc) .. controls +(0,-.75,0) and +(-.2,0,-.2) .. (botii) .. controls +(+.5,0,+.5) and +(0,-1.25,0) .. (Ec) .. controls +(0,+1,0) and +(0,-1,0) .. (Ec2);
\end{tikzpicture}
\quad = \quad
\begin{tikzpicture}[baseline=(middle), scale=0.7]
  \path (0,-1,0) coordinate (middle);
  \path
  (-1,0,-1)
  +(-1,0,0)  coordinate (Wu1) 
  +(+1,0,0)  coordinate (Eu1) 
  +(0,0,+1)  coordinate (Su1) 
  +(0,0,-1)  coordinate (Nu1) 
  ++(-.25,-1.15,0)
  +(-.5,0,-.5) coordinate (boti1)
  +(+.5,0,+.5) coordinate (botii1)
  (+1,0,+1)
  +(-1,0,0)  coordinate (Wu2) 
  +(+1,0,0)  coordinate (Eu2) 
  +(0,0,+1)  coordinate (Su2) 
  +(0,0,-1)  coordinate (Nu2) 
  ++(-.25,-1.15,0)
  +(-.5,0,-.5) coordinate (boti2)
  +(+.5,0,+.5) coordinate (botii2)
  ;
  \draw[string, black!50!red]  (Eu1) -- (Nu1);
  \draw[string, black!50!blue] (Nu1) -- (Wu1);
  \draw[string] (Wu1) .. controls +(0,-.75,0) and +(-.2,0,-.2) .. (boti1) .. controls +(+.5,0,+.5) and +(0,-1.25,0) .. (Nu1);
  \path[fill=gray!50!blue,opacity=0.5] (Wu1) .. controls +(0,-.75,0) and +(-.2,0,-.2) .. (boti1) .. controls +(+.5,0,+.5) and +(0,-1.25,0) .. (Nu1) -- (Wu1);
  \path[fill=gray!50!red,opacity=0.5] (boti1) .. controls +(+.5,0,+.5) and +(0,-1.25,0) .. (Nu1) -- (Eu1) .. controls +(0,-1.25,0) and +(+.5,0,+.5) .. (botii1) -- (boti1);
  \path[fill=gray!50!red,opacity=0.5] (Wu1) .. controls +(0,-.75,0) and +(-.2,0,-.2) .. (boti1) -- (botii1) .. controls +(-.2,0,-.2) and +(0,-.75,0)  .. (Su1) -- (Wu1);
  \path[fill=gray!50!blue,opacity=0.5] (Su1) .. controls +(0,-.75,0) and +(-.2,0,-.2) .. (botii1) .. controls +(+.5,0,+.5) and +(0,-1.25,0) .. (Eu1) -- (Su1);
  \draw[string, black!50!red]  (Wu1) -- (Su1);
  \draw[string, black!50!blue] (Su1) -- (Eu1);
  \draw[string] (Su1) .. controls +(0,-.75,0) and +(-.2,0,-.2) .. (botii1) .. controls +(+.5,0,+.5) and +(0,-1.25,0) .. (Eu1);
  \draw[string, black!50!red]  (Eu2) -- (Nu2);
  \draw[string, black!50!blue] (Nu2) -- (Wu2);
  \draw[string] (Wu2) .. controls +(0,-.75,0) and +(-.2,0,-.2) .. (boti2) .. controls +(+.5,0,+.5) and +(0,-1.25,0) .. (Nu2);
  \path[fill=gray!50!blue,opacity=0.5] (Wu2) .. controls +(0,-.75,0) and +(-.2,0,-.2) .. (boti2) .. controls +(+.5,0,+.5) and +(0,-1.25,0) .. (Nu2) -- (Wu2);
  \path[fill=gray!50!red,opacity=0.5] (boti2) .. controls +(+.5,0,+.5) and +(0,-1.25,0) .. (Nu2) -- (Eu2) .. controls +(0,-1.25,0) and +(+.5,0,+.5) .. (botii2) -- (boti2);
  \path[fill=gray!50!red,opacity=0.5] (Wu2) .. controls +(0,-.75,0) and +(-.2,0,-.2) .. (boti2) -- (botii2) .. controls +(-.2,0,-.2) and +(0,-.75,0)  .. (Su2) -- (Wu2);
  \path[fill=gray!50!blue,opacity=0.5] (Su2) .. controls +(0,-.75,0) and +(-.2,0,-.2) .. (botii2) .. controls +(+.5,0,+.5) and +(0,-1.25,0) .. (Eu2) -- (Su2);
  \draw[string, black!50!red]  (Wu2) -- (Su2);
  \draw[string, black!50!blue] (Su2) -- (Eu2);
  \draw[string] (Su2) .. controls +(0,-.75,0) and +(-.2,0,-.2) .. (botii2) .. controls +(+.5,0,+.5) and +(0,-1.25,0) .. (Eu2);
\end{tikzpicture}
\\[12pt]
\label{eqn:aug}
\begin{tikzpicture}[baseline=(middle), scale=0.7]
  \path (0,-2,0) coordinate (middle);
  \path 
  (0,-.5,0)
  +(-1,0,0)  coordinate (Wm)
  +(+1,0,0)  coordinate (Em)
  +(0,0,+1)  coordinate (Sm)
  +(0,0,-1)  coordinate (Nm)
    ++(-.05,1.15,0)
  +(-.5,0,+.5) coordinate (topi)
  +(+.5,0,-.5) coordinate (topii)
  (+.25,-1.85,0)
  +(-.5,0,-.5) coordinate (pi)
  +(+.5,0,+.5) coordinate (pii)
  (-1,-3,+1)
  +(-1,0,0)  coordinate (Wm1) 
  +(+1,0,0)  coordinate (Em1) 
  +(0,0,+1)  coordinate (Sm1) 
  +(0,0,-1)  coordinate (Nm1) 
  (+1,-3,-1)
  +(-1,0,0)  coordinate (Wm2) 
  +(+1,0,0)  coordinate (Em2) 
  +(0,0,+1)  coordinate (Sm2) 
  +(0,0,-1)  coordinate (Nm2) 
  ;
  \draw[string, black!50!red] (Em1) -- (Nm1);
  \draw[string, black!50!blue] (Nm1) -- (Wm1);
  \draw[string, black!50!red] (Em2) -- (Nm2);
  \draw[string, black!50!blue] (Nm2) -- (Wm2);  
  \draw[string] (Nm2) .. controls +(0,+1,0) and +(0,-1,0) .. (Nm) .. controls +(0,+.75,0) and +(-.05,0,+.05) .. (topii);
  \path[fill=gray!50!blue,opacity=0.5] (Wm1) .. controls +(0,+1,0) and +(0,-1,0) .. (Wm) .. controls +(0,+.75,0) and +(-.05,0,+.05) .. (topi) -- (topii) .. controls +(-.05,0,+.05) and +(0,+.75,0) .. (Nm) .. controls +(0,-1,0) and +(0,+1,0) .. (Nm2) -- (Wm2) 
  .. controls +(0,+.75,0) and +(+.2,0,+.2) .. (pi)  .. controls +(-.5,0,-.5) and +(0,+1.25,0)  ..  (Nm1) -- (Wm1);
  \draw[string] (Nm1) .. controls +(0,+1.25,0) and +(-.5,0,-.5) .. (pi) .. controls +(+.2,0,+.2) and +(0,+.75,0)  .. (Wm2);
  \path[fill=gray!50!red,opacity=0.5] (Nm2) .. controls +(0,+1,0) and +(0,-1,0) .. (Nm) .. controls +(0,+.75,0) and +(-.05,0,+.05) .. (topii) .. controls +(+.15,0,-.15) and +(0,+1,0) .. (Em) .. controls +(0,-1,0) and +(0,+1,0) .. (Em2) -- (Nm2);
  \draw[string] (Wm1) .. controls +(0,+1,0) and +(0,-1,0) .. (Wm) .. controls +(0,+.75,0) and +(-.05,0,+.05) .. (topi);
  \draw[string] (topii) .. controls +(+.15,0,-.15) and +(0,+1,0) .. (Em) .. controls +(0,-1,0) and +(0,+1,0)   .. (Em2) ;
  \path[fill=gray!50!red,opacity=0.5] (Wm1) .. controls +(0,+1,0) and +(0,-1,0) .. (Wm) .. controls +(0,+.75,0) and +(-.05,0,+.05) .. (topi) .. controls +(+.15,0,-.15) and +(0,+1,0) .. (Sm) .. controls +(0,-1,0) and +(0,+1,0) .. (Sm1) -- (Wm1);
  \path[fill=gray!50!red,opacity=0.5] (Nm1) .. controls +(0,+1.25,0) and +(-.5,0,-.5) .. (pi) -- (pii) .. controls +(-.5,0,-.5) and +(0,+1.25,0)  ..  (Em1) -- (Nm1);
  \path[fill=gray!50!red,opacity=0.5] (Sm2) .. controls +(0,+.75,0) and +(+.2,0,+.2) .. (pii) -- (pi) .. controls +(+.2,0,+.2) and +(0,+.75,0)  .. (Wm2) -- (Sm2);
  \path[fill=gray!50!blue,opacity=0.5] (Sm1) .. controls +(0,+1,0) and +(0,-1,0) .. (Sm) 
  .. controls +(0,+1,0) and +(+.15,0,-.15)  .. (topi) -- (topii) .. controls +(+.15,0,-.15)  and +(0,+1,0) ..
   (Em) .. controls +(0,-1,0) and +(0,+1,0) .. (Em2) -- (Sm2) .. controls +(0,+.75,0) and +(+.2,0,+.2) .. (pii)  .. controls +(-.5,0,-.5) and +(0,+1.25,0)  ..  (Em1) -- (Sm1);
  \draw[string] (Sm1) .. controls +(0,+1,0) and +(0,-1,0) .. (Sm) .. controls +(0,+1,0) and +(+.15,0,-.15)  .. (topi) ;
  \draw[string] (Em1) .. controls +(0,+1.25,0) and +(-.5,0,-.5) .. (pii) .. controls +(+.2,0,+.2) and +(0,+.75,0)  .. (Sm2);
  \draw[string, black!50!red] (Wm1) -- (Sm1);
  \draw[string, black!50!blue] (Sm1) -- (Em1);
  \draw[string, black!50!red] (Wm2) -- (Sm2);
  \draw[string, black!50!blue] (Sm2) -- (Em2);
\end{tikzpicture}
\quad = \quad
\begin{tikzpicture}[baseline=(middle),scale=0.7]
  \path (0,-2,0) coordinate (middle);
  \path 
  (-1,-1.5,+1)
  +(-1,0,0)  coordinate (Wm)
  +(0,0,+1)  coordinate (Sm)
   ++(-.05,1.15,0)
  +(-.5,0,+.5) coordinate (topi)
   (+1,-1.5,-1)
  +(+1,0,0)  coordinate (Em)
  +(0,0,-1)  coordinate (Nm)
    ++(-.05,1.15,0)
  +(+.5,0,-.5) coordinate (topii)
  (+.25,-1.85,0)
  +(-.5,0,-.5) coordinate (pi)
  +(+.5,0,+.5) coordinate (pii)
  (-1,-3,+1)
  +(-1,0,0)  coordinate (Wm1) 
  +(+1,0,0)  coordinate (Em1) 
  +(0,0,+1)  coordinate (Sm1) 
  +(0,0,-1)  coordinate (Nm1) 
  (+1,-3,-1)
  +(-1,0,0)  coordinate (Wm2) 
  +(+1,0,0)  coordinate (Em2) 
  +(0,0,+1)  coordinate (Sm2) 
  +(0,0,-1)  coordinate (Nm2) 
  ;
  \draw[string, black!50!red] (Em1) -- (Nm1);
  \draw[string, black!50!blue] (Nm1) -- (Wm1);
  \draw[string, black!50!red] (Em2) -- (Nm2);
  \draw[string, black!50!blue] (Nm2) -- (Wm2);  
  \draw[string] (Nm2) .. controls +(0,+1,0) and +(0,-1,0) .. (Nm) .. controls +(0,+.75,0) and +(-.05,0,+.05) .. (topii);
  \path[fill=gray!50!blue,opacity=0.5] (Wm1) .. controls +(0,+1,0) and +(0,-1,0) .. (Wm) .. controls +(0,+.75,0) and +(-.05,0,+.05) .. (topi) -- (topii) .. controls +(-.05,0,+.05) and +(0,+.75,0) .. (Nm) .. controls +(0,-1,0) and +(0,+1,0) .. (Nm2) -- (Wm2) 
  .. controls +(0,+.75,0) and +(+.2,0,+.2) .. (pi)  .. controls +(-.5,0,-.5) and +(0,+1.25,0)  ..  (Nm1) -- (Wm1);
  \draw[string] (Nm1) .. controls +(0,+1.25,0) and +(-.5,0,-.5) .. (pi) .. controls +(+.2,0,+.2) and +(0,+.75,0)  .. (Wm2);
  \path[fill=gray!50!red,opacity=0.5] (Nm2) .. controls +(0,+1,0) and +(0,-1,0) .. (Nm) .. controls +(0,+.75,0) and +(-.05,0,+.05) .. (topii) .. controls +(+.15,0,-.15) and +(0,+1,0) .. (Em) .. controls +(0,-1,0) and +(0,+1,0) .. (Em2) -- (Nm2);
  \draw[string] (Wm1) .. controls +(0,+1,0) and +(0,-1,0) .. (Wm) .. controls +(0,+.75,0) and +(-.05,0,+.05) .. (topi);
  \draw[string] (topii) .. controls +(+.15,0,-.15) and +(0,+1,0) .. (Em) .. controls +(0,-1,0) and +(0,+1,0)   .. (Em2) ;
  \path[fill=gray!50!red,opacity=0.5] (Wm1) .. controls +(0,+1,0) and +(0,-1,0) .. (Wm) .. controls +(0,+.75,0) and +(-.05,0,+.05) .. (topi) .. controls +(+.15,0,-.15) and +(0,+1,0) .. (Sm) .. controls +(0,-1,0) and +(0,+1,0) .. (Sm1) -- (Wm1);
  \path[fill=gray!50!red,opacity=0.5] (Nm1) .. controls +(0,+1.25,0) and +(-.5,0,-.5) .. (pi) -- (pii) .. controls +(-.5,0,-.5) and +(0,+1.25,0)  ..  (Em1) -- (Nm1);
  \path[fill=gray!50!red,opacity=0.5] (Sm2) .. controls +(0,+.75,0) and +(+.2,0,+.2) .. (pii) -- (pi) .. controls +(+.2,0,+.2) and +(0,+.75,0)  .. (Wm2) -- (Sm2);
  \path[fill=gray!50!blue,opacity=0.5] (Sm1) .. controls +(0,+1,0) and +(0,-1,0) .. (Sm) 
  .. controls +(0,+1,0) and +(+.15,0,-.15)  .. (topi) -- (topii) .. controls +(+.15,0,-.15)  and +(0,+1,0) ..
   (Em) .. controls +(0,-1,0) and +(0,+1,0) .. (Em2) -- (Sm2) .. controls +(0,+.75,0) and +(+.2,0,+.2) .. (pii)  .. controls +(-.5,0,-.5) and +(0,+1.25,0)  ..  (Em1) -- (Sm1);
  \draw[string] (Sm1) .. controls +(0,+1,0) and +(0,-1,0) .. (Sm) .. controls +(0,+1,0) and +(+.15,0,-.15)  .. (topi) ;
  \draw[string] (Em1) .. controls +(0,+1.25,0) and +(-.5,0,-.5) .. (pii) .. controls +(+.2,0,+.2) and +(0,+.75,0)  .. (Sm2);
  \draw[string, black!50!red] (Wm1) -- (Sm1);
  \draw[string, black!50!blue] (Sm1) -- (Em1);
  \draw[string, black!50!red] (Wm2) -- (Sm2);
  \draw[string, black!50!blue] (Sm2) -- (Em2);
\end{tikzpicture} 
\quad = \quad
\begin{tikzpicture}[baseline=(middle),scale=0.7]
  \path (0,-2.5,0) coordinate (middle);
  \path
  (-1,-3,+1)
  +(-1,0,0)  coordinate (Wcu)
  +(+1,0,0)  coordinate (Ecu)
  +(0,0,+1)  coordinate (Scu)
  +(0,0,-1)  coordinate (Ncu)
  ++(-.05,1.15,0)
  +(-.5,0,+.5) coordinate (topi)
  +(+.5,0,-.5) coordinate (topii)
  ;
  \draw[string, black!50!red]  (Ecu) -- (Ncu);
  \draw[string, black!50!blue] (Ncu) -- (Wcu);
  \draw[string] (Wcu) .. controls +(0,+.75,0) and +(-.05,0,+.05) .. (topi) .. controls +(+.15,0,-.15) and +(0,+1,0) .. (Scu);
  \draw[string] (Ncu) .. controls +(0,+.75,0) and +(-.05,0,+.05) .. (topii) .. controls +(+.15,0,-.15) and +(0,+1,0) .. (Ecu);
  \path[fill=gray!50!red,opacity=0.5] (Ncu) .. controls +(0,+.75,0) and +(-.05,0,+.05) .. (topii) .. controls +(+.15,0,-.15) and +(0,+1,0) .. (Ecu) -- (Ncu);
  \path[fill=gray!50!blue,opacity=0.5] (Wcu) .. controls +(0,+.75,0) and +(-.05,0,+.05) .. (topi) -- (topii) .. controls +(-.05,0,+.05) and +(0,+.75,0) .. (Ncu) -- (Wcu);
  \path[fill=gray!50!red,opacity=0.5] (Wcu) .. controls +(0,+.75,0) and +(-.05,0,+.05) .. (topi) .. controls +(+.15,0,-.15) and +(0,+1,0) .. (Scu) -- (Wcu);
  \draw[string, black!50!red]  (Wcu) -- (Scu);
  \draw[string, black!50!blue] (Scu) -- (Ecu);
  \path[fill=gray!50!blue,opacity=0.5] (topi) .. controls +(+.15,0,-.15) and +(0,+1,0) .. (Scu) -- (Ecu) .. controls +(0,+1,0)  and +(+.15,0,-.15) .. (topii) -- (topi);
  \path
  (+1,-3,-1)
  +(-1,0,0)  coordinate (Wcu2)
  +(+1,0,0)  coordinate (Ecu2)
  +(0,0,+1)  coordinate (Scu2)
  +(0,0,-1)  coordinate (Ncu2)
  ++(-.05,1.15,0)
  +(-.5,0,+.5) coordinate (topi2)
  +(+.5,0,-.5) coordinate (topii2)
  ;
  \draw[string, black!50!red]  (Ecu2) -- (Ncu2);
  \draw[string, black!50!blue] (Ncu2) -- (Wcu2);
  \draw[string] (Wcu2) .. controls +(0,+.75,0) and +(-.05,0,+.05) .. (topi2) .. controls +(+.15,0,-.15) and +(0,+1,0) .. (Scu2);
  \draw[string] (Ncu2) .. controls +(0,+.75,0) and +(-.05,0,+.05) .. (topii2) .. controls +(+.15,0,-.15) and +(0,+1,0) .. (Ecu2);
  \path[fill=gray!50!red,opacity=0.5] (Ncu2) .. controls +(0,+.75,0) and +(-.05,0,+.05) .. (topii2) .. controls +(+.15,0,-.15) and +(0,+1,0) .. (Ecu2) -- (Ncu2);
  \path[fill=gray!50!blue,opacity=0.5] (Wcu2) .. controls +(0,+.75,0) and +(-.05,0,+.05) .. (topi2) -- (topii2) .. controls +(-.05,0,+.05) and +(0,+.75,0) .. (Ncu2) -- (Wcu2);
  \path[fill=gray!50!red,opacity=0.5] (Wcu2) .. controls +(0,+.75,0) and +(-.05,0,+.05) .. (topi2) .. controls +(+.15,0,-.15) and +(0,+1,0) .. (Scu2) -- (Wcu2);
  \draw[string, black!50!red]  (Wcu2) -- (Scu2);
  \draw[string, black!50!blue] (Scu2) -- (Ecu2);
  \path[fill=gray!50!blue,opacity=0.5] (topi2) .. controls +(+.15,0,-.15) and +(0,+1,0) .. (Scu2) -- (Ecu2) .. controls +(0,+1,0)  and +(+.15,0,-.15) .. (topii2) -- (topi2);
\end{tikzpicture}
\\[15pt]
\label{eqn:unitcounit}
\begin{tikzpicture}[baseline=(middle),scale=0.7]
  \path 
  (0,0,0)
  +(0,0,0) coordinate (middle)
  +(-1,0,0)  coordinate (W)
  +(+1,0,0)  coordinate (E)
  +(0,0,+1)  coordinate (S)
  +(0,0,-1)  coordinate (N)
  (-.25,-1.15,0)
  +(-.5,0,-.5) coordinate (boti)
  +(+.5,0,+.5) coordinate (botii)
  (-.05,1.15,0)
  +(-.5,0,+.5) coordinate (topi)
  +(+.5,0,-.5) coordinate (topii)
  ;
  \draw[string] (topi) .. controls +(-.05,0,+.05) and +(0,+.75,0)  .. (W) .. controls +(0,-.75,0) and +(-.2,0,-.2) .. (boti) .. controls +(+.5,0,+.5) and +(0,-1.25,0) .. (N) .. controls +(0,+.75,0) and +(-.05,0,+.05) .. (topii);
  \path[fill=gray!50!blue,opacity=0.5] (topi) .. controls +(-.05,0,+.05) and +(0,+.75,0)  ..(W) .. controls +(0,-.75,0) and +(-.2,0,-.2) .. (boti) .. controls +(+.5,0,+.5) and +(0,-1.25,0) .. (N) .. controls +(0,+.75,0) and +(-.05,0,+.05) .. (topii) -- (topi);
  \path[fill=gray!50!red,opacity=0.5] (boti) .. controls +(+.5,0,+.5) and +(0,-1.25,0) .. (N) .. controls +(0,+.75,0) and +(-.05,0,+.05) .. (topii) .. controls +(+.15,0,-.15)  and +(0,+1,0)   .. (E) .. controls +(0,-1.25,0)  and +(+.5,0,+.5) .. (botii) -- (boti);
  \path[fill=gray!50!red,opacity=0.5] (boti) .. controls +(-.2,0,-.2) and +(0,-.75,0)  .. (W) .. controls +(0,+.75,0) and +(-.05,0,+.05)    .. (topi)  .. controls +(+.15,0,-.15) and +(0,+1,0) .. (S) .. controls +(0,-.75,0) and +(-.2,0,-.2) .. (botii) -- (boti);
  \path[fill=gray!50!blue,opacity=0.5] (topi) .. controls +(+.15,0,-.15) and +(0,+1,0) .. (S) .. controls +(0,-.75,0) and +(-.2,0,-.2) .. (botii) .. controls +(+.5,0,+.5) and +(0,-1.25,0) .. (E) .. controls +(0,+1,0) and +(+.15,0,-.15)  .. (topii);
  \draw[string] (topi) .. controls +(+.15,0,-.15) and +(0,+1,0) .. (S) .. controls +(0,-.75,0) and +(-.2,0,-.2) .. (botii) .. controls +(+.5,0,+.5) and +(0,-1.25,0) .. (E) .. controls +(0,+1,0) and +(+.15,0,-.15)  .. (topii);
  ;
\end{tikzpicture}
\quad = \quad 1
\\[12pt]
\begin{tikzpicture}[baseline=(middle),scale=0.7]
  \path (0,0,0) coordinate (middle);
  \path
  (-2,0,0)
  +(-1,0,0)  coordinate (Ww)
  +(+1,0,0)  coordinate (Ew)
  +(0,0,+1)  coordinate (Sw)
  +(0,0,-1)  coordinate (Nw)
  (+2,0,0)
  +(-1,0,0)  coordinate (We)
  +(+1,0,0)  coordinate (Ee)
  +(0,0,+1)  coordinate (Se)
  +(0,0,-1)  coordinate (Ne)
  (0,0,+2)
  +(-1,0,0)  coordinate (Ws)
  +(+1,0,0)  coordinate (Es)
  +(0,0,+1)  coordinate (Ss)
  +(0,0,-1)  coordinate (Ns)
  (0,0,-2)
  +(-1,0,0)  coordinate (Wn)
  +(+1,0,0)  coordinate (En)
  +(0,0,+1)  coordinate (Sn)
  +(0,0,-1)  coordinate (Nn)
  ;
  \path
  (-1,3,-1)
  +(-1,0,0)  coordinate (Wt1)
  +(+1,0,0)  coordinate (Et1)
  +(0,0,+1)  coordinate (St1)
  +(0,0,-1)  coordinate (Nt1)
  ++(+.25,-1.85,0)
  +(-.5,0,-.5) coordinate (pi1)
  +(+.5,0,+.5) coordinate (pii1)
  (+1,3,+1)
  +(-1,0,0)  coordinate (Wt2)
  +(+1,0,0)  coordinate (Et2)
  +(0,0,+1)  coordinate (St2)
  +(0,0,-1)  coordinate (Nt2)
  ++(+.25,-1.85,0)
  +(-.5,0,-.5) coordinate (pi2)
  +(+.5,0,+.5) coordinate (pii2)
  (-1,-3,+1)
  +(-1,0,0)  coordinate (Wb1)
  +(+1,0,0)  coordinate (Eb1)
  +(0,0,+1)  coordinate (Sb1)
  +(0,0,-1)  coordinate (Nb1)
  ++(+.05,1.85,0)
  +(-.5,0,+.5) coordinate (qi1)
  +(+.5,0,-.5) coordinate (qii1)
  (+1,-3,-1)
  +(-1,0,0)  coordinate (Wb2)
  +(+1,0,0)  coordinate (Eb2)
  +(0,0,+1)  coordinate (Sb2)
  +(0,0,-1)  coordinate (Nb2)
  ++(+.05,1.85,0)
  +(-.5,0,+.5) coordinate (qi2)
  +(+.5,0,-.5) coordinate (qii2)
  ;
  \draw[string, black!50!red]  (Et1) -- (Nt1);
  \draw[string, black!50!blue] (Nt1) -- (Wt1);
  \draw[string, black!50!red]  (Et2) -- (Nt2);
  \draw[string, black!50!blue] (Nt2) -- (Wt2);
  \draw[string, black!50!red]  (Eb1) -- (Nb1);
  \draw[string, black!50!red]  (Eb2) -- (Nb2);
  \draw[string, black!50!blue] (Nb2) -- (Wb2);
  \draw[string] (Nb2) .. controls +(0,1,0) and +(0,-1,0) .. (Nn) .. controls +(0,1,0) and +(0,-1,0) .. (Nt1);
  \path[fill=gray!50!blue,opacity=0.5] 
    (Nb2) .. controls +(0,1,0) and +(0,-1,0) .. (Nn) .. controls +(0,1,0) and +(0,-1,0) .. (Nt1) -- (Wt1) .. controls +(0,-1,0) and +(0,1,0) .. (Ww) .. controls +(0,-1,0) and +(0,1,0) .. (Wb1) -- (Nb1) .. controls +(0,1,0) and +(0,-1,0) .. (Nw) 
    .. controls +(0,+1.25,0) and +(-.5,0,-.5) .. (pi1) .. controls +(+.2,0,+.2) and +(0,+.75,0)  ..
    (Wn) .. controls +(0,-1,0) and +(0,1,0) .. (Wb2) -- (Nb2)
  ;
  \draw[string] (Wt1) .. controls +(0,-1,0) and +(0,1,0) .. (Ww) .. controls +(0,-1,0) and +(0,1,0) .. (Wb1);
  \draw[string] (Nb1) .. controls +(0,1,0) and +(0,-1,0) .. (Nw) 
  .. controls +(0,+1.25,0) and +(-.5,0,-.5) .. (pi1) .. controls +(+.2,0,+.2) and +(0,+.75,0)  ..
  (Wn) .. controls +(0,-1,0) and +(0,1,0) .. (Wb2);
  \path[fill=gray!50!red,opacity=0.5] 
    (Nb2) .. controls +(0,1,0) and +(0,-1,0) .. (Nn) .. controls +(0,1,0) and +(0,-1,0) .. (Nt1) -- 
    (Et1) .. controls +(0,-1,0) and +(0,1,0) .. (En) .. controls +(0,-1,0) and +(-.15,0,+.15) .. (qii2) .. controls  +(+.05,0,-.05) and +(0,-.75,0) .. (Ne) .. controls +(0,1,0) and +(0,-1,0) .. (Nt2) --
    (Et2)  .. controls +(0,-1,0) and +(0,1,0) .. (Ee) .. controls +(0,-1,0) and +(0,1,0) .. (Eb2) 
    -- (Nb2)
   ;
   \draw[string] (Et1) .. controls +(0,-1,0) and +(0,1,0) .. (En) .. controls +(0,-1,0) and +(-.15,0,+.15) .. (qii2) .. controls  +(+.05,0,-.05) and +(0,-.75,0) .. (Ne) .. controls +(0,1,0) and +(0,-1,0) .. (Nt2);
   \draw[string] (Et2)  .. controls +(0,-1,0) and +(0,1,0) .. (Ee) .. controls +(0,-1,0) and +(0,1,0) .. (Eb2) ;
   \path[fill=gray!50!red,opacity=0.5]
   (pi1) --
   (pii1) .. controls +(-.5,0,-.5) and +(0,+1.25,0)  .. (Ew) .. controls +(0,-1,0) and +(-.15,0,+.15) .. (qii1) 
   .. controls  +(+.05,0,-.05) and +(0,-.75,0) .. (Ns)
   .. controls +(0,+1.25,0)  and +(-.5,0,-.5)  .. (pi2) -- (pii2)
   .. controls +(-.5,0,-.5) and +(0,+1.25,0)  .. (Es)
   .. controls +(0,-1,0) and +(0,1,0) .. (Eb1) -- (Nb1) .. controls +(0,1,0) and +(0,-1,0) .. (Nw)
   .. controls +(0,+1.25,0)  and +(-.5,0,-.5)  .. (pi1)
   ;
   \path[fill=gray!50!blue,opacity=0.5]
   (Et1) .. controls +(0,-1,0) and +(0,1,0) .. (En) .. controls +(0,-1,0) and +(-.15,0,+.15) .. (qii2) -- (qi2)
   .. controls +(-.15,0,+.15)  and  +(0,-1,0) .. (Sn)
   .. controls +(0,+.75,0) and +(+.2,0,+.2) .. (pii1) .. controls +(-.5,0,-.5) and +(0,+1.25,0)  .. (Ew)
   .. controls +(0,-1,0) and +(-.15,0,+.15) .. (qii1) -- (qi1)
   .. controls +(-.15,0,+.15)  and  +(0,-1,0) .. (Sw)
   .. controls +(0,1,0) and +(0,-1,0) .. (St1) -- (Et1)
   ;
   \draw[string, black!50!blue] (St1) -- (Et1);
   \draw[string, black!50!blue] (Nb1) -- (Wb1);
   \draw[string] (qi2)
   .. controls +(-.15,0,+.15)  and  +(0,-1,0) .. (Sn)
   .. controls +(0,+.75,0) and +(+.2,0,+.2) .. (pii1) .. controls +(-.5,0,-.5) and +(0,+1.25,0)  .. (Ew)
   .. controls +(0,-1,0) and +(-.15,0,+.15) .. (qii1);
   \draw[string] 
   (qii1) 
   .. controls  +(+.05,0,-.05) and +(0,-.75,0) .. (Ns)
   .. controls +(0,+1.25,0)  and +(-.5,0,-.5)  .. (pi2)
    .. controls +(+.2,0,+.2) and +(0,+.75,0)  .. (We)
    .. controls +(0,-.75,0) and +(+.05,0,-.05)  .. (qi2)
    ;
   \path[fill=gray!50!red,opacity=0.5]
   (qi2)
   .. controls +(-.15,0,+.15)  and  +(0,-1,0) .. (Sn)
   .. controls +(0,+.75,0) and +(+.2,0,+.2) .. (pii1) -- (pi1)
   .. controls +(+.2,0,+.2) and +(0,+.75,0)  ..
    (Wn) .. controls +(0,-1,0) and +(0,1,0) .. (Wb2) -- (Sb2) .. controls +(0,1,0) and +(0,-1,0) .. (Se)
    .. controls +(0,+.75,0) and +(+.2,0,+.2) .. (pii2) -- (pi2)
    .. controls +(+.2,0,+.2) and +(0,+.75,0)  .. (We)
    .. controls +(0,-.75,0) and +(+.05,0,-.05)  .. (qi2)
   ;
   \path[fill=gray!50!blue,opacity=0.5]
   (qi1) -- (qii1) 
   .. controls  +(+.05,0,-.05) and +(0,-.75,0) .. (Ns)
   .. controls +(0,+1.25,0)  and +(-.5,0,-.5)  .. (pi2)
    .. controls +(+.2,0,+.2) and +(0,+.75,0)  .. (We)
    .. controls +(0,-.75,0) and +(+.05,0,-.05)  .. (qi2)
    -- (qii2) .. controls  +(+.05,0,-.05) and +(0,-.75,0) .. (Ne) .. controls +(0,1,0) and +(0,-1,0) .. (Nt2) -- (Wt2)
    .. controls +(0,-1,0) and +(0,1,0) .. (Ws) .. controls +(0,-.75,0) and +(+.05,0,-.05)  .. (qi1)
   ;
   \draw[string, black!50!red] (Wb2) -- (Sb2);
   \draw[string]  (Sb2) .. controls +(0,1,0) and +(0,-1,0) .. (Se)
    .. controls +(0,+.75,0) and +(+.2,0,+.2) .. (pii2)
   .. controls +(-.5,0,-.5) and +(0,+1.25,0)  .. (Es)
   .. controls +(0,-1,0) and +(0,1,0) .. (Eb1)
   ;
   \draw[string]
   (Wt2)
    .. controls +(0,-1,0) and +(0,1,0) .. (Ws) .. controls +(0,-.75,0) and +(+.05,0,-.05)  .. (qi1)
   .. controls +(-.15,0,+.15)  and  +(0,-1,0) .. (Sw)
   .. controls +(0,1,0) and +(0,-1,0) .. (St1)
   ;
   \path[fill=gray!50!red,opacity=0.5]
   (Wt2)
    .. controls +(0,-1,0) and +(0,1,0) .. (Ws) .. controls +(0,-.75,0) and +(+.05,0,-.05)  .. (qi1)
   .. controls +(-.15,0,+.15)  and  +(0,-1,0) .. (Sw)
   .. controls +(0,1,0) and +(0,-1,0) .. (St1) -- (Wt1)
    .. controls +(0,-1,0) and +(0,1,0) .. (Ww) .. controls +(0,-1,0) and +(0,1,0) .. (Wb1)
    -- (Sb1) .. controls +(0,1,0) and +(0,-1,0) .. (Ss) .. controls +(0,1,0) and +(0,-1,0) .. (St2) -- (Wt2)
   ;
   \draw[string, black!50!red] (Wb1) -- (Sb1);
   \draw[string, black!50!red] (Wt1) -- (St1);
   \draw[string, black!50!red] (Wt2) -- (St2);
   \path[fill=gray!50!blue,opacity=0.5]
   (Sb2) .. controls +(0,1,0) and +(0,-1,0) .. (Se)
    .. controls +(0,+.75,0) and +(+.2,0,+.2) .. (pii2)
   .. controls +(-.5,0,-.5) and +(0,+1.25,0)  .. (Es)
   .. controls +(0,-1,0) and +(0,1,0) .. (Eb1)
   -- (Sb1) .. controls +(0,1,0) and +(0,-1,0) .. (Ss) .. controls +(0,1,0) and +(0,-1,0) .. (St2) -- (Et2)
   .. controls +(0,-1,0) and +(0,1,0) .. (Ee) .. controls +(0,-1,0) and +(0,1,0) .. (Eb2) -- (Sb2)
   ;
   \draw[string, black!50!blue] (Eb2) -- (Sb2);
   \draw[string, black!50!blue] (Eb1) -- (Sb1);
   \draw[string, black!50!blue] (Et2) -- (St2);
   \draw[string] (Sb1) .. controls +(0,1,0) and +(0,-1,0) .. (Ss) .. controls +(0,1,0) and +(0,-1,0) .. (St2);
\end{tikzpicture}
\quad=
\begin{tikzpicture}[baseline=(middle),scale=0.7]
  \path (0,0,0) coordinate (middle);
  \path
  (0,0,0)
  +(-1,0,0)  coordinate (W)
  +(+1,0,0)  coordinate (E)
  +(0,0,+1)  coordinate (S)
  +(0,0,-1)  coordinate (N)
  (+.25,-1.85,0)
  +(-.5,0,-.5) coordinate (pi)
  +(+.5,0,+.5) coordinate (pii)
  (+.05,1.85,0)
  +(-.5,0,+.5) coordinate (qi)
  +(+.5,0,-.5) coordinate (qii)
  ;
  \path
  (-1,3,-1)
  +(-1,0,0)  coordinate (Wt1)
  +(+1,0,0)  coordinate (Et1)
  +(0,0,+1)  coordinate (St1)
  +(0,0,-1)  coordinate (Nt1)
  (+1,3,+1)
  +(-1,0,0)  coordinate (Wt2)
  +(+1,0,0)  coordinate (Et2)
  +(0,0,+1)  coordinate (St2)
  +(0,0,-1)  coordinate (Nt2)
  (-1,-3,+1)
  +(-1,0,0)  coordinate (Wb1)
  +(+1,0,0)  coordinate (Eb1)
  +(0,0,+1)  coordinate (Sb1)
  +(0,0,-1)  coordinate (Nb1)
  (+1,-3,-1)
  +(-1,0,0)  coordinate (Wb2)
  +(+1,0,0)  coordinate (Eb2)
  +(0,0,+1)  coordinate (Sb2)
  +(0,0,-1)  coordinate (Nb2)
  ;
  \draw[string, black!50!red]  (Et1) -- (Nt1);
  \draw[string, black!50!blue] (Nt1) -- (Wt1);
  \draw[string, black!50!red]  (Et2) -- (Nt2);
  \draw[string] (Nt1) .. controls +(0,-1,0) and +(0,1,0) .. (N) .. controls +(0,-1,0) and +(0,1,0) .. (Nb2);
  \draw[string, black!50!red]  (Eb1) -- (Nb1);
  \draw[string, black!50!blue] (Nb1) -- (Wb1);
  \draw[string, black!50!red]  (Eb2) -- (Nb2);
  \draw[string, black!50!blue] (Nb2) -- (Wb2);
  \path[fill=gray!50!blue,opacity=0.5]
   (Nt1) .. controls +(0,-1,0) and +(0,1,0) .. (N) .. controls +(0,-1,0) and +(0,1,0) .. (Nb2) -- (Wb2)
   .. controls +(0,+.75,0) and +(+.2,0,+.2) .. (pi)
   .. controls +(-.5,0,-.5) and +(0,+1.25,0)  .. (Nb1) -- (Wb1)
   .. controls +(0,1,0) and +(0,-1,0) .. (W) .. controls +(0,1,0) and +(0,-1,0) .. (Wt1) -- (Nt1)
  ;
  \draw[string] 
   (Wb2)
   .. controls +(0,+.75,0) and +(+.2,0,+.2) .. (pi)
   .. controls +(-.5,0,-.5) and +(0,+1.25,0)  .. (Nb1)
  ;
  \draw[string]
   (Wb1)
   .. controls +(0,1,0) and +(0,-1,0) .. (W) .. controls +(0,1,0) and +(0,-1,0) .. (Wt1)
  ;
  \path[fill=gray!50!red,opacity=0.5]
   (Nt1) .. controls +(0,-1,0) and +(0,1,0) .. (N) .. controls +(0,-1,0) and +(0,1,0) .. (Nb2) -- (Eb2)
   .. controls +(0,1,0) and +(0,-1,0) .. (E) .. controls +(0,1,0) and +(0,-1,0) .. (Et2) -- (Nt2)
   .. controls +(0,-.75,0) and +(+.05,0,-.05)  .. (qii)
   .. controls +(-.15,0,+.15)  and  +(0,-1,0) .. (Et1) -- (Nt1)
  ;
  \draw[string]
   (Nt2)
   .. controls +(0,-.75,0) and +(+.05,0,-.05)  .. (qii)
   .. controls +(-.15,0,+.15)  and  +(0,-1,0) .. (Et1)
  ;
  \draw[string]
   (Eb2)
   .. controls +(0,1,0) and +(0,-1,0) .. (E) .. controls +(0,1,0) and +(0,-1,0) .. (Et2)
  ;
  \path[fill=gray!50!red,opacity=0.5] (Nb1) .. controls +(0,+1.25,0) and +(-.5,0,-.5) .. (pi) -- (pii) .. controls +(-.5,0,-.5) and +(0,+1.25,0)  ..  (Eb1) -- (Nb1);
  \path[fill=gray!50!red,opacity=0.5] (Sb2) .. controls +(0,+.75,0) and +(+.2,0,+.2) .. (pii) -- (pi) .. controls +(+.2,0,+.2) and +(0,+.75,0)  .. (Wb2) -- (Sb2);
  \draw[string, black!50!red] (Wb2) -- (Sb2);
  \draw[string] (Sb2) .. controls +(0,+.75,0) and +(+.2,0,+.2) .. (pii) .. controls +(-.5,0,-.5) and +(0,+1.25,0)  ..  (Eb1);
  \path[fill=gray!50!blue,opacity=0.5] (qii) .. controls +(-.15,0,+.15) and +(0,-1,0) .. (Et1) -- (St1) .. controls +(0,-1,0) and +(-.15,0,+.15) .. (qi) -- (qii);
  \draw[string] (Wt2) .. controls +(0,-.75,0) and +(+.05,0,-.05) .. (qi) .. controls +(-.15,0,+.15) and +(0,-1,0) .. (St1);
  \path[fill=gray!50!blue,opacity=0.5] (qi) .. controls +(+.05,0,-.05) and +(0,-.75,0)  .. (Wt2) -- (Nt2) .. controls +(0,-.75,0) and +(+.05,0,-.05) .. (qii) -- (qi);
  \draw[string, black!50!blue] (St1) -- (Et1);
  \draw[string, black!50!blue] (Wt2) -- (Nt2);
  \path[fill=gray!50!red,opacity=0.5]
   (Wt2) .. controls +(0,-.75,0) and +(+.05,0,-.05) .. (qi) .. controls +(-.15,0,+.15) and +(0,-1,0) .. (St1) -- (Wt1) 
   .. controls +(0,-1,0) and +(0,1,0) .. (W) .. controls +(0,-1,0) and +(0,1,0) .. (Wb1) -- (Sb1)
   .. controls +(0,1,0) and +(0,-1,0) .. (S) .. controls +(0,1,0) and +(0,-1,0) .. (St2) -- (Wt2)
  ;
  \draw[string, black!50!red] (Wt1) -- (St1);
  \draw[string, black!50!red] (Wt2) -- (St2);
  \draw[string, black!50!red] (Wb1) -- (Sb1);
  \path[fill=gray!50!blue,opacity=0.5]
    (Sb2) .. controls +(0,+.75,0) and +(+.2,0,+.2) .. (pii) .. controls +(-.5,0,-.5) and +(0,+1.25,0)  ..  (Eb1) 
    -- (Sb1) .. controls +(0,1,0) and +(0,-1,0) .. (S) .. controls +(0,1,0) and +(0,-1,0) .. (St2)
    -- (Et2) .. controls +(0,-1,0) and +(0,1,0) .. (E) .. controls +(0,-1,0) and +(0,1,0) .. (Eb2) -- (Sb2)
  ;
  \draw[string, black!50!blue] (Et2) -- (St2);
  \draw[string, black!50!blue] (Eb2) -- (Sb2);
  \draw[string, black!50!blue] (Eb1) -- (Sb1);
  \draw[string] (Sb1) .. controls +(0,1,0) and +(0,-1,0) .. (S) .. controls +(0,1,0) and +(0,-1,0) .. (St2) ;
\end{tikzpicture}
\label{eqn:bial}
\end{gather}

As the second-named author explained in the talk \cite{David2017Talk}, these axioms hold for topological reasons: for example, the triviality of the bigon $1 \overset\alpha\Rightarrow gf \overset{\alpha^{-1}}\Rightarrow 1$ is what allows the ``neck-cutting'' of \eqref{eqn:coaug} and \eqref{eqn:aug}, and the triviality of the other composite $gf  \overset{\alpha^{-1}}\Rightarrow 1 \overset\alpha\Rightarrow gf$ is responsible for the ``hole-boring'' of \eqref{eqn:bial}.

\begin{remark} In~\cite{David2017Talk}, it is also explained how dropping the triviality of the composite $1 \Rightarrow gf \Rightarrow 1$ and only requiring triviality of  $gf  \Rightarrow 1 \Rightarrow gf$   leads to \emph{weak Hopf algebras} in the sense of \cite{MR1726707} instead of true Hopf algebras. 
\end{remark}

Such three-dimensional arguments are dangerous, however, when there is incomplete adjunctibility. For example, \cite{David2017Talk} supplies also a topological proof of the existence of an antipode which requires extra adjunctibility beyond what is mandated by $\Adj \owedge \Adj$.

\subsection{Extra dualizability and (co)integrals} \label{subsec:extradualizability} The goal of this section is to explore this extra adjunctibility.

\begin{prop}\label{prop:dualizability}
Consider an adjunctible retract in a pointed $(\infty,3)$-category $(\cC, 1_{\cC})$, i.e.\ a section-retraction pair $f: 1_{\cC} \to X$, $g:X \to 1_{\cC}$, and $\alpha: gf \simeq \id_{1_{\cC}}$  for which $f$ admits a right adjoint $f^R$, $g$ admits a left adjoint $g^L$, and $\alpha$ satisfies one (and hence all) of the equivalent conditions from Corollary~\ref{cor:smashretract}: 
\begin{itemize}
\item $\alpha^{\rmate}$ admits a left adjoint; 
\item $\alpha^{\lmate}$ admits a right adjoint;
\item $g\,\ev_f$ admits a left adjoint; 
\item $\ev_g\, f$ admits a right adjoint.
\end{itemize}
If it moreover also satisfies one (and hence all) of the following equivalent conditions
\begin{itemize}
\item $\alpha^{\rmate}$ admits a right adjoint; 
\item $\alpha^{\lmate}$ admits a left adjoint;
\item $g\,\ev_f$ admits a  right adjoint; 
\item $\ev_g \,f$ admits a left adjoint,
\end{itemize}
Then the underlying object of the Hopf algebra $H(X,f,g,\alpha)$ is dualizable in the braided monoidal $(\infty,1)$-category $\Omega^2 \cC$.
\end{prop}
\begin{proof}By Lemma~\ref{lem:adjointsofmates}, these conditions are indeed equivalent. Dualizability of $H(X, f,g,\alpha)$ is immediate from Remark~\ref{rem:simplesquare} or also directly from~\eqref{eqn:unpackedpasting}. 
\end{proof}

Dualizability of a Hopf algebra is closely related to the existence of integrals and cointegrals on the underlying bialgebra.
\begin{definition}Let $\cB$ be a $\EE_2$-monoidal $(\infty,1)$-category, $I \in \cB$ an object and $H \in \cB$ a bialgebra. An \define{$I$-valued left integral} on $H$ is a left  $H$-comodule map $\mathrm{int} : H \to I$, where $I$ is made into an $H$-comodule via the unit. Dually, an \define{$I$-covalued left cointegral} in $H$ is a left module map $\mathrm{coint}: I \to H$, where $I$ is made into an $H$-module via the counit.\end{definition}
The following fundamental result in the theory of Hopf algebras in braided monoidal categories is proved in \cite{MR1685417,MR1759389}:
\begin{proposition}[\cite{MR1685417,MR1759389}]\label{prop:integrals}
  For a dualizable bialgebra $H$ in a braided monoidal $(1,1)$-category $\cB$, the following are equivalent:
\begin{enumerate}
  \item $H$ is a Hopf algebra. \label{integral1}
  \item $H$ is a Hopf algebra with invertible antipode. \label{integral2}
  \item \label{integral3} There exists unique (up to unique isomorphism) invertible objects $I_{\mathrm{int}}$ and $I_{\mathrm{coint}}$ in the Karoubi completion of $\cB$ and an $I_{\mathrm{int}}$-valued left integral $\mathrm{int}$ and an $I_{\mathrm{coint}}$-covalued left cointegral $\mathrm{coint}$  such that the composition $I_{\mathrm{coint}} \to[\mathrm{coint}] H \to[\mathrm{int}] I_{\mathrm{int}}$ is an identity.\qed
\end{enumerate}
\end{proposition}

To show that \ref{integral3} implies \ref{integral1}, given an integral and cointegral as in \ref{integral3}, one verifies that the following defines an antipode: 
\begin{multline}\label{eqn:integralformulaforantipode}
H \simeq H \otimes I_{\mathrm{int}}^{-1}\otimes I_{\mathrm{int}} 
\to[H \otimes I_{\mathrm{int}}^{-1} \otimes (\mathrm{int} \circ \mathrm{coint})^{-1}]  H \otimes I_{\mathrm{int}}^{-1}\otimes I_{\mathrm{coint}} 
\to[H \otimes I_{\mathrm{coint}}^{-1} \otimes \mathrm{coint}] H \otimes I_{\mathrm{int}}^{-1} \otimes H
\\ \to[ H \otimes I_{\mathrm{int}}^{-1} \otimes \Delta] H \otimes I_{\mathrm{int}}^{-1} \otimes H \otimes H
 \to[\mathrm{br}^{-1}_{H, I_{\mathrm{int}}^{-1} \otimes H} \otimes H] I_{\mathrm{int}}^{-1} \otimes H \otimes H \otimes H
 \\ \to[I_{\mathrm{int}}^{-1} \otimes m \otimes H] I_{\mathrm{int}}^{-1} \otimes H \otimes H \to [I_{\mathrm{int}}^{-1} \otimes \mathrm{int} \otimes H] I_{\mathrm{int}}^{-1} \otimes I_{\mathrm{int}} \otimes H \simeq H.
\end{multline}
The other (more difficult) implications are proven in \cite{MR1685417,MR1759389}.

Examples are given in \cite{MR1685417,MR1759389} of Hopf algebras $H$ whose associated invertible object $I$ is not isomorphic to the unit object $1 \in \cB$. In particular, in these examples, $I_{\mathrm{int}} \cong I_{\mathrm{coint}}$ is not an object of the full subcategory of $\cB$ on the tensor powers of $H$. This shows that the Karoubi completion in condition \ref{integral3} of the Proposition can be necessary.

\begin{remark} Although we will not prove it in this paper, we expect the proof of Proposition~\ref{prop:integrals} to carry over to the $\infty$-categorical context. We do however point out that it does not suffice to apply Proposition~\ref{prop:integrals} to the homotopy category $h_1\cB$ of a braided monoidal $(\infty,1)$-category $\cB$ since Karoubi completeness and Karoubi completion  are not compatible with taking homotopy categories \cite[Warning~1.2.4.8]{HA}.
\end{remark}

In the remainder of this section, we will show that if we assume more adjunctibility than mandated in Proposition~\ref{prop:dualizability}, then these universal (co)integral and cointegral for our Hopf algebra $H(X,f,g,\alpha)$ already exist ($\infty$-categorically coherently) in $\Omega^2\cC$  without Karoubi completion. 

Indeed, the topological pictures in Section~\ref{subsec:unpacked} might lead one to think that, in the presence of sufficient adjunctibility, there should be a valid morphism $H(X,f,g,\alpha) \to 1$ of the following shape:
\begin{equation}\label{eqn:nonexistentintegral}
\begin{tikzpicture}[scale=-1,baseline=(middle), scale=0.7]
  \path 
  (0,0,0)
  +(0,-.5,0) coordinate (middle)
  +(-1,0,0)  coordinate (Wu)
  +(+1,0,0)  coordinate (Eu)
  +(0,0,+1)  coordinate (Su)
  +(0,0,-1)  coordinate (Nu)
  ++(-.25,-1.15,0)
  +(-.5,0,-.5) coordinate (boti)
  +(+.5,0,+.5) coordinate (botii)
  ;
  \draw[string, black!50!red]  (Eu) -- (Nu);
  \draw[string, black!50!blue] (Nu) -- (Wu);
  \draw[string] (Wu) .. controls +(0,-.75,0) and +(-.2,0,-.2) .. (boti) .. controls +(+.5,0,+.5) and +(0,-1.25,0) .. (Nu);
  \path[fill=gray!50!blue,opacity=0.5] (Wu) .. controls +(0,-.75,0) and +(-.2,0,-.2) .. (boti) .. controls +(+.5,0,+.5) and +(0,-1.25,0) .. (Nu) -- (Wu);
  \path[fill=gray!50!red,opacity=0.5] (boti) .. controls +(+.5,0,+.5) and +(0,-1.25,0) .. (Nu) -- (Eu) .. controls +(0,-1.25,0) and +(+.5,0,+.5) .. (botii) -- (boti);
  \path[fill=gray!50!red,opacity=0.5] (Wu) .. controls +(0,-.75,0) and +(-.2,0,-.2) .. (boti) -- (botii) .. controls +(-.2,0,-.2) and +(0,-.75,0)  .. (Su) -- (Wu);
  \path[fill=gray!50!blue,opacity=0.5] (Su) .. controls +(0,-.75,0) and +(-.2,0,-.2) .. (botii) .. controls +(+.5,0,+.5) and +(0,-1.25,0) .. (Eu) -- (Su);
  \draw[string, black!50!red]  (Wu) -- (Su);
  \draw[string, black!50!blue] (Su) -- (Eu);
  \draw[string] (Su) .. controls +(0,-.75,0) and +(-.2,0,-.2) .. (botii) .. controls +(+.5,0,+.5) and +(0,-1.25,0) .. (Eu);
\end{tikzpicture}
\end{equation}
It turns out that one cannot draw the morphism \eqref{eqn:nonexistentintegral} due to ``framing'' data suppressed in this topological notation: 
the
Tangle Hypothesis \cite{baezdolan,lurie} describes categories with adjunctibility in terms of diagrams made out of framed manifolds, and the acts of clockwise or counterclockwise $180^\circ$-rotation of the framing are not homotopic. (We will not assume the Tangle Hypothesis in our formulas, but we will use it as a source of inspiration for the formulas themselves.) Indeed, if there were not framing issues to contend with, then as explained in \cite{David2017Talk}, the above morphism would be a left and right integral valued in the unit object, and $H(X,f,g,\alpha)$ would be a unimodular Hopf algebra. The goal of the remainder of this section is to write down formulas for the integrals and cointegrals on $H(X,f,g,\alpha)$ that keep track of all framing data.

\begin{definition} \label{def:fullyadjunctible} A $1$-morphism $f:B\to C$ in an $(\infty,3)$-category is \emph{fully adjunctible} if it admits a right adjoint $f^R$ and its counit and unit  $\ev_f: ff^R \Rightarrow \id_{C}$ and $\coev_f : \id_{1_\cC} \Rightarrow f^Rf$ themselves admit both left and right adjoint $2$-morphisms. 
\end{definition}
We will justify the terminology in Remark~\ref{rem:justifyfullyadj}.
Both the choice of a right adjoint and of a left adjoint of $\ev_f$ and $\coev_f$ exhibit $f^R$ as a  left adjoint of $f$. The intertwining isomorphism between these two identification is given by the following: 

\begin{definition}\label{def:Radford}
Given a fully adjunctible $1$-morphism $f:B\to C$ in an  $(\infty,3)$-category,  its \define{Radford automorphism} $\Rad_f : f \Isom f$ is the composition \begin{equation*} 
  \begin{tikzpicture}[xscale=1.5,baseline=(C.base)]
    \path
    (0,0) coordinate (beginning)
    (0,1.5) coordinate (Rad)
    (0,3) coordinate (ending)
    ;
    \draw[string, black!50!red] (beginning) -- node[auto,swap] {$\scriptstyle f$} (Rad);
    \draw[string, black!50!red] (Rad) -- node[auto,swap] {$\scriptstyle f$} (ending);
    \path[fill=gray!50!orange,opacity=0.25] (beginning) -- (Rad) -- (ending) -- ++(-1.5,0) -- ++(0,-3) -- cycle;
    \path[fill=gray!50!yellow,opacity=0.25] (beginning) -- (Rad) -- (ending) -- ++(+1.5,0) -- ++(0,-3) -- cycle;
    \path (Rad) node[draw, black!50!red, ellipse, inner sep=0.5pt,fill=white]  {$\scriptstyle \Rad_f$} ;
    \path (-1,1.5) node (C) {$X$} (+1, 1.5) node {$B$};
  \end{tikzpicture}
  \quad := \quad
    \begin{tikzpicture}[xscale=1.5,baseline=(C.base)]
      \path
      (2.5,0) coordinate (beginning) 
      (2,2) coordinate (coevR) 
      (1,1) coordinate (evL)
      (.5,3) coordinate (ending) 
      ;
      \draw[string, black!50!red] (beginning) .. controls +(0,1) and +(.5,0) .. node[auto,swap,pos=0.3] {$\scriptstyle f$} (coevR);
      \draw[string, black!50!red] (coevR) .. controls +(-.5,0) and +(.5,0) .. node[auto] {$\scriptstyle f^R$} (evL);
      \draw[string, black!50!red] (evL) .. controls +(-.5,0) and +(0,-1) .. node[auto,swap] {$\scriptstyle f$} (ending);
      \path[fill=gray!50!orange,opacity=0.25] (beginning) .. controls +(0,1) and +(.5,0) ..  (coevR) .. controls +(-.5,0) and +(.5,0) .. (evL) .. controls +(-.5,0) and +(0,-1) .. (ending) -- ++(-1,0) -- ++(0,-3) -- cycle;
      \path[fill=gray!50!yellow,opacity=0.25] (beginning) .. controls +(0,1) and +(.5,0) ..  (coevR) .. controls +(-.5,0) and +(.5,0) .. (evL) .. controls +(-.5,0) and +(0,-1) .. (ending) -- ++(3,0) -- ++(0,-3) -- cycle;
      \path 
      (coevR) node[draw, black!50!red, ellipse, inner sep=0.5pt,fill=white] {$\scriptstyle \coev_f^R$}
      (evL) node[draw, black!50!red, ellipse, inner sep=0.5pt,fill=white]  {$\scriptstyle \ev_f^L$}
      (0,1.5) node (C) {$X$} (3,1.5) node {$B$}
      ;
    \end{tikzpicture}.
  \end{equation*}
\end{definition}
The Radford automorphism is indeed invertible with inverse 
 \begin{equation} \label{eqn:radfordinverse}
  \begin{tikzpicture}[xscale=1.5,baseline=(C.base)]
    \path
    (0,0) coordinate (beginning)
    (0,1.5) coordinate (Rad)
    (0,3) coordinate (ending)
    ;
    \draw[string, black!50!red] (beginning) -- node[auto,swap] {$\scriptstyle f$} (Rad);
    \draw[string, black!50!red] (Rad) -- node[auto,swap] {$\scriptstyle f$} (ending);
    \path[fill=gray!50!orange,opacity=0.25] (beginning) -- (Rad) -- (ending) -- ++(-1.5,0) -- ++(0,-3) -- cycle;
    \path[fill=gray!50!yellow,opacity=0.25] (beginning) -- (Rad) -- (ending) -- ++(+1.5,0) -- ++(0,-3) -- cycle;
    \path (Rad) node[draw, black!50!red, ellipse, inner sep=0.5pt,fill=white]  {$\scriptstyle \Rad_f^{-1}$} ;
    \path (-1,1.5) node (C) {$X$} (+1, 1.5) node {$B$};
  \end{tikzpicture}
  \quad \simeq \quad
    \begin{tikzpicture}[xscale=1.5,baseline=(C.base)]
      \path
      (2.5,0) coordinate (beginning) 
      (2,2) coordinate (coevR) 
      (1,1) coordinate (evL)
      (.5,3) coordinate (ending) 
      ;
      \draw[string, black!50!red] (beginning) .. controls +(0,1) and +(.5,0) .. node[auto,swap,pos=0.3] {$\scriptstyle f$} (coevR);
      \draw[string, black!50!red] (coevR) .. controls +(-.5,0) and +(.5,0) .. node[auto] {$\scriptstyle f^R$} (evL);
      \draw[string, black!50!red] (evL) .. controls +(-.5,0) and +(0,-1) .. node[auto,swap] {$\scriptstyle f$} (ending);
       \path[fill=gray!50!orange,opacity=0.25] (beginning) .. controls +(0,1) and +(.5,0) ..  (coevR) .. controls +(-.5,0) and +(.5,0) .. (evL) .. controls +(-.5,0) and +(0,-1) .. (ending) -- ++(-1,0) -- ++(0,-3) -- cycle;
      \path[fill=gray!50!yellow,opacity=0.25] (beginning) .. controls +(0,1) and +(.5,0) ..  (coevR) .. controls +(-.5,0) and +(.5,0) .. (evL) .. controls +(-.5,0) and +(0,-1) .. (ending) -- ++(3,0) -- ++(0,-3) -- cycle;
      \path 
      (coevR) node[draw, black!50!red, ellipse, inner sep=0.5pt,fill=white] {$\scriptstyle \coev_f^L$}
      (evL) node[draw, black!50!red, ellipse, inner sep=0.5pt,fill=white]  {$\scriptstyle \ev_f^R$}
      (0,1.5) node (C) {$X$} (3,1.5) node {$B$}
      ;
    \end{tikzpicture}.
  \end{equation}

  \begin{remark}\label{rem:justifyfullyadj}
  If $f$ is fully adjunctible, then any choice of left or right dual of $\ev_f$ and $\coev_f$ exhibits $f^R$ as a left adjoint of $f$; hence $f$ is part of a $2$-periodic chain of adjoints 
  \[   \cdots \dashv f^L \dashv f \dashv f^R \simeq f^L \dashv f \dashv \cdots
  \]
Moreover, by definition of the Radford automorphism, \begin{equation}\label{eq:Radev}\ev_f^L \simeq (\Rad_f f^R) \circ  \ev_f^R  \hspace{1cm} \coev_f^L \circ (f^R \Rad_f) \simeq \coev_f^R \end{equation}  and hence both $\ev_f$ and $\coev_f$ are part of an infinite chain of adjoints
 \[   \cdots  \dashv  \ev_f \circ  (\Rad_f^{-1} f^R)  \dashv (\Rad_f f^R)\circ ev_f^R   \dashv \ev_f \dashv \ev_f^R \dashv \ev_f  \circ (\Rad_f f^R) \dashv \cdots \]  and similarly for $\coev_f$, justifying the terminology ``fully adjunctible'' in Definition~\ref{def:fullyadjunctible}. Compare~\cite{AraujoThesis}.
  \end{remark}

Up to isomorphism, Definition~\ref{def:Radford} does not depend on the various choices of adjunction data. In particular, it follows that if $f:B \to C$ and $g:C \to D$ are fully adjunctible $1$-morphisms, then there is a canonical isomorphism \begin{equation}\label{eq:Radcomposition}
\Rad_g\Rad_f \simeq \Rad_{g \circ f},\end{equation} and if $\beta: f \To f'$ is a $2$-isomorphism, then \begin{equation}\label{eq:Radiso}\Rad_{f'} \circ \beta \simeq \beta \circ \Rad_{f}.\end{equation} 

\begin{remark}
The isomorphism~\eqref{eq:Radiso} can be generalized to non-invertible $2$-morphisms $\beta: f \To f'$ which admit an infinite chain of adjoints $\cdots \dashv \beta^L \dashv \beta \dashv \beta^R \dashv \cdots $ in which case\footnote{Equivalently, we can define a $3$-isomorphism  $ {\beta^{LL}} \circ {\Rad_f} \simeq {\Rad_{f'}} \circ {\beta^{RR}}.$  Invertibility of $\Rad_f$ implies that $\Rad_f \simeq (\Rad_f^{-1})^{L} \simeq (\ev_f f) \circ (f \coev_f^{LL})$.  For a $2$-morphism $\beta: f \To f'$, we let $\beta^{\rrmate}: (f')^R \To f^R$ and $\beta^{\llmate}: (f')^L \To f^L$ denote their ``full mates/180-degree rotations.''  Together with functoriality  of $(-)^{LL}$, we can recognize  $ \beta^{LL} \circ \Rad_f \simeq \Rad_{f'} \circ \beta^{{\rrmate} LL {\llmate}}.$  On the other hand, Lemma~\ref{lem:adjointsofmates} supplies an isomorphism  $ \beta^{{\rrmate} LL {\llmate}} \simeq \beta^{RR}. $}
\[\beta  \circ \Rad_f \simeq \Rad_{f'} \circ \beta^{RRRR}.
\]
These isomorphisms are natural in $\beta$ and should define an identity-on-objects natural isomorphism between the identity functor $\id_{\cC}$ and the functor  $(-)^{RRRR}: \cC \to \cC$, whose existence is a categorical manifestation of the isomorphism $\pi_1 O(3) = \mathbb{Z}/2$. We will not develop this further, but see~\cite{DSPS} from whence we get the name ``Radford.''  
\end{remark}

In the following, it will be useful to express the Radford automorphism in terms of a left adjoint of a $1$-morphism. 

    \begin{lemma} \label{lem:RadForL}A $1$-morphism $g:B \to C$ in an $(\infty,3)$-category is fully adjunctible if and only if it admits a left adjoint $g^L$ whose counit and unit $\ev_g: g^L g \To \id_{B} $ and $\coev_g : \id_A \To g g^L$ admit both left and right adjoints. In this case, the Radford automorphism is isomorphic to   
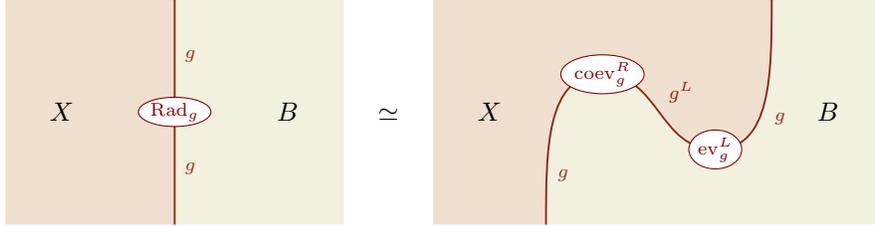
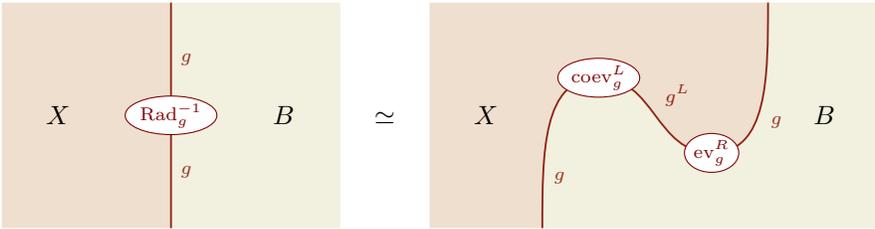
\[\begin{tikzpicture}[xscale=1.5,baseline=(C.base)]
    \path
    (0,0) coordinate (beginning)
    (0,1.5) coordinate (Rad)
    (0,3) coordinate (ending)
    ;
    \draw[string, black!50!red] (beginning) -- node[auto,swap] {$\scriptstyle g$} (Rad);
    \draw[string, black!50!red] (Rad) -- node[auto,swap] {$\scriptstyle g$} (ending);
    \path[fill=gray!50!orange,opacity=0.25] (beginning) -- (Rad) -- (ending) -- ++(-1.5,0) -- ++(0,-3) -- cycle;
    \path[fill=gray!50!yellow,opacity=0.25] (beginning) -- (Rad) -- (ending) -- ++(+1.5,0) -- ++(0,-3) -- cycle;
    \path (Rad) node[draw, black!50!red, ellipse, inner sep=0.5pt,fill=white]  {$\scriptstyle \Rad_g$} ;
    \path (-1,1.5) node (C) {$X$} (+1, 1.5) node {$B$};
  \end{tikzpicture}
  \quad \simeq \quad
    \begin{tikzpicture}[xscale=-1.5,baseline=(C.base)]
      \path
      (2.5,0) coordinate (beginning) 
      (2,2) coordinate (coevR) 
      (1,1) coordinate (evL)
      (.5,3) coordinate (ending) 
      ;
      \draw[string, black!50!red] (beginning) .. controls +(0,1) and +(.5,0) .. node[auto,swap,pos=0.3] {$\scriptstyle g$} (coevR);
      \draw[string, black!50!red] (coevR) .. controls +(-.5,0) and +(.5,0) .. node[auto] {$\scriptstyle g^L$} (evL);
      \draw[string, black!50!red] (evL) .. controls +(-.5,0) and +(0,-1) .. node[auto,swap] {$\scriptstyle g$} (ending);
      \path[fill=gray!50!yellow, opacity=0.25] (beginning) .. controls +(0,1) and +(.5,0) ..  (coevR) .. controls +(-.5,0) and +(.5,0) .. (evL) .. controls +(-.5,0) and +(0,-1) .. (ending) -- ++(-1,0) -- ++(0,-3) -- cycle;
      \path[fill=gray!50!orange,opacity=0.25] (beginning) .. controls +(0,1) and +(.5,0) ..  (coevR) .. controls +(-.5,0) and +(.5,0) .. (evL) .. controls +(-.5,0) and +(0,-1) .. (ending) -- ++(3,0) -- ++(0,-3) -- cycle;
      \path 
      (coevR) node[draw, black!50!red, ellipse, inner sep=0.5pt,fill=white] {$\scriptstyle \coev_g^R$}
      (evL) node[draw, black!50!red, ellipse, inner sep=0.5pt,fill=white]  {$\scriptstyle \ev_g^L$}
      (0,1.5) node (C) {$B$} (3,1.5) node {$X$}
      ;
    \end{tikzpicture}
 \]
 with inverse
 \[\begin{tikzpicture}[xscale=1.5,baseline=(C.base)]
    \path
    (0,0) coordinate (beginning)
    (0,1.5) coordinate (Rad)
    (0,3) coordinate (ending)
    ;
    \draw[string, black!50!red] (beginning) -- node[auto,swap] {$\scriptstyle g$} (Rad);
    \draw[string, black!50!red] (Rad) -- node[auto,swap] {$\scriptstyle g$} (ending);
    \path[fill=gray!50!orange,opacity=0.25] (beginning) -- (Rad) -- (ending) -- ++(-1.5,0) -- ++(0,-3) -- cycle;
    \path[fill=gray!50!yellow,opacity=0.25] (beginning) -- (Rad) -- (ending) -- ++(+1.5,0) -- ++(0,-3) -- cycle;
    \path (Rad) node[draw, black!50!red, ellipse, inner sep=0.5pt,fill=white]  {$\scriptstyle \Rad_g^{-1} $} ;
    \path (-1,1.5) node (C) {$X$} (+1, 1.5) node {$B$};
  \end{tikzpicture}
  \quad \simeq \quad
    \begin{tikzpicture}[xscale=-1.5,baseline=(C.base)]
      \path
      (2.5,0) coordinate (beginning) 
      (2,2) coordinate (coevR) 
      (1,1) coordinate (evL)
      (.5,3) coordinate (ending) 
      ;
      \draw[string, black!50!red] (beginning) .. controls +(0,1) and +(.5,0) .. node[auto,swap,pos=0.3] {$\scriptstyle g$} (coevR);
      \draw[string, black!50!red] (coevR) .. controls +(-.5,0) and +(.5,0) .. node[auto] {$\scriptstyle g^L$} (evL);
      \draw[string, black!50!red] (evL) .. controls +(-.5,0) and +(0,-1) .. node[auto,swap] {$\scriptstyle g$} (ending);
      \path[fill=gray!50!yellow, opacity=0.25] (beginning) .. controls +(0,1) and +(.5,0) ..  (coevR) .. controls +(-.5,0) and +(.5,0) .. (evL) .. controls +(-.5,0) and +(0,-1) .. (ending) -- ++(-1,0) -- ++(0,-3) -- cycle;
      \path[fill=gray!50!orange,opacity=0.25] (beginning) .. controls +(0,1) and +(.5,0) ..  (coevR) .. controls +(-.5,0) and +(.5,0) .. (evL) .. controls +(-.5,0) and +(0,-1) .. (ending) -- ++(3,0) -- ++(0,-3) -- cycle;
      \path 
      (coevR) node[draw, black!50!red, ellipse, inner sep=0.5pt,fill=white] {$\scriptstyle \coev_g^L$}
      (evL) node[draw, black!50!red, ellipse, inner sep=0.5pt,fill=white]  {$\scriptstyle \ev_g^R$}
      (0,1.5) node (C) {$B$} (3,1.5) node {$X$}
      ;
    \end{tikzpicture}.
 \]
 \end{lemma}
 \begin{proof}
The right adjoints $\ev_g^R$ and $\coev_g^R$ exhibit $g^L$ as a right adjoint of $g$. The unit $\ev_g^R$ and counit $\coev_g^R$ of this ``new'' adjunction $g \dashv g^L$ by definition have left adjoints.  Similar to Remark~\ref{rem:justifyfullyadj}, we will now show that they also have right adjoints. Denote the right-hand side of the above equation by $R_g$. It follows that $ ( g^L R_g)  \circ \ev_g^R  \simeq \ev_g^L$ and that $\coev_g^L \circ (R_g g^L) \simeq \coev_g^R$.  Hence, the unit and counit also have right adjoints, namely $\ev_g \circ (g^L R_g)$ and $(R_g^{-1} g^L) \circ \coev_g$, respectively. Using these choices of unit and counit, the Radford isomorphism of $g$ unpacks to $\Rad_g \simeq (g \ev_g) \circ (\coev_g g ) \circ R_g \simeq R_g$.  \end{proof}

With this set up, we are now ready to describe the canonical left integral and cointegral for $H(X,f,g,\alpha)$.
\begin{definition}
  Given a retract $(X,f,g,\alpha)$ in a pointed $(\infty,3)$-category for which both $f$ and $g$ are fully adjunctible. Its \define{Radford bubbles}  $I_{\mathrm{int}}(X,f,g,\alpha), I_{\mathrm{coint}}(X,f,g,\alpha) \in \Omega^2\cC$ are the following composites:
  \begin{align*}
  I_{\mathrm{coint}}(X,f,g,\alpha)  :=&  
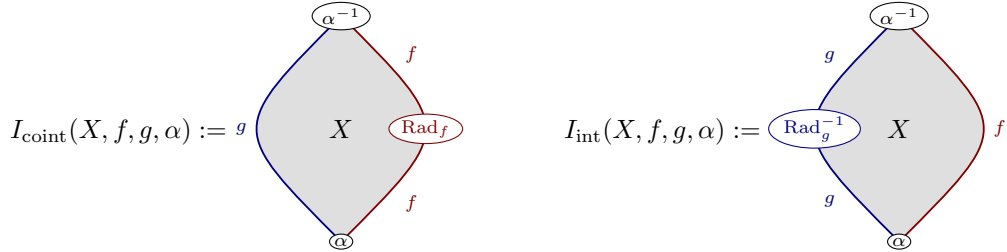
  \begin{tikzpicture}[yscale=1.5, baseline=(C.base)]
    \path
    (0,0) coordinate (alpha)
    (0,2) coordinate (alpham)
    ;
    \path[fill=gray,opacity=0.25] (alpha) .. controls +(+1.5,1) and +(+1.5,-1) .. (alpham) .. controls +(-1.5,-1) and +(-1.5,1) .. (alpha);
    \draw[string, black!50!red] (alpha) .. controls +(+1.5,1) and +(+1.5,-1) .. node[draw, black!50!red, ellipse, line width=0.4pt, inner sep=0.5pt,fill=white] {$\scriptstyle \Rad_f$} node[auto,swap,pos=0.2] {$\scriptstyle f$} node[auto,swap,pos=0.8] {$\scriptstyle f$} (alpham);
    \draw[string, black!50!blue] (alpha) .. controls +(-1.5,1) and +(-1.5,-1) .. node [auto] {$\scriptstyle g$} (alpham);
    \path 
    (0,1) node (C) {$X$}
    (alpha) node[draw, ellipse, inner sep=0.5pt,fill=white] {$\scriptstyle \alpha$}
    (alpham) node[draw, ellipse, inner sep=0.5pt,fill=white] {$\scriptstyle \alpha^{-1}$}
    ;
  \end{tikzpicture}
&\qquad I_{\mathrm{int}}(X,f,g,\alpha)  :=&  \;
  \begin{tikzpicture}[xscale=-1,yscale=1.5, baseline=(C.base)]
    \path
    (0,0) coordinate (alpha)
    (0,2) coordinate (alpham)
    ;
    \path[fill=gray,opacity=0.25] (alpha) .. controls +(+1.5,1) and +(+1.5,-1) .. (alpham) .. controls +(-1.5,-1) and +(-1.5,1) .. (alpha);
    \draw[string, black!50!blue] (alpha) .. controls +(+1.5,1) and +(+1.5,-1) .. node[draw,  ellipse, line width=0.4pt, inner sep=0.5pt,fill=white] {$\scriptstyle \Rad_g^{-1}$} node[auto,pos=0.2] {$\scriptstyle g$} node[auto,pos=0.8] {$\scriptstyle g$} (alpham);
    \draw[string, black!50!red] (alpha) .. controls +(-1.5,1) and +(-1.5,-1) .. node [auto,swap] {$\scriptstyle f$} (alpham);
    \path 
    (0,1) node (C) {$X$}
    (alpha) node[draw, ellipse, inner sep=0.5pt,fill=white] {$\scriptstyle \alpha$}
    (alpham) node[draw, ellipse, inner sep=0.5pt,fill=white] {$\scriptstyle \alpha^{-1}$}
    ;
  \end{tikzpicture}
\\
\intertext{  The Radford bubbles are manifestly invertible, with inverse}
  I_{\mathrm{coint}}(X,f,g,\alpha)^{-1}  \simeq &  
  \begin{tikzpicture}[yscale=1.5, baseline=(C.base)]
    \path
    (0,0) coordinate (alpha)
    (0,2) coordinate (alpham)
    ;
    \path[fill=gray,opacity=0.25] (alpha) .. controls +(+1.5,1) and +(+1.5,-1) .. (alpham) .. controls +(-1.5,-1) and +(-1.5,1) .. (alpha);
    \draw[string, black!50!red] (alpha) .. controls +(+1.5,1) and +(+1.5,-1) .. node[draw, black!50!red, ellipse, line width=0.4pt, inner sep=0.5pt,fill=white] {$\scriptstyle \Rad_f^{-1}$} node[auto,swap,pos=0.2] {$\scriptstyle f$} node[auto,swap,pos=0.8] {$\scriptstyle f$} (alpham);
    \draw[string, black!50!blue] (alpha) .. controls +(-1.5,1) and +(-1.5,-1) .. node [auto] {$\scriptstyle g$} (alpham);
    \path 
    (0,1) node (C) {$X$}
    (alpha) node[draw, ellipse, inner sep=0.5pt,fill=white] {$\scriptstyle \alpha$}
    (alpham) node[draw, ellipse, inner sep=0.5pt,fill=white] {$\scriptstyle \alpha^{-1}$}
    ;
  \end{tikzpicture} 
&\qquad I_{\mathrm{coint}}^{-1}(X,f,g,\alpha)  \simeq&  \;
  \begin{tikzpicture}[xscale=-1,yscale=1.5, baseline=(C.base)]
    \path
    (0,0) coordinate (alpha)
    (0,2) coordinate (alpham)
    ;
    \path[fill=gray,opacity=0.25] (alpha) .. controls +(+1.5,1) and +(+1.5,-1) .. (alpham) .. controls +(-1.5,-1) and +(-1.5,1) .. (alpha);
    \draw[string, black!50!blue] (alpha) .. controls +(+1.5,1) and +(+1.5,-1) .. node[draw,  ellipse, line width=0.4pt, inner sep=0.5pt,fill=white] {$\scriptstyle \Rad_g$} node[auto,pos=0.2] {$\scriptstyle g$} node[auto,pos=0.8] {$\scriptstyle g$} (alpham);
    \draw[string, black!50!red] (alpha) .. controls +(-1.5,1) and +(-1.5,-1) .. node [auto,swap] {$\scriptstyle f$} (alpham);
    \path 
    (0,1) node (C) {$X$}
    (alpha) node[draw, ellipse, inner sep=0.5pt,fill=white] {$\scriptstyle \alpha$}
    (alpham) node[draw, ellipse, inner sep=0.5pt,fill=white] {$\scriptstyle \alpha^{-1}$}
    ;
  \end{tikzpicture}
  \end{align*}
  \end{definition}
  
  \begin{remark}
    The Radford bubbles $I_{\mathrm{coint}}(X,f,g,\alpha)$ and $I_{\mathrm{int}}(X,f,g,\alpha)$ are isomorphic by the naturality equations~\eqref{eq:Radcomposition} and~\eqref{eq:Radiso}.
  \end{remark}
  
  \begin{prop}\label{prop:intcoint}
  Given a section-retraction pair $(X,f,g,\alpha)$ in a pointed $(\infty,3)$-category $(\cC, 1_{\cC})$ for which  $f$ and $g$ are fully adjunctible. Then, $(X,f,g,\alpha)$ satisfies the conditions of Proposition~\ref{prop:dualizability} and hence gives rise to a dualizable Hopf algebra $H(X,f,g,\alpha)$ in $\Omega^2 \cC$. Moreover, the composite 
    \[
  I_{\mathrm{coint}}(X,f,g,\alpha)~\simeq
   \begin{tikzpicture}[yscale=0.75, xscale=1,scale=0.8, baseline=(C.center)]
    \path
    (0,0) coordinate (alpha)
    (0,1) coordinate (bot)
    (1,2) coordinate (botR)
   (1,3) coordinate (topR) 
    (0,4) coordinate (top)
    (0,5) coordinate (alpham)
    ;
    \draw[string, black!50!red] (alpha) to[out=45, in=-45] node[auto,swap,pos=0.5] {$\scriptstyle f$} (botR) to [out=-135, in= 45]
     (bot) to [out=135, in =-135, looseness=1.5]node[auto,swap,pos=0.5] {$\scriptstyle f$} (top) to [out=-45, in=135] 
     (topR) to [out=45, in =-45] node[auto,swap,pos=0.5] {$\scriptstyle f$} (alpham);
    \draw[string, black!50!blue] (alpha) to [out=135, in=-135, looseness=1.5] node[auto,pos=0.5] {$\scriptstyle g$} (alpham);
     \path[fill=gray,opacity=0.25] (alpha)  to[out=45, in=-45] (botR) to [out=-135, in= 45] (bot) to [out=135, in =-135, looseness=1.5] (top) to [out=-45, in=135] (topR) to [out=45, in =-45] (alpham) to [out=-135, in =135, looseness=1.5] (alpha); 
    
    \path 
    (0,2.5) node (C) {}
    (alpha) node[draw, ellipse, inner sep=0.5pt,fill=white] {$\scriptstyle \alpha$}
     (top) node[draw, ellipse, inner sep=0.5pt,fill=white] {$\scriptstyle \ev_f$}
     (bot)  node[draw, ellipse, inner sep=0.5pt,fill=white] {$\scriptstyle \ev_f^L$}
    (alpham) node[draw, ellipse, inner sep=0.5pt,fill=white] {$\scriptstyle \alpha^{-1}$}
   (botR) node [draw, ellipse, inner sep=0.5pt,fill=white] {$\scriptstyle \coev_f^R$}
   (topR)  node [draw, ellipse, inner sep=0.5pt,fill=white] {$\scriptstyle \coev_f$}
    ;
  \end{tikzpicture} 
  \To[\ev_{\coev_f}]
   \begin{tikzpicture}[yscale=0.75,scale=0.8, baseline=(C.center)]
    \path
    (0,0) coordinate (alpha)
    (0,1) coordinate (bot)
    (1,2) coordinate (botR)
   (1,3) coordinate (topR) 
    (0,4) coordinate (top)
    (0,5) coordinate (alpham)
    ;
    \draw[string, black!50!red] (alpha) to[out=45, in=-45, looseness=1.5] node[auto,swap,pos=0.5] {$\scriptstyle f$} (alpham);
    \draw[string, black!50!blue] (alpha) to [out=135, in=-135, looseness=1.5] node[auto,pos=0.5] {$\scriptstyle g$} (alpham);
         \path[fill=gray,opacity=0.25] (alpha)  to[out=45, in=-45, looseness=1.5] (alpham) to [out=-135, in= 135, looseness=1.5] (alpha);
      \draw[string,black!50!red ,fill =white  ] (bot) to [out=45, in=-45, looseness=1.5]node[auto,pos=0.5] {$\scriptstyle f^R$} (top) to [out=-135, in =135, looseness=1.5] node[auto,pos=0.5] {$\scriptstyle f$}(bot);
 
     \path 
    (0,2.5) node (C) {}
    (alpha) node[draw, ellipse, inner sep=0.5pt,fill=white] {$\scriptstyle \alpha$}
     (top) node[draw, ellipse, inner sep=0.5pt,fill=white] {$\scriptstyle \ev_f$}
     (bot)  node[draw, ellipse, inner sep=0.5pt,fill=white] {$\scriptstyle \ev_f^L$}
    (alpham) node[draw, ellipse, inner sep=0.5pt,fill=white] {$\scriptstyle \alpha^{-1}$}
    ;
  \end{tikzpicture} 
\stackrel{\eqref{eq:simplification1}}{ \simeq }  ~  H(X,f,g,\alpha)
  \]
  defines a left $I_{\mathrm{coint}}(X,f,g,\alpha)$-valued cointegral and the composite 
  \[
  H(X,f,g,\alpha)~\stackrel{\eqref{eq:simplification2}}{ \simeq }
   \begin{tikzpicture}[yscale=0.75,scale=0.8,  baseline=(C.center)]
    \path
    (0,0) coordinate (alpha)
    (0,1) coordinate (bot)
    (1,2) coordinate (botR)
   (1,3) coordinate (topR) 
    (0,4) coordinate (top)
    (0,5) coordinate (alpham)
    ;
    \draw[string, black!50!red] (alpha) to[out=45, in=-45, looseness=1.5] node[auto,swap,pos=0.5] {$\scriptstyle f$} (alpham);
    \draw[string, black!50!blue] (alpha) to [out=135, in=-135, looseness=1.5] node[auto,pos=0.5] {$\scriptstyle g$} (alpham);
         \path[fill=gray,opacity=0.25] (alpha)  to[out=45, in=-45, looseness=1.5] (alpham) to [out=-135, in= 135, looseness=1.5] (alpha);
      \draw[string,black!50!blue ,fill =white  ] (bot) to [out=45, in=-45, looseness=1.5]node[auto,pos=0.5] {$\scriptstyle g$} (top) to [out=-135, in =135, looseness=1.5] node[auto,pos=0.5] {$\scriptstyle g^L$}(bot);
 
     \path 
    (0,2.5) node (C) {}
    (alpha) node[draw, ellipse, inner sep=0.5pt,fill=white] {$\scriptstyle \alpha$}
     (top) node[draw, ellipse, inner sep=0.5pt,fill=white] {$\scriptstyle \ev_g$}
     (bot)  node[draw, ellipse, inner sep=0.5pt,fill=white] {$\scriptstyle \ev_g^R$}
    (alpham) node[draw, ellipse, inner sep=0.5pt,fill=white] {$\scriptstyle \alpha^{-1}$}
    ;
  \end{tikzpicture} 
  \To[\coev_{\coev_g}]
   \begin{tikzpicture}[yscale=0.75, xscale=-1, scale=0.8, baseline=(C.center)]
    \path
    (0,0) coordinate (alpha)
    (0,1) coordinate (bot)
    (1,2) coordinate (botR)
   (1,3) coordinate (topR) 
    (0,4) coordinate (top)
    (0,5) coordinate (alpham)
    ;
    \draw[string, black!50!blue] (alpha) to[out=45, in=-45] node[auto,pos=0.5] {$\scriptstyle g$} (botR) to [out=-135, in= 45]
    (bot) to [out=135, in =-135, looseness=1.5]node[auto,pos=0.5] {$\scriptstyle g$} (top) to [out=-45, in=135] 
    (topR) to [out=45, in =-45] node[auto,pos=0.5] {$\scriptstyle g$} (alpham);
    \draw[string, black!50!red] (alpha) to [out=135, in=-135, looseness=1.5] node[auto,swap,pos=0.5] {$\scriptstyle f$} (alpham);
     \path[fill=gray,opacity=0.25] (alpha)  to[out=45, in=-45] (botR) to [out=-135, in= 45] (bot) to [out=135, in =-135, looseness=1.5] (top) to [out=-45, in=135] (topR) to [out=45, in =-45] (alpham) to [out=-135, in =135, looseness=1.5] (alpha); 
    
    \path 
    (0,2.5) node (C) {}
    (alpha) node[draw, ellipse, inner sep=0.5pt,fill=white] {$\scriptstyle \alpha$}
     (top) node[draw, ellipse, inner sep=0.5pt,fill=white] {$\scriptstyle \ev_g$}
     (bot)  node[draw, ellipse, inner sep=0.5pt,fill=white] {$\scriptstyle \ev_g^R$}
    (alpham) node[draw, ellipse, inner sep=0.5pt,fill=white] {$\scriptstyle \alpha^{-1}$}
   (botR) node [draw, ellipse, inner sep=0.5pt,fill=white] {$\scriptstyle \coev_g^L$}
   (topR)  node [draw, ellipse, inner sep=0.5pt,fill=white] {$\scriptstyle \coev_g$}
    ;
  \end{tikzpicture} 
  \simeq~I_{\mathrm{int}}(X,f,g,\alpha)
   \]
defines a left $I_{\mathrm{int}}(X,f,g,\alpha)$-valued integral. The composite $\mathrm{int}_{(X,f,g,\alpha)} \circ \mathrm{coint}_{(X,f,g,\alpha)}$ is invertible.
  \end{prop}
  \begin{proof}
  By Remark~\ref{rem:justifyfullyadj}, the conditions of Proposition~\ref{prop:dualizability} are evidently satisfied. 
  
  We now show that $\mathrm{coint}_{(X,f,g,\alpha)}: I_{\mathrm{coint}}(X,f,g,\alpha) \to H(X,f,g,\alpha)$ defines a cointegral, the integral case is completely analogous. 
 Let $\gamma:= \alpha^{-1} \circ (g \ev_f f) : gff^R f \To \id_{1_{\cC}}$. Recall from Remark~\ref{rem:simplesquare} that composing any $2$-morphism $\mu: \id_{1_{\cC}} \To g ff^R f$ with $\gamma$ defines a coherent left module $\gamma \circ \mu \in \Omega^2 \cC$ of the algebra $H(X,f,g,\alpha) \simeq \gamma \circ \gamma^L$ and whiskering any $3$-morphism $\mu \Rrightarrow \mu'$ with $\gamma$ induces a module map $\gamma \circ \mu \to \gamma \circ \mu'$. For $\mu_0 := (gf \coev_f) \circ \alpha$, this module is precisely the unit $\gamma \circ \mu_0 \simeq \id_{\id_{1_{\cC}}}$ with its module structure induced from the counit of $H(X,f,g,\alpha)$. Let now $\mu_1:= (gf \coev_f) \circ (g \Rad_f) \circ \alpha$. Since $\mu_1 \simeq \mu_0 \circ  I_{\mathrm{coint}}(X,f,g,\alpha) $, the induced module structure on $\gamma \circ \mu_1 \simeq I_{\mathrm{coint}}(X,f,g,\alpha)$ agrees with the one induced by $\mu_0$, i.e.\ where $H(X,f,g,\alpha)$ acts via the counit. By definition, our map $\mathrm{coint}_{(X,f,g,\alpha)}$ from above is given by whiskering a $3$-morphism $\mu_1 \to \gamma^L$ with $\gamma$ and hence defines a left $H(X,f,g,\alpha)$-module map, i.e.\ a cointegral. 
 
 {We leave checking the invertibility of the composite $\mathrm{int}_{(X,f,g,\alpha)} \circ \mathrm{coint}_{(X,f,g,\alpha)}$ to the reader.}  
\end{proof}
       
  \begin{remark} Assuming only the weak adjunctibility conditions from Proposition~\ref{prop:dualizability}, Proposition~\ref{prop:integrals} still asserts the existence of a universal (co)integral $I_{\mathrm{coint}}(X,f,g,\alpha) \cong I_{\mathrm{int}}(X,f,g,\alpha)$ in the Karoubi completion of $\Omega^2\cC$ (at least if $\cC$ is a $(3,3)$-category). We do not know whether the full assumptions of Proposition~\ref{prop:intcoint} are necessary to ensure this object already lives in $\Omega^2 \cC$ or whether the weaker conditions from Proposition~\ref{prop:dualizability} already suffice.    \end{remark}
  
  We find it instructive to give a description of integral and cointegral directly in terms of the more symmetric composite~\eqref{eqn:unpackedstring}. Using formula \eqref{eqn:Hasanalgebra} for $H(X,f,g,\alpha)$ and the  isomorphism $(\alpha^{\rmate})^L \circ \Rad_g \simeq (\alpha^{\rmate})^R$ induced from~\eqref{eq:Radev}, we can rewrite  $H(X,f,g,\alpha)$ as 
  
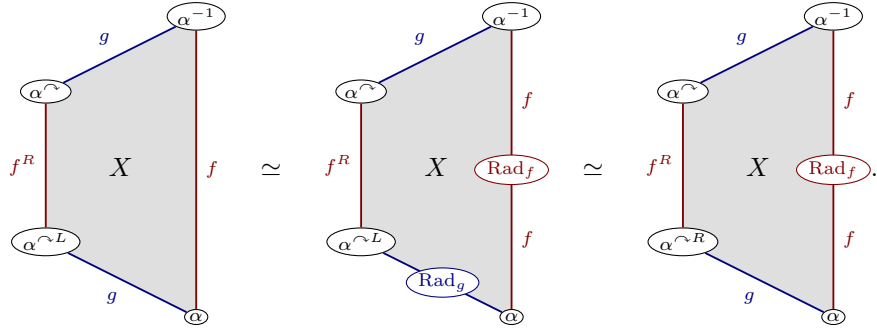
\[
  \begin{tikzpicture}[scale=2,baseline=(M.base)]
\path[fill=gray,opacity=0.25] (-.5,.5) -- (.5,+1) -- (+.5,-1) -- (-.5,-.5) -- cycle;
\path 
(-.5,.5) node[draw, ellipse, inner sep=0.5pt,fill=white] (W) {$\scriptstyle \alpha^\rmate$}
(+.5,-1) node[draw, ellipse, inner sep=0.5pt,fill=white] (E) {$\scriptstyle \alpha$}
(.5,+1) node[draw, ellipse, inner sep=0.5pt,fill=white] (N) {$\scriptstyle \alpha^{-1}$}
(-.5,-.5) node[draw, ellipse, inner sep=0.5pt,fill=white] (S) {$\scriptstyle \alpha^{\rmate L}$}
(0,0) node (M) {$X$}
;
\draw[string, black!50!red] (S) -- node[auto] {$\scriptstyle f^R$} (W);
\draw[string, black!50!blue] (S) -- node[auto,swap] {$\scriptstyle g$} (E);
\draw[string, black!50!blue] (W) -- node[auto] {$\scriptstyle g$} (N);
\draw[string, black!50!red] (E) -- node[auto,swap] {$\scriptstyle f$} (N);
\end{tikzpicture}
\quad\simeq\quad
\begin{tikzpicture}[scale=2,baseline=(M.base)]
\path[fill=gray,opacity=0.25] (-.5,.5) -- (.5,+1) -- (+.5,-1) -- (-.5,-.5) -- cycle;
\path 
(-.5,.5) node[draw, ellipse, inner sep=0.5pt,fill=white] (W) {$\scriptstyle \alpha^\rmate$}
(+.5,-1) node[draw, ellipse, inner sep=0.5pt,fill=white] (E) {$\scriptstyle \alpha$}
(.5,+1) node[draw, ellipse, inner sep=0.5pt,fill=white] (N) {$\scriptstyle \alpha^{-1}$}
(-.5,-.5) node[draw, ellipse, inner sep=0.5pt,fill=white] (S) {$\scriptstyle \alpha^{\rmate L}$}
(0,0) node (M) {$X$}
;
\draw[string, black!50!red] (S) -- node[auto] {$\scriptstyle f^R$} (W);
\draw[string, black!50!blue] (S) -- node[draw, line width=0.4pt, ellipse, inner sep=0.5pt,fill=white] {$\scriptstyle \Rad_g$} (E);
\draw[string, black!50!blue] (W) -- node[auto] {$\scriptstyle g$} (N);
\draw[string, black!50!red] (E) -- node[auto,swap,pos=0.25] {$\scriptstyle f$} node[auto,swap,pos=0.75] {$\scriptstyle f$} node[draw, line width=0.4pt, ellipse, inner sep=0.5pt,fill=white] {$\scriptstyle \Rad_f$} (N);
\end{tikzpicture}
\quad\simeq\quad
\begin{tikzpicture}[scale=2,baseline=(M.base)]
\path[fill=gray,opacity=0.25] (-.5,.5) -- (.5,+1) -- (+.5,-1) -- (-.5,-.5) -- cycle;
\path 
(-.5,.5) node[draw, ellipse, inner sep=0.5pt,fill=white] (W) {$\scriptstyle \alpha^\rmate$}
(+.5,-1) node[draw, ellipse, inner sep=0.5pt,fill=white] (E) {$\scriptstyle \alpha$}
(.5,+1) node[draw, ellipse, inner sep=0.5pt,fill=white] (N) {$\scriptstyle \alpha^{-1}$}
(-.5,-.5) node[draw, ellipse, inner sep=0.5pt,fill=white] (S) {$\scriptstyle \alpha^{\rmate R}$}
(0,0) node (M) {$X$}
;
\draw[string, black!50!red] (S) -- node[auto] {$\scriptstyle f^R$} (W);
\draw[string, black!50!blue] (S) -- node[auto,swap] {$\scriptstyle g$} (E);
\draw[string, black!50!blue] (W) -- node[auto] {$\scriptstyle g$} (N);
\draw[string, black!50!red] (E) -- node[auto,swap,pos=0.25] {$\scriptstyle f$} node[auto,swap,pos=0.75] {$\scriptstyle f$} node[draw, line width=0.4pt, ellipse, inner sep=0.5pt,fill=white] {$\scriptstyle \Rad_f$} (N);
\end{tikzpicture}  .\]
In these terms, the integral is given by the evaluation $3$-morphism $\ev_{\alpha^\rmate}$ for the adjunction $\alpha^\rmate \dashv (\alpha^\rmate)^R$:
  \begin{equation}
  \label{eqn:integral}
  \mathrm{int}_{H(X,f,g,\alpha)} := 
    \begin{tikzpicture}[yscale=1.5, baseline=(C.base)]
    \path
    (0,0) coordinate (alpha)
    (0,2) coordinate (alpham)
    ;
    \path[fill=gray,opacity=0.25] (alpha) .. controls +(+1.5,1) and +(+1.5,-1) .. (alpham) .. controls +(-1.5,-1) and +(-1.5,1) .. (alpha);
    \draw[string, black!50!red] (alpha) .. controls +(+1.5,1) and +(+1.5,-1) .. node[draw, black!50!red, ellipse, line width=0.4pt, inner sep=0.5pt,fill=white] {$\scriptstyle \Rad_f$} node[auto,swap,pos=0.2] {$\scriptstyle f$} node[auto,swap,pos=0.8] {$\scriptstyle f$} (alpham);
    \draw[string, black!50!blue] (alpha) .. controls +(-1.5,1) and +(-1.5,-1) .. 
    node[draw, black, line width=1pt, inner sep=4pt,fill=white] {$\scriptstyle \ev_{\alpha^\rmate}$} node[auto,pos=0.2] {$\scriptstyle g$} node[auto,pos=0.8] {$\scriptstyle g$}
     (alpham);
    \path 
    (0,1) node (C) {$X$}
    (alpha) node[draw, ellipse, inner sep=0.5pt,fill=white] {$\scriptstyle \alpha$}
    (alpham) node[draw, ellipse, inner sep=0.5pt,fill=white] {$\scriptstyle \alpha^{-1}$}
    ;
  \end{tikzpicture}
 \quad  : \quad H(X,f,g,\alpha) \Rrightarrow I_{\mathrm{int}}(X,f,g,\alpha).
  \end{equation}
  A similar logic starting with \eqref{eqn:Hasacoalgebra} identifies $H(X,f,g,\alpha)$ with
  $$ 
  
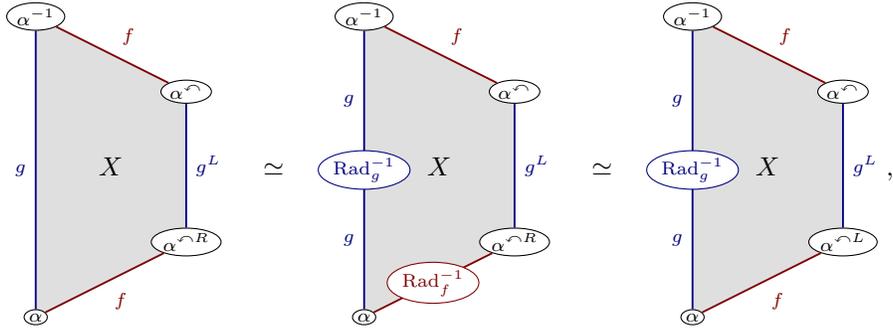
\begin{tikzpicture}[scale=2,baseline=(M.base)]
\path[fill=gray,opacity=0.25] (-.5,-1) -- (-.5,+1) -- (+.5,.5) -- (.5,-.5) -- cycle;
\path 
(-.5,-1) node[draw, ellipse, inner sep=0.5pt,fill=white] (W) {$\scriptstyle \alpha$}
(+.5,.5) node[draw, ellipse, inner sep=0.5pt,fill=white] (E) {$\scriptstyle \alpha^\lmate$}
(-.5,+1) node[draw, ellipse, inner sep=0.5pt,fill=white] (N) {$\scriptstyle \alpha^{-1}$}
(.5,-.5) node[draw, ellipse, inner sep=0.5pt,fill=white] (S) {$\scriptstyle \alpha^{\lmate R}$}
(0,0) node (M) {$X$}
;
\draw[string, black!50!red] (S) -- node[auto] {$\scriptstyle f$} (W);
\draw[string, black!50!blue] (S) -- node[auto,swap] {$\scriptstyle g^L$} (E);
\draw[string, black!50!blue] (W) -- node[auto] {$\scriptstyle g$} (N);
\draw[string, black!50!red] (E) -- node[auto,swap] {$\scriptstyle f$} (N);
\end{tikzpicture}
  \quad\simeq\quad
  \begin{tikzpicture}[scale=2,baseline=(M.base)]
\path[fill=gray,opacity=0.25] (-.5,-1) -- (-.5,+1) -- (+.5,.5) -- (.5,-.5) -- cycle;
\path 
(-.5,-1) node[draw, ellipse, inner sep=0.5pt,fill=white] (W) {$\scriptstyle \alpha$}
(+.5,.5) node[draw, ellipse, inner sep=0.5pt,fill=white] (E) {$\scriptstyle \alpha^\lmate$}
(-.5,+1) node[draw, ellipse, inner sep=0.5pt,fill=white] (N) {$\scriptstyle \alpha^{-1}$}
(.5,-.5) node[draw, ellipse, inner sep=0.5pt,fill=white] (S) {$\scriptstyle \alpha^{\lmate R}$}
(0,0) node (M) {$X$}
;
\draw[string, black!50!red] (S) -- node[draw, line width=0.4pt, ellipse, inner sep=0.5pt,fill=white] {$\scriptstyle \Rad_f^{-1}$} (W);
\draw[string, black!50!blue] (S) -- node[auto,swap] {$\scriptstyle g^L$} (E);
\draw[string, black!50!blue] (W) -- node[auto,pos=0.25] {$\scriptstyle g$} node[auto,pos=0.75] {$\scriptstyle g$} node[draw, line width=0.4pt, ellipse, inner sep=0.5pt,fill=white] {$\scriptstyle \Rad_g^{-1}$} (N);
\draw[string, black!50!red] (E) -- node[auto,swap] {$\scriptstyle f$} (N);
\end{tikzpicture}
  \quad\simeq\quad
  \begin{tikzpicture}[scale=2,baseline=(M.base)]
\path[fill=gray,opacity=0.25] (-.5,-1) -- (-.5,+1) -- (+.5,.5) -- (.5,-.5) -- cycle;
\path 
(-.5,-1) node[draw, ellipse, inner sep=0.5pt,fill=white] (W) {$\scriptstyle \alpha$}
(+.5,.5) node[draw, ellipse, inner sep=0.5pt,fill=white] (E) {$\scriptstyle \alpha^\lmate$}
(-.5,+1) node[draw, ellipse, inner sep=0.5pt,fill=white] (N) {$\scriptstyle \alpha^{-1}$}
(.5,-.5) node[draw, ellipse, inner sep=0.5pt,fill=white] (S) {$\scriptstyle \alpha^{\lmate L}$}
(0,0) node (M) {$X$}
;
\draw[string, black!50!red] (S) -- node[auto] {$\scriptstyle f$} (W);
\draw[string, black!50!blue] (S) -- node[auto,swap] {$\scriptstyle g^L$} (E);
\draw[string, black!50!blue] (W) -- node[auto,pos=0.25] {$\scriptstyle g$} node[auto,pos=0.75] {$\scriptstyle g$} node[draw, line width=0.4pt, ellipse, inner sep=0.5pt,fill=white] {$\scriptstyle \Rad_g^{-1}$} (N);
\draw[string, black!50!red] (E) -- node[auto,swap] {$\scriptstyle f$} (N);
\end{tikzpicture}
  ,$$
  and the cointegral is given by the coevaluation $3$-morphism for the adjunction $\alpha^{\lmate L} \dashv \alpha^\lmate$:
  \begin{equation}
  \label{eqn:cointegral}
  \mathrm{coint}_{H(X,f,g,\alpha)} := 
    \begin{tikzpicture}[yscale=1.5, baseline=(C.base)]
    \path
    (0,0) coordinate (alpha)
    (0,2) coordinate (alpham)
    ;
    \path[fill=gray,opacity=0.25] (alpha) .. controls +(+1.5,1) and +(+1.5,-1) .. (alpham) .. controls +(-1.5,-1) and +(-1.5,1) .. (alpha);
    \draw[string, black!50!red] (alpha) .. controls +(+1.5,1) and +(+1.5,-1) .. 
    node[draw, black, line width=1pt, inner sep=4pt,fill=white] {$\scriptstyle \coev_{\alpha^\lmate}$} 
     node[auto,swap,pos=0.2] {$\scriptstyle f$} node[auto,swap,pos=0.8] {$\scriptstyle f$} (alpham);
    \draw[string, black!50!blue] (alpha) .. controls +(-1.5,1) and +(-1.5,-1) .. 
    node[draw, black!50!blue, ellipse, line width=0.4pt, inner sep=0.5pt,fill=white] {$\scriptstyle \Rad_g^{-1}$} node[auto,pos=0.2] {$\scriptstyle g$} node[auto,pos=0.8] {$\scriptstyle g$}
     (alpham);
    \path 
    (0,1) node (C) {$X$}
    (alpha) node[draw, ellipse, inner sep=0.5pt,fill=white] {$\scriptstyle \alpha$}
    (alpham) node[draw, ellipse, inner sep=0.5pt,fill=white] {$\scriptstyle \alpha^{-1}$}
    ;
  \end{tikzpicture}
 \quad  : \quad I_{\mathrm{coint}}(X,f,g,\alpha) \Rrightarrow H(X,f,g,\alpha).
  \end{equation}

\begin{remark}
Plugging the above formulas for integral and cointegral into the formula \eqref{eqn:integralformulaforantipode} for the antipode, one finds that in the $3$-dimensional pictures where we have suppressed all framing data including applications of $\Rad_f$, the antipode is simply the $180^\circ$ twist:
$$
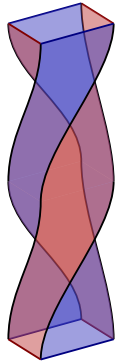
\begin{tikzpicture}[scale=0.7]
  \path (0,0,0) coordinate (middle);
  \path
  (0,0,0)
  +(-1,0,0)  coordinate (W)
  +(+1,0,0)  coordinate (E)
  +(0,0,+1)  coordinate (S)
  +(0,0,-1)  coordinate (N)
  (0,3,0)
  +(-1,0,0)  coordinate (Wt)
  +(+1,0,0)  coordinate (Et)
  +(0,0,+1)  coordinate (St)
  +(0,0,-1)  coordinate (Nt)
  (0,-3,0)
  +(-1,0,0)  coordinate (Wb)
  +(+1,0,0)  coordinate (Eb)
  +(0,0,+1)  coordinate (Sb)
  +(0,0,-1)  coordinate (Nb)  
  ;
  \draw[string, black!50!red]  (Et) -- (Nt);
  \draw[string, black!50!blue] (Nt) -- (Wt);
  \draw[string, black!50!red]  (Eb) -- (Nb);
  \draw[string, black!50!blue] (Nb) -- (Wb);
  \draw[string] (Nb) .. controls +(0,1,0) and +(0,-1,0) .. (W);
  \draw[string] (Eb) .. controls +(0,1,0) and +(0,-1,0) .. (N) .. controls +(0,1,0) and +(0,-1,0) .. (Wt);
  \draw[string] (E) .. controls +(0,1,0) and +(0,-1,0) .. (Nt);
  \path[fill=gray!50!red,opacity=0.5] (Eb) .. controls +(0,1,0) and +(0,-1,0) .. (N) -- (W) .. controls +(0,-1,0) and +(0,+1,0) .. (Nb) -- (Eb);
  \path[fill=gray!50!blue,opacity=0.5] (E) .. controls +(0,1,0) and +(0,-1,0) .. (Nt) -- (Wt) .. controls +(0,-1,0) and +(0,+1,0) .. (N) -- (E);
  \path[fill=gray!50!blue,opacity=0.5] (Sb) .. controls +(0,1,0) and +(0,-1,0) .. (E) -- (N) .. controls +(0,-1,0) and +(0,+1,0) .. (Eb) -- (Sb);
  \path[fill=gray!50!red,opacity=0.5] (S) .. controls +(0,1,0) and +(0,-1,0) .. (Et) -- (Nt) .. controls +(0,-1,0) and +(0,+1,0) .. (E) -- (S);
  \path[fill=gray!50!blue,opacity=0.5] (Nb) .. controls +(0,1,0) and +(0,-1,0) .. (W) -- (S) .. controls +(0,-1,0) and +(0,+1,0) .. (Wb) -- (Nb);
  \path[fill=gray!50!red,opacity=0.5] (N) .. controls +(0,1,0) and +(0,-1,0) .. (Wt) -- (St) .. controls +(0,-1,0) and +(0,+1,0) .. (W) -- (N);
  \path[fill=gray!50!red,opacity=0.5] (Wb) .. controls +(0,1,0) and +(0,-1,0) .. (S) -- (E) .. controls +(0,-1,0) and +(0,+1,0) .. (Sb) -- (Wb);
  \path[fill=gray!50!blue,opacity=0.5] (W) .. controls +(0,1,0) and +(0,-1,0) .. (St) -- (Et) .. controls +(0,-1,0) and +(0,+1,0) .. (S) -- (W);
  \draw[string] (Sb) .. controls +(0,1,0) and +(0,-1,0) .. (E);
  \draw[string] (Wb) .. controls +(0,1,0) and +(0,-1,0) .. (S) .. controls +(0,1,0) and +(0,-1,0) .. (Et);
  \draw[string] (W) .. controls +(0,1,0) and +(0,-1,0) .. (St);  
  \draw[string, black!50!red]  (Wt) -- (St);
  \draw[string, black!50!blue] (St) -- (Et);
  \draw[string, black!50!red]  (Wb) -- (Sb);
  \draw[string, black!50!blue] (Sb) -- (Eb);
\end{tikzpicture}
$$
Without the extra dualizability assumptions of Proposition~\ref{prop:dualizability} (let alone those of Proposition~\ref{prop:intcoint}), Corollary~\ref{thm:adjinduceshopf} still guarantees the existence of a (generally non-invertible) antipode, which we find hard to picture topologically. 
\end{remark}

\subsection{Tannakian reconstruction}\label{subsec:tannakian}

In this final section, we will explore our construction in a convenient $3$-category in which to do higher linear algebra. We will end up recovering the classical Tannakian reconstruction (of quantum groups), which recovers a Hopf algebra $H$ from its category of comodules \cite{MR1098991,MR1173027,MR1623637}. One goal of this section is to emphasize that our main Theorem~\ref{thm:mainHopftheorem} can indeed reconstruct (co)Hopf algebras with noninvertible antipode. Such Hopf algebras already appear in the foundational work \cite{MR1098991} (note that Ulbrich uses the term ``rigid'' for categories with one-sided duals) but are often omitted in more recent works, for example the diagrammatic/TQFT-theoretic constructions of \cite{MR1915348,MR4520300,TudorWenjun}.

Let $\Pr$ denote the symmetric monoidal $(\infty,2)$-category of presentable\footnote{Recall that a \emph{presentable $(\infty,1)$-category} is a large, locally small  $(\infty,1)$-category $\cC$ which admits small colimits and which is \emph{accessible} in the sense that there is a regular cardinal $\kappa$ so that $\cC$ is generated under $\kappa$-filtered colimits by $\kappa$-compact objects; see~\cite[\S 5.5]{HTT}. Lurie constructs the symmetric monoidal structure on the underlying $(\infty,1)$-category of $\Pr$ in \cite[Prop.~4.8.1.15]{HA}. That this extends to the $(\infty,2)$-category is explained e.g.\ in~\cite{2011.03035}. Indeed, Lurie focuses on the underlying $(\infty,1)$-category of $\Pr$ because it controls the story: the $(\infty,2)$-structure can be recovered from the inner homs, adjunct to Lurie's tensor product, together with  the forgetful functor $\Pr \to \widehat{\cat{Cat}}_{(\infty,1)}$. Because of this, the literature tends to write $\Pr$ or $\mathrm{Pr}$ just for the underlying $(\infty,1)$-category of presentable $(\infty,1)$-categories and colimit-preserving functors. We will use this same notation for the $(\infty,2)$-category, and trust the reader to forgive this mild notational overload.} $(\infty,1)$-categories, colimit-preserving functors, and natural transformations. The symmetric monoidal structure $\boxtimes$ on $\Pr$ is universally determined by the property that for presentable $(\infty,1)$-categories $\cC$, $\cD$, and $\cE$, the space of colimit-preserving functors $\cC \boxtimes \cD \to  \cE$ agrees with the space of functors $\cC \times \cD \to \cE$ which preserve colimits separately in both arguments.

 Throughout this section, we fix a $\cV \in \Alg_{\mathbb{E}_2}(\Pr)$, i.e.\ a  \define{presentably braided monoidal $(\infty,1)$-category}: a presentable $(\infty,1)$-category equipped with a braided monoidal structure so that the tensor product functor $\otimes : \cV \times \cV \to \cV$ preserves small colimits  separately in both arguments. Our main example will be $\cV = \Vect$, the ($(1,1)-$)category of (small) vector spaces. 
 
We denote by $\PrV:= \RMod_{\cV}(\Pr)$ the $(\infty,2)$-category of right $\cV$-modules in $\Pr$ and refer to it as the $(\infty,2)$-category of \emph{presentable $\cV$-linear categories}\footnote{Equivalently, as explained in \cite[\S A.3]{MR4788123}, a presentable $\cV$-linear $(\infty,1)$-category $\cM$ is a large $\cV$-enriched $(\infty,1)$-category, in the sense of~\cite{MR3345192},  whose underlying category is presentable, admits tensors, and if for every $v\in \cV$, the induced functor $ -\triangleleft v: \cM \to \cM$ between the underlying categories preserves small colimits. Here, a $\cV$-enriched category is said to \emph{admit tensors} if for all $v\in \cV, m \in \cM$ the induced functor $\hom(v, \mathrm{hom}(m, - )): \cM \to \cV$ is corepresentable. Writing $m \triangleleft v$ for the corepresenting object, this induces an action of $\cV$ on $\cM$.}: presentable $(\infty,1)$-categories $\cM$ equipped with a right action by $\cV$ so that the action functor $\cM \times \cV \to \cM$ preserves small colimits in both arguments.  Since $\cV$ is braided, the category $\PrV$ inherits a monoidal structure $-\boxtimes_{\cV}-$, computed as the relative tensor product over $\cV$ in $\Pr$, which preserves colimits separately in both arguments.

The symmetric monoidal $(\infty,2)$-category $\PrV$  satisfies the conditions of~\cite[Def.~8.3]{JFS} and hence by \cite{JFS, MR3650080} gives rise to a Morita $(\infty,3)$-category $\Mor(\PPrV)$. Intuitively, $\Mor(\PPrV)$  may be described as follows:

\begin{itemize}
\item Objects are given by  algebra objects $\cA, \cB, \ldots \in \Alg(\PrV)$, i.e.\ presentably monoidal $\cV$-linear  categories;
\item $\mathrm{Hom}(\cA, \cB) = {}_{\cB}\cat{Bim}_{\cA}(\PPrV)$ is the $(\infty,2)$-category of presentable and $\cV$-linear $\cB$-$\cA$-bimodule categories;
\item Composition $\cA \to[_{\cB} \cM_{\cA}] \cB \to[{}_{\cC}\cN_{\cB}]\cC$ of bimodules given by the relative tensor product ${}_{\cC}\cN \otimes_{\cB}\cM_{\cA} $.
\end{itemize}
We point that for $\cV$ a $(1,1)$-category, $\PPrV$ is a $(2,2)$-category (namely, the $(2,2)$-category of presentable $\cV$-linear categories in the sense of \cite{MR1294136,MR648793}) and $\Mor(\PPrV)$ is a $(3,3)$-category. 
For $\cV= \Vect$, the symmetric monoidal  $(2,2)$-category $\PPrV$  is a reasonable choice for the 2-category of ``$2$-vector spaces,'' and the symmetric monoidal $(3,3)$-category $\cat{Mor}(\PPrV)$ is a reasonable choice for the 3-category of ``$3$-vector spaces''. 

\begin{remark} 
Another canonical choice for  the $3$-category of $3$-vector spaces is the $3$-category of presentable $\cV$-linear $2$-categories $\tPPrV$, introduced in \cite{2011.03035} as a natural categorification of $\PPrV$. 
All our arguments and computations below apply just as well to $\tPPrV$  and the inclined reader may just replace all occurrences of $\Mor(\PPrV)$ by $\tPPrV$.
\end{remark}

Using the canonical functor $\Alg(\PrV) \to \Mor(\PPrV)$ sending a morphism $F: \cA \to \cB$ of algebras to the bimodule ${}_{\cB} \cB_{\cA}$ with $\cA$-action twisted by $F$, any retract in $\Alg(\PrV)$ induces a retract in $\Mor(\PPrV)$. Since $\cV \in \Alg(\PrV)$ is initial, any morphism $\cA \to \cV$ in $\Alg(\PrV)$ therefore exhibits $\cV$ as a retract of $\cA$ in $\Alg(\PrV)$ and hence in $\Mor(\PPrV)$. Unpacked, this looks as follows: 
\begin{lemma} \label{lem:easyretract}Let $\cA \to \cV$ be a morphism in $\Alg(\PrV)$, i.e.\ a presentably monoidal $\cV$-linear category  $\cA$ equipped with a \emph{fiber functor},  a small colimit preserving $\cV$-linear monoidal functor $F:\cA \to \cV$.  Then, the unit inclusion $\cV \to \cA$ and the fiber functor $\cA \to \cV$ define a retract in $\Mor(\PPrV)$ 
\begin{equation}\label{eqn:moritaretract}
 \begin{tikzcd}[sep=3em]
\cV \arrow[r, equal] \arrow[d, "f := {{}_{\cA} \cA_{\cV}}"'] &\arrow[d, equal] \cV \\
\cA \arrow[r, "g:= {{}_{\cV} \cV_{\cA}} "'] 
\arrow[Rightarrow, ur, phantom, sloped, "\overset{\textstyle\sim} \Rightarrow"] 
& \cV
\end{tikzcd}.
\end{equation}
for which both $f$ and $g$ are left adjoints in $\Mor(\PPrV)$ with right adjoints $$f^R ={}_{\cV} \cA_{\cA}\hspace{1cm} g^R = {}_{\cA}\cV_{\cV}.$$

The mates $f^R \Rightarrow g$ and $f \Rightarrow g^R$ of the isomorphism $gf \simeq \id$ both unpack to the fiber functor $F$, thought of as either a map of  $\cV$-$\cA$-bimodules or of $\cA$-$\cV$-bimodules:
\[ f^R = {}_{\cV}\cA_\cA \to[F] {}_{\cV}\cV_\cA = g, \qquad f = {_\cA \cA_{\cV}} \to[ F] {_\cA \cV_{\cV}} = g^R. \]
\end{lemma}
\begin{proof}
The functor $\Alg(\PrV) \to \Mor(\PPrV)$ sends the retract in $\Alg(\PrV)$ given by the unit inclusion $\cV \to \cA$ and the fiber functor $\cA \to \cV$ to the retract  $gf \simeq \id$ in $\Mor(\PPrV)$. Moreover, the bimodule ${}_{\cC}\cC_{\cB}$ induced by any monoidal functor $G: \cB \to \cC$ has right adjoint ${}_{\cB} \cC_{\cC}$ with unit ${}_{\cB} \cB_{\cB} \to {}_{\cB}\cC \boxtimes_{\cC} \cC_{\cB} \simeq {}_{\cB} \cC_{\cB}$ given by the functor $G$ itself considered and   counit $\cC \otimes_{\cB} \cC \to \cC$ given by the tensor product of $\cC$. Using these, the mates $f^R \To g$ and $f \To g^R$ are as described. 
\end{proof}

In other words, to understand the various versions of our adjunctibility criteria from Theorem~\ref{thm:mainHopftheorem} and Remark~\ref{rem:oppositevariants}, it suffices to understand when a bimodule, thought of as a $1$-morphism in $\Mor(\PPrV)$, is a right adjoint, and to understand when a bimodule morphism, thought of as a $2$-morphism in $\Mor(\PPrV)$, is a left or right adjoint.

To this end, we recall the following adjunctibility criteria in $\PPrV$ and $(\infty,2)$-categories of bimodules therein. A morphism $G:\cM \to \cN$ in $\Pr$ is a colimit-preserving functor between presentable categories and hence admits by the adjoint functor theorem~\cite[Cor.~5.5.2.9]{HTT} a right adjoint functor $G^R: \cN \to \cM$. A necessary and sufficient condition for adjunctibility of $G$ in the $(\infty,2)$-category $\PPr$ is that this right adjoint is itself a morphism in $\PPr$, i.e.\ preserves colimits. Enhancing to bimodule categories, we have:

\begin{prop}[{see e.g.~\cite[Prop.~5.2]{MR4695687}}]\label{prop:doctrinaladjunction}
Let $\cA, \cB \in \Alg(\PrV)$ and $G:{}_{\cB}\cM_{\cA} \to {}_{\cB} \cN_{\cA}$ be a $1$-morphism in ${}_{\cB}\cat{Bim}_{\cA}(\PPrV)$. Then, $G$ is a left adjoint in the $(\infty,2)$-category ${}_{\cB}\cat{Bim}_{\cA}(\PPrV)$ if and only if 
the right adjoint $G^R: \cN \to \cM$ of the underlying functor $G: \cM \to \cN$ preserves small colimits and commutes with the $\cA$ and $\cB$ action in the sense that the induced morphisms 
\[
b \triangleright G^R(m) \triangleleft a \to G^R(b \triangleright m \triangleleft a)
\]
are isomorphisms.   \qed
\end{prop}

We recall the following terminology; see~\cite{ MR4695687, 2410.21537,RZ} for details.

\begin{definition}
Let $\cA \in \Alg(\Pr)$ be a presentable monoidal $\infty$-category and $\cM \in \RMod_{\cA}(\Pr)$. An object $m \in \cM$ is called \define{$\cA$-atomic} (also called \define{$\cA$-tiny}, see e.g.~\cite{MR3361309}) if the induced functor $m\otimes -: \cA \to \cM$ is an internal left adjoint in $\RMod_{\cA}(\PPr)$, equivalently if the internal-hom functor $\underline{\mathrm{hom}}_{\cM}(m,-): \cM \to \cA$ preserves small colimits and commutes with the the right $\cA$-action in the sense that the canonical morphism 
\[\underline{\mathrm{hom}}(m, n) \otimes a \to \underline{\mathrm{hom}}(m, n \otimes a)
\] 
is an isomorphism for all $a\in \cA, m,n \in \cM$. 
We say that $\cM$ is \define{generated by $\cA$-atomics} if there is a set of $\cA$-atomic objects in $\cM$ such that $\cM$ is the smallest full subcategory of $\cM$ which contains this set and which is closed under colimits and the $\cA$-action. 
\end{definition}
For example, the atomic objects in $\cV$ are precisely the dualizable objects.

\begin{prop}[{see e.g.~\cite[Prop.~1.40]{2410.21537}}]\label{lem:tinygeneration}
Let $\cA, \cB \in \Alg(\PrV)$ and $_\cB \cM_\cA \in {_{\cB}\cat{Bim}_{\cA}}(\PrV)$ a presentable $\cV$-linear bimodule category. Then, if the right $\cA$-module $\cM_{\cA} \in \RMod_{\cA}(\PrV)$ is generated by $\cA$-atomics, it follows\footnote{This condition is sufficient but not necessary:  e.g.\ for $\cA = \cB = \cV$, it follows from ~\cite[Thm.~1.49]{2410.21537} that $\cM \in \PPrV$ is a right dual if and only if $\cM$ is a retract of an atomically generated category. } that the $1$-morphism ${}_{\cB}\cM_{\cA}: \cA \to \cB$ in the $(\infty,3)$-category  $\Mor(\PPrV)$ is a right adjoint. 
\end{prop}
\begin{proof} The proof in \cite{2410.21537} concerns dualizability in $\PPrV$ for a presentably symmetric monoidal $\cV \in \mathrm{CAlg}(\Pr)$, but applies verbatim to the less commutative situation of Proposition~\ref{lem:tinygeneration}. 
\end{proof}

If $C$ is a small $\cV$-enriched $(\infty,1)$-category in the sense of~\cite{MR3345192}, its \emph{enriched presheaf category} $\operatorname{Psh}_\cV(C) = \Fun_\cV(C^\op,\cV)$ is an object of $\RMod_{\cV}(\Pr)$, equipped with an enriched Yoneda embedding $C \to \cA$ sending $c\mapsto \hom_C(-,c)$. In fact, $\cA$ is the free cocompletion of $C$ under colimits and tensoring, in the sense that any $\cV$-enriched functor $G:C \to \cD$ into a $\cD \in \RMod_{\cV}(\Pr)$ extends uniquely to a colimit-preserving and $\cV$-linear functor $\PShV(C) \to \cD$. We refer the reader  to \cite{MR4567127,MR4554672,MR4695687, RZ} for details on the $\infty$-categorical enriched Yoneda lemma.

\begin{prop}\label{prop:Tannaka1} Let $C$ be a small  $\cV$-enriched monoidal $(\infty,1)$-category, equipped with a $\cV$-enriched monoidal functor $F : C \to \cV$. Suppose furthermore that $C$ has left, but not necessarily right, duals.

Then, the induced presentably monoidal $\cV$-linear category $\cA:=\PShV(C)$ with fiber functor $\PShV(C) \to \cV$ given by the unique colimit-preserving $\cV$-linear extension of $F$ induces via Lemma~\ref{lem:easyretract} a  retract in $\Mor(\PPrV)$ for which: 
\begin{itemize}
\item $g$ (and $f$) is a left adjoint and $f$ is a right adjoint;
\item $f \To g^R$ is a left adjoint
\end{itemize}
Hence, by Theorem~\ref{thm:mainHopftheorem} and Remark~\ref{rem:oppositevariants} this gives rise to a coHopf algebra in $\cV$.
\end{prop}
\begin{proof}
By Lemma~\ref{lem:easyretract}, both $f$ and $g$ are left adjoints. The fact that $f$ is also a right adjoint follows from Proposition~\ref{lem:tinygeneration} since $\cA = \PShV(C)$ is as a presheaf category $\cV$-atomically generated (see e.g.~\cite[Obs.~5.8]{RZ}).

We now show that $f \To g^R$ is a left adjoint, i.e.\ that ${ F} : {_\cA\cA} \to {_\cA\cV}$ is an internal left adjoint in $\LMod_\cA(\PPrV)$. We first show that $F$ is an internal left adjoint in $\PPrV$: 
 By~\cite[Prop.~3.32]{RZ} a functor out of a presheaf category is an internal left adjoint if and only if it sends representable presheaves to atomics. Since $\cV$-atomics in $\cV$ are precisely the dualizable objects, it therefore suffices to show that $F$ sends all objects of $C$ to dualizable objects in $\cV$. This follows from our assumption that all objects of $C$ have left duals.  By Proposition~\ref{prop:doctrinaladjunction}, to check that $F$ lifts to an internal left adjoint in $\LMod_\cA(\PPrV)$, it therefore suffices to check compatibility between its right adjoint $F^R$ and the left $\cA$-action. Since $\cA$ is generated by $C$, it suffices to check compatibility between $F^R$ and the left $C$-action: we need to show that for all $c \in C$ and all $v \in \cV$, the comparison map
  \[ c \otimes F^R(v) \to F^R(F(c) \otimes v) \]
  is an iso in $\cA$ (where the $\otimes$ on the left denotes the multiplication in $\cA$ and on the right denotes the multiplication in $\cV$). 
  We can test this comparison map by mapping into both sides from $a \in \cA$. Using that $c$ has a left dual $c^\vee$ and that $F$ is monoidal, we find:
  \begin{multline*}
    \Map_\cA(a, c \otimes F^R(v)) \simeq \Map_\cA(c^\vee \otimes a, F^R(v)) \simeq \Map_\cV(F(c^\vee \otimes a), v) 
    \\
    \simeq \Map_\cV(F(c^\vee) \otimes F(a), v) \simeq \Map_\cV(F(c)^\vee \otimes F(a), v) 
    \\
    \simeq \Map_\cV(F(a), F(c) \otimes v) \simeq \Map_\cA(a, F^R(F(c) \otimes v) )\qedhere
  \end{multline*}
\end{proof}
Since $C$ did not necessarily have right duals, this coHopf algebra constructed via Proposition~\ref{prop:Tannaka1}  is not necessarily Hopf.
   Comparing with the formulas from \S\ref{subsec:unpacked}  recovers the foundational reconstruction from~\cite{MR1098991} and e.g.~\cite{2311.14221}.

In applications of (co)Hopf algebras, one sometimes wants to work with their categories of modules and sometimes their categories of comodules; and this affects exactly which (co)Hopf algebra one reconstructions. In the former case, it is standard to work just with finite-dimensional modules; that's how to think about the small category $C$ in Proposition~\ref{prop:Tannaka1}. On the other hand, in algebraic and quantum group theory, one often wants to work with categories of all, possibly infinite-dimensional, comodules. We end with an application of our main result to that version of Tannakian reconstruction. 

For a coalgebra $C$ in $\cV$, we define the $\infty$-category of comodules $$\cat{LcoMod}^C(\cV):= \LMod_C(\cV^{\op})^{\op}.$$ Dualizing~\cite[Cor.~4.2.4.8]{HA}, there is the cofree-forgetful adjunction 
\begin{equation}\label{eq:cofree}
\mathrm{Forget}: \cat{LcoMod}^C(\cV) \rightleftarrows \cV: \mathrm{CoFree} := C \otimes -.
\end{equation}

\begin{prop}\label{prop:comodules} For $\cV \in \Alg(\PrL)$ and $C$ a coalgebra in $\cV$, the category $\cat{LcoMod}^C(\cV)$ is a  presentable $\cV$-module category, i.e.\ is in $\PrV$, and the functor $\mathrm{Forget}:  \cat{LcoMod}^C(\cV) \to \cV$ is an internal left adjoint in $\PrV$. 

Corestriction $\cat{LcoMod}^C(\cV) \to \cat{LcoMod}^D(\cV)$ along coalgebra homomorphisms $C\to D$ assembles into a functor 
\[\cat{LcoMod}^{-}(\cV): \coAlg(\cV) \to \PrV.
\]
If $\cV \in \Alg_{\mathbb{E}_2}(\PrL)$, then this functor inherits a lax monoidal structure. 
\end{prop}
\begin{proof}
We write $\widehat{\cat{Cat}}_{(\infty,1)}$ for the $\infty$-category of large $(\infty,1)$-categories. Note that for $\cV \in \Alg_{\mathbb{E}_2}(\widehat{\cat{Cat}}_{(\infty,1)})$, the functor (sending algebra homomorphism $A\to B$ to the restriction functors $\LMod_B(\cV) \to \LMod_A(\cV)$)
\[\LMod_{-}(\cV): \Alg(\cV)^{\op} \to \RMod_{\cV}(\widehat{\cat{Cat}}_{(\infty,1)})
\]
carries a lax monoidal structure\footnote{We give a very brief sketch of a coherent construction of this lax monoidal structure, for details and notation we refer to~\cite[\S 3.2.4]{HA}, also see~\cite[\S 8.1]{MR4567127}.  Let $\mathrm{LM}, \mathrm{RM}, \mathrm{BM}$ denote the operads corepresenting left, right, and bimodules, and let $\otimes^{\mathrm{BV}}$ denote the Boardman--Vogt tensor product of operads. For any operad $\cO$ and $\cV \in \Alg_{\cO\otimes^{\mathrm{BV}} \mathbb{E}_1}(\widehat{\cat{Cat}}_{(\infty,1)})$, let $\LMod(\cV) \to \Alg(\cV)$ denote the Cartesian fibration classifying the functor $\LMod_{-}(\cV): \Alg(\cV)^{\op} \to \widehat{\cat{Cat}}_{(\infty,1)}$.  Via~\cite[Prop.~3.2.4.3]{HA}, the map of operads $\mathrm{LM}\otimes^{\mathrm{BV}}\mathrm{RM} \to \mathrm{BM} \to \mathbb{E}_1$ induces an $\cO\otimes^{\mathrm{BV}} \mathrm{RM}$-monoidal structure on the functor $\Alg_{\mathrm{LM}/\cO \otimes^{\mathrm{BV}} \mathbb{E}_1}(\cV) \to \Alg_{\mathbb{E}_1/\cO \otimes^{\mathrm{BV}} \mathbb{E}_1}(\cV)$. By~\cite[Rem.~3.2.4.2]{HA}, its underlying $\cO\otimes^{\mathrm{BV}} \mathbb{E}_1$-monoidal functor is $\cV \to *$ acting on the $\cO$-monoidal functor $\LMod(\cV) \to \Alg(\cV)$. In fact, it can be checked that this is an $\cO\otimes^{\mathrm{BV}} \mathrm{RM}$-monoidal Cartesian fibration (the opposite notion to~\cite[Def.~2.1.2.13]{HA}) and hence can be straightened to give a lax $\cO \otimes^{\mathrm{BV}} \mathrm{RM}$-monoidal functor $\Alg(\cV)^{\op} \to \widehat{\cat{Cat}}_{(\infty,1)}$ whose underlying lax $\cO\otimes^{\mathrm{BV}} \mathbb{E}_1$-monoidal functor is $* \to \widehat{\cat{Cat}}_{(\infty,1)}$ sending $* \mapsto \cV$. Taking modules at $* \in \Alg(*)$ therefore results in a lax $\cO$-monoidal functor $\Alg(\cV)^{\op} \to \RMod_{\cV}(\widehat{\cat{Cat}}_{(\infty,1)})$. }
for the monoidal structure on $\Alg(\cV)$ from Lemma~\ref{lem:algmonoidal}  and the monoidal structure on $\RMod_{\cV}(\widehat{\cat{Cat}}_{(\infty,1)})$ given by the relative tensor product over~$\cV$. 
Taking opposites, this induces a lax monoidal structure on the functor $$\LcoMod^{-}(\cV): \coAlg(\cV) \to \RMod_{\cV}(\widehat{\cat{Cat}}_{(\infty,1)})$$ (sending a coalgebra morphism to the corestriction functor between categories of comodules). For $\cV \in \Alg_{\mathbb{E}_2}(\Pr)$, we will show that it factors (as a lax monoidal functor) through the subcategory $\PrV \subseteq  \RMod_{\cV}(\widehat{\cat{Cat}}_{(\infty,1)})$. This entails: 
\begin{enumerate}
\item $\LcoMod^C(\cV)$ is presentable;
\item The action functor  $\LcoMod^C(\cV) \times \cV \to \LcoMod^C(\cV)$ preserves colimits separately in both arguments; 
\item Corestriction $\LcoMod^C(\cV) \to \LcoMod^D(\cV)$ along any coalgebra morphism $C\to D$ preserves small colimits;
\item For coalgebras $C,D$, the functor $\LcoMod^C(\cV) \times \LcoMod^D(\cV) \to \LcoMod^{C\otimes D} (\cV)$ preserves small colimits separately in both variables. 
\end{enumerate}
Item 1 follows from \cite[Prop.~3.8]{2410.21537}  since the adjunction~\eqref{eq:cofree} is comonadic and the associated comonad $C\otimes -: \cV \to \cV$ preserves small colimits. Since $\mathrm{Forget}:\cat{LcoMod}^C(\cV) \to \cV$ is a conservative left adjoint, it creates colimits\footnote{\label{ftn:create}A functor $\cA \to \cB$  is said to \emph{create colimits} if it preserves colimits and if  every cocone in $\cA$ whose image in $\cB$ is colimiting is already a colimit cocone in $\cA$. Any conservative functor which preserves colimits also creates them.}.  Since it moreover agrees with corestriction along the terminal coalgebra homomorphism $C\to 1_{\cV}$, it follows that corestriction along any coalgebra homomorphism preserves small colimits, proving item 3. Item 2 follows from item 4 for $D= 1_{\cV}$. Item 4 follows analogously to item 2. 

To see that $\mathrm{Forget}$ is in fact an internal left adjoint, note that its right adjoint $\mathrm{coFree}$ evidently preserves the right $\cV$-action and preserves colimits since $\mathrm{Forget} \circ \mathrm{coFree} = C\otimes -$ does and $\mathrm{Forget}$ creates colimits. 
\end{proof}

\begin{cor}\label{cor:Tannaka2}
If $H \in \BiAlg(\cV)$ is a bialgebra, then $\cA:=\cat{LcoMod}^H(\cV)$ inherits a presentably $\cV$-linear monoidal structure, i.e.\ $\cA \in \Alg(\PrV)$. Moreover, the forgetful functor $\mathrm{Forget}: \cat{LcoMod}^H(\cV) \to \cV$ is a morphism in $\Alg(\PrV)$ and hence gives rise, via Lemma~\ref{lem:easyretract},  to a retract in $\Mor(\PPrV)$. If (and only if) $H$ is a Hopf algebra, this retract satisfies (in the notation of Lemma~\ref{lem:easyretract}): 
\begin{itemize}
\item $f$ (and $g$) is a left adjoint, and $g$ is a right adjoint; 
\item $f^R \To g$ is a left adjoint. 
\end{itemize}
Hence, by Theorem~\ref{thm:mainHopftheorem} and Remark~\ref{rem:oppositevariants}, this gives rise to a Hopf algebra in $\Omega^2 \Mor(\PPrV) \simeq \cV$, namely $H$. 
\end{cor}
\begin{proof}
By Proposition~\ref{prop:comodules}, the functor $\LcoMod^{-}(\cV): \coAlg(\cV) \to \PrV$ is lax monoidal and hence induces a functor \[\LcoMod^{-}(\cV): \BiAlg(\cV) = \Alg(\coAlg(\cV)) \to \Alg(\PrV).\] The unique bialgebra homorphism $H \to 1_{\cV}$ is sent to $\mathrm{Forget}: \LcoMod^H(\cV) \to \cV$, thus lifting the latter to a morphism in $\Alg(\PrV)$. We will now show that $f^R \Rightarrow g$ is a left adjoint, i.e.\ that $\mathrm{Forget}$ is an internal  left adjoint in $\RMod_{\LcoMod^H(\cV)}(\PPrV)$. By Proposition~\ref{prop:comodules}, it is an internal left adjoint in $\PPrV$. Thus, it suffices by Proposition~\ref{prop:doctrinaladjunction} to show that its right adjoint commutes with the right $\LcoMod^H(\cV)$-action,  i.e.\ that for every $v\in \cV$ and $m \in \LcoMod^H(\cV)$, the canonical morphism in $\LcoMod^H(\cV)$
\[\mathrm{coFree}(v) \otimes m \to \mathrm{coFree}(v\otimes \mathrm{Forget}(m))
\]
is an isomorphism. Since both functors are morphisms in $\PrV$, it suffices to set $v= 1_{\cV}$ and since $\cat{LcoMod}^H(\cV)$ is cogenerated by the regular comodule $H$, it suffices to set $m = H$. Then $\mathrm{coFree}(1_\cV) \otimes H$ becomes $H \otimes H$ made into an $H$-comodule via the coaction that coacts via comultiplication on both copies of $H$, whereas $\mathrm{coFree}(1_{\cV} \otimes \mathrm{Forget}(H))$ becomes $H \otimes H$ made into an $H$-comodule via the coaction that coacts via comultiplication on the left copy of $H$ but that coacts trivially on the right copy of $H$. The canonical map between these $H$-comodules unpacks to the shear map denoted $\sh^\nwarrow$ in Notation~\ref{not:shearmaps}. By Lemma~\ref{lem:reversingashearmap}, this is invertible if and only if $H$ is Hopf.
  
 Since $F:\cA_{\cA} \to \cV_{\cA}$ is internal left adjoint, it preserves $\cA$-atomics and hence $F(1_{\cA}) \simeq 1_{\cV}$ is $\cA$-atomic. But $1_{\cV}$ generates $\cV$ under colimits and $\cV$-tensoring and hence (since $\cV$ is a retract of $\cA$) also under colimits and $\cA$-tensoring\footnote{In general,  $\cM_\cA \in \RMod_\cA(\Pr)$ is generated by a single $\cA$-atomic object $M$ if and only if 
  the canonical map $\LMod_{\End(M)}(\cA) \to \cM$ sending the regular module to $M$ is an equivalence in $\RMod_\cA(\Pr)$ (see e.g.~\cite[Cor.~5.13]{RZ}). Taking $(\cM,M) = (\cV, 1_\cV)$, we thus recover the so-called \define{fundamental theorem of Hopf modules} \cite[Thm. 3.6]{1202.3613} asserting that $\LMod_H\cat{LcoMod}^H(\cV) \to \cV$ is an equivalence iff $H$ is Hopf.}. Thus, it follows from Lemma~\ref{lem:tinygeneration} that $k = {}_{\cV}\cV_{\cA}$ is a right adjoint bimodule. 
  
Unpacking our construction (or rather its $1$-op variant from Remark~\ref{rem:oppositevariants}) indeed recovers $H$.
\end{proof}
\begin{remark} We note that the conditions from Proposition~\ref{prop:Tannaka1} and Corollary~\ref{cor:Tannaka2} differ: the former corresponds to $L_3(\Adj \owedge \Adj)^{\text{1-op, 2-op}}$ and gives rise to a coHopf algebra in $\cV$ while the latter corresponds to $L_3(\Adj \owedge \Adj)^{\text{1-op}}$ and gives rise to a Hopf algebra in $\cV$; compare Remark~\ref{rem:oppositevariants}. And indeed, already when $\cV = \Vect$ and if $H$ is not left-semiperfect\footnote{A coalgebra $H$ in $\Vec$ is called \define{left-semiperfect} if $ \cat{LcoMod}^H(\Vec)$ has enough projectives. If $G$ is an affine algebraic group, then $\cO(G)$ is left-semiperfect if and only if $G$ is virtually reductive, i.e.\ a finite-index extension of a reductive group. For example, the Hopf algebra $k[x] = \cO(\bG_a)$ with comultiplication $\Delta x = x \otimes 1 + 1 \otimes x$ is not left-semiperfect.}, then $\LcoMod^H(\cV)$ is not generated by $\cV$-atomic objects and Theorem~1.3 of~\cite{MR3361309} implies that moreover $f = {_{\LcoMod^H(\cV)} \LcoMod^H(\cV)_\cV}$ is not a right adjoint. Thus, the retract from Corollary~\ref{cor:Tannaka2} does not satisfy the conditions from Proposition~\ref{prop:Tannaka1}.
\end{remark}

Corollary~\ref{cor:Tannaka2} shows that our construction can indeed recover \emph{any} Hopf algebra:

\begin{corollary}\label{cor:anyHopf}
Let $\cV \in \Alg_{\EE_2}(\Pr)$ be a presentably braided monoidal $(\infty,1)$-category and $H$ be a Hopf algebra in $\cV$. Then, $H$ arises from our construction applied to the $(\infty,3)$-category $\Mor(\PPrV)$ and the retract from Corollary~\ref{cor:Tannaka2}. \qed
\end{corollary}

\bibliographystyle{initalpha}
\bibliography{Bibliography}

\end{document}

%% file: preamble.tex
\usepackage{amsmath}       
\usepackage{amsthm}        
\usepackage{amssymb}       
\usepackage{amsfonts}
\usepackage{yhmath}
\usepackage[all]{xy}
\usepackage{xspace}
\usepackage{calc}
\usepackage{pgfplots}
\usepgflibrary{fpu}
\usepackage{mathtools}
\usepackage{tabularx}
\usepackage{ifthen}		
\usepackage{graphicx}
\usepackage{docmute}
\usepackage{array}
\usepackage{subfloat} 
\usepackage{bbm}            
\usepackage{etoolbox}  
\usepackage{verbatim}
\usepackage{stmaryrd}
\usepackage{thm-restate}
\usepackage{multicol}
\usepackage{nccmath}
\usepackage{stackrel}
\usepackage[T1]{fontenc}

\usepackage[margin=1.5in]{geometry}
\renewcommand{\to}[1][]{\ensuremath{\xrightarrow{#1}}}
\newcommand{\To}[1][]{\ensuremath{\xRightarrow{#1}}}

\allowdisplaybreaks[1]

\theoremstyle{plain} 
\newtheorem{maintheorem}{Theorem}

\newtheorem{theorem}{Theorem}[section]
\newtheorem{lemma}[theorem]{Lemma}
\newtheorem{corollary}[theorem]{Corollary}          
\newtheorem{proposition}[theorem]{Proposition}              
\newtheorem{cor}[theorem]{Corollary}          
\newtheorem{prop}[theorem]{Proposition}

\theoremstyle{definition} 
\newtheorem{definition}[theorem]{Definition}

\theoremstyle{remark}  

\newtheorem{remark}[theorem]{Remark}
\newtheorem{example}[theorem]{Example}
\newtheorem{construction}[theorem]{Construction}
\newtheorem{warn}[theorem]{{Warning}}
\newtheorem{question}[theorem]{{Question}}
\newtheorem*{guide*}{Guide}
\newtheorem*{outline*}{Outline}
\newtheorem*{remarkohc*}{Remark on higher categories and the cobordism hypothesis}
\newtheorem*{assumpfield*}{Assumptions on the base field}
\newtheorem{notation}[theorem]{Notation}

\newcommand\subsektion[1]{\vspace{6pt}\noindent\textbf{#1.}}

\usepackage{bbm}

\DeclareMathOperator\Mod{\cat{Mod}}
\DeclareMathOperator\RMod{\cat{RMod}}
\DeclareMathOperator\LMod{\cat{LMod}}

\newcommand\rmate{\curvearrowright}
\newcommand\lmate{\curvearrowleft}

\newcommand\longto\longrightarrow

\newcommand\Isom{\overset{\!\sim}\Rightarrow}
\newcommand\Iisom{\mathrel{\raisebox{1.5ex}{\ensuremath{\scriptstyle\sim}}\hspace{-1.5ex}{\ensuremath\Rrightarrow}}}

\DeclareMathOperator{\Hom}{Hom}
\DeclareMathOperator{\End}{End}

\DeclareMathOperator{\Fun}{Fun}
\newcommand{\id}{\mathrm{id}}

\newcommand{\Aut}{\mathrm{Aut}}

\newcommand{\Map}{\mathrm{Map}}

\renewcommand{\Vec}{\cat{Vec}}
\newcommand{\Vect}{\Vec}

\newcommand{\Alg}{\cat{Alg}}
\newcommand{\coAlg}{\cat{coAlg}}
\newcommand{\BiAlg}{\mathrm{BiAlg}}
\newcommand{\Hopf}{\cat{Hopf}}
\newcommand{\coHopf}{\cat{Hopf}}

\newcommand{\Cat}{\cat{Cat}}
\newcommand{\Spaces}{\cat{Spaces}}
\newcommand{\Set}{\cat{Set}}
\newcommand\strCat{\cat{StrCat}}

\newcommand{\PrL}{\cat{Pr}}
\renewcommand\Pr{\cat{Pr}}
\newcommand\PPr{\mathbb{P}\mathrm{r}}
\renewcommand\PPr{\Pr}

\newcommand{\ev}{\mathrm{ev}}
\newcommand{\coev}{\mathrm{coev}}

\newcommand{\op}{\mathrm{op}}
\newcommand{\co}{\mathrm{co}}
\newcommand\mop{\mathrm{mp}}

\renewcommand\Vec{\cat{Vec}}

\newcommand\br{\mathrm{braid}}

\newcommand\rev{\mathrm{rev}}

\newcommand\Adj{\mathrm{Adj}}

\newcommand\Mnd{\mathrm{Mnd}}
\newcommand\Funlax{\mathrm{Fun}^{\mathrm{lax}}}
\newcommand\Funoplax{\mathrm{Fun}^{\mathrm{oplax}}}
\newcommand\Funstrong{\mathrm{Fun}^{\mathrm{strong}}}
\newcommand\globe{\theta}
\newcommand\Gaunt{\cat{Gaunt}}
\newcommand\Arr[1]{\mathrm{Arr}(#1)}
\newcommand\Sone{\vec{S}^1}
\DeclareMathOperator\colim{colim}
\newcommand\eHom{\underline{\mathrm{Hom}}}
\newcommand\Mon{\cat{Mon}}
\newcommand\coMon{\cat{coMon}}
\newcommand\coCat{\cat{coCat}}
\DeclareMathOperator\ob{ob}

\DeclareMathOperator\Gau{Gau}

\DeclareMathOperator\maps{Map}

\newcommand{\rrmate}{\text{\rotatebox[origin=c]{90}{$\circlearrowright$}}}
\newcommand{\llmate}{\text{\reflectbox{\rotatebox[origin=c]{90}{$\circlearrowright$}}}}

\newcommand\Ret{\cat{Ret}}
\newcommand\RetWalking{\mathrm{Ret}}
\renewcommand\BiAlg{\cat{BiAlg}}
\renewcommand\Hopf{\cat{Hopf}}
\renewcommand\coHopf{\cat{coHopf}}

\newcommand\someadj{\mathrm{adj}}

\newcommand\strict{\mathrm{str}}

\newcommand\mono\hookrightarrow

\DeclareMathOperator\sh{sh}

\newcommand\oplax{\mathrm{oplax}}

\newcommand\Rad{\mathrm{Rad}}

\newcommand{\PSh}{\operatorname{PSh}}
\newcommand{\PShV}{\operatorname{PSh}_{\cV}}
\newcommand\LcoMod{\cat{LcoMod}}

\newcommand{\PrV}{\cat{Pr}_{\cV}}
\newcommand{\PPrV}{\mathbb{P}\mathrm{r}_{\cV}}
\newcommand{\tPPrV}{2\mathbb{P}\mathrm{r}_{\cV}}
\newcommand{\Mor}{\cat{Mor}}
\renewcommand\PPrV{\Pr_\cV}
\renewcommand\tPPrV{\cat{2Pr}_\cV}

\newcommand\vop{\textnormal{v-op}}
\newcommand\hop{\textnormal{h-op}}

\newcommand\fancysquare{\raisebox{-0.5pt}{\large\ensuremath{\square\!\!\!\!\!\square}}} 

\def\cA{\mathcal A}\def\cB{\mathcal B}\def\cC{\mathcal C}\def\cD{\mathcal D}
\def\cE{\mathcal E}

\def\cM{\mathcal M}\def\cN{\mathcal N}\def\cO{\mathcal O}

\def\cV{\mathcal V}\def\cX{\mathcal X}
\def\cY{\mathcal Y}

\def\EE{\mathbb E}

\def\OO{\mathbb O}

\newcommand\bG{\mathbb G}

\newcommand\bN{\mathbb N}

\newcommand\bZ{\mathbb Z}

\newcommand\rB{\mathrm B}

\usepackage[mathscr]{euscript}

\usepackage[pdftex, colorlinks=true, linkcolor=black, citecolor=black, urlcolor=black, hypertexnames=false,
pdfauthor={\authorpaper},
pdftitle={\titlepaper}
	]{hyperref}

\newcommand\define[1]{\emph{#1}}
\newcommand\cat[1]{\mathbf{#1}}

\newcommand\ignore[1]{}

\newcommand\arXiv[1]{\href{http://arxiv.org/abs/#1}{\nolinkurl{arXiv:#1}}}
\newcommand\MRnumber[2]{}
\newcommand\DOI[1]{\href{http://dx.doi.org/#1}{\nolinkurl{DOI:#1}}}
\newcommand\MAILTO[1]{\href{mailto:#1}{\nolinkurl{#1}}}

\usepackage{tikz}
\usetikzlibrary{matrix}
\usetikzlibrary{cd}
\usetikzlibrary{knots}
\usetikzlibrary{arrows}
\usetikzlibrary{arrows.meta}
\usetikzlibrary{decorations.pathreplacing}
\usetikzlibrary{decorations.markings}
\usetikzlibrary{arrows}
\usetikzlibrary{calc}
\usetikzlibrary{shapes.misc}
\usetikzlibrary{fit}
\usepgflibrary{decorations.pathmorphing}
\usepgflibrary{shapes.geometric}
\usetikzlibrary{arrows.meta}

\tikzset{arrow data/.style 2 args={
      decoration={
         markings,
         mark=at position #1 with {\arrow[scale=0.7]{#2}}}, 
         postaction=decorate}
}

\pgfdeclarelayer{back}
\pgfdeclarelayer{backback}
\pgfdeclarelayer{front}
\pgfsetlayers{backback,back,main,front}

\makeatletter
\pgfkeys{%
  /tikz/on layer/.code={
    \pgfonlayer{#1}\begingroup
    \aftergroup\endpgfonlayer
    \aftergroup\endgroup
  },
  /tikz/node on layer/.code={
     \gdef\node@@on@layer{%
      \setbox\tikz@tempbox=\hbox\bgroup\pgfonlayer{#1}\unhbox\tikz@tempbox\endpgfonlayer\egroup}
    \aftergroup\node@on@layer
  },
  /tikz/end node on layer/.code={
    \endpgfonlayer\endgroup\endgroup
  }
}
\def\node@on@layer{\aftergroup\node@@on@layer}

\tikzset{label/.style = {scale=0.5}}
\tikzset{braid/.style={ preaction={draw=white, line width =3pt, on layer=#1}, draw =black, line width=0.7pt}}
\tikzset{braid/.default=main}
\tikzset{box/.style={line width =0.3pt}}
\tikzset{string/.style={draw=black, line width=0.7pt}}
\tikzset{dinat/.style={draw=red, line width=0.7pt}}
\tikzset{hb/.style={draw=blue,  line width=0.7pt}}

 \def\dotscale{0.3}
\tikzset{dot/.style n args={1}{draw,fill=black,circle,line width=0.7pt,scale=\dotscale}}
\tikzset{odot/.style n args={1}{draw,circle,fill=white,line width=0.7pt,scale=\dotscale}}
\tikzset{smalldot/.style ={node on layer=front, fill = black, rectangle, line width =0.7pt, scale=1.5*\dotscale}}
\tikzset{smalldotgreen/.style ={node on layer=front, fill = black!20!green,, circle, line width =0.7pt, scale=0.75*\dotscale}}

\tikzset{std/.style = {string, scale=0.55}}
\tikzset{box/.style={rectangle, draw, fill=white, minimum width =#1, label}}

\def\shift{0.075cm}
\tikzset{upshift/.style={yshift=\shift}}
\tikzset{downshift/.style={yshift=-\shift}}

\tikzset{curve/.style={settings={#1},to path={(\tikztostart)
    .. controls ($(\tikztostart)!\pv{pos}!(\tikztotarget)!\pv{height}!270:(\tikztotarget)$)
    and ($(\tikztostart)!1-\pv{pos}!(\tikztotarget)!\pv{height}!270:(\tikztotarget)$)
    .. (\tikztotarget)\tikztonodes}},
    settings/.code={\tikzset{quiver/.cd,#1}
        \def\pv##1{\pgfkeysvalueof{/tikz/quiver/##1}}},
    quiver/.cd,pos/.initial=0.35,height/.initial=0}

\tikzset{between/.style n args={2}{/tikz/spath/at end path construction={
    \tikzset{spath/split at keep middle={current}{#1}{#2}}
}}}

\tikzset{tail reversed/.code={\pgfsetarrowsstart{tikzcd to}}}
\tikzset{2tail/.code={\pgfsetarrowsstart{Implies[reversed]}}}
\tikzset{2tail reversed/.code={\pgfsetarrowsstart{Implies}}}
\tikzset{no body/.style={/tikz/dash pattern=on 0 off 1mm}}

\usetikzlibrary{3d}
\tikzset{xzplane/.style={canvas is xz plane at y=#1}}
\tikzset{td/.style={
					y={(0.6cm, 0.4cm)},
					x={(1cm, 0cm)},					
                    z  = {(0cm,1cm)},
                    scale = 0.55}}
\tikzset{slice/.style={draw = gray!50, line width = 0.7pt}}              
  \tikzset{obj/.style={scale=0.6, color=gray}}

\tikzset{21braid/.style = {draw=black, line width =0.3pt}}

  \tikzset{surface/.style={ fill=gray!60, opacity=0.8}}
\tikzset{hb3d/.style={smalldot}}
\tikzset{1mor/.style= {string}}

\def\maskpointheight{5pt}
\tikzset{mask point/.style={ transform shape, sloped, 
minimum width=15pt, minimum height=\maskpointheight, inner sep=0pt, ultra thin, font=\tiny, node on layer=#1}}
\tikzset{mask point/.default=main}

\tikzset{
clip even odd rule/.code={\pgfseteorule},
invclip/.style={clip,insert path=[clip even odd rule]{
   [reset cm](-\maxdimen,-\maxdimen)rectangle(\maxdimen,\maxdimen)
    }}}

\def\circclip{2pt}
\tikzset{circ mask point/.style={ circle,
minimum width=\circclip, inner sep=0pt, ultra thin, font=\tiny, node on layer=#1}}
\tikzset{circ mask point/.default=main}

\def\boxmaskpointheight{6pt}
\tikzset{box mask point/.style={ transform shape, sloped, 
minimum width=3.5pt, minimum height=\boxmaskpointheight, inner sep=0pt, ultra thin, font=\tiny}}

\usetikzlibrary{matrix,arrows,decorations.pathmorphing,decorations.pathreplacing,decorations.markings}
\tikzset{
    dot/.style={circle,draw,fill,inner sep=1pt},
    arrow/.style={->,thick,shorten <=2pt,shorten >=2pt},
    twoarrow/.style={double,double distance=1.5pt,shorten <=9pt,shorten >=10pt,decoration={markings,mark=at position -8pt with {\arrow[scale=1.75]{>}}},preaction={decorate}},
    twoarrowlonger/.style={double,double distance=1.5pt,shorten <=5pt,shorten >=6pt,decoration={markings,mark=at position -4pt with {\arrow[scale=1.75]{>}}},preaction={decorate}},
    twoarrowshorter/.style={double,double distance=1.5pt,shorten <=13pt,shorten >=14pt,decoration={markings,mark=at position -12pt with {\arrow[scale=1.75]{>}}},preaction={decorate}},
    twoarrowshorthead/.style={double,double distance=1.5pt,shorten <=9pt,shorten >=20pt,decoration={markings,mark=at position -18pt with {\arrow[scale=1.75]{>}}},preaction={decorate}},
    threearrowpart1/.style={ thick,double,double distance=3pt,shorten <=9pt,shorten >=11pt},
    threearrowpart2/.style={ thick,shorten <=9pt,shorten >=10pt},
    threearrowpart3/.style={ shorten <=9pt,shorten >=10pt,decoration={markings,mark=at position -8pt with {\arrow[scale=3]{>}}},preaction={decorate}},
fourarrowpart1/.style={thick, double,double distance=4pt,shorten <=1pt,shorten >=2.75pt},
fourarrowpart2/.style={thick,  double,double distance=1pt,shorten <=1pt,shorten >=1.25pt,decoration={markings,mark=at position -.05pt with {\arrow[scale=3,ultra thin]{>}}},preaction={decorate} }
}

\newcommand\fourarrow{%
\begin{tikzpicture}
  \path coordinate (l) ++(3ex,0) coordinate (r);
  \draw[double,double distance=4pt,shorten <=1pt,shorten >=2.75pt] (l) -- (r);
  \draw[double,double distance=1pt,shorten <=1pt,shorten >=1.25pt,decoration={markings,mark=at position -.05pt with {\arrow[scale=3,ultra thin]{>}}},preaction={decorate}] (l) -- (r);
\end{tikzpicture}%
}

\newcommand\threearrowlong{%
\begin{tikzpicture}
  \path coordinate (l) ++(5ex,0) coordinate (r);
  \draw[,double,double distance=3pt,shorten <=1pt,shorten >=2.5pt] (l) -- (r);
  \draw[shorten <=1pt,shorten >=.5pt,decoration={markings,mark=at position -.05pt with {\arrow[scale=3,ultra thin]{>}}},preaction={decorate}] (l) -- (r);
\end{tikzpicture}%
}

\newcommand\fourarrowlong{%
\begin{tikzpicture}
  \path coordinate (l) ++(5ex,0) coordinate (r);
  \draw[double,double distance=4pt,shorten <=1pt,shorten >=2.75pt] (l) -- (r);
  \draw[double,double distance=1pt,shorten <=1pt,shorten >=1.25pt,decoration={markings,mark=at position -.05pt with {\arrow[scale=3,ultra thin]{>}}},preaction={decorate}] (l) -- (r);
\end{tikzpicture}%
}

\tikzset{wiggly/.style={decorate, decoration={snake, segment length=5pt, amplitude=1pt}}}

\tikzset{Rightarrow/.style={double equal sign distance,>={Implies},->},
triple/.style={-,preaction={draw,Rightarrow}},
quadruple/.style={preaction={draw,Rightarrow,shorten >=0pt},shorten >=1pt,-,double,double
distance=0.2pt}}

\tikzset{
    graydot/.style={gray!50!white,circle,draw,fill,inner sep=1pt},
    grayarrow/.style={gray!50!white,->,thick,shorten <=2pt,shorten >=2pt},
    graytwoarrow/.style={gray!50!white,double,double distance=1.5pt,shorten <=9pt,shorten >=10pt,decoration={markings,mark=at position -8pt with {\arrow[scale=1.75,gray!50!white]{>}}},preaction={decorate}},
}

\usetikzlibrary{intersections}

